\documentclass[
11pt,  reqno]{amsart}
\usepackage[margin=1.2in,marginparwidth=1.5cm, marginparsep=0.5cm]{geometry}
\pdfoutput=1

\usepackage{booktabs} 
\usepackage{microtype}
\usepackage{amssymb}
\usepackage{mathrsfs}

\usepackage{upgreek} 

\usepackage{color}
\usepackage[implicit=true]{hyperref}

\usepackage{xsavebox}

\usepackage{cases}

\allowdisplaybreaks[2]

\definecolor{gr}{rgb}   {0.,   0.69,   0.23 }
\definecolor{bl}{rgb}   {0.,   0.5,   1. }
\definecolor{mg}{rgb}   {0.85,  0.,    0.85}
\definecolor{yl}{rgb}   {0.8,  0.7,   0.}
\definecolor{or}{rgb}  {0.7,0.2,0.2}

\usepackage{tikz} 
\usepackage{marginnote}
\usepackage{scalerel} 

\usetikzlibrary{shapes.misc}
\usetikzlibrary{shapes.symbols}
\usetikzlibrary{decorations}
\usetikzlibrary{decorations.markings}

\tikzset{
	dot/.style={circle,fill=black,draw=black,inner sep=0pt,minimum size=0.5mm},
	>=stealth,
	}

\tikzset{
	dot2/.style={circle,fill=black,draw=black,inner sep=0pt,minimum size=0.2mm},
	>=stealth,
	}

\tikzset{
	ddot/.style={circle,fill=white,draw=black,inner sep=0pt,minimum size=0.8mm},
	>=stealth,
	}

\tikzset{decision/.style={ 
        draw,
        diamond,
        aspect=1.5
    }}

\tikzset{dia2/.style
={diamond,fill=white,draw=black,inner sep=0pt,minimum size=1mm},
	>=stealth,
	}

\tikzset{dia/.style
={star,fill=black,draw=black,inner sep=0pt,minimum size=1mm},
	>=stealth,
	}

\tikzset{dia/.style
={diamond,fill=black,draw=black,inner sep=0pt,minimum size=1.3mm},
	>=stealth,
	}

\makeatletter
\def\DeclareSymbol#1#2#3{\xsavebox{#1}{\tikz[baseline=#2,scale=0.15]{#3}}}
\def\<#1>{\xusebox{#1}}
\makeatother

\newsavebox{\peA}
\newsavebox{\pneA}
\newsavebox{\plA}
\newsavebox{\pgA}
\newsavebox{\pleA}
\newsavebox{\pgeA}
\newsavebox{\pezA}

\savebox{\peA}{\tikz \draw (0,0) node[shape=circle,draw,inner sep=0pt,minimum size=8.5pt] {\scriptsize  $=$};}
\savebox{\pneA}{\tikz \draw (0,0) node[shape=circle,draw,inner sep=0pt,minimum size=8.5pt] {\footnotesize $\neq$};}
\savebox{\plA}{\tikz \draw (0,0) node[shape=circle,draw,inner sep=0pt,minimum size=8.5pt] {\scriptsize $<$};}
\savebox{\pgA}{\tikz \draw (0,0) node[shape=circle,draw,inner sep=0pt,minimum size=8.5pt] {\scriptsize $>$};}
\savebox{\pleA}{\tikz \draw (0,0) node[shape=circle,draw,inner sep=0pt,minimum size=8.5pt] {\scriptsize $\leqslant$};}
\savebox{\pgeA}{\tikz \draw (0,0) node[shape=circle,draw,inner sep=0pt,minimum size=8.5pt] {\scriptsize $\geqslant$};}
\savebox{\pezA}{\tikz \draw (0,0) node[shape=circle,draw,
fill=white, 
inner sep=0pt,minimum size=8.5pt]{} ;}

\def \peB{\mathchoice
{\scalebox{.7}{{\usebox{\peA}}}}
{\scalebox{.7}{{\usebox{\peA}}}}
{\scalebox{.7}{{\usebox{\peA}}}}
{}
}

\def \plB{\mathchoice
{\scalebox{.7}{{\usebox{\plA}}}}
{\scalebox{.7}{{\usebox{\plA}}}}
{\scalebox{.7}{{\usebox{\plA}}}}
{}
}
\def \pgB{\mathchoice
{\scalebox{.7}{{\usebox{\pgA}}}}
{\scalebox{.7}{{\usebox{\pgA}}}}
{\scalebox{.7}{{\usebox{\pgA}}}}
{}
}

\def \pgeB{\mathchoice
{\scalebox{.7}{{\usebox{\pgeA}}}}
{\scalebox{.7}{{\usebox{\pgeA}}}}
{\scalebox{.7}{{\usebox{\pgeA}}}}
{}
}
\def \pezB{\mathchoice
{\scalebox{.7}{{\usebox{\pezA}}}}
{\scalebox{.7}{{\usebox{\pezA}}}}
{\scalebox{.7}{{\usebox{\pezA}}}}
{}
}

\newcommand{\pe}{\mathbin{{\peB}}}

\newcommand{\pl}{\mathbin{{\plB}}}
\newcommand{\pg}{\mathbin{{\pgB}}}

\newcommand{\pge}{\mathbin{{\pgeB}}}
\newcommand{\pez}{\mathbin{{\pezB}}}

\usetikzlibrary{decorations.pathmorphing, patterns,shapes}

\DeclareSymbol{dot}{-2.7}{\draw (0,0) node[dot]{};}

\DeclareSymbol{1}{0}{\draw[white] (-.4,0) -- (.4,0); \draw (0,0)  -- (0,1.2) node[dot] {};}
\DeclareSymbol{2}{0}{\draw (-0.5,1.2) node[dot] {} -- (0,0) -- (0.5,1.2) node[dot] {};}
\DeclareSymbol{3}{0}{\draw (0,0) -- (0,1.2) node[dot] {}; \draw (-.7,1) node[dot] {} -- (0,0) -- (.7,1) node[dot] {};}
\DeclareSymbol{30}{-3}{\draw (0,0) -- (0,-1); \draw (0,0) -- (0,1.2) node[dot] {}; \draw (-.7,1) node[dot] {} -- (0,0) -- (.7,1) node[dot] {};}
\DeclareSymbol{31}{-3}{\draw (0,0) -- (0,-1) -- (1,0) node[dot] {}; \draw (0,0) -- (0,1.2) node[dot] {}; \draw (-.7,1) node[dot] {} -- (0,0) -- (.7,1) node[dot] {};}
\DeclareSymbol{32}{-3}{\draw (0,0) -- (0,-1) -- (1,0) node[dot] {}; \draw (0,0) -- (0,-1) -- (-1,0) node[dot] {}; \draw (0,0) -- (0,1.2) node[dot] {}; \draw (-.7,1) node[dot] {} -- (0,0) -- (.7,1) node[dot] {};}
\DeclareSymbol{20}{-3}{\draw (0,0) -- (0,-1);\draw (-.7,1) node[dot] {} -- (0,0) -- (.7,1) node[dot] {};}
\DeclareSymbol{22}{-3}{\draw (0,0.3) -- (0,-1) -- (1,0) node[dot] {}; \draw (0,0.3) -- (0,-1) -- (-1,0) node[dot] {};\draw (-.7,1) node[dot] {} -- (0,0.3) -- (.7,1) node[dot] {};}
\DeclareSymbol{31p}{-3}{\draw (0,0) -- (0,-1) -- (1,0) node[dot] {}; \draw (0,0) -- (0,1.2) node[dot] {}; \draw (-.7,1) node[dot] {} -- (0,0) -- (.7,1) node[dot] {}; 
\draw (0,-1) node{\scaleobj{0.5}{\pez}}; 
\draw (0,-1) node{\scaleobj{0.5}{\pe}}; }
\DeclareSymbol{32p}{-3}{\draw (0,0) -- (0,-1) -- (1,0) node[dot] {}; \draw (0,0) -- (0,-1) -- (-1,0) node[dot] {}; \draw (0,0) -- (0,1.2) node[dot] {}; \draw (-.7,1) node[dot] {} -- (0,0) -- (.7,1) node[dot] {}; 
\draw (0,-1) node{\scaleobj{0.5}{\pez}};
\draw (0,-1) node{\scaleobj{0.5}{\pe}};}
\DeclareSymbol{22p}{-3}{\draw (0,0.3) -- (0,-1) -- (1,0) node[dot] {}; \draw (0,0.3) -- (0,-1) -- (-1,0) node[dot] {};\draw (-.7,1) node[dot] {} -- (0,0.3) -- (.7,1) node[dot] {}; 
\draw (0,-1) node{\scaleobj{0.5}{\pez}};
\draw (0,-1) node{\scaleobj{0.5}{\pe}};}
\DeclareSymbol{21p}{-3}{\draw (0,0.3) -- (0,-1) -- (1,0) node[dot] {}; 
\draw (-.7,1) node[dot] {} -- (0,0.3) -- (.7,1) node[dot] {}; 
\draw (0,-1) node{\scaleobj{0.5}{\pez}};
\draw (0,-1) node{\scaleobj{0.5}{\pe}};}

\DeclareSymbol{11p}{-3}{\draw (0,0) node[dot2] {} -- (0,-1) -- (1,0) node[dot] {}; 
\draw (0,0) -- (0,1.2) node[dot] {};  
\draw (0,-1) node{\scaleobj{0.5}{\pez}}; 
\draw (0,-1) node{\scaleobj{0.5}{\pe}}; }
\DeclareSymbol{100}{-3}
{\draw[white] (-.4,0) -- (.4,0); 
\draw (0,0) node[dot2] {} -- (0,-1)  ; 
\draw (0,0) -- (0,1.2) node[dot] {};}  
\usetikzlibrary{mindmap,backgrounds,trees,arrows,decorations.pathmorphing,decorations.pathreplacing,positioning,shapes.geometric,shapes.misc,decorations.markings,decorations.fractals,calc,patterns}

\tikzset{>=stealth',
         cvertex/.style={circle,draw=black,inner sep=1pt,outer sep=3pt},
         vertex/.style={circle,fill=black,inner sep=1pt,outer sep=3pt},
         star/.style={circle,fill=yellow,inner sep=0.75pt,outer sep=0.75pt},
         tvertex/.style={inner sep=1pt,font=\scriptsize},
         gap/.style={inner sep=0.5pt,fill=white}}
\tikzstyle{mybox} = [draw=black, fill=blue!10, very thick,
    rectangle, rounded corners, inner sep=10pt, inner ysep=20pt]
\tikzstyle{boxtitle} =[fill=blue!50, text=white,rectangle,rounded corners]

\tikzstyle{decision} = [diamond, draw, fill=blue!20,
    text width=4.5em, text badly centered, node distance=2.5cm, inner sep=0pt]
\tikzstyle{block} = [rectangle, draw, fill=blue!20,
    text width=2.7cm, text centered, rounded corners, minimum height=3em]

\tikzstyle{line} = [draw, very thick, color=black!50, -latex']

\tikzstyle{cloud} = [draw, ellipse,fill=red!40, 
node distance=2.5cm,
    minimum height=1em]

\tikzstyle{cloud2} = [draw, ellipse,fill=red!30, text=white,text width=10em, node distance=2.5cm, text centered, minimum height=4em]

\tikzstyle{cloud3} = [draw, ellipse, fill=cyan!30, 
node distance=2.5cm,
    minimum height=3em, 
    text width=1.7cm]

\tikzstyle{cloud4} = [draw, ellipse,fill=orange!70, node distance=2.5cm,
   text width=2.1cm,  
           minimum height=3em]

\tikzstyle{cloud5} = [draw, ellipse,fill=red!20, node distance=2.5cm,
    text centered, 
    text width=3.0cm, 
    minimum height=3em]

\tikzstyle{cloud6} = [draw, ellipse,fill=red!20, node distance=2.5cm,
    text width=1.5cm, 
    text centered, 
    minimum height=3em]

\tikzset{
    position/.style args={#1:#2 from #3}{
        at=(#3.#1), anchor=#1+180, shift=(#1:#2)
    }
}

\newtheorem{theorem}{Theorem} [section]

\newtheorem{lemma}[theorem]{Lemma}
\newtheorem{proposition}[theorem]{Proposition}
\newtheorem{remark}[theorem]{Remark}

\newtheorem{corollary}[theorem]{Corollary}

\DeclareMathOperator*{\supp}{supp}

\newcommand{\1}{\hspace{0.2mm}\text{I}\hspace{0.2mm}}
\newcommand{\II}{\text{I \hspace{-2.8mm} I} }
\newcommand{\III}{\text{I \hspace{-2.9mm} I \hspace{-2.9mm} I}}
\newcommand{\noi}{\noindent}
\newcommand{\Z}{\mathbb{Z}}
\newcommand{\R}{\mathbb{R}}
\newcommand{\T}{\mathbb{T}}

\let\Re=\undefined\DeclareMathOperator*{\Re}{Re}
\let\Im=\undefined\DeclareMathOperator*{\Im}{Im}

\let\P= \undefined
\newcommand{\P}{\mathbf{P}}

\newcommand{\E}{\mathbb{E}}

\renewcommand{\L}{\mathcal{L}}

\newcommand{\K}{\mathcal{K}}

\newcommand{\F}{\mathcal{F}}

\newcommand{\al}{\alpha}
\newcommand{\be}{\beta}
\newcommand{\dl}{\delta}

\newcommand{\nb}{\nabla}

\newcommand{\Dl}{\Delta}
\newcommand{\eps}{\varepsilon}
\newcommand{\kk}{\kappa}
\newcommand{\g}{\gamma}

\newcommand{\ld}{\lambda}
\newcommand{\Ld}{\Lambda}
\newcommand{\s}{\sigma}
\newcommand{\Si}{\Sigma}
\newcommand{\ft}{\widehat}

\newcommand{\wt}{\widetilde}
\newcommand{\cj}{\overline}

\newcommand{\dt}{\partial_t}
\newcommand{\dd}{\partial}

\newcommand{\ud}{\underline}

\newcommand{\ta}{\theta}

\renewcommand{\l}{\ell}
\renewcommand{\o}{\omega}
\renewcommand{\O}{\Omega}

\newcommand{\les}{\lesssim}
\newcommand{\ges}{\gtrsim}

\newcommand{\jb}[1]
{\langle #1 \rangle}

\newcommand{\jbb}[1]
{[\hspace{-0.6mm}[ #1 ]\hspace{-0.6mm}]}

\renewcommand{\b}{\beta}
\newcommand{\ind}{\mathbf 1}

\newcommand{\N}{\mathbb{N}}

\renewcommand{\H}{\mathcal{H}}

\DeclareMathOperator{\Lip}{Lip}

\newtheorem*{ackno}{Acknowledgements}

\newcommand{\I}{\mathcal{I}}

\newcommand{\If}{\mathfrak{I}}

\newcommand{\RR}{\mathcal{R}}

\newcommand{\A}{\mathcal{A}}

\newcommand{\C}{\mathcal{C}}
\numberwithin{equation}{section}
\numberwithin{theorem}{section}

\newcommand{\Q}{\mathbb{Q}}
\newcommand{\PP}{\mathbb{P}}

\DeclareMathOperator{\Law}{Law}

\newcommand{\ZZ}{\mathfrak{Z}}

\newcommand{\Zc}{\mathcal{Z}}

\newcommand{\muu}{\vec{\mu}}

\newcommand{\rhoo}{\vec{\rho}}

\newcommand{\W}{\mathcal{W}}
\newcommand{\U}{\mathcal{U}}

\newcommand{\dr}{\theta}
\newcommand{\Dr}{\Theta}

\newcommand{\Ha}{\mathbb{H}_a}
\newcommand{\Hc}{\mathbb{H}_c}

\newcommand{\NN}{\mathcal{N}}
\newcommand{\D}{\mathcal{D}}

\newcommand{\Res}{\mathfrak{R}}
\newcommand{\Qxy}{Q_{X,Y}}

\newcommand{\QxyN}{Q_{X_N,Y_N}}

\newcommand{\dia}{\diamond}
\newcommand{\Y}{\mathbb{Y}}

\newcommand{\too}{\longrightarrow}

\newcommand{\proj}{\Pi}

\newcommand{\Ups}{\Upsilon}
\newcommand{\UUps}{{\ud \Upsilon}}

\newcommand{\Ab}{\mathbb{A}}

\newcommand{\tf}{\mathfrak{t}}

\usepackage{comment}

\makeatletter
\@namedef{subjclassname@2020}{%
  \textup{2020} Mathematics Subject Classification}
\makeatother

\begin{document}
\baselineskip = 14pt

\title[Focusing $\Phi^4_3$-model with a Hartree-type nonlinearity]
{Focusing $\Phi^4_3$-model with a Hartree-type nonlinearity}

\author[T.~Oh, M.~Okamoto, and L.~Tolomeo]
{Tadahiro Oh, Mamoru Okamoto, and Leonardo Tolomeo}

\address{
Tadahiro Oh, School of Mathematics\\
The University of Edinburgh\\
and The Maxwell Institute for the Mathematical Sciences\\
James Clerk Maxwell Building\\
The King's Buildings\\
Peter Guthrie Tait Road\\
Edinburgh\\ 
EH9 3FD\\
 United Kingdom,
 and
 School of Mathematics and Statistics, Beijing Institute of Technology, Beijing 100081, China}

\email{hiro.oh@ed.ac.uk}

\address{
Mamoru Okamoto\\
Department of Mathematics\\
 Graduate School of Science\\ Osaka University\\
Toyonaka\\ Osaka\\ 560-0043\\ Japan}
\email{okamoto@math.sci.osaka-u.ac.jp}

\address{
Leonardo Tolomeo\\ 
Mathematical Institute, Hausdorff Center for Mathematics, Universit\"at Bonn, Bonn, Germany}

\email{tolomeo@math.uni-bonn.de}

\subjclass[2020]{35L71, 60H15, 81T08, 60L40, 35K15}

\keywords{Hartree $\Phi^4_3$-measure;
stochastic quantization;
stochastic nonlinear wave equation; nonlinear wave equation;
Gibbs measure; 
paracontrolled calculus;
nonlinear heat equation}

\begin{abstract}
Lebowitz, Rose, and Speer (1988) initiated the study 
of focusing Gibbs measures, which was continued
by Brydges and Slade (1996), Bourgain (1997, 1999), 
and Carlen, Fr\"ohlich, and Lebowitz (2016)
among others.
In this paper, we complete the program
on the (non-)construction of the focusing Hartree Gibbs measures
in the three-dimensional setting.
More precisely, 
we  study a focusing $\Phi^4_3$-model with a Hartree-type nonlinearity, 
where the potential 
for the Hartree nonlinearity is given by the Bessel potential of order~$\be$.
We first construct 
the focusing Hartree $\Phi^4_3$-measure 
for $\beta > 2$, 
while we prove its non-normalizability for $\be < 2$.
Furthermore,  we establish the following phase transition 
at the critical value  $\be = 2$:
  normalizability in the weakly nonlinear regime
and 
 non-normalizability in the strongly nonlinear regime.
We then study 
the  canonical stochastic quantization of the focusing Hartree $\Phi^4_3$-measure, 
namely, 
 the three-dimensional stochastic damped nonlinear wave equation (SdNLW) with 
a cubic 
nonlinearity of Hartree-type, 
forced by an additive space-time white noise, 
and 
prove almost sure global well-posedness
and invariance of the focusing Hartree $\Phi^4_3$-measure
for $\be > 2$ (and  $\be = 2$ in the weakly nonlinear regime).
In view of the  non-normalizability result, 
our almost sure global well-posedness result  is sharp.
In Appendix, 
we also discuss  the (parabolic) stochastic quantization
for the focusing Hartree $\Phi^4_3$-measure. 

We also consider  the defocusing case.
By adapting our argument from the focusing case, 
we first construct the defocusing Hartree $\Phi^4_3$-measure
and the associated invariant dynamics for the defocusing Hartree SdNLW
for $\be > 1$.
By introducing further renormalizations at $\be = 1$ and $\be = \frac 12$, 
we extend the construction of 
the defocusing Hartree $\Phi^4_3$-measure
for $\be > 0$, where the resulting measure is shown to be singular with respect to the
reference Gaussian free field for $0 < \be \le  \frac 12$.

\end{abstract}

%
\maketitle
%


\tableofcontents

\section{Introduction}
\label{SEC:1}


\subsection{Focusing Hartree $\Phi^4_3$-measure
and its canonical stochastic quantization}

In this paper, we study 
the Gibbs measure 
$\rho$ 
with a Hartree-type nonlinearity
on 
the three-dimensional torus on $\T^3 = (\R/2\pi\Z)^3$, formally written as\footnote{In this introduction,
we keep our discussion at a formal level and do not worry about various renormalizations 
required 
to give a  proper meaning to various objects.}
\begin{align}
d\rho(u) = Z^{-1} \exp \bigg(\frac{\s}4 \int_{\T^3} (V*u^2)u^2 dx\bigg) d\mu(u), 
\label{H1}
\end{align}

\noi
and its associated stochastic quantization.
Here, $\mu$ is the massive Gaussian free field on~$\T^3$
(see~\eqref{gauss0} with $s =1$)
and the coupling constant $\s \in \R\setminus \{0\}$.
The associated energy functional 
for the Gibbs measure~$\rho$ in \eqref{H1} 
is given by
\begin{align}
E(u)
= \frac 12 \int_{\T^3} |\jb{\nabla} u|^2 dx 
- \frac \s4 \int_{\T^3} (V \ast u^2 ) u^2 dx.
\label{H2}
\end{align}

\noi
The main interest in this paper is
to investigate the construction 
of the Hartree Gibbs measures in the {\it focusing}
case ($\s > 0$). 
In the seminal work \cite{LRS}, 
Lebowitz, Rose, and Speer initiated
the study of  focusing Gibbs measures
in the one-dimensional setting.
In this work, they constructed
the one-dimensional focusing Gibbs measures\footnote{As pointed out by 
Carlen, Fr\"ohlich, and Lebowitz \cite{CFL}, there is in fact an error in the Gibbs measure
construction in \cite{LRS}, 
which was amended in  \cite{BO94, OST}.
In particular,   in \cite{OST}, the first and third authors
with Sosoe  completed the 
focusing Gibbs measure construction program
in the one-dimensional setting, including 
the critical case ($p = 6$) at the optimal $L^2$-threshold.
See \cite{OST} for more details 
on the (non-)construction of the focusing Gibbs measures in the one-dimensional setting.}
in the $L^2$-(sub)critical setting (i.e.~$2 < p\leq 6$)
with an $L^2$-cutoff:
\begin{align}
d\rho(u) = Z^{-1} 
\ind_{\{\int_\T u^2 dx \le K \}} \exp \bigg(\frac{1}{p} \int_{\T} |u|^p dx\bigg) d\mu(u)
\label{AX1}
\end{align}

\noi
or with a taming by the $L^2$-norm:
\begin{align}
d\rho(u) = Z^{-1} 
 \exp \bigg(\frac{1}{p} \int_{\T} |u|^p dx - A \Big(\int_\T u^2 dx\Big)^{q}\bigg) d\mu(u)
\label{AX2}
\end{align}

\noi
for some appropriate $q = q(p)$, 
where $\mu$ denotes the periodic Wiener measure on $\T$.
See Remark 2.1 in \cite{LRS}, 
Here, the parameter $A > 0$ denotes the so-called 
(generalized) chemical potential
and the expression \eqref{AX2} is referred to as the generalized grand-canonical Gibbs measure.
See also the work by 
Carlen, Fr\"ohlich, and Lebowitz \cite{CFL}
for a further discussion, where they describe the details
of the construction of the 
generalized grand-canonical Gibbs measure in \eqref{AX2}.
In the two-dimensional setting, 
Brydges and Slade \cite{BS}
continued the study on the focusing Gibbs measures
and showed that with the quartic interaction ($p = 4$), 
the focusing Gibbs measure $\rho$ in \eqref{AX1} (and hence $\rho$ in~\eqref{AX2};
see \eqref{pa00})
is not normalizable as a probability measure
(even with proper renormalization on the potential energy
$\frac{1}{4} \int_{\T^2} |u|^4 dx$ and on the $L^2$-cutoff).
See also~\cite{OS}.
We point out that with the cubic interaction ($p = 3$),
Jaffe constructed a (renormalized) $\Phi^3_2$-measure 
with a Wick-ordered $L^2$-cutoff.
See~\cite{BO95, OS}.
Following a suggestion by Lebowitz \cite{BO99}
to consider a Hartree-type nonlinearity
in order  to overcome
the difficulty of the focusing Gibbs measure construction in higher dimensions, 
Bourgain investigated the construction of the focusing Gibbs measures
with a Hartree-type nonlinearity in \eqref{H1} (with $p = 4$)
in the two- and three-dimensional setting \cite{BO97, BO99}.
In particular, by taking 
 $V$ to be 
the kernel for 
the Bessel potential 
of order $\be$:\footnote{In the following, 
we simply refer to $V$ in  \eqref{Bessel1}
as the Bessel potential of order $\be$.}
\begin{align}
 V*f = \jb{\nabla}^{-\beta} f = (1-\Dl)^{-\frac{\be}{2}} f, 
 \label{Bessel1}
\end{align}

\noi
Bourgain  constructed 
(with a proper renormalization and a Wick-ordered $L^2$-cutoff) 
the focusing Hartree Gibbs measure $\rho$ in \eqref{H1} for $\be > 2$
(in the complex-valued setting); see~\eqref{Gibbs7} below.
Furthermore, he studied the associated 
 Hartree nonlinear Schr\"odinger equation (NLS)
on~$\T^3$:
\begin{align}
i \dt u + (1- \Dl) u - \s (V*|u|^2) u = 0, 
\label{NLS1}
\end{align}

\noi
and constructed invariant Gibbs dynamics for \eqref{NLS1}
when $\be > 2$.\footnote{By combining
the construction of the focusing Hartree Gibbs measure
in the critical case ($\be = 2$) with $0 < \s \ll 1$ (Theorem \ref{THM:Gibbs0}) and 
the local well-posedness result in \cite{DNY4}, 
this result by Bourgain \cite{BO97} can be extended to the critical case $\be = 2$
(in the weakly nonlinear regime $0 < \s \ll 1$).
See also Remark \ref{REM:end1}.}
In the same paper \cite{BO97}, 
Bourgain proposed to further 
 investigate 
the (non-)normalizability issue of the focusing (Hartree) Gibbs measures
 as a continuation of 
\cite{LRS, BS, BO97}.
See also Section 5 in \cite{LRS}.
In this paper, 
we complete this program
on the (non-)construction of the focusing Hartree Gibbs measures
\eqref{H1} 
in the three-dimensional setting.
More precisely, 
in the focusing case ($\s > 0$),
\begin{itemize}
\item[(i)] We construct the focusing Hartree Gibbs measure 
for $\be \ge 2$ (with $0 < \s \ll 1$ when $\be = 2$), 

\smallskip

\item[(ii)] We prove that 
the focusing Hartree Gibbs measure is not normalizable
for $\be < 2$ or for $\be = 2$ and $\s \gg 1$.

\end{itemize}

\noi
See Theorem \ref{THM:Gibbs0}.
In particular, we establish a phase transition
in two respects:
(i)  the 
 focusing Hartree Gibbs measure
is constructible for $\be > 2$, 
while it is not for $\be < 2$
and (ii)~when $\be = 2$, 
the  focusing Hartree Gibbs measure
is constructible for $0 < \s \ll 1$, 
while it is not for $\s \gg 1$.
In this paper, 
we also construct the (canonical) stochastic quantization dynamics;
see Theorem \ref{THM:GWP0} and Remark \ref{REM:heat}.

We point out that such a Gibbs measure with a (Wick-ordered) $L^2$-cutoff 
is not suitable for stochastic quantization in the heat and wave settings
due to the lack of the $L^2$-conservation.
For this reason, 
we consider 
 the following
 generalized grand-canonical Gibbs measure
 formulation of the 
  focusing Hartree Gibbs measure 
(namely,  with a taming by the Wick-ordered $L^2$-norm):
\begin{align}
d\rho(u) = Z^{-1}
 \exp \bigg( \frac \s4 \int_{\T^3} (V*:\! u^2\!:) :\!u^2\! :\, dx - A
\bigg|\int_{\T^3} :\,u^2: \, dx\bigg|^\g\bigg) d \mu(u)
\label{H3}
\end{align}

\noi
for suitable $A, \g > 0$.

We now state our first main  result in a somewhat formal manner.
See Theorems \ref{THM:Gibbs} and~\ref{THM:Gibbs2}
in Subsection \ref{SUBSEC:Gibbs}
for the precise statements.
We also study 
 the defocusing case $(\s < 0$), 
 where
 we construct 
 the defocusing Hartree Gibbs measure $\rho$ in \eqref{H1}
  (without a cutoff or taming by the Wick-ordered $L^2$-norm)
for any $\be > 0$.

\begin{theorem}\label{THM:Gibbs0}
Given  $\be>0$, 
let $V$ be the Bessel potential of order $\b$.

\smallskip

\noi
\textup{(i)} \textup{(focusing case).}
Let $\s> 0$. Then, the following statements hold\textup{:}

\begin{itemize}
\item
Let $\be > 2$ and  $\max \big(\frac{\be+1}{\be-1},2\big) \le  \g <3$
with $\g > 2$ when $\be = 3$.
Then, 
the focusing Hartree Gibbs measure $\rho$ in~\eqref{H3}
exists as a limit of the truncated Gibbs measures, 
provided that 
 $A >0$ is sufficiently large.

\smallskip

\item
Let $1< \be < 2$.
Then, 
the focusing Hartree Gibbs measure $\rho$ in \eqref{H3}
is not normalizable \textup{(}i.e.~$Z = \infty$\textup{)}
for any $A, \g > 0$.

\smallskip

\item
\textup{(critical case)}.
Let $\be = 2$.
Then, by choosing $\g = 3$, the focusing Hartree Gibbs measure 
$\rho$ in~\eqref{H3} exists in the weakly nonlinear regime
\textup{(}$0 < \s \ll 1$\textup{)}, provided that $A = A(\s) >0$ is
 sufficiently large. 
On the other hand, 
in the strongly nonlinear regime
\textup{(}i.e.~$\s \gg1$\textup{)}, 
the focusing Hartree Gibbs measure 
$\rho$ in~\eqref{H3} is not normalizable 
for any $\g > 0$ and any $A>0$.

\end{itemize}

\noi
Furthermore, when the focusing Hartree Gibbs measure $\rho$ exists, 
it is equivalent to the base massive Gaussian free field $\mu$.

\smallskip

\noi
\textup{(ii)} \textup{(defocusing case)}.\footnote{After 
\label{FOOT1}
the completion of this paper, 
we learned that Bringmann \cite{Bring1} independently studied the construction of the
Hartree Gibbs measures in the defocusing case and obtained analogous results for $\be > 0$.
We point out some differences between \cite{Bring1} and our work in the defocusing case.
Bringmann 
proves tightness of the truncated defocusing Hartree
Gibbs measures, using the Laplace transform
as in  a recent work \cite{BG} by Barashkov and Gubinelli.
This yields  convergence of the truncated Gibbs measures up to a subsequence.
However, uniqueness of the limiting Gibbs measure is not studied  in \cite{Bring1}.
In this paper, 
we establish tightness by a more direct argument
and also prove uniqueness of the limiting Gibbs measure
(which implies convergence of the entire sequence);
see Section~\ref{SEC:def}
for  the most intricate case $0 < \beta \le \frac 12$.
In \cite{Bring1}, 
Bringmann also proves singularity of the defocusing Hartree Gibbs measure
with respect to the massive Gaussian free field $\mu$
in the  range $0 < \be <  \frac 12$.
This is done 
by first establishing singularity of the reference shifted measure
with respect to $\mu$
as in~\cite{BG2}.
In Subsection \ref{SUBSEC:def5}, we present a direct proof of singularity of the Gibbs measure
without referring to the shifted measure
for $0 < \be \le  \frac 12$, including the  endpoint $\be = \frac 12$ 
which is not covered in \cite{Bring1}. 
See Remark~\ref{REM:sing} and Appendix \ref{SEC:D}
on absolute continuity of the Gibbs measure with respect to the shifted measure.
We point out that the focusing case is not studied in~\cite{Bring1}.

As for the dynamical problem, our results are complementary.
Our main focus in this paper is to study the focusing case.
In Theorem \ref{THM:GWP0}, 
we establish a sharp  result 
on almost sure global well-posedness
of the focusing Hartree SdNLW \eqref{SNLW0}
and invariance of the  focusing Hartree Gibbs measure.
In the defocusing case, we only handle the range $\be > 1$, 
where we  need the same renormalization as in the focusing case.

In the second preprint \cite{Bring2}, 
Bringmann studies the dynamical problem in the defocusing case, 
more precisely,  the defocusing Hartree NLW
\eqref{NLW1} with $\s < 0$
and his analysis goes much further
than that presented in our paper.
In this remarkable work, Bringmann  proves its almost sure global well-posedness
and invariance of the defocusing Hartree Gibbs measures
for the entire range $\be > 0$.}
Let $\s < 0$.
Given any $\be > 1$, 
the defocusing Hartree Gibbs measure $\rho$ in~\eqref{H3}
with $A = 0$
exists as a limit of the truncated Gibbs measures.
By introducing further renormalizations at $\be = 1$
and  $\be = \frac 12$, 
the defocusing Hartree Gibbs measure can be constructed
 as a limit of the truncated Gibbs measures
 for $\be > 0$.

For $\be >\frac 12$, the defocusing Hartree Gibbs measure $\rho$ 
 is equivalent to the base massive Gaussian free field $\mu$, 
 while they are mutually singular for $0 < \be \le \frac 12$.

\end{theorem}

We point out that the Gibbs measure is constructed as
a strong limit in the theorem above
{\it except} for the defocusing case with $0 < \be \le \frac 12$, 
where the limiting Gibbs measure is constructed only as a weak limit
of the truncated Gibbs measures.
Theorem \ref{THM:Gibbs0} provides
a complete picture\footnote{The non-normalizability in Theorem \ref{THM:Gibbs0}\,(i)
for $1 < \be < 2$ may be extended for lower values of $\be$ by introducing 
further renormalizations as in the defocusing case.
 We, however,  do not pursue this issue.} on the construction of the Hartree Gibbs measures
on $\T^3$, which is of particular interest in the focusing case
due to its critical nature at $\be = 2$.
The most important novelty in Theorem \ref{THM:Gibbs0}
is the non-normalizability of the focusing Hartree Gibbs measure
for (i) $\be < 2$  or (ii) $\be = 2$ and $\s \gg 1$, 
where we introduce a new strategy for such a non-normalizability argument.
See also \cite{TW, OS}.
The results in Theorem~\ref{THM:Gibbs0} also apply to the
Hartree Gibbs measure with a Wick-ordered  $L^2$-cutoff
studied by Bourgain~\cite{BO97},
showing essential sharpness of his result for $\be > 2$ in the focusing case.
Theorem~\ref{THM:Gibbs0}
 extends the construction of 
 the focusing 
Hartree Gibbs measure with a Wick-ordered  $L^2$-cutoff
 in \cite{BO97} (see~\eqref{Gibbs7} below)
to the critical case ($\be = 2$) in the weakly nonlinear regime ($0 < \s \ll 1$), 
while it establishes the non-normalizability 
for $\be < 2$ and for $\b= 2$ in the strongly nonlinear regime ($ \s \gg 1$), 
thus completing the picture also for the
focusing Hartree Gibbs measure with a Wick-ordered  $L^2$-cutoff.
See Remark \ref{REM:end1} below.
In the defocusing case, 
Theorem \ref{THM:Gibbs0}
also improves Bourgain's Gibbs measure construction for $\beta > \frac 32$ \cite{BO97}
to $\be > 0$.

\begin{remark}\rm
(i) The Hartree Gibbs measures of the form \eqref{H1} 
with various potentials
appear in different contexts in mathematical physics,  in particular as limits of 
the corresponding 
many-body quantum Gibbs states
\cite{LNR1, FKSS1,   LNR2,  LNR3,  LNR4,  LNR5, Sohinger, FKSS2}.
See also \cite{BO97, BO99}.

\smallskip

\noi
(ii)
In the defocusing case ($\s < 0$), 
the Gibbs measure $\rho$ in \eqref{H1} 
corresponds
to the well-studied $\Phi^4_3$-measure
when  $\be = 0$.
The construction of the $\Phi^4_3$-measure
is 
one of the early achievements in constructive Euclidean quantum field theory;
see
\cite{Glimm, GlimmJ, Feldman, Park, 
BFS, AK, BG, GH18b}.
For an overview of the  constructive program
with respect to the $\Phi^4_3$-model, 
see the introductions in \cite{AK, GH18b}.
From the scaling point of view (see \eqref{NLWc} below),  when $\be > 0$, the defocusing Hartree Gibbs measure in~\eqref{H1}
corresponds to 
$\Phi^{4-\eps}_3$-measure
for $\eps =\frac{2\be}{2+\be} > 0$, which tends to 0 as $\be \to 0$.

\smallskip

\noi
(iii)
Note that when $\be = 2$, the potential $V$ essentially corresponds to the Coulomb potential
$V(x) = |x|^{-1}$, which is of particular physical relevance; see \eqref{Be2}.

\smallskip

\noi
(iv)
A precise value
of $\s$ does not play any role unless $\be = 2$ in the focusing case
(and it plays no role in the defocusing case ($\s < 0$))
and thus we simply set $\s = \pm 1$ except for this endpoint focusing case ($\be = 2$).

\end{remark}

\medskip

Next, we discuss stochastic dynamics associated with the Gibbs measures 
constructed in Theorem \ref{THM:Gibbs0}.
This process is known as 
 stochastic quantization \cite{PW}.
While we may consider the usual parabolic stochastic quantization,\footnote{See Remark
\ref{REM:heat} and Appendix \ref{SEC:B}
for the parabolic stochastic quantization of the Hartree Gibbs measure.}
where the linear part is given by the heat operator, 
we consider  the following stochastic damped nonlinear wave equation (SdNLW)  
with a cubic nonlinearity of Hartree-type, 
posed
on $\T^3$:
\begin{align}
\dt^2 u + \dt u + (1 -  \Dl)  u  - \s (V \ast u^2)u  = \sqrt{2} \xi,
\qquad (x, t) \in \T^3\times \R_+,
\label{SNLW0}
\end{align}
where 
$\s \in \R\setminus \{0\}$,  $u$ is an unknown function, 
and $\xi$ denotes a (Gaussian) space-time white noise on $\T^3\times \R_+$
with the space-time covariance given by
\[ \E\big[ \xi(x_1, t_1) \xi(x_2, t_2) \big]
= \dl(x_1 - x_2) \dl (t_1 - t_2).\]

With  $\vec{u} = (u, \dt u)$, 
define the energy $\mathcal{E}(\vec u)$ by 
\begin{align}
\begin{split}
\mathcal{E}(\vec{u})
& = E(u) +  \frac 12 \int_{\T^3} (\dt u)^2 dx \\
& = \frac 12 \int_{\T^3} |\jb{\nabla} u|^2 dx + \frac 12 \int_{\T^3} (\dt u)^2 dx 
- \frac \s4 \int_{\T^3} (V \ast u^2 ) u^2 dx, 
\end{split}
\label{Hamil1}
\end{align}

\noi
where $E(u)$ is as in \eqref{H2}.
This is precisely the energy (= Hamiltonian) 
of  the  (deterministic) nonlinear wave equation (NLW) on $\T^3$
with a cubic Hartree-type nonlinearity:
\begin{align}
\dt^2 u + (1 -  \Dl)  u - \s (V \ast u^2)u  = 0.
\label{NLW1}
\end{align}

\noi
Then,
by letting $v = \dt u$, we can write \eqref{SNLW0} as
\begin{align}
\dt   \begin{pmatrix}
u \\ v
\end{pmatrix}
=     
\begin{pmatrix} 
\frac{\dd\mathcal{E}}{\dd v}
\rule[-3mm]{0pt}{2mm}
\\
- \frac{\dd\mathcal{E}}{\dd u}
\end{pmatrix}
+      \begin{pmatrix} 
0 
\\
- \frac{\dd\mathcal{E}}{\dd v} + \sqrt 2 \xi
\end{pmatrix}.
\label{SNLW00}
\end{align}

\noi
Thus, it is easy to see that 
the Gibbs measure $\rhoo$, formally given by 
\begin{align}
``d\rhoo(\vec u ) = Z^{-1}e^{-\mathcal E(\vec u  )}d\vec u 
=  d\rho \otimes d\mu_0(\vec u )\text{''}
\label{Gibbs0}
\end{align}

\noi 
remains invariant under the dynamics of Hartree SdNLW~\eqref{SNLW0}.
Here, $\rho$ is the Hartree Gibbs measure in~\eqref{H1}
and $\mu_0$ denotes the  white noise measure;
see \eqref{gauss0} with $s = 0$.
Namely, 
Hartree SdNLW~\eqref{SNLW0}
is  the so-called canonical stochastic quantization equation\footnote{Namely, the Langevin equation
 with the momentum $v = \dt u$.}
 for the Gibbs measure $\rhoo$,
 and thus is of importance in mathematical physics. See~\cite{RSS}.

Stochastic nonlinear wave equations (SNLW)
 have been studied extensively
in various settings; 
see \cite[Chapter 13]{DZ} and \cite{OOcomp} for the references therein.
In recent years, we have seen a rapid progress
in the well-posedness theory of 
SNLW with space-time  white noise forcing:\footnote{Some of the works mentioned below
are on SNLW without damping.}
\begin{align}
\dt^2 u + \dt u + (1 -  \Dl)  u + \NN(u)  =  \xi
\label{SNLW0a}
\end{align}

\noi
for a power-type nonlinearity
 \cite{GKO, GKO2, GKOT,   ORTz, OOR, OOcomp, Tolomeo2, OWZ}
and for trigonometric and exponential nonlinearities
\cite{ORSW,  ORW,  ORSW2}.
We also  mention the works 
 \cite{OTh2,   OPTz,  OOTz}
 on
 nonlinear wave equations with  rough random initial data
and
\cite{Deya1, Deya2} 
on SNLW with 
more singular (both in space and time) noises.
In \cite{GKO2}, 
Gubinelli, Koch, and the first author studied the hyperbolic $\Phi^3_3$-model
(i.e.~\eqref{SNLW0a} on $\T^3$ with $\NN(u) = u^2$)
by combining  the paracontrolled calculus~\cite{GIP, CC, MW1}, 
originally introduced in the parabolic setting, 
with the multilinear harmonic analytic approach, 
more traditional in studying dispersive equations.
In particular, one of the new ingredients in \cite{GKO2}
was the introduction of 
{\it  paracontrolled operators}
(namely, random operators with an embedded paracontrolled structure)
as a part of the pre-defined enhanced data set.
These paracontrolled operators introduced in \cite{GKO2}
play an important role in studying well-posedness
of Hartree SdNLW~\eqref{SNLW0}.
See Subsection~\ref{SUBSEC:1.4}.

We now state our main result on the dynamical problem.

\begin{theorem}\label{THM:GWP0}

Let $V$ be the Bessel potential of order $\be$ with

\medskip

\begin{itemize}
\item[\textup{(i)}] $\be \ge 2$
 in the focusing case \textup{(}$\s>0$\textup{)}, and

\smallskip

\item[\textup{(ii)}]   $\be > 1$  in the defocusing case \textup{(}$\s<0$\textup{)}.

\end{itemize}

\medskip

\noi
In the focusing case with $\be = 2$, 
we also assume that  $\s > 0$ is sufficiently small.
Then, the cubic  Hartree SdNLW~\eqref{SNLW0} 
on the three-dimensional torus $\T^3$ 
\textup{(}with a proper renormalization\textup{)} is almost surely globally well-posed
with respect to the random initial data distributed
by the \textup{(}renormalized\textup{)} Gibbs measure $\rhoo$ in \eqref{Gibbs0}.
Furthermore, the Gibbs measure $\rhoo$ is invariant under the resulting dynamics.

\end{theorem}

See Theorem~\ref{THM:GWP}
for the precise statement.
Theorem \ref{THM:GWP0}
is a wave-analogue of Bourgain's result in \cite{BO97}
on the Hartree NLS \eqref{NLS1}
for $\be > 2$
mentioned above.
In the focusing case, we extend the result
to the endpoint case $\be = 2$ in the weakly nonlinear regime.
In view of the 
 non-normalizability of the focusing Hartree Gibbs measure  (Theorem \ref{THM:Gibbs0}),
 Theorem \ref{THM:GWP0} is sharp in the focusing case.\footnote{Recall that finiteness
 of a limiting measure is needed for Bourgain's invariant measure argument.}
In terms of the scaling, 
 Theorem \ref{THM:GWP0} 
 for $\be > 1$ in the defocusing case\footnote{As mentioned earlier, 
 this result was improved to $\be > 0$ by Bringmann \cite{Bring2}.} may be viewed as a (slight) improvement
 from  \cite{GKO2} on the quadratic nonlinearity
 (corresponding to $\be = 2$).
 
 Given the construction of the Gibbs measure in Theorem~\ref{THM:Gibbs0}, 
 the main task in proving Theorem~\ref{THM:GWP0}
is  the construction of local-in-time dynamics
almost surely with respect to the Gibbs measure.
We go over
the well-posedness aspects
in Section \ref{SEC:NLW}.
In particular, 
in Subsection~\ref{SUBSEC:1.4}, 
by using the ideas from the paracontrolled calculus, 
we rewrite (the renormalized version of) Hartree SdNLW~\eqref{SNLW0}
into a system of three unknowns, 
for which we prove local well-posedness.

\begin{remark}\rm

(i) 
Let us study   \eqref{SNLW0}
from the scaling point of view.
Recall that the  Bessel potential of order $\be$ on $\T^3$
can be written (for some $c>0$) as 
\begin{align} \label{Be2}
V(x) = c |x|^{\be-3} + K(x)
\end{align}
for $0<\be<3$ and $x \in \T^3 \setminus \{ 0 \}$,
where $K$ is a smooth function on $\T^3$.
See Lemma 2.2 in \cite{ORSW}.
In order to study the scaling property 
of Hartree SdNLW \eqref{SNLW0}, let us consider the following 
 nonlinear wave equation (NLW) on $\R^3$ (without damping):
\begin{align} 
\dt^2 u- \Dl u \pm \big( |x|^{\be-3} \ast u^2 \big) u =0.
\label{NLWb}
\end{align}

\noi
A simple calculation shows that \eqref{NLWb} is invariant under the following scaling:
\[
u(x,t) \mapsto
u^{\ld} (x,t) = \ld^{1+ \frac \be 2} u(\ld x, \ld t)
\]

\noi
for $\ld >0$.
Namely, the equation \eqref{NLWb}
with a cubic Hartree nonlinearity
scales like
the following NLW
with a power nonlinearity:
\begin{align} 
\dt^2 u- \Dl u\pm |u|^{\frac 4{2+\be}}u =0.
\label{NLWc}
\end{align}

\noi
From this scaling point of view, 
the quadratic SNLW studied in \cite{GKO2}
corresponds to Hartree SdNLW~\eqref{SNLW0}
with $\be = 2$.
See Remark \ref{REM:quad} below.

\smallskip

\noi
(ii) In a recent work \cite{DNY1}, 
Deng, Nahmod, and Yue introduced the notion of probabilistic scaling
and the associated critical regularity, 
based on the 
observation that the Picard second iterate should be (at least) as smooth as a stochastic convolution (or a random linear solution in the context of the random data well-posedness theory).
The probabilistic scaling critical 
regularity for \eqref{NLWb} 
on $\T^3$ (with  Gaussian random initial data)
are given by
$s_{\text{prob}}^{\text{Hartree}} = \max\big(-\frac{\be+2}3, - \frac 32\big)$.
See Figure 2 in \cite{Bring2}.
As observed in the recent works \cite{OTh2, GKO, GKOT, ORTz}, 
the study of 
SNLW with the space-time white noise forcing is
closely related to that of the deterministic NLW with 
the Gaussian free field as initial data (see \eqref{IV2} below)
with regularity $s= - \frac 12 - \eps$.
Comparing this regularity with 
$s_{\text{prob}}^{\text{Hartree}}$ above, 
we see that SdNLW \eqref{SNLW0} with the space-time white noise forcing
is subcritical 
for $\be > -\frac 12$
(coming from the condition $s_{\text{prob}}^{\text{Hartree}} < -\frac 12$).
%
From this probabilistic scaling point of view, 
one may hope to solve \eqref{SNLW0}
for the entire subcritical range but this is a very challenging problem.
See also Remark~\ref{REM:0} below.

Lastly, we point out that while the probabilistic
scaling criticality is relevant for constructing local-in-time dynamics, 
the critical value $\be = 2$ in the focusing case comes 
from the viewpoint of the measure construction (Theorem \ref{THM:Gibbs0}),
which is relevant for constructing global-in-time dynamics.

\end{remark}

\begin{remark} \label{REM:Vp} \rm
In view of \eqref{Be2}, 
(the kernel of)
the Bessel potential $V(x)$ is not non-negative\footnote{Note that, in view 
of \eqref{Be2}, the potential $V$ is uniformly bounded 
from below by a (possibly negative) constant.}
 on $\T^3$.
Nonetheless, 
the potential part of the energy in \eqref{Hamil1}
(for a smooth function $u$)
is 
non-negative.
Indeed, Parseval's identity yields 
\begin{align*}
\int_{\T^3} (V \ast u^2) \, u^2  dx
= \sum_{n \in \Z^3} \ft V(n) |\ft {u^2} (n)|^2
\ge 0.
\end{align*}

\noi
This justifies the use of the terminology `defocusing\,/\,focusing'.
\end{remark}

\begin{remark}\rm 
We point out that a slight modification of our proof of Theorem \ref{THM:GWP0}
yields the corresponding results
(namely, almost sure global well-posedness and invariance of the associated Gibbs measure)
for the (deterministic) cubic Hartree NLW \eqref{NLW1}
 on $\T^3$
for 
(i)  $\be > 2$ (and $\be = 2$ in the weakly nonlinear regime) in the focusing case
and (ii)  $\be > 1$ in the defocusing case.
As pointed above, this result is sharp in the focusing case.

\end{remark}

\begin{remark}\label{REM:heat}\rm 
 In Appendix \ref{SEC:B}, 
we consider the parabolic stochastic quantization of the 
focusing Hartree Gibbs measure $\rho$
constructed in Theorem \ref{THM:Gibbs0}.
Namely, we study the following stochastic nonlinear heat equation on $\T^3$
with a focusing Hartree nonlinearity ($\s > 0$):
\begin{align}
 \dt u + (1 -  \Dl)  u - \s (V*u^2) u   =  \sqrt 2\xi.
\label{heat0a}
\end{align}

\noi
When $\be >  2$,
  (and $\be = 2$ in the weakly nonlinear regime, $0 < \s \ll 1$),  
we prove almost sure global well-posedness of \eqref{heat0a} and invariance of the 
focusing Hartree Gibbs measure.
In view of the non-normalizability result in Theorem \ref{THM:Gibbs0}, 
this result is also sharp.

\end{remark}

\begin{remark}\label{REM:quad}\rm
In terms of scaling, 
the critical focusing Hartree model  ($\be = 2$) 
corresponds to the $\Phi^3_3$-model.
In~\cite{OOTolo}, 
we study the construction of the $\Phi^3_3$-measure:
\begin{align*}
d\rho(u) = Z^{-1} \exp \bigg(\frac{\s}3 \int_{\T^3} u^3  dx\bigg) d\mu(u), 
\end{align*}

\noi
and its canonical stochastic quantization. 
This $\Phi^3_3$-model also turns out to be critical.
In the measure construction part, 
 we exhibit 
a phase transition between the 
weakly  nonlinear regime ($ |\s| \ll 1 $)
and 
the strongly nonlinear regime
($|\s| \gg1$) 
for the $\Phi^3_3$-measure, just 
as in the critical $\be = 2$ case of Theorem~\ref{THM:Gibbs0}\,(i).
In the weakly nonlinear regime, 
we also 
 extend the local-in-time solutions 
 to  the hyperbolic $\Phi^3_3$-model
(i.e.~\eqref{SNLW0a} on $\T^3$ with $\NN(u) = u^2$), 
constructed in~\cite{GKO2},  globally in time.
While 
the 
focusing Hartree Gibbs measure in \eqref{H3}
is absolutely continuous with respect the base Gaussian free field
even in the critical case ($\be = 2$), 
it turns out that 
  the $\Phi^3_3$-measure
  is singular with respect to the base Gaussian free field. 
This singularity of the $\Phi^3_3$-measure
introduces additional difficulties
 in both the measure (non-)construction part
and the dynamical part in \cite{OOTolo}.
See \cite{OOTolo} for a further discussion.

\end{remark}

\begin{remark}\rm

In \cite{Tolomeo1}, the third author introduced 
a new approach to establish unique
ergodicity of Gibbs measures
for stochastic dispersive/hyperbolic equations.
In particular, ergodicity of the 
Gibbs measures was shown in \cite{Tolomeo1} for 
the cubic SdNLW on $\T$
and the cubic stochastic damped nonlinear beam equation on $\T^3$.
See also 
\cite{FT} on the asymptotic Feller property
of the invariant Gibbs dynamics for these models.
In  \cite{Tolomeo3}, 
 the third  author 
further developed the methodology
and managed to prove 
ergodicity of the hyperbolic $\Phi^4_2$-model, 
i.e.~\eqref{SNLW0a} on $\T^2$ with $\NN(u) = u^3$.

\end{remark}

\begin{remark}\label{REM:0}

\rm

In the defocusing case, the threshold value  $\be=1$ in Theorem \ref{THM:GWP0}
is by no means sharp but 
a further renormalization is required in order to treat the problem for $\be \le 1$
(as mentioned in Theorem \ref{THM:Gibbs0}).\footnote{As mentioned in Footnote \ref{FOOT1}, 
Bringmann \cite{Bring2} studied the  defocusing Hartree NLW
\eqref{NLW1} with $\s < 0$
and proved its almost sure global well-posedness
and invariance of the defocusing Hartree Gibbs measures
for the entire range $\be > 0$.
We expect that his analysis also applies to the defocusing Hartree SdNLW \eqref{SNLW0}
and yields the corresponding well-posedness result for $\be > 0$.
}
When $\be = 0$, 
Hartree SdNLW  \eqref{SNLW0} with $\s = -1$ reduces to the following 
 hyperbolic $\Phi^4_3$-model on $\T^3$:
\begin{align}
\dt^2 u + \dt u + (1 -  \Dl)  u + u^3  = \sqrt{2} \xi.
\label{NLWd}
\end{align}

\noi
In the parabolic setting, we have seen a tremendous
progress
in the study of singular stochastic partial differential equations (PDEs)
over the last ten years
and, in particular, 
the well-posedness theory of  the parabolic $\Phi^4_3$-model:
\begin{align}
 \dt u +  (1 -  \Dl)  u + u^3  =  \sqrt 2\xi, 
\label{heat0}
\end{align}

\noi
has been studied by many authors.
See
 \cite{Hairer, GIP, CC, Kupi, MW1, MWX, AK, GH18} and references therein. 
Up to date, the well-posedness issue of the  hyperbolic $\Phi^4_3$-model \eqref{NLWd}
 remains as an important open problem.\footnote{In a very recent breakthrough work \cite{BDNY}, 
 Bringmann, Deng, Nahmod, and Yue resolved this open problem
 and proved that the hyperbolic $\Phi^4_3$-model is indeed
 almost surely globally well-posed with respect to the (defocusing) $\Phi^4_3$-measure.}
In a recent preprint \cite{OWZ}, by smoothing out the noise in \eqref{NLWd}
(i.e.~replacing $\xi$ by $\jb{\nb}^{-\be}\xi$ for any $\be >0$), 
Y.~Wang, Zine, and the first author proved local well-posedness
of the cubic  SNLW on $\T^3$ with an almost space-time white noise forcing.

We also note that the well-posedness issue of NLS \eqref{NLS1} 
with the Gibbs measure for  $\be = 0$, 
corresponding to the dispersive $\Phi^4_3$-model, 
is a challenging open problem, expected to be much harder 
than the hyperbolic $\Phi^4_3$-model mentioned above.
We mention a recent 
breakthrough~\cite{DNY2}
by Deng, Nahmod, and Yue,
making an important step in this direction.

 \end{remark}

\subsection{Hartree Gibbs measures}
\label{SUBSEC:Gibbs}

In this subsection, we describe
a renormalization procedure 
(and also a taming by the Wick-ordered $L^2$-norm in the focusing case) 
required to construct
the Gibbs measure $\rhoo$ in \eqref{Gibbs0}
and make 
precise statements
on the Gibbs measure construction
(Theorems~\ref{THM:Gibbs} and \ref{THM:Gibbs2}).
For this purpose,  we first fix some notations. 
Given $ s \in \R$, 
let $\mu_s$ denote
a Gaussian measure,   formally defined by
\begin{align}
 d \mu_s 
   = Z_s^{-1} e^{-\frac 12 \| u\|_{{H}^{s}}^2} du
& =  Z_s^{-1} \prod_{n \in \Z^3} 
 e^{-\frac 12 \jb{n}^{2s} |\ft u(n)|^2}   
 d\ft u(n) , 
\label{gauss0}
\end{align}

\noi
where 
  $\jb{\,\cdot\,} = \big(1+|\,\cdot\,|^2\big)^\frac{1}{2}$
and $\ft u(n)$  denotes the Fourier transforms of $u$.
Note that 
$\mu_s$ corresponds to 
 the massive Gaussian free field $\mu_1$ when  $s  = 1$
 and to the white noise measure $\mu_0$ when $s = 0$.
On $\T^3$, it is well known that $\mu_s$ is a Gaussian probability measure supported
on $W^{s - \frac 32 - \eps, p}(\T^3)$ for any $\eps > 0$ and $1 \leq p \leq \infty$.
For simplicity, we set $\mu = \mu_1$
and 
\begin{align}
\muu = \mu_1 \otimes \mu_{0} .
\label{gauss1}
\end{align}

\noi
Note that $\mu$ and $\muu$ serve
as the reference Gaussian measures
for the Gibbs measures $\rho$ in \eqref{H1} and 
$\rhoo$ in \eqref{Gibbs0}, respectively.

We now go over the Fourier representation
of functions distributed by $\mu$ and $\muu$.
Define the index set $\Ld$ and $\Ld_0$ by 
\begin{align}
\Ld = \bigcup_{j=0}^{2} \Z^j\times \N \times \{0\}^{2-j}
\qquad \text{and}\qquad \Ld_0 = \Ld \cup\{(0, 0, 0)\}
\label{index}
\end{align}

\noi
such that $\Z^3 = \Ld \cup (-\Ld) \cup \{(0, 0, 0)\}$.
Then, 
let 
$\{ g_n \}_{n \in \Ld_0}$ and $\{ h_n \}_{n \in \Ld_0}$
 be sequences of mutually independent standard complex-valued\footnote
{This means that $g_0,h_0\sim\NN_\R(0,1)$
and  
$\Re g_n, \Im g_n, \Re h_n, \Im h_n \sim \NN_\R(0,\tfrac12)$
for $n \ne 0$.}
 Gaussian random variables on 
a probability space $(\O,\F,\PP)$ and 
set $g_{-n} := \cj{g_n}$ and $h_{-n} := \cj{h_n}$ for $n \in \Ld_0$.
Moreover, we assume that 
$\{ g_n \}_{n \in \Ld_0}$ and $\{ h_n \}_{n \in \Ld_0}$ are independent from the space-time white noise $\xi$ in \eqref{SNLW0}.
We now define random distributions $u= u^\o$ and $v = v^\o$ by 
the following  Gaussian Fourier series:\footnote{By convention, 
 we endow $\T^3$ with the normalized Lebesgue measure $dx_{\T^3}= (2\pi)^{-3} dx$.}
\begin{equation} 
u^\o = \sum_{n \in \Z^3 } \frac{ g_n(\o)}{\jb{n}} e_n
 \qquad
\text{and}
\qquad
v^\o = \sum_{n\in \Z^3}  h_n(\o) e_n, 
\label{IV2}
\end{equation}

\noi
where  $e_n=e^{i n\cdot x}$.
Denoting the law of a random variable $X$ by $\Law(X)$, 
we then have
\[\Law((u, v)) = \muu_1 = \mu \otimes \mu_0\] 

\noi
for $(u, v)$ in \eqref{IV2}.
Note that  $\Law((u, v)) = \muu$ is supported on
\begin{align*}
\H^{s}(\T^3): = H^{s}(\T^3)\times H^{s - 1}(\T^3)
\end{align*}

\noi
for $s < -\frac 12$ but not for $s \geq -\frac 12$.

\begin{remark}\label{REM:Gibbs1} \rm

 In the following, we only discuss the construction
 and non-normalizability 
 of the (renormalized) Gibbs measure~$\rho$ on $u$, formally written in \eqref{H1}.
The Gibbs measure $\rhoo$
on a vector $\vec u= (u, \dt u)$
for SdNLW \eqref{SNLW0} and NLW \eqref{NLW1}, formally defined in   
 \eqref{Gibbs0},   decouples
as the Gibbs measure $\rho$   on the first component $u$
and the white noise measure $\mu_0$ on the second component $ \dt u$.
Thus, once we prove Theorem \ref{THM:Gibbs0}
for the Gibbs measure $\rho$ on $u$, 
by setting 
\begin{align*}
d\rhoo(\vec u ) 
=  d\rho \otimes d\mu_0(\vec u ), 
\end{align*}

\noi
we see that 
the corresponding results extend to the Gibbs measure 
 $\rhoo$.
See also Remarks~\ref{REM:Gibbs1a}
and~\ref{REM:Gibbs2a}.

\end{remark}

\noi
$\bullet$ {\bf Defocusing case:}
Let us first consider the defocusing case.
A precise value of $\s < 0$ in \eqref{H1} does not play any role
and thus we simply set $\s = -1$.
In view of 
 \eqref{H2}, 
 we can write the formal expression~\eqref{H1}
 for the Gibbs measure $\rho$ as\footnote{Hereafter, 
 we simply use $Z$, $Z_N$, etc. to denote various normalization constants.}
\begin{align}
\text{``}\, d \rho(u) 
= Z^{-1} e^{-E(u)} d  u
= Z^{-1} \exp \bigg( - \frac 1 4 \int_{\T^3} (V \ast u^2) u^2 dx \bigg) d \mu(u) \, \text{"}.
\label{Gibbs0a}
\end{align}

\noi
Since $u$ in the support of $\mu$ is not a function, 
the quartic potential energy is not well defined
 and thus a proper renormalization is required to give a meaning to 
 \eqref{Gibbs0a}.
In order to explain the renormalization process, we first study the regularized model.
Given $N \in \N$, we define the (spatial) frequency projector $\pi_N$  by 
\begin{align}
\pi_N f = 
\sum_{ |n| \leq N}  \ft f (n)  e_n.
\label{pi}
\end{align}

\noi
Let   $u$  be as in \eqref{IV2}
and 
set $u_N = \pi_N u$.
Note that, for each fixed $x \in \T^3$, 
 $u_N(x)$ is
a mean-zero real-valued Gaussian random variable with variance
\begin{align}
\s_N = \E\big[u_N^2(x)\big] = \sum _{|n| \le N} \frac1{\jb{n}^2}
\sim N \too \infty, 
\label{sigma1}
\end{align}

\noi
as $N \to \infty$.
\noi
See also \eqref{sigma1a} below.
We then define the Wick power $:\! u_N^2 \!: $ by 
\begin{align}
:\! u_N^2 \!: \, = u_N^2 - \s_N.
\label{Wick1}
\end{align}

Let us consider the renormalized potential energy.
By Parseval's identity, we have
\begin{align}
\begin{split}
\int_{\T^3} (V \ast :\! u_N^2 \!:) :\! u_N^2 \!: dx
& = \sum_{n \in \Z^3}
\ft V(n) |\ft{\,:\!u_N^2 \!:\,}(n)|^2 \\
& = \sum_{n\in \Z^3} \ft V(n) 
\bigg( \sum_{
\substack{n_1,n_2 \in \Z^3 \\ |n_1|, |n_2| \le N\\
n_1 + n_2 = n}} \ft u(n_1) \ft u(n_2) -\ind_{n = 0}\cdot  \s_N\bigg) \\
& \quad 
\times  \bigg( \sum_{
 \substack{n_1',n_2' \in \Z^3 \\ |n_1'|, |n_2'| \le N\\
\\ n_1' + n_2' = n} }\cj{\ft u(n_1') \ft u(n_2')} -\ind_{n = 0} \cdot \s_N\bigg). 
\end{split}
\label{Wick2}
\end{align}

\noi
While the Wick renormalization \eqref{Wick1} removes certain singularities, 
we still need to subtract a divergent contribution from the renormalized potential energy in \eqref{Wick2}.
By setting 
\begin{align}
\al_N := \sum_{\substack{n_1,n_2 \in \Z^3 \\ 
|n_1|, |n_2| \le N\\n_1 + n_2 \ne 0}} \frac{\ft V(n_1+n_2)}{\jb{n_1}^2 \jb{n_2}^2}, 
\label{alN}
\end{align}

\noi
we define the full renormalized potential energy $R_N(u)$ by 
\begin{align}
\begin{split}
R_N (u)
&= \frac 14 \int_{\T^3} (V \ast :\! u_N^2 \!:) :\! u_N^2 \!: dx
- \frac 12 \al_N.
\end{split}
\label{K1}
\end{align}

\noi
With \eqref{Bessel1} and Lemma \ref{LEM:SUM} below,  
we see that $\al_N$ is uniformly bounded in $N \in \N$ when $\be > 2$ and thus 
the subtraction of $\frac 12\al_N$ in \eqref{K1} is not necessary in this case.
Thanks to the presence of $\al_N$ in \eqref{K1}, 
we can show that $R_N $ 
converges to some limit $R$
in $L^p(\mu)$ when $\be > 1$.
See Lemma \ref{LEM:conv} below.

Define the truncated renormalized Gibbs measure $\rho_N$ by 
\begin{align}
d \rho_N (u) = Z_N^{-1} e^{-R_N(u)} d \mu(u).
\label{GibbsN}
\end{align}

\noi
Then, we have the following uniform exponential integrability of the density, 
which allows us to construct the limiting Gibbs measure $\rho$.

\begin{theorem}[defocusing case] \label{THM:Gibbs}
Let $V$ be the Bessel potential of order $\be>0$.

\smallskip

\noi
\textup{(i)}
Let $\be > 1$.
Then, given any finite $ p \ge 1$, 
there exists $C_p > 0$ such that 
\begin{equation}
\sup_{N\in \N} \Big\| e^{-R_N(u)}\Big\|_{L^p( \mu)}
\leq C_p  < \infty.
\label{exp1}
\end{equation}

\noi
Moreover, we have
\begin{equation}\label{exp2}
\lim_{N\rightarrow\infty}e^{ -R_N(u)}=e^{-R(u)}
\qquad \text{in } L^p( \mu).
\end{equation}

\noi
As a consequence, 
the truncated renormalized Gibbs measure $\rho_N$ in \eqref{GibbsN} converges, in the sense 
of~\eqref{exp2}, 
 to the defocusing Hartree Gibbs measure $\rho$ given by
\begin{align}
d\rho(u)= Z^{-1} e^{-R(u)}d\mu(u).
\label{Gibbs2}
\end{align}

\noi
The resulting Gibbs measure $\rho$ is equivalent 
to the base massive Gaussian free field  $\mu = \mu_1$.

\smallskip

\noi
\textup{(ii)}
By introducing further renormalizations
at $\be = 1$ and $\be = \frac 12$, 
we replace  the potential energy $R_N(u)$  in \eqref{K1} by 
the new renormalized potential energies\textup{:}
\[ 
\text{$R_N^\dia(u)$ \ for $\tfrac 12 < \be \le 1$
\qquad and  
\qquad  $R_N^{\dia\dia} (u)$ \ for $0 < \be \le \tfrac 12$}.\] 

\noi
Then, 
the uniform exponential integrability \eqref{exp1}
holds 
for \textup{(a)} any finite $p \geq 1$ when $ \tfrac 12 < \be \le 1$
and \textup{(b)} 
 $p = 1$ when 
 $0 < \be \le \tfrac 12$.

\smallskip

\begin{itemize}
\item[\textup{(ii.a)}]
Let  $\frac 12 < \be \le 1$.
Then, 
 $R_N^\dia $ 
converges to some limit $R^\dia$
in $L^p(\mu)$ and we have
\begin{equation}
\lim_{N\rightarrow\infty}e^{ -R_N^\dia (u)}=e^{-R^\dia (u)}
\qquad \text{in } L^p( \mu).
\label{exp2a}
\end{equation}

\noi
As a consequence, 
the truncated renormalized Gibbs measure $\rho_N$ in \eqref{GibbsN} 
\textup{(}with $R_N$ replaced by $R_N^\dia$\textup{)} converges, in the sense 
of~\eqref{exp2}, 
 to the defocusing Hartree Gibbs measure $\rho$ 
 in \eqref{Gibbs2}
 \textup{(}with $R$ replaced by $R^\dia$\textup{)}. 
The resulting Gibbs measure $\rho$ is equivalent 
to the base massive Gaussian free field  $\mu$.

\smallskip

\item[\textup{(ii.b)}]
Let $0 < \be \le \frac 12$.
The truncated renormalized Gibbs measure $\rho_N$ in \eqref{GibbsN} 
\textup{(}with $R_N$ replaced by $R_N^{\dia\dia}$\textup{)} converges
weakly to a unique limit $\rho$.
In this case, the resulting Gibbs measure $\rho$ 
and 
 the base massive Gaussian free field  $\mu$
 are mutually singular.

\end{itemize}

\end{theorem}

See \eqref{K1r} and \eqref{K5}
for the definitions of $R_N^\dia$ and $R_N^{\dia\dia}$.
Theorem \ref{THM:Gibbs} is an improvement of the defocusing Hartree Gibbs measure construction
by Bourgain \cite{BO97}, 
where he essentially proved an analogue of Theorem \ref{THM:Gibbs}
for $\be > \frac 32$.
See \cite{BO97} for a precise statement.

The main task  in proving Theorem \ref{THM:Gibbs} 
is to show  the uniform exponential bound \eqref{exp1}.
 We establish the bound \eqref{exp1} by applying the variational approach introduced by Barashkov and Gubinelli \cite{BG} in the construction of
the $\Phi^4_3$-measure. See also \cite{GOTW, ORSW2}.
We point out that further renormalizations
are required
in order to go below the thresholds $\be = 1$ and $\be = \frac 12$
and that the renormalization introduced for  $0 < \be \le \frac 12$
(see \eqref{K5})
only appears at the level of the Gibbs measure but not
in the associated equation.
See Remarks \ref{REM:ren} and \ref{REM:Bb} and Subsection \ref{SUBSEC:def3} below.
When $\be = 0$, the Gibbs measure 
corresponds to the $\Phi^4_3$-measure
whose construction requires a further renormalization to remove a logarithmic divergence;
see
\cite{Glimm, GlimmJ, Feldman, Park, 
BFS, AK, BG, GH18b}.
If we consider a $\Phi^4_3$-measure but with a smoother base Gaussian measure $\mu_s$, $s > 1$, 
such a logarithmic divergence does not appear
and thus the second renormalization is not needed.
Thus, it is interesting to see that the defocusing Hartree Gibbs measure $\rho$ requires
renormalizations at $\be = 1$ and $\frac 12$.

Once the uniform bound \eqref{exp1} is established, 
the $L^p$-convergence \eqref{exp2} of the densities follows from 
(softer) convergence in measure of the densities. 
See Remark 3.8 in \cite{Tz08}.  
For $0 < \be \le \frac 12$, 
such convergence in measure of the densities no longer holds, 
which is essentially the source of the singularity
of the Gibbs measure in this range.
See Remark~\ref{REM:Bb}. 
For this range of $\be$, we use the more refined
Bou\'e-Dupuis variational formula (Lemma \ref{LEM:var3})
to prove uniqueness of the limiting Gibbs measures
and its singularity with respect to the base Gaussian free field.
Our proof of the singularity is strongly inspired
by a recent work 
\cite{BG2} by Barashkov and Gubinelli, 
where they proved the ``folklore'' singularity
of the $\Phi^4_3$-measure with respect to the base Gaussian free field.
While the proof of the singularity in \cite{BG2} goes through the shifted measure, 
we present a direct argument without referring to shifted measures.
 See Remark~\ref{REM:sing}.

We present the proof of Theorem~\ref{THM:Gibbs}\,(i) 
for $\be > 1$ in Section \ref{SEC:Gibbs}, 
while 
the proof of Theorem~\ref{THM:Gibbs}\,(ii) 
for $0 < \be \leq 1$ is discussed in detail in 
Section \ref{SEC:def}.

\begin{remark}\label{REM:Gibbs1a}\rm
Let $\be > 1$.
Define
 the renormalized energy:
\begin{align} 
E^\flat( u)
= \frac 12 \int_{\T^3} |\jb{\nabla} u|^2 dx + R(u).
\label{Energy0}
\end{align}

\noi
In view of the definition of $\mu = \mu_1$, \eqref{Gibbs2},  and \eqref{Energy0}, we 
can also write the defocusing Hartree Gibbs measure $\rho$ formally as 
\begin{align*} 
d \rho = Z^{-1} e^{-E^\flat ( u)} d  u.
\end{align*}

\noi
Similarly, by defining the renormalized energy 
for SdNLW \eqref{SNLW0} and NLW \eqref{NLW1} by 
\begin{align} 
\mathcal{E}^\flat (\vec u)
= \frac 12 \int_{\T^3} |\jb{\nabla} u|^2 dx + \frac 12 \int_{\T^3} (\dt u)^2 dx + R(u), 
\label{Energy1}
\end{align}

\noi
we can write 
 the defocusing Hartree 
Gibbs measure $\rhoo = \rho\otimes \mu_0$ on a vector $\vec u = (u, \dt u)$ 
as
\begin{align} 
d \rhoo = Z^{-1} e^{-\mathcal{E}^\flat ( \vec u)} d  \vec u.
\label{Gibbs2x}
\end{align}

\noi
In the following subsections, we discuss
well-posedness of the SdNLW dynamics,
emanating from the renormalized energy $\mathcal{E}^\flat(\vec u)$ in \eqref{Energy1}. 

\end{remark}

\begin{remark} \label{REM:ren}\rm
We briefly discuss the renormalization required for $\be \leq 1$.
See Subsection \ref{SUBSEC:def3} for 
a further renormalization required for $\be \le \frac 12$.
Define $\kk_N(n)$ 
by 
\begin{align} 
\kappa_N (n) =
\sum_{\substack{n_1 \in \Z^3 \\ n_1 \neq -n \\ |n_1| \le N}} \ft V(n+n_1) \jb{n_1}^{-2}.
\label{kappa1}
\end{align}

\noi 
Note that the limit
$\kappa (n) = \lim_{N \to \infty} \kappa_N(n) $ exists if and only if $\be > 1$.
This term exactly cancels the divergence part
of $R_N(u)$ which emerges at $\be = 1$.
See Remark~\ref{REM:Bb}.
With a slight abuse of notation, 
define $K_N$ and $K_N^\frac{1}{2}$ by 
\begin{align}
 K_N(x) = \sum_{n \in \Z^3} \kk_N(n) e_n(x)
 \qquad \text{and}\qquad
 K_N^\frac{1}{2}(x) = \sum_{n \in \Z^3} \kk_N^\frac{1}{2}(n) e_n(x).
\label{kappa2}
\end{align}

\noi
Then, 
for  $\frac 12 < \be \le 1$, 
we can introduce 
a further renormalization
to $R_N(u)$ in \eqref{K1} by setting
\begin{align} 
\begin{split}
 R_N^\dia (u)
=
R_N(u)
- \int_{\T^3} :\! (K_N^\frac{1}{2} *   u_N)^2\!:\, dx.
\end{split}
\label{K1r}
\end{align}

\noi
The truncated renormalized Gibbs measure $\rho_N$ is then given by 
\begin{align}
d \rho_N (u) = Z_N^{-1} e^{- R_N^\dia(u)} d \mu(u), 
\label{Gibbsx}
\end{align}

\noi
for which we prove the following 
 uniform exponential integrability:
\begin{equation}
\sup_{N\in \N} \Big\| e^{- R_N^\dia(u)}\Big\|_{L^p(\mu)}
\leq C_p  < \infty
\label{exp1a}
\end{equation}

\noi
for any finite $p \geq 1$
and the convergence claimed in Theorem \ref{THM:Gibbs} (ii.a).
This allows us to construct 
 the  Gibbs measure $\rho$ given by 
\begin{align}
d\rho(u)= Z^{-1} e^{- R^\dia(u)}d\mu(u)
\label{Gibbs2a}
\end{align}

\noi
as a limit of 
the truncated renormalized Gibbs measures $\rho_N$ in \eqref{Gibbsx}, 
provided that $\be > \frac 12$.

For  $0 < \be \le \frac 12$, 
we introduce another renormalization, based on a change of variables \eqref{YZ13}
as in~\cite{BG},  
and prove the uniform exponential integrability
for a new renormalized potential energy $R_N^{\dia\dia}(u)$:
\begin{equation}
\sup_{N\in \N} \E_\mu\Big[ e^{- R_N^{\dia\dia}(u)}\Big]
 < \infty.
\label{exp1c}
\end{equation}

\noi
We can prove the uniform exponential integrability
only for $p = 1$ due to 
the renormalization introduced at $\be = \frac 12$
(which is aimed to cancel a second order interaction).
Unfortunately, the convergence of $R_N^{\dia\dia}(u)$
or the density no longer holds in this case.
We establish uniqueness of the limiting Gibbs measure in 
a direct manner.  See Subsection \ref{SUBSEC:def4}.

\end{remark}

\begin{remark}\label{REM:sing}\rm
As mentioned above, 
 our proof of the singularity 
of the Gibbs measure does not make use of the shifted measure.
In Appendix \ref{SEC:D}, we show that the Gibbs
measure $\rho$ is absolutely continuous
with respect to the shifted measure, 
more precisely, to the law of $Y(1) - \ZZ(1) + \W(1)$, 
where $Y(1)$ is as in \eqref{P2} with $\Law(Y(1) ) = \mu$, 
 $\ZZ = \ZZ(Y)$ is the limit of $\ZZ^N$ defined in \eqref{YZ12}, 
 and the auxiliary process $\W = \W(Y)$ is defined in \eqref{AC0}.

\end{remark}

\noi
$\bullet$ {\bf Focusing case:}
Let us first  go over  the Gibbs measure construction in the two-dimensional setting.
In the defocusing case, 
the standard Wick renormalization and
Nelson's argument \cite{Nelson2} allow us to construct
the  (defocusing) $\Phi^4_2$-measure on $\T^2$: 
\begin{align*}
d\rho (u) = Z^{-1}e^{-\frac 14 \int_{\T^2}  :u^4: \, dx  }d \mu(u) .
\end{align*}

\noi
See \cite{Simon, GlimmJ87, DT, OTh}.
On the other hand, 
in the focusing case, 
Brydges and Slade \cite{BS} proved
 non-normalizability of 
$\Phi^4_2$-measure, even with a (Wick-ordered) $L^2$-cutoff.
In \cite{BO95}, 
 Bourgain reported Jaffe's construction of a  $\Phi^3_2$-measure endowed with a Wick-ordered
 $ L^2$-cutoff:
\begin{align*}
d\rho = Z^{-1}
 \ind_{\{\int_{\T^2} :\,u^2: \, dx\, \leq K\}} 
 e^{ \int_{\T^2}  :u^3: \, dx  }d \mu(u) .
\end{align*}

\noi 
Unfortunately, this measure is not suitable
for studying the associated heat and wave  dynamics
due to the lack of the $L^2$-conservation in the deterministic setting.\footnote{This measure
does not make sense in the complex-valued setting
and hence is not suitable also for the Schr\"odinger dynamics.}
In \cite{BO95}, 
Bourgain instead proposed to consider the  Gibbs measure
of the form:
\begin{align}
d\rhoo(\vec u ) = Z^{-1}
 e^{ \int_{\T^2}  :u^3: \, dx  - A
\big(\int_{\T^2} :\,u^2: \, dx\big)^2} d \muu(\vec u ) 
\label{Gibbs6}
\end{align}

\noi
(for sufficiently large $A>0$)
in studying NLW dynamics on $\T^2$.
See \cite{OTh2} for the construction of the associated NLW dynamics.

Let us now discuss the focusing Hartree Gibbs measure 
in the three-dimensional setting.
In~\cite{BO97}, Bourgain studied
the construction of the Gibbs measure 
for the Hartree NLS~\eqref{NLS1} on~$\T^3$. 
In the focusing case, he constructed the Gibbs measure
with a Wick-ordered $L^2$-cutoff
(for complex-valued $u$):
\begin{align}
d\rho(u) = Z^{-1} \ind_{\{\int_{\T^3} :\,|u|^2: \, dx\leq K\}} \, e^{\frac 14 \int_{\T^3} (V*:|u|^2:)\,  :|u|^2 :\, dx} d\mu(u)
\label{Gibbs7}
\end{align}

\noi
for $\be > 2$.
As in the two-dimensional case, 
such a measure is not suitable  for studying the NLW or heat dynamics
due to the non-conservation of the $L^2$-norm. 
Following Bourgain's proposition~\eqref{Gibbs6}
in the two-dimensional case
\cite{BO95}, 
we consider the following  Hartree Gibbs measure
on $\T^3$ in the focusing case ($\s > 0$):
\begin{align}
d\rho(u) = Z^{-1}
 e^{\frac \s4 \int_{\T^3} (V*:u^2:) :u^2 :\, dx - A
\big|\int_{\T^3} :\,u^2: \, dx\big|^\g} d \mu(u)
\label{Gibbs8}
\end{align}

\noi
for some suitable $A, \g > 0$.
Thus, 
we replace $R_N$ in \eqref{K1} by 
\begin{align}
\begin{split}
\RR_N (u)
&= \frac \s 4 \int_{\T^3} (V \ast :\! u_N^2 \!:) :\! u_N^2 \!: dx
 - A \, \bigg| \int_{\T^3} :\! u_N^2 \!: dx\bigg|^\g - \frac \s2 \al_N
\end{split}
\label{K2}
\end{align}

\noi
and 
define the truncated renormalized Gibbs measure $\rho_N$ by 
\begin{align}
d \rho_N (u) = Z_N^{-1} e^{\RR_N(u)} d \mu(u).
\label{GibbsN1}
\end{align}

\noi
Then, we have the following result in the focusing case.

\begin{theorem}[focusing case] \label{THM:Gibbs2}
Let $\s> 0$ and $V$ be the Bessel potential of order $\be>1$.
Then,  for any $A > 0$ and $\g > 0$, 
 $\RR_N $ defined in \eqref{K2}
converges to some limit $\RR$
in $L^p(\mu)$.

\smallskip

\noi
\textup{(i)}
Given   $\be>2$, let 
 $\max \big(\frac{\be+1}{\be-1},2\big) \le \g <3$, 
 with $\g > 2$ when $\be = 3$.
Then, 
given any finite $ p \ge 1$, 
there exists   $A  =  A(p) > 0$ such that 
\begin{equation}
\sup_{N\in \N} \Big\| e^{\RR_N(u)}\Big\|_{L^p(\mu)}
\leq C_p  < \infty
\label{exp3}
\end{equation}

\noi
for some  $C_p > 0$.
In particular,  we have
\begin{equation}
\lim_{N\rightarrow\infty}e^{ \RR_N(u)}=e^{\RR(u)}
\qquad \text{in } L^p(\mu).
\label{exp4}
\end{equation}

\noi
As a consequence, 
the truncated renormalized Gibbs measure $\rho_N$ in \eqref{GibbsN1} converges, in the sense of \eqref{exp4}, 
 to the focusing Hartree Gibbs measure $\rho$ given by
\begin{align}
d\rho(u)= Z^{-1} e^{\RR(u)}d\mu(u).
\label{Gibbs9}
\end{align}

\noi
Furthermore, 
the resulting Gibbs measure $\rho$ is equivalent 
to the base massive Gaussian free field~$\mu$.

\smallskip

\noi
\textup{(ii) (non-normalizability).}
Let   $1 < \be <  2$.
Then, for any $A > 0$ and $\g> 0$, we have
\begin{align}
\sup_{N \in \N}  \E_{\mu} \Big[e^{\RR_N (u)} \Big] = \infty.
\label{Gibbs10}
\end{align}

\noi
In particular, the focusing Hartree Gibbs measure $\rho$ in \eqref{Gibbs9}
can not be defined as a probability measure for $1< \be < 2$.

\smallskip

\noi
\textup{(iii) (critical case).}
Let $\be = 2$.
Then, there exist $\s_1 \geq \s_0 > 0$ such that 

\smallskip

\begin{itemize}
\item[\textup{(iii.a)}] \textup{(strongly nonlinear regime)}. For $\s > \s_1$, the 
 focusing Hartree Gibbs measure $\rho$ in \eqref{Gibbs9}
 is not normalizable in the sense of \eqref{Gibbs10}
 for any $A> 0$ and $\g > 0$.

\smallskip

\item[\textup{(iii.b)}] 
\textup{(weakly nonlinear regime)}. 
For $0 < \s < \s_0$, 
then by choosing $\g = 3$ and $A = A(\s) >0$ sufficiently large, 
we can construct the 
 focusing Hartree Gibbs measure $\rho$ in \eqref{Gibbs9}
 as in Part \textup{(i)}. In particular, \eqref{exp3} and \eqref{exp4}
 hold with a restricted range $1 \le p <  p(\s)$ in this case.
\end{itemize}

\end{theorem}

We present the proof of  Theorem
\ref{THM:Gibbs2}
in Section \ref{SEC:Gibbs}.
As in the defocusing case, 
we prove Theorem \ref{THM:Gibbs2}, 
using  the variational approach  by Barashkov and Gubinelli in \cite{BG}.
In the focusing case, 
the potential energy 
for the drift $\Dr$ appears with the $-$ sign
and we need the lower bound $\g\ge \frac{\be+1}{\be-1}$ to control this part.
See \eqref{YY14} below.
Furthermore,   in the non-endpoint case $\be > 2$, the upper bound $\g<3$ essentially ensures that $|\int_{\T^3} \Dr^2 dx|^\g$ is the leading part of the second term on the right-hand side of \eqref{K2}.
See Lemma \ref{LEM:Dr5} below.
In the critical case $\be = 2$
under the weakly nonlinear assumption ($0 < \s < \s_0$), 
the Gibbs measure construction requires  a more refined argument.
See Subsection~\ref{SUBSEC:foc3}.

Theorem \ref{THM:Gibbs2}
shows that our Gibbs measure construction
in the focusing case 
is sharp.
Our argument also shows that Bourgain's 
construction \cite{BO97} of the focusing Hartree Gibbs measure~\eqref{Gibbs7}
for $\be > 2$
is also sharp modulo the endpoint case $\be = 2$, 
where an analogous  dichotomy\,/\,phase transition  follows as a corollary to Theorem \ref{THM:Gibbs2}\,(iii).
See Remark \ref{REM:end1}.

Let us consider the following truncated Gibbs measure
with a Wick-ordered $L^2$-cutoff:
\begin{align}
d \wt \rho_N (u) = \wt Z_N^{-1} \ind_{\{ |\int_{\T^3} \, : u_N^2 :\, dx | \le K\}}
 e^{\s R_N(u)} d \mu(u), 
\label{C1}
\end{align}

\noi
where $\s > 0$ and $R_N(u)$ is as in \eqref{K1}.
Then, by 
noting that
\begin{align}
 \ind_{\{|x| \le K\}} \le \exp\big( -  A |x|^\gamma\big) \exp\big(A K^\g\big)
\label{pa00}
\end{align}

\noi
for any $x \in \R$, $K>0$, $\g>0$, and $A>0$, 
the uniform integrability \eqref{exp3} in Theorem \ref{THM:Gibbs2}
implies
\begin{equation*}
\sup_{N\in \N} \Big\|  \ind_{\{ |\int_{\T^3} \, : u_N^2 :\, dx | \le K\}}
 e^{\s R_N(u)}\Big\|_{L^p(\mu)}
\leq \wt C_p  < \infty
\end{equation*}

\noi
for $\be > 2$ or $\be = 2$ with sufficiently small $\s > 0$.
A modification of the proof of Theorem \ref{THM:Gibbs2}
yields convergence of the truncated Gibbs measure $\wt \rho_N$ in \eqref{C1}
to the limiting Gibbs measure
\begin{align}
d \wt \rho (u) =  Z_N^{-1} \ind_{\{ |\int_{\T^3} \, : u^2 :\, dx | \le K\}}
 e^{\s R(u)} d \mu(u), 
\label{C2}
\end{align}

\noi
in the sense of convergence of the truncated density 
$  \ind_{\{ |\int_{\T^3} \, : u_N^2 :\, dx | \le K\}}
 e^{\s R_N(u)}$, 
analogous to~\eqref{exp4}.

Our proof of the non-normalizability in Theorem \ref{THM:Gibbs2}
is in fact based on
showing non-normalizability of 
the focusing Hartree Gibbs measure $\wt \rho$ 
with a Wick-ordered $L^2$-cutoff in~\eqref{C2}.
See Proposition \ref{PROP:Gibbs5}.\footnote{While the proof of Proposition \ref{PROP:Gibbs5}
works only for sufficiently large $K \gg 1$, it is possible to modify the argument
so that the conclusion of Proposition \ref{PROP:Gibbs5} holds for any $K > 0$.
See Remark \ref{REM:OS}.}
Our main strategy
for proving non-normalizability
of $\wt \rho$ in~\eqref{C2}
is inspired by a recent work by 
Weber and the third author \cite{TW}
on the non-construction of the Gibbs measure 
for the focusing cubic NLS on the real line,
giving an alternative proof of Rider's result \cite{Rider},  
and 
is also based on the variational formulation due to Barashkov and Gubinelli \cite{BG}.
For this approach, 
we need to construct a drift $\Dr$ which achieves the desired divergence.
See Remark \ref{REM:choice} below.
The lower threshold $\be = 1$ in Theorem~\ref{THM:Gibbs2}\,(ii) 
naturally appears due to the necessity of a further renormalization
for $\be \le 1$
(required even in the defocusing case).
See Remark \ref{REM:ren}.
We expect that once we endow with a proper renormalization, 
the non-normalizability result may be extended for 
lower values of $\be \le 1$.
We  point out that a similar argument yields
the exact analogue of Theorem \ref{THM:Gibbs2}
for the focusing Hartree Gibbs measure in \eqref{Gibbs7}
with an Wick-ordered $L^2$-cutoff (but without an absolute value
on the Wick-ordered $L^2$-norm), where 
we introduce a general  coupling constant $\s >0$
as in \eqref{C1} and \eqref{C2}. 
See Remark \ref{REM:end1}.
Lastly, we also mention  related works
\cite{LRS, BS, Rider, BB14, OST,OOTolo, OS}
on the non-normalizability (and other issues)
for focusing Gibbs measures.

\begin{remark}\label{REM:Gibbs2}\rm
(i) 
While we stated 
Theorem \ref{THM:Gibbs2} for the Bessel potential, 
the Gibbs measure construction holds
for any Hartree potential 
 $V$, satisfying 
\begin{align} \label{Vb}
| \ft V(n)| \les \jb{n}^{-\b}
\end{align}
for $n \in \Z^3$
and 
the non-normalizability holds
for any Hartree potential 
 $V$, satisfying 
$ \ft V(n) \ges \jb{n}^{-\b}$
for $n \in \Z^3$.

\smallskip

\noi
(ii) In the two-dimensional case, 
the focusing Hartree Gibbs measure $\rho$ in \eqref{Gibbs8}
(also  $\rho$ in~\eqref{Gibbs7} with a Wick-ordered $L^2$ cutoff)
can be easily constructed for $\be > 0$ 
(and suitable $\g > 2$) via the variational argument.
When $\be = 0$, it is not normalizable in view of the result~\cite{BS} by Brydges and Slade.
See also \cite{OS}.

\smallskip

\noi
(iii)  In \cite{OQ}, Quastel and the first author studied
the construction of the focusing Gibbs measure
on the one-dimensional torus $\T$, 
with a specified $L^2$-norm $\ind_{\{\int_\T u^2 dx = K\}}$
(and a specified momentum).
It is of interest to investigate the construction 
(in particular, non-normalizability) of 
the focusing Hartree Gibbs measure on $\T^3$
with a specified Wick-ordered $L^2$-norm:
\begin{align*}
d \ft \rho (u) =  Z_N^{-1} \ind_{\{ |\int_{\T^3} \, : u^2 :\, dx | = K\}}
 e^{\s R(u)} d \mu(u).
\end{align*}

\end{remark}

\begin{remark}\rm
When $\be < 2$ (or $\be =2$ with $\s \gg 1$), Theorem \ref{THM:Gibbs2}
states that the focusing Hartree Gibbs measure is not normalizable.
A natural question may be then 
to wonder if it is possible to find diverging constants $C_N \to \infty$
such that $e^{\RR_N(u) - C_N}$
remains uniformly integrable with respect to the massive Gaussian free field $\mu$.
In view of the convergence of $\RR_N$
 in $L^p(\mu)$ stated in Theorem \ref{THM:Gibbs2}, 
we see that $e^{\RR_N(u)}$ converges in measure.
This in turn implies
that  $e^{\RR_N(u) - C_N}$ converges in measure to $0$,
showing that there is no hope to find a good candidate for the limiting focusing Hartree Gibbs measure
as a probability measure which is absolutely continuous
with respect to the base Gaussian free field
in this case.
Furthermore, by slightly modifying the proof of 
Theorem 1.8\,(ii) in \cite{OOTolo} (see also Proposition 4.4 in \cite{OOTolo}), 
we can also show\footnote{Strictly speaking, in order to prove this non-convergence claim, 
we need to modify our frequency projector
(projecting onto  a ball $\{|n|\le N\}$)
 to that onto a cube $[-N, N]^3$ as in \cite{OOTolo}.} that,
 as probability measures on $\C^{-\frac 12-}(\T^3)$,  the truncated focusing Hartree Gibbs measures $\rho_N$ in \eqref{GibbsN1}
do not converges to any weak limit, not even up to any subsequence.
\end{remark}

\begin{remark}\label{REM:Gibbs2a}
\rm

Let $\be \ge 2$.
Define
 the renormalized energy:
\begin{align} 
E^\sharp( u)
= \frac 12 \int_{\T^3} |\jb{\nabla} u|^2 dx - \RR(u), 
\label{Energy0a}
\end{align}

\noi
where $\RR(u)$ is the limit of $\RR_N(u)$.
Then, as in Remark \ref{REM:Gibbs1a}, 
we can also write $\rho$ in \eqref{Gibbs8} formally as 
\begin{align*} 
d \rho = Z^{-1} e^{-E^\sharp ( u)} d  u.
\end{align*}

\noi
Similarly, by defining the renormalized energy 
for SdNLW \eqref{SNLW0} and NLW \eqref{NLW1} by 
\begin{align} 
\mathcal{E}^\sharp (\vec u)
= \frac 12 \int_{\T^3} |\jb{\nabla} u|^2 dx + \frac 12 \int_{\T^3} (\dt u)^2 dx - \RR(u), 
\label{Energy1a}
\end{align}

\noi
we can write the 
focusing Hartree Gibbs measure $\rhoo = \rho\otimes \mu_0$ on a vector $\vec u = (u, \dt u)$ 
as
\begin{align} 
d \rhoo = Z^{-1} e^{-\mathcal{E}^\sharp ( \vec u)} d  \vec u.
\label{Gibbs11}
\end{align}

\noi
In the focusing case, 
the second term in \eqref{K2}
introduces an extra term
for the resulting equations.
See \eqref{SNLWA2} and 
\eqref{SNLH}.

\end{remark}

\section{Invariant dynamics for Hartree  SdNLW}
\label{SEC:NLW}

In this section, we consider the canonical stochastic quantization
for the Hartree Gibbs measure constructed in Theorems \ref{THM:Gibbs}
and \ref{THM:Gibbs2}
and describe our strategy for constructing global-in-time invariant Gibbs dynamics.

\subsection{Main results}
\label{SUBSEC:1.3}

Let $\rhoo$ be the focusing Hartree Gibbs measure ($\s> 0$) constructed
in Theorem~\ref{THM:Gibbs2}.
As pointed out in Remark \ref{REM:Gibbs2a}, 
the energy for $\rhoo$ is given by $\mathcal{E}^\sharp(u)$ in~\eqref{Energy1a}.
By considering the Langevin equation,
i.e.~\eqref{SNLW00} with $\mathcal{E}$ replaced by $\mathcal{E}^\sharp$, 
we obtain the following 
focusing Hartree  SdNLW:
\begin{align}
\dt^2 u + \dt u + (1 -  \Dl)  u - \s(V \ast \,:\! u^2 \!:\,)u + M_\g (\,:\! u^2 \!:\,) u = \sqrt{2} \xi,
\label{SNLWA2}
\end{align}

\noi
where $M_\g$ is defined by 
\begin{align}
M_\g (w) := 2A\g \bigg| \int_{\T^3} w dx \bigg|^{\g-2} \int_{\T^3} w dx
\label{focusnon}
\end{align}

\noi
and 
 $:\! u^2 \!:$ denotes the Wick renormalization of $u^2$.\footnote{In order to give a proper meaning to 
 $:\! u^2 \!:$, we need to assume a structure on $u$ (see \eqref{decomp3}).
 We postpone this discussion to the next subsection.} 
The last term $M_\g (:\! u^2 \!:) u$ on the left-hand side of \eqref{SNLWA2}
appears due to the taming via a power of the Wick-ordered $L^2$-norm
in~\eqref{Gibbs8} and~\eqref{K2}.
Given $N \in \N$, we also consider the truncated focusing Hartree SdNLW: 
\begin{align}
\begin{split}
\dt^2 & u_N + \dt u_N  + (1 -  \Dl)  u_N  \\
& 
-\s  \pi_N \big( (V \ast  :\! (\pi_N u_N)^2 \!:\, ) \pi_N u_N \big)
+  M_\g (\,:\! (\pi_N u_N)^2 \!:\,) \pi_N u_N 
= \sqrt{2} \xi, 
\end{split}
\label{SNLWA2r}
\end{align}

\noi
where 
$:\! (\pi_N u_N)^2 \!:  \, = 
(\pi_N u_N)^2 -\s_N.$
Our main goal here is to  construct invariant Gibbs dynamics 
for the focusing Hartree SdNLW \eqref{SNLWA2} as a limit of the truncated dynamics  \eqref{SNLWA2r}.

In the defocusing case ($\s < 0$), 
the energy for the Gibbs measure (for $\be > 1$) is given by $\mathcal{E}^\flat(u)$
in \eqref{Energy1}, giving rise to the following 
defocusing Hartree SdNLW:
\begin{align}
\dt^2 u + \dt u + (1 -  \Dl)  u  -\s (V \ast :\! u^2 \!:)u
= \sqrt{2} \xi
\label{SNLWA}
\end{align}

\noi
and its truncated version:
\begin{align}
\dt^2 u_N + \dt u_N + (1 -  \Dl)  u_N  
-\s  \pi_N \big( (V \ast  :\! (\pi_N u_N)^2 \!:\,) \pi_N u_N \big)
= \sqrt{2} \xi
\label{SNLWAr}
\end{align}
for $N \in \N$.

\begin{theorem}\label{THM:GWP}
Let $V$ be the Bessel potential of order $\be$ with

\medskip

\hspace{8mm}
\textup{(i)}\  $\be \ge 2$
 in the focusing case  \qquad and \qquad
\textup{(ii)} \  $\be > 1$  in the defocusing case.

\medskip

\noi
In the focusing case with $\be = 2$, 
we also assume that  $\s > 0$ is sufficiently small.

\smallskip

\noi
\textup{(i)
(focusing case)}.
Let  $A > 0$ be sufficiently large
and $\g > 0$ satisfy 
 $\max \big(\frac{\be+1}{\be-1},2\big) \le \g <3$
 with $\g > 2$ when $\be = 3$.
Then, the focusing Hartree SdNLW \eqref{SNLWA2} is almost surely globally well-posed
with respect to the random initial data distributed
by the Gibbs measure $\rhoo = \rho\otimes \mu_0$ in \eqref{Gibbs11}.
Furthermore, $\rhoo$ is invariant under the resulting dynamics.

More precisely,
there exists a non-trivial stochastic process $(u, \dt u) \in C(\R_+; \H^{-\frac 12 -\eps}(\T^3))$ for any $\eps>0$ such that,
given any $T>0$, the solution $(u_N, \dt u_N)$ to the truncated Hartree SdNLW~\eqref{SNLWA2r}
with the random initial data distributed by the truncated Gibbs measure 
$\rhoo_N = \rho_N \otimes \mu_0$, where $\rho_N$ is as in~\eqref{GibbsN1}, 
converges to $(u, \dt u)$ in $C([0,T]; \H^{-\frac 12-\eps}(\T^3))$.
Furthermore, 
we have $\Law\big((u(t), \dt u(t))\big) = \rhoo$ for any $t \in \R_+$.

\smallskip

\noi
\textup{(ii) (defocusing case)}
Let $\be > 1$.
Then, the corresponding results from Part \textup{(i)}
hold for the defocusing Hartree SdNLW \eqref{SNLWA}, 
its truncated version  \eqref{SNLWAr}, 
and the Gibbs measure $\rhoo$ in~\eqref{Gibbs2x}.

\end{theorem}

In view of Theorem \ref{THM:Gibbs2}, 
Theorem \ref{THM:GWP}\,(i) on the focusing case is sharp.
On the other hand, the threshold $\be = 1$ in 
Theorem \ref{THM:GWP}\,(ii) is by no means sharp.
As we saw in Theorem \ref{THM:Gibbs}
on the Gibbs measure construction, 
we need to introduce another renormalization to go below $\be = 1$.
Since our main goal in this paper is to obtain a sharp result
in the focusing case, we limit ourselves only to the range $\be > 1$
in the defocusing case, where the same renormalization as
in the focusing case suffices.

The main task in proving Theorem \ref{THM:GWP}
is the construction of local-in-time solutions.
Our  strategy for constructing local-in-time dynamics
is to adapt the  paracontrolled approach
in the hyperbolic\,/\,dispersive setting 
as in \cite{GKO2}, where the quadratic SNLW on $\T^3$ was studied.
By viewing 
the cubic Hartree nonlinearity $(V \ast :\! u^2 \!:)u$ 
as iterated  bilinear interactions,\footnote{In \cite{BO97}, 
Bourgain used this view point in studying the cubic Hartree NLS \eqref{NLS1}.} 
the exact paracontrolled operators
used in \cite{GKO2} appear in the study of 
the cubic Hartree SdNLW~\eqref{SNLWA2} and \eqref{SNLWA}.
We, however,  point out that, in order to treat the ill-defined product $:\! u^2 \!:$
in $(V \ast :\! u^2 \!:)u$ (see also $M_\g (\,:\! u^2 \!:\,)u$ in \eqref{SNLWA2}), 
the paracontrolled analysis in \cite{GKO2} is not sufficient.
In order to overcome this difficulty, 
we view the ill-defined (resonant) product (see \eqref{Res1} below)
as a new unknown
and  rewrite the equation into a system for {\it three} unknowns.
(Note that in \cite{GKO2}, the resulting system
was for two unknowns.)
In the next subsection, we describe the basic setup
of our paracontrolled approach.

Once we establish local well-posedness, 
we  adapt Bourgain's invariant measure argument \cite{BO94, BO96}
to the stochastic PDE setting (as in \cite{GKOT, ORTz})
to prove   the desired almost sure global well-posedness
 and invariance of the Gibbs measure.
Due to the use of the paracontrolled structure
in the local-in-time analysis, 
however, we need to proceed with care, 
in particular in proving 
convergence of the truncated dynamics to the full dynamics
on any large time interval $[0, T]$, 
where we  make use of the paracontrolled structure
on a large time scale (i.e.~not locally in time).
See  Section~\ref{SEC:GWP}
for  details.

\begin{remark}\rm
Here, we used the sharp frequency cutoff $\pi_N$.
It is, however, possible to use 
regularization via a mollification
and show that 
the limiting process is independent of mollification kernels.
See \cite{GKO2} for a further discussion.
We also point out that there are certain approximations
which lead to a wrong (and even divergent) limit.
See \cite{OOTz} for such an example
in the context of the deterministic NLW with random initial data.

\end{remark}

\subsection{Paracontrolled approach: defocusing case}
\label{SUBSEC:1.4}

In this subsection, we 
go over a paracontrolled approach  in the simpler defocusing case
($\s < 0$).
Since a precise value of $\s < 0$ does not play any role, 
we set $\s = -1$.
Proceeding 
 in the spirit of 
\cite{CC, MW1, GKO2}, 
we 
transform the defocusing Hartree SdNLW \eqref{SNLWA} to a system of PDEs.
Unlike the previous works \cite{CC, MW1, GKO2}, 
the resulting  system (see \eqref{SNLW5} below) consists of  {\it three} equations.
We then state our local well-posedness result
 of the resulting system.
The focusing case is treated in the next subsection.

The main difficulty in studying Hartree SdNLW \eqref{SNLWA}
comes from the roughness of the space-time white noise.
This is already manifested at the level of the linear equation.
Let $\Psi$  be the solution to the following linear stochastic damped wave equation:
\begin{align*}
\begin{cases}
\dt^2 \Psi + \dt\Psi +(1-\Dl)\Psi  = \sqrt{2}\xi\\
(\Psi,\dt\Psi)|_{t=0}=(\phi_0,\phi_1),
\end{cases}
\end{align*}

\noi
where  $(\phi_0, \phi_1) = (\phi_0^\o, \phi_1^\o)$ 
is a pair of 
 the Gaussian random distributions with 
$\Law( (\phi_0^\o, \phi_1^\o)) = \muu = \mu_1 \otimes \mu_0$.
Define the linear damped wave propagator $\D(t)$ by 
\begin{equation} 
\D(t) = e^{-\frac t2 }\frac{\sin\Big(t\sqrt{\tfrac34-\Dl}\Big)}{\sqrt{\tfrac34-\Dl}}, 
\label{D2a}
\end{equation}

\noi
viewed as a Fourier multiplier operator.
By setting
\begin{align}
\jbb{n} = \sqrt{\frac 34+|n|^2},
\label{jbbn}
\end{align}

\noi
we have
\begin{align}
\D(t) f  =  \sum_{n\in \Z^3}  \ft \D_n(t) 
 \ft f (n) e_n
 := 
\sum_{n\in \Z^3} 
e^{-\frac t2}\frac{\sin (t \jbb{n})}{\jbb{n}}
 \ft f (n) e_n.
\label{W3}
\end{align}

\noi
Then, the stochastic convolution $\Psi$ can be expressed as 
\begin{equation}
\Psi (t) =  \dt\D(t)\phi_0 +  \D(t) (\phi_0+\phi_1) + \sqrt 2 \int_0^t\D(t-t')dW(t'), 
\label{W2}
\end{equation}

\noi
where  $W$ denotes a cylindrical Wiener process on $L^2(\T^3)$:
\begin{align}
W(t)
 = \sum_{n \in \Z^3} B_n (t) e_n
\label{W1}
\end{align}

\noi
and  
$\{ B_n \}_{n \in \Z^3}$ 
is defined by 
$B_n(t) = \jb{\xi, \ind_{[0, t]} \cdot e_n}_{ x, t}$.
Here, $\jb{\cdot, \cdot}_{x, t}$ denotes 
the duality pairing on $\T^3\times \R$.
As a result, 
we see that $\{ B_n \}_{n \in \Ld_0}$ is a family of mutually independent complex-valued\footnote
{In particular, $B_0$ is  a standard real-valued Brownian motion.} 
Brownian motions conditioned so that $B_{-n} = \cj{B_n}$, $n \in \Z^3$. 
Note that we have, for any $n \in \Z^2$,  
 \[\text{Var}(B_n(t)) = \E\Big[
 \jb{\xi, \ind_{[0, t]} \cdot e_n}_{x, t}\cj{\jb{\xi, \ind_{[0, t]} \cdot e_n}_{x, t}}
 \Big] = \|\ind_{[0, t]} \cdot e_n\|_{L^2_{x, t}}^2 = t.\]

\noi
It is easy to see that $\Psi$ almost surely lies in 
$C (\R_+;W^{-\frac{1}{2}-\eps, \infty}(\T^3))$
for any $\eps > 0$.
See Lemma~\ref{LEM:stoconv} below.

Given $N \in \N$, we define the truncated stochastic convolution  $\Psi_N$ by
\begin{align}
\Psi_N &= \pi_N \Psi,
\label{so4a}
\end{align}

\noi
where $\pi_N$ is the (spatial) frequency projection defined  in \eqref{pi}.
Then, for each fixed $(x,t) \in \T^3 \times \R_+$, 
 a direct computation with \eqref{W3} and \eqref{W2} shows that 
 the random variable $\Psi_N(x, t)$ is a mean-zero real-valued Gaussian random variable with variance
\begin{align}
\s_N &  = \E \big[ \Psi_N (x,t)^2\big]
 =  \sum_{|n| \le N} \frac{1}{ \jb{n}^2}  \sim  N
\to \infty, 
\label{sigma1a}
\end{align}

\noi
as $N \to \infty$
(which agrees with $\s_N$ defined in \eqref{sigma1}).
We then define the truncated Wick power $:\! \Psi^2_N \!:$ by 
\begin{align}
:\! \Psi^2_N \!: & = (\Psi_N)^2 - \s_N.
\label{so4b}
\end{align}

\noi
A standard computation shows that 
$:\! \Psi^2 \!: \, = \lim_{N \to \infty} :\! \Psi^2_N \!:$ 
belongs to $C([0,T];W^{-1 -\eps,\infty}(\T^3))$ almost surely for $\eps>0$.
See Lemma~\ref{LEM:stoconv} below.

\medskip

In the following, we keep our discussion at a formal level\footnote{In the following, 
we directly work on \eqref{SNLWA}.
A rigorous treatment, however, needs to start with the truncated equation 
\eqref{SNLWAr} and take a limit $N \to \infty$.}
 and only  discuss spatial regularities (= differentiability) of various objects
without worrying about precise spatial Sobolev spaces that they belong to.
We also use the following ``rules'':\footnote{In the remaining part of the paper, 
we will justify these rules.}

\begin{itemize}

\item A product of functions of regularities $s_1$ and $s_2$
is defined if $s_1 + s_2 > 0$.
When $s_1 > 0$ and $s_1 \geq s_2$, the resulting product has regularity $s_2$.

\item A product of stochastic objects (not depending on the unknown)
is always well defined, possibly with a renormalization.
The product of stochastic objects of regularities $s_1$ and $s_2$
has regularity $\min( s_1, s_2, s_1 + s_2)$.

\end{itemize}

\smallskip

We now write $u$ in the first order expansion as in \cite{McKean, BO96, DPD2}:
\begin{align}
u = \Psi+ v.
\label{decomp3}
\end{align}

\noi
Then, it follows from \eqref{SNLWA} and \eqref{decomp3} that $v$ satisfies
\begin{align}
\begin{split}
(\dt^2 + \dt +1 -  \Dl)  v 
&  =  - \Big( V \ast :\!(v+\Psi)^2\!:  \Big) (v+\Psi) \\
& = - \Big( V \ast (v^2 + 2 v \Psi + :\! \Psi^2 \!:\, ) \Big) v
- \Big( V \ast (v^2 + 2 v \Psi + :\! \Psi^2 \!:\, ) \Big) \Psi.
\end{split}
\label{SNLW2}
\end{align}

\noi
The second  term on the right-hand side
has regularity\footnote{Hereafter,   we use $a-$ 
(and $a+$) to denote $a- \eps$ (and $a+ \eps$, respectively)
for arbitrarily small $\eps > 0$.
If this notation appears in an estimate, 
then  an implicit constant 
is allowed to depend on $\eps> 0$ (and it usually diverges as $\eps \to 0$).}
$-\frac 12 -$, inheriting the worse regularity of $\Psi$.
In view of one degree of smoothing under the damped wave operator, 
we  expect\footnote{Here, we do not expect to have any multilinear smoothing. See Remark \ref{REM:no} below.} $v$ to have regularity at most $\frac 12- = (-\frac 12-) + 1$.
Then, the product 
$v \Psi$ is {\it not} well defined since $(\frac 12-) + (-\frac 12-) < 0$.

\begin{remark} \rm
Note that the second term 
on the right-hand side of \eqref{SNLW2} 
(ignoring $v\Psi$)
has regularity $-\frac 12-$ even if $\be \gg 1$.
Namely, the smoothing property of the Bessel potential $V$
does {\it not} improve the regularity of this term.
Furthermore, we point out, when $\be > 1$, 
the purely stochastic term 
 $(V \ast :\! \Psi^2 \!:) \Psi$ and 
 the terms   $\big( V * v^2  \big) \Psi$, involving the unknown $v$,  have the same regularity $-\frac 12-$.
This makes it  difficult to apply a higher order expansion as in \cite{GKO2, OPTz}, since the worst part depends not only on $\Psi$ but also on the unknown $v$.

\end{remark}

We now proceed with the paracontrolled calculus.
The main ingredients for the paracontrolled approach in the parabolic setting, 
introduced by Gubinelli, Imkeller, and Perkowski~\cite{GIP}, 
are (i) a paracontrolled ansatz 
and (ii) commutator estimates.
As pointed out in \cite{GKO2}, 
however, 
there seems to be no smoothing for a certain relevant commutator
for the wave equation.
In order to overcome this difficulty, 
 Gubinelli, Koch, and the first author~\cite{GKO2}
 introduced the so-called paracontrolled operators
 (see \eqref{X3} and \eqref{X6} below) in studying SNLW
 with a quadratic nonlinearity.
While our nonlinearity is cubic, 
the presence of the Bessel potential makes it more convenient
to view it as iterated bilinear interactions
(as in the Schr\"odinger case by Bourgain \cite{BO97}).
As a result,  the   (essentially) same paracontrolled operators from \cite{GKO2}
will play an important role in our analysis.

In the following, 
the  paraproduct decomposition:
\begin{align}
fg 
= f\pl g + f \pe g + f \pg g
\label{para}
\end{align}

\noi
 plays an important role.
See Section \ref{SEC:2} for a precise definition.
The first term 
$f\pl g$ (and the third term $f\pg g$) is called the paraproduct of $g$ by $f$
(the paraproduct of $f$ by $g$, respectively)
and it is always well defined as a distribution
of regularity $\min(s_2, s_1+ s_2)$.
On the other hand, 
the resonant product $f \pe g$ is well defined in general 
only if $s_1 + s_2 > 0$.
We also use the notation $f\pge g := f\pg g + f\pe g$.

With this notation, we introduce our paracontrolled ansatz:\footnote{We say that a distribution $f$ is paracontrolled (by a given distribution $g$)
if there exists $f'$ such that
$ f= f' \pl g + h$, 
where $h$ is a ``smoother'' remainder.
See Definition 3.6 in \cite{GIP}
for a precise definition.
Formally speaking, 
via the decomposition \eqref{decomp3a}
with \eqref{SNLW3} and the regularity assumption
$0 < s_1 < s_2$, 
we are postulating 
$(\dt^2  +\dt + 1 - \Dl)v$ is paracontrolled by $\Psi$.
}
\begin{align}
v = X+ Y,
\label{decomp3a}
\end{align}

\noi
where $X$ and $Y$ satisfy
\begin{align}
\begin{split}
 (\dt^2 + \dt  +1 - \Dl) X  & =  - \Big( V \ast \big( (X+Y)^2 + 2 (X+Y) \Psi + :\! \Psi^2 \!: \big) \Big) \pl \Psi,
\label{SNLW3}
\end{split}\\
\begin{split}
 (\dt^2 + \dt +1  - \Dl) Y 
&  = - \Big( V \ast \big( (X+Y)^2 + 2 (X+Y) \Psi + :\! \Psi^2 \!: \big) \Big) (X+Y) \\
& \hphantom{XXX}
- \Big( V \ast \big( (X+Y)^2 + 2 (X+Y) \Psi + :\! \Psi^2 \!: \big) \Big) \pge \Psi.
\end{split}
\label{SNLW4}
\end{align}

\noi
In view of the paraproduct decomposition \eqref{para}, 
the right-hand side of the $X$-equation \eqref{SNLW3}
consists of the worst nonlinear terms in \eqref{SNLW2}.
We  postulate that  both $X$ and $Y$ have positive regularities $s_1$ and $s_2$, respectively,
with $0 < s_1 < s_2$.
If we ignore for now  the potentially ill-defined resonant products of the unknowns
with $\Psi$, 
then we expect that $X$ has regularity 
 $\frac 12- = (-\frac 12-)+1$ (at best).
In the second equation, 
the worst term is given by 
 the purely stochastic resonant product 
 \begin{align}
 (V \ast :\! \Psi^2 \!:\,) \pe \Psi
\label{Psi2}
\end{align}

\noi
  which has regularity $\be - \frac 32 - $.
 See Lemma \ref{LEM:IV} below.
Thus, we expect that $Y$ has regularity $\frac 12+$
when $\be > 1$ is close to $1$.
See Remark \ref{REM:no} below for a further discussion.

\begin{remark}\label{REM:no}\rm

We point out that there is no 
multilinear dispersive smoothing for \eqref{Psi2}
(and hence for $Y$).
This is due to the fact that the third term $Z_{13}$  on the right-hand side of  \eqref{Z13} 
in the proof of Lemma \ref{LEM:IV}
is  the linear solution (namely, the stochastic convolution)
with an explicit smoothing of order $\be-1$, coming from 
\[\sum_{\substack{n_2 \in \Z^3 \\ |n+n_2| \sim |n_2|\ne 0 }}
\ft V(n+n_2)  \jb{n_2}^{-2} \les \frac{1}{\jb{n}^{\be-1}}\]

\noi
where the inequality follows from 
Lemma \ref{LEM:SUM}.
Putting together with one degree of smoothing
coming from the Duhamel integral operator, 
we expect the regularity 
of $Y$ to be $(-\frac 12 -) + (\be - 1) + 1 = \be - \frac 12-$, 
which is $\frac 12+$ when $\be > 1$ is close to $1$.
\end{remark}

The main new feature  of our formulation \eqref{SNLW3} - \eqref{SNLW4},
when compared with the previous works \cite{CC, MW1, GKO2}, 
is that the first equation \eqref{SNLW3} (for $X$)
 is {\it nonlinear} in the unknowns $X$ and $Y$, 
while the paracontrolled parts 
in \cite{CC, MW1, GKO2}
were {\it linear} in the unknowns. 
This makes our analysis different from  that in 
\cite{CC, MW1, GKO2}.
In these previous works, 
 the main difficulty 
was to make sense of the resonant product 
$\pe$ (for example $X\pe \Psi$ in \cite{GKO2})
in the  second equation \eqref{SNLW4} (for $Y$), 
which was overcome using the Duhamel formulation of the $X$-equation
(and then via the commutator estimates in the parabolic setting
and via the paracontrolled operators in the wave case \cite{GKO2}).

In our case, 
the resonant product with $\Psi$ in the
second term on the right-hand side of the second equation \eqref{SNLW4}
is not so much of an issue thanks
to the smoothing property of $V$.
On the other hand, 
we expect from \eqref{SNLW3} that 
$X$ has regularity $\frac 12-$
and thus   $X \pe \Psi$ is not well defined since 
 the sum of the regularities is negative.
Note that this resonant product 
$X \pe \Psi$  appears in both equations.
Furthermore, the smoothing of $V$ does not help the situation
since the (ill-defined) resonant product $X \pe \Psi$ appears inside
the convolution with $V$.
Our main new idea is to define the resonant product
\begin{align}
\text{``}\,\Res = X \pe \Psi\,\text{''}
\label{Res1}
\end{align}

\noi 
as a new unknown
and reduce to a system of three unknowns $(X, Y, \Res)$.
More precisely, we substitute the Duhamel formulation of
the $X$-equation \eqref{SNLW3}
into \eqref{Res1} and define $\Res$ by 
\begin{align}
\Res &=
- \I \Big( \big(V \ast (\Qxy  + 2\Res + :\! \Psi^2 \!:\,)\big) \pl \Psi\Big)\pe \Psi
\label{Res2}
\end{align}

\noi
where 
 $\I = (\dt^2 +\dt +1 - \Dl)^{-1}$ is the Duhamel integral operator
  given by 
\[\I F (t) = \int_0^t \D(t - t') F(t') dt'\]

\noi
 and 
$Q_{X, Y}$ denotes a good part of $:\!u^2\!: $, defined by 
\begin{align}
\Qxy =
(X+Y)^2 + 2 X \pl \Psi + 2 X \pg \Psi + 2 Y \Psi.
\label{Pxy}
\end{align}

\noi
Note that 
 all the terms in \eqref{Pxy} make sense for $0 < s_1 < \frac 12 < s_2$
 and that $\Qxy$ has (expected) regularity 
 $ - \frac 12-$.
Recalling
\eqref{decomp3}
and \eqref{decomp3a},  we have
\begin{align}
:\!u^2\!: \, \, = \Qxy + 2\Res \, + \! :\! \Psi^2 \!:.
\label{Pxy2}
\end{align}

Due to the paraproduct structure (with the high frequency part given by $\Psi$)
in the Duhamel integral operator $\I$ in \eqref{Res2}, 
we see that the resonant product in \eqref{Res2} is not well defined at this point.
In order to give a precise meaning to the right-hand side of \eqref{Res2}, 
we now recall the paracontrolled operators introduced in \cite{GKO2}.\footnote
{Strictly speaking, the paracontrolled operators introduced in \cite{GKO2}
are for the undamped wave equation. 
Since the local-in-time mapping property remains unchanged, 
we ignore this minor point.}
We point out that
in the parabolic setting, 
it is at this step where
one would introduce commutators  and exploit their smoothing properties.
For our dispersive problem, however, such an argument does not seem to work.
See \cite[Remark~1.17]{GKO2}.

Given a function $w $ of positive regularity on $\T^3 \times \R_+$, 
 define
\begin{align}
\begin{split}
 \If_{\pl}(w) (t)
:  \! &  =  \I (w\pl \Psi)(t) \\
& =       \sum_{n \in \Z^3}
e_n  \sum_{\substack{n =  n_1 +  n_2\\  |n_1| \ll |n_2|}}
\int_0^t e^{-\frac{t-t'}2} \frac{\sin ((t-t')\jbb{n})}{\jbb{n}}
\ft w(n_1, t')\,  \ft{ \Psi}(n_2, t') dt',
\end{split}
\label{X2}
\end{align}

\noi
where $\jbb{n}$ is as in  \eqref{jbbn}.
Here, $|n_1| \ll |n_2|$ signifies the paraproduct $\pl$
in the definition of $\If_{\pl}$.\footnote{For simplicity of the presentation, 
we use the less precise definitions of paracontrolled operators
in the remaining part of this introduction.
See
\eqref{XX2a},
\eqref{A0a}, and \eqref{A0b}
for the precise definitions of 
the paracontrolled operators $ \If_{\pl}^{(1)}$ 
and $\If_{\pl, \pe}$.}
As mentioned above, 
the regularity of 
$ \If_{\pl}(w)$ is (at best) $\frac 12-$
and thus the resonant product 
$ \If_{\pl}(w) \pe \Psi$ does not make sense
in terms of deterministic analysis.

Proceeding as in  \cite{GKO2}, 
we divide the paracontrolled operator $\If_{\pl}$
into two parts.
Fix small $\theta  > 0$.
Denoting by $n_1$ and $n_2$ the spatial frequencies
of $w$ and $\Psi$  in \eqref{X2}, 
we  define $ \If_{\pl}^{(1)}$ and $ \If_{\pl}^{(2)}$ 
as the restrictions of $\If_{\pl}$ onto $\{ |n_1| \ges |n_2|^{\theta}\}$
and  $\{ |n_1| \ll |n_2|^{\theta}\}$.
More concretely, we set 
\begin{align}
 \If_{\pl}^{(1)}(w)(t) 
:= \sum_{n \in \Z^3}
e_n  \sum_{\substack{n =  n_1 +  n_2\\    |n_2|^\theta \les |n_1| \ll |n_2|}}
\int_0^t e^{-\frac{t-t'}2} \frac{\sin ((t-t') \jbb{n})}{\jbb{n}}
\ft w(n_1, t')\,  \ft{\Psi}(n_2, t') dt'
\label{X3}
\end{align}

\noi
and $ \If_{\pl}^{(2)}(w)   := \If_{\pl}(w) - \If_{\pl}^{(1)}(w)$.
As for the first paracontrolled operator $ \If_{\pl}^{(1)}$, 
 the lower bound  $|n_1| \ges |n_2|^\theta$ 
 and the positive regularity of $w$
 allow us to prove 
a smoothing property
such that 
 the resonant product 
$ \If_{\pl}^{(1)}(w) \pe \Psi$ is well defined.
See Lemma~\ref{LEM:sto3}
below.

As noted in \cite{GKO2},
the second paracontrolled operator $ \If_{\pl}^{(2)}$
does not seem to possess a (deterministic) smoothing property.
One of the main novelty in \cite{GKO2}
was then to directly 
 study the  operator $\If_{\pl, \pe}$ defined by
\begin{align}
\begin{split}
\If_{\pl, \pe}(w) (t)
: \!  & =
  \If_{\pl}^{(2)}(w)\pe \Psi(t) \\
&  = \sum_{n \in \Z^3}e_n 
    \int_0^{t} 
\sum_{n_1 \in \Z^3}
\ft w(n_1, t') \A_{n, n_1} (t, t') dt', 
\end{split}
\label{X6}
\end{align}

\noi
where $\A_{n, n_1} (t, t')$ is given by 
\begin{align}
\A_{n, n_1} (t, t')
= \ind_{[0 , t]}(t') \sum_{\substack{n - n_1 =  n_2 + n_3\\ |n_1| \ll |n_2|^\theta \\ |n_1 + n_2|\sim |n_3|}}
 e^{-\frac{t-t'}{2}} \frac{\sin ( (t - t') \jbb{n_1+n_2})}{\jbb{n_1+n_2}}
   \ft{\Psi}(n_2, t')\,   \ft{\Psi}(n_3, t) .
\label{X7}
\end{align}

\noi
Here, 
the condition  $ |n_1 + n_2|\sim |n_3|$
is used to denote the  Fourier  multiplier corresponding 
to the resonant product $\pe$ in \eqref{X6}.
See \eqref{A0b} for a more precise  definition.

In \cite{GKO2}, 
by combining stochastic analysis and multilinear dispersion, 
Gubinelli, Koch, and the first author proved
the following almost sure boundedness
property
of the paracontrolled operator $\If_{\pl, \pe}$ defined in \eqref{X6}.
 Given Banach spaces $B_1$ and $B_2$, 
 we use $\L(B_1; B_2)$ to denote the space
 of bounded linear operators from $B_1$ to $B_2$.

\begin{lemma}\label{LEM:sto4}
Let  $ s_3 < 0$ and $T > 0$.
Then, there exist small $\theta = \theta (s_3) > 0$
and $\eps > 0$
such that
the paracontrolled operator $\If_{\pl, \pe}$ 
defined in \eqref{X6}
belongs to  the class:
\begin{align*}
\L_1(T) := \L(C([0, T]; L^2(\T^3) ) \cap C^1([0,T];H^{-1 -\eps}(\T^3))
\, ;\, 
 C([0, T]; H^{s_3}(\T^3) )),
\end{align*}
almost surely.

\end{lemma}

The kernel 
$\A_{n, n_1} (t, t')$ in \eqref{X7}
can be divided into two parts: a stochastic part
and a deterministic counter term.
See \eqref{X9} below.
In order to control a part 
of the deterministic counter term, 
the time differentiability of the input function $w$ was exploited in \cite{GKO2}.
Unfortunately, 
 Lemma \ref{LEM:sto4} is not suitable
for our purpose
due to the lack of differentiability in the range of $\L_1(T)$. 
One of the terms in 
 \eqref{Res2}, giving rise to $\Res$, is given by 
$\If_{\pl, \pe}(V*\Res)$.
Hence, we need to prove an almost sure mapping property
 with the same time differentiability for the domain and the range.
In Section \ref{SEC:po}, we prove the following proposition.

\begin{proposition}\label{PROP:sto4J}
Let  $ s_3 < 0$ and $T>0$.
Then, there exist small $\theta = \theta (s_3) > 0$
such that, for any finite $ q >1$, 
the paracontrolled operator $\If_{\pl, \pe}$
defined in \eqref{X6}
belongs to
\begin{align}
\L_2(q, T) := \L(L^q([0, T]; L^2(\T^3) )
\, ;\, 
 L^\infty([0, T]; H^{s_3}(\T^3))), 
\label{L1ast}
\end{align}

\noi
almost surely.
Furthermore the following tail estimate holds
for some  $C,c >0$:
\begin{align}
\PP \Big( \| \If_{\pl, \pe} \|_{\L_2 (q, T)} > \ld \Big)
\le
\begin{cases}
C \exp \big( -  \frac{\ld}{T^{c}} \big),  & \text{when }0<T\le1, \\
C T \exp ( -  \ld),  & \text{when }  T > 1
\end{cases}
\label{pote2}
\end{align}
for any $\ld\gg 1$.

If we define  the paracontrolled operator $\If_{\pl, \pe}^N$, $N \in \N$, 
by replacing 
$\Psi$ in \eqref{X6} and \eqref{X7}
with the truncated stochastic convolution $\Psi_N$ in \eqref{so4a}, 
then
the truncated paracontrolled operators $\If_{\pl, \pe}^N$ converge almost surely to $\If_{\pl, \pe}$ 
in $\L_2 (q, T)$.
Furthermore, the tail estimate~\eqref{pote2}
holds for the truncated paracontrolled operators $\If_{\pl, \pe}^N$ with the constants independent of $N\in \N$. 
\end{proposition}

We are now ready to present the resulting system for the three unknowns $(X, Y, \Res)$.
Putting together \eqref{SNLW3}, \eqref{SNLW4}, \eqref{Res2}, \eqref{Pxy2}, 
\eqref{X3}, and  \eqref{X6}, we arrive at the following system:
\begin{align}
\begin{split}
 (\dt^2 + \dt  +1 - \Dl) X  & =  - \Big( V \ast 
( \Qxy + 2\Res \, + \! :\! \Psi^2 \!: ) \Big) \pl \Psi,\\
 (\dt^2 + \dt +1  - \Dl) Y 
&  = - \Big( V \ast ( \Qxy + 2\Res \, + \! :\! \Psi^2 \!:) \Big) (X+Y) \\
& \hphantom{X}
- \Big( V \ast ( \Qxy + 2\Res \, + \! :\! \Psi^2 \!: ) \Big) \pge \Psi, \\
\Res 
&= - \If_{\pl}^{(1)}
 \big(V \ast (\Qxy  + 2\Res + :\! \Psi^2 \!:\,)\big)\pe \Psi\\
& \hphantom{X}
 - \If_{\pl, \pe}
 \big(V \ast (\Qxy  + 2\Res + :\! \Psi^2 \!:\,)\big), \\
(X, \dt X, Y, \dt Y, \Res)|_{t = 0} & = (X_0, X_1, Y_0, Y_1, 0).
\end{split}
\label{SNLW5}
\end{align}

\noi
Here, we included general initial data for $X$ and $Y$.
By  viewing the following random distributions and operator:
$ \Psi$,  $:\! \Psi^2 \!:$, $ (V \ast :\! \Psi^2 \!:) \pe \Psi$,
and $ \If_{\pl, \pe}$
as predefined deterministic data
with certain regularity properties, 
we prove the following local well-posedness 
 of the system~\eqref{SNLW5}.
Given $s \in \R$ and $T>0$, define $X^s(T)$ by 
\begin{align}
 X^{s}(T)
 & = C([0,T ];H^{s}(\T^3))\cap C^1([0,T]; H^{s-1}(\T^3)).
\label{M0}
\end{align}

\begin{theorem} \label{THM:1}
Let $V$ be the Bessel potential of order $\b> 1$.
Let $\frac 14<s_1 < \frac 12 < s_2<1$ and $-\frac 12 <s_3<0$ satisfy
\begin{align}
\be > -3s_1 + s_2 + 2.
\label{Z0}
\end{align}
Then,
there exist $\ta = \ta(s_3) > 0$ 
and $\eps = \eps(s_1, s_2, s_3, \be ) > 0$
such that 
if 
\begin{itemize}
\item   $\Psi $ is a distribution-valued function belonging to 
$C([0, 1]; W^{-\frac 12 - \eps, \infty}(\T^3))\cap C^1([0, 1]; W^{-\frac 32 - \eps, \infty}(\T^3))$, 

\smallskip
\item
$:\! \Psi^2 \!:
$ is a distribution-valued function belonging to $ C([0, 1]; W^{-1 - \eps, \infty}(\T^3))$,

\smallskip
\item
$(V \ast :\! \Psi^2 \!:) \pe \Psi
$ is a distribution-valued function belonging to $ C([0, 1]; W^{\be -\frac 32 - \eps, \infty}(\T^3))$, 

\smallskip

\item 
 the operator
 $\If_{\pl, \pe}$ 
belongs to  the class
 $\L_2\big(\frac 32, 1\big)$ in \eqref{L1ast}, 

\smallskip

\end{itemize}

\noi
then the system \eqref{SNLW5} is locally well-posed in 
$\H^{s_1}(\T^3) \times \H^{s_2}(\T^3) \times \{0\}$.
More precisely, 
given any $(X_0, X_1, Y_0, Y_1) \in \H^{s_1}(\T^3)\times \H^{s_2}(\T^3)$, 
there exists $T > 0$ 
such that there exists  a unique solution $(X, Y, \Res) $ to the 
defocusing Hartree SdNLW system \eqref{SNLW5} on $[0, T]$
in the class\textup{:}
\begin{align}
 Z^{s_1, s_2, s_3}(T)
 =  X^{s_1}(T) \times  X^{s_2}(T) \times 
 L^3([0, T ];H^{s_3}(\T^3)). 
\label{Z1}
\end{align}

\noi
Furthermore, the solution $(X, Y, \Res)$
depends  continuously 
on the enhanced data set\textup{:}
\begin{align}
\Xi = \big(X_0, X_1, Y_0, Y_1, 
 \Psi, \,  :\! \Psi^2 \!:,  \,  (V \ast :\! \Psi^2 \!:) \pe \Psi,
\,  \If_{\pl, \pe}\big)
\label{data1}
\end{align}

\noi
in the class\textup{:}
\begin{align*}
\mathcal{X}^{s_1, s_2, \eps}_T
& = \H^{s_1}(\T^3) \times 
\H^{s_2}(\T^3)\\
& \hphantom{X}
\times 
\big( C([0,T]; W^{-\frac 12 - \eps, \infty}(\T^3))\cap C^1([0, T]; W^{-\frac 32 - \eps, \infty}(\T^3))\big)\\
& \hphantom{X}
\times 
C([0,T]; W^{-1 - \eps, \infty}(\T^3))
\times C([0,T];W^{\be -\frac 32 -\eps,\infty}(\T^3))\\
& \hphantom{X}
\times
 \L_2\big(\tfrac 32, T\big).
\end{align*}

\end{theorem}

Note that, given $\be > 1$, the condition \eqref{Z0} is satisfied
by taking both $s_1$ and $ s_2$ sufficiently close to $\frac 12$.
Given the a priori regularities of the enhanced data, 
Theorem \ref{THM:1} follows
from the standard energy estimate for the damped wave equation
(see \eqref{EE1} below).
Namely, we do not need to rely on the Strichartz estimates
thanks to the strong smoothing of the Bessel potential~$V$.
See Section \ref{SEC:LWP} for the proof.

\begin{remark} \label{REM:LWP}\rm
(i) The choice of the temporal integrability 
$L^3_T$ for $\Res$
and $ \L_2\big(\tfrac 32, T\big)$
comes from  the focusing case presented in the next subsection.

\smallskip

\noi
(ii) For the sake of the well-posedness
of the system \eqref{SNLW5}, 
we considered 
general initial data $(X_0, X_1, Y_0, Y_1) \in \H^{s_1}(\T^3)\times \H^{s_2}(\T^3)$
in Theorem \ref{THM:1}.
However, in order to go back from the system~\eqref{SNLW5}
to the defocusing Hartree SdNLW \eqref{SNLWA}
with the identification \eqref{Res1}
(in the limiting sense), 
we need to set $(X_0, X_1) = (0, 0)$
since  the resonant product
of  
$ \dt\D(t)X_0 +  \D(t) (X_0+X_1)$ and $\Psi$ is not well defined
in general.
The same comment applies to Theorem \ref{THM:2}
in the focusing case.


\smallskip

\noi
(iii) In proving  the local well-posedness result 
of  the system \eqref{SNLW5} stated in Theorem \ref{THM:1}, 
we do not need to use the $C^1_TW^{-\frac32-\eps, \infty}_x$-norm
for the stochastic convolution $\Psi$.
However, we will need
the $C^1_TW^{-\frac32-\eps, \infty}_x$-norm for $\Psi$
in the globalization argument presented in Section~\ref{SEC:GWP}
and thus have included it in the hypothesis
of Theorem \ref{THM:1}
and the definition 
of the space $\mathcal{X}^{s_1, s_2, \eps}_T$.
See also \eqref{data3} and Remark \ref{REM:conv0}.
Furthermore, with this definition of the space $\mathcal{X}^{s_1, s_2, \eps}_T$, 
the map 
from an enhanced data set in \eqref{data1} (with $(X_0, X_1, Y_0, Y_1) = (0, 0, u_0, u_1)$)
to $(u, \dt u)$, where $ u = \Psi  + X + Y$ as in~\eqref{decomp3} and \eqref{decomp3a}
is  a continuous map
from $\mathcal{X}^{s_1, s_2, \eps}_T$
to $ C([0, T]; \H^{-\frac{1}{2}-\eps}(\T^3))$.

Consider 
the following defocusing Hartree SdNLW $(\s < 0$) with the truncated noise for $N \in \N$:
\begin{align}
\begin{cases}
\dt^2 u_N + \dt u_N + (1 -  \Dl)  u_N  
-\s   \big( (V \ast  :\! u_N^2 \!:\,) \, u_N \big)
= \sqrt{2}\pi_N \xi\\
(u_N, \dt u_N)|_{t= 0} = (u_0, u_1) + \pi_N (\phi_0^\o, \phi_1^\o)
\end{cases}
\label{SNLWn}
\end{align}

\noi
where $(u_0, u_1) \in \H^{s_2}(\T^3)$, 
$\Law( (\phi_0^\o, \phi_1^\o)) = \muu = \mu_1 \otimes \mu_0$, 
and 
$:\!  u_N^2 \!:  \, =  u_N^2 -\s_N$.
Then, 
together with the almost sure convergence
of the truncated enhanced data set:
\[ 
\Xi_N = \big(0, 0, u_0, u_1, 
 \Psi_N, \,  :\! \Psi^2_N \!:,  \,  (V \ast :\! \Psi^2_N \!:) \pe \Psi_N,
\,  \If_{\pl, \pe}^N\big)\]
 
 \noi
 in $\mathcal{X}^{s_1, s_2, \eps}_1$
 (see Lemmas \ref{LEM:stoconv} and  \ref{LEM:IV} and Proposition \ref{PROP:sto4J}), 
 the discussion above shows that 
 the solution $(u_N, \dt u_N)$ to \eqref{SNLWn}
 converges almost surely
 to some limiting process $(u, \dt u)$ in 
 $ C([0, T_\o]; \H^{-\frac{1}{2}-\eps}(\T^3))$, 
 where $T_\o$ is an almost surely positive stopping time, 
thus  yielding local well-posedness of 
the defocusing Hartree SdNLW \eqref{SNLWA}
in the usual sense in the study of singular stochastic PDEs.

The same comment applies to Theorem \ref{THM:2}
in the focusing case.

\end{remark}

\subsection{Focusing case}
\label{SUBSEC:1.5}

In the following, we briefly describe the required modification
to prove local well-posedness of 
the focusing Hartree SdNLW \eqref{SNLWA2} 
for $\be \geq  2$.
Since a precise value of $\s > 0$ does not play any role, 
we set $\s = 1$.
In the focusing case, 
we have  an extra term $M_\g(:\!u^2\!:) u$ in the equation.
From \eqref{focusnon}, 
\eqref{decomp3}, \eqref{decomp3a}, and \eqref{Pxy2}, we have
\begin{align}
M_\g(:\!u^2\!:) u
= M_\g(\Qxy + 2\Res \, + \! :\! \Psi^2 \!:) \Psi
+ M_\g(\Qxy + 2\Res \, + \! :\! \Psi^2 \!:)(X+Y).
\label{gam1}
\end{align}

\noi
Then, 
by including the first term on the right-hand side of \eqref{gam1}
in the $X$-equation and the second term in the $Y$-equation, we 
end up with the system:
\begin{align}
\begin{split}
 (\dt^2 + \dt  +1 - \Dl) X  
& =   \Big( V \ast 
( \Qxy + 2\Res \, + \! :\! \Psi^2 \!:) \Big) \pl \Psi\\
& \hphantom{X}
- M_\g(\Qxy + 2\Res \, + \! :\! \Psi^2 \!:) \Psi,\\
 (\dt^2 + \dt +1  - \Dl) Y
&  =  \Big( V \ast ( \Qxy + 2\Res \, + \! :\! \Psi^2 \!:) \Big) (X+Y) \\
& \hphantom{X}
+ \Big( V \ast ( \Qxy + 2\Res \, + \! :\! \Psi^2 \!: ) \Big) \pge \Psi\\
& \hphantom{X}
-M_\g(\Qxy + 2\Res \, + \! :\! \Psi^2 \!:)(X+Y),\\
\Res 
&=  \If_{\pl}^{(1)}
 \big(V \ast (\Qxy  + 2\Res + :\! \Psi^2 \!:\,)\big)\pe \Psi\\
& \hphantom{X}
+   \If_{\pl, \pe}
 \big(V \ast (\Qxy  + 2\Res + :\! \Psi^2 \!:\,)\big)\\
& \hphantom{X}
- \I\big( M_\g(\Qxy + 2\Res \, + \! :\! \Psi^2 \!:) \Psi\big) \pe \Psi, \\
(X, \dt X, Y, \dt Y, \Res)|_{t = 0} & = (X_0, X_1, Y_0, Y_1, 0).
\end{split}
\label{SNLW6}
\end{align}

\noi
Here, $\g$ is as in Theorem \ref{THM:Gibbs2}
and in particular, we have $ \g = 3$ when $\be = 2$.
The last term in the $\Res$-equation puts a restriction
on the temporal integrability for $\Res$.
By the energy estimate, 
we can place 
$M_\g(\Qxy + 2\Res \, + \! :\! \Psi^2 \!:) $ in $L^1([0, T])$
(ignoring the spatial regularity).
In order to perform a contraction argument, we need to save some time integrability
and thus need to place $|\int \Res \, dx |^{\g - 1}$ in $L^{1+}([0, T])$, 
namely, $\int \Res \,dx $ in $L^{2+ }([0, T])$ when $\g = 3$.
This explains the choice $L^3_T$ for $\Res$ in \eqref{Z1}.

In order to handle the last term in the $\Res$-equation, we also need to introduce the following stochastic term:
\begin{align}
\Ab(x, t, t') = \sum_{n \in \Z^3} e_n(x) \sum_{\substack{n = n_1 + n_2\\|n_1|\sim| n_2|}}
 e^{-\frac{t-t'}{2}} \frac{\sin ( (t - t') \jbb{n_1})}{\jbb{n_1}}
 \ft \Psi(n_1, t') \ft \Psi (n_2, t)
\label{sto1}
\end{align}

\noi
for $t \geq t' \geq 0$, 
where $|n_1|\sim| n_2|$ signifies the resonant product.
Then, we interpret 
the last term in the $\Res$-equation
as
\begin{align}
\Big( \I\big( M_\g(\Qxy + 2\Res \, + \! :\! \Psi^2 \!:) \Psi\big) \pe \Psi\Big)(t)
= \int_0^t 
M_\g(\Qxy + 2\Res \, + \! :\! \Psi^2 \!:) (t') \Ab(t, t') dt'.
\label{sto1a}
\end{align}

\noi
We point out that the Fourier transform $\ft \Ab(n, t, t')$
corresponds to $\A_{n, 0}(t, t')$ defined in \eqref{X7}
and thus the analysis for $\Ab$ is closely related to 
that for the paracontrolled
operator $\If_{\pl, \pe}$ in~\eqref{X6}.
See
 Lemma \ref{LEM:sto1} below.

As a result, we obtain the following local well-posedness
of the focusing Hartree SdNLW system~\eqref{SNLW6}.

\begin{theorem} \label{THM:2}
Let $V$ be the Bessel potential of order $\b\geq 2$, 
$A\in \R$, and $2 < \g \le 3$.
Let $\frac 14<s_1 < \frac 12 < s_2<1$ and $-\frac 12 <s_3<0$, 
satisfying \eqref{Z0}.
Then,
there exist $\ta = \ta(s_3) > 0$ 
and $\eps = \eps(s_1, s_2, s_3, \be ) > 0$
such that 
if 
\begin{itemize}
\item   $\Psi $ is a distribution-valued function belonging to 
$C([0, 1]; W^{-\frac 12 - \eps, \infty}(\T^3))\cap C^1([0, 1]; W^{-\frac 32 - \eps, \infty}(\T^3))$, 

\smallskip
\item
$:\! \Psi^2 \!:
$ is a distribution-valued function belonging to $ C([0, 1]; W^{-1 - \eps, \infty}(\T^3))$,

\smallskip
\item
$\Ab(t, t')$ 
is a distribution-valued function belonging to $ L^\infty_{t'} L^3_{t}(\Dl_2(1); H^{ - \eps}(\T^3))$, 
where $\Dl_2(T) \subset [0, T]^2 $ is given by 
\begin{align}
\Dl_2(T) = \{(t, t') \in \R_+^2: 0 \leq t' \leq t \leq T\}, 
\label{sto2}
\end{align}

\smallskip

\item 
 the operator
 $\If_{\pl, \pe}$ 
belongs to  the class
 $ \L_2\big(\tfrac 32, 1\big)$ in \eqref{L1ast}, 

\smallskip

\end{itemize}

\noi
then the system  \eqref{SNLW6}  is locally well-posed in 
$\H^{s_1}(\T^3) \times \H^{s_2}(\T^3) \times \{0\}$.
More precisely, 
given any $(X_0, X_1, Y_0, Y_1) \in \H^{s_1}(\T^3)\times \H^{s_2}(\T^3)$, 
there exists $T > 0$ 
such that there exists  a unique solution $(X, Y, \Res) $ to the 
focusing Hartree SdNLW system \eqref{SNLW6} on $[0, T]$
in the class 
$ Z^{s_1, s_2, s_3}(T)$
defined in \eqref{Z1}.
Furthermore, the solution $(X, Y, \Res)$
depends  continuously 
on the enhanced data set\textup{:}
\begin{align}
\Xi = \big(X_0, X_1, Y_0, Y_1, 
 \Psi, \,  :\! \Psi^2 \!: ,  \, \Ab, 
\,  \If_{\pl, \pe}\big)
\label{data2}
\end{align}

\noi
in the class\textup{:}
\begin{align*}
\mathcal{Y}^{s_1, s_2, \eps}_T
& = \H^{s_1}(\T^3) \times 
\H^{s_2}(\T^3)\\
& \hphantom{X}
\times 
\big(C([0,T]; W^{-\frac 12 - \eps, \infty}(\T^3))\cap C^1([0,T]; W^{-\frac 32 - \eps, \infty}(\T^3))\big)\\
& \hphantom{X}
\times 
C([0,T]; W^{-1 - \eps, \infty}(\T^3))
\times  L^\infty_{t'}L^3_t(\Dl_2(T); H^{ - \eps}(\T^3))
\times
 \L_2\big(\tfrac 32, T\big).
\end{align*}

\end{theorem}

For $\be > \frac 32$, 
we can make sense of 
the resonant product in 
$(V \ast :\! \Psi^2 \!:) \pe \Psi$ 
 in a deterministic manner
(given the pathwise regularities of $\Psi$ and $:\!\Psi^2\!:$\,)
and thus there is no need to include this term in the enhanced data set.

\begin{remark}\rm
By including 
$ (V \ast :\! \Psi^2 \!:) \pe \Psi$
in the enhanced data set, we may extend
Theorem~\ref{THM:2}
for $\be > 1$ under the condition \eqref{Z0}.
Note, however, that 
it is not very meaningful to consider 
the focusing SdNLW \eqref{SNLWA2} and thus the system 
\eqref{SNLW6} for $\be < 2$ in view of Theorem \ref{THM:Gibbs2},
since the nonlinearity, especially the terms involving $M_\g$, 
is derived from the potential energy in the Gibbs measure. 
\end{remark}

\section{Notations and basic lemmas}
\label{SEC:2}

In describing regularities of functions and distributions, 
we  use $\eps > 0$ to denote a small constant.
We often  suppress the dependence on such $\eps > 0$ in an estimate.

\subsection{Sobolev 
 and Besov spaces}
\label{SUBSEC:21}

Let $s \in \R$ and $1 \leq p \leq \infty$.
We define the $L^2$-based Sobolev space $H^s(\T^d)$
by the norm:
\begin{align*}
\| f \|_{H^s} = \| \jb{n}^s \ft f (n) \|_{\l^2_n}.
\end{align*}

\noi
We also define the $L^p$-based Sobolev space $W^{s, p}(\T^d)$
by the norm:
\begin{align*}
\| f \|_{W^{s, p}} = \big\| \F^{-1} [\jb{n}^s \ft f(n)] \big\|_{L^p}.
\end{align*}

\noi
When $p = 2$, we have $H^s(\T^d) = W^{s, 2}(\T^d)$.

Let $\phi:\R \to [0, 1]$ be a smooth  bump function supported on $[-\frac{8}{5}, \frac{8}{5}]$ 
and $\phi\equiv 1$ on $\big[-\frac 54, \frac 54\big]$.
For $\xi \in \R^d$, we set $\phi_0(\xi) = \phi(|\xi|)$
and 
\[\phi_{j}(\xi) = \phi\big(\tfrac{|\xi|}{2^j}\big)-\phi\big(\tfrac{|\xi|}{2^{j-1}}\big)\]

\noi
for $j \in \N$.
Then, for $j \in \Z_{\geq 0} := \N \cup\{0\}$, 
we define  the Littlewood-Paley projector  $\P_j$ 
as the Fourier multiplier operator with a symbol $\varphi_j$
given by 
\begin{align}
 \varphi_j(\xi) = \frac{\phi_j(\xi)}{\sum_{k \in \Z_{\geq 0}} \phi_k(\xi)}.
\label{phi1}
 \end{align}

\noi
Note that, for each $\xi \in \R^d$,  the sum in the denominator is over finitely many $k$'s.
Thanks to the normalization \eqref{phi1}, 
we have 
\[ f = \sum_{j = 0}^\infty \P_j f. \]

Let us now recall
 the definition and basic properties of  paraproducts
 introduced by Bony~\cite{Bony}.
 See~\cite{BCD, GIP} for further details.
Given two functions $f$ and $g$ on $\T^3$
of regularities $s_1$ and $s_2$, 
we write the product $fg$ as
\begin{align}
fg & \hspace*{0.3mm}
= f\pl g + f \pe g + f \pg g\notag \\
& := \sum_{j < k-2} \P_{j} f \, \P_k g
+ \sum_{|j - k|  \leq 2} \P_{j} f\,  \P_k g
+ \sum_{k < j-2} \P_{j} f\,  \P_k g.
\label{para1}
\end{align}

Next, we  recall the basic properties of the Besov spaces $B^s_{p, q}(\T^d)$
defined by the norm:
\begin{equation*}
\| u \|_{B^s_{p,q}} = \Big\| 2^{s j} \| \P_{j} u \|_{L^p_x} \Big\|_{\l^q_j(\Z_{\geq 0})}.
\end{equation*}

\noi
We denote the H\"older-Besov space by  $\C^s (\T^d)= B^s_{\infty,\infty}(\T^d)$.
Note that (i)~the parameter $s$ measures differentiability and $p$ measures integrability, 
(ii)~$H^s (\T^d) = B^s_{2,2}(\T^d)$,
and (iii)~for $s > 0$ and not an integer, $\C^s(\T^d)$ coincides with the classical H\"older spaces $C^s(\T^d)$;
see \cite{Graf}.

We recall the basic estimates in Besov spaces.
See \cite{BCD, GOTW} for example.

\begin{lemma}
The following estimates hold.

\noi
\textup{(i) (interpolation)} 
Let $s, s_1, s_2 \in \R$ and $p, p_1, p_2 \in (1,\infty)$
such that $s = \ta s_1 + (1-\ta) s_2$ and $\frac 1p = \frac \ta{p_1} + \frac{1-\ta}{p_2}$
for some $0< \ta < 1$.
Then, we have\footnote{We use the convention that the symbol $\lesssim$ indicates that inessential constants are suppressed in the inequality.}
\begin{equation}
\| u \|_{W^{s,  p}} \les \| u \|_{W^{s_1, p_1}}^\ta \| u \|_{W^{s_2, p_2}}^{1-\ta}.
\label{interp}
\end{equation}

\noi
\textup{(ii) (immediate  embeddings)}
Let $s_1, s_2 \in \R$ and $p_1, p_2, q_1, q_2 \in [1,\infty]$.
Then, we have
\begin{align} 
\begin{split}
\| u \|_{B^{s_1}_{p_1,q_1}} 
&\les \| u \|_{B^{s_2}_{p_2, q_2}} 
\qquad \text{for $s_1 \leq s_2$, $p_1 \leq p_2$,  and $q_1 \geq q_2$},  \\
\| u \|_{B^{s_1}_{p_1,q_1}} 
&\les \| u \|_{B^{s_2}_{p_1, \infty}}
\qquad \text{for $s_1 < s_2$},\\
\| u \|_{B^0_{p_1, \infty}}
 &  \les  \| u \|_{L^{p_1}}
 \les \| u \|_{B^0_{p_1, 1}}.
\end{split}
\label{embed}
\end{align}

\smallskip

\noi
\textup{(iii) (algebra property)}
Let $s>0$. Then, we have
\begin{equation}
\| uv \|_{\C^s} \les \| u \|_{\C^s} \| v \|_{\C^s}.
\label{alge}
\end{equation}

\smallskip

\smallskip

\noi
\textup{(iv) (Besov embedding)}
Let $1\leq p_2 \leq p_1 \leq \infty$, $q \in [1,\infty]$,  and  $s_2 = s_1 + d \big(\frac{1}{p_2} - \frac{1}{p_1}\big)$. Then, we have
\begin{equation*}
 \| u \|_{B^{s_1}_{p_1,q}} \les \| u \|_{B^{s_2}_{p_2,q}}.
\end{equation*}

\smallskip

\noi
\textup{(v) (duality)}
Let $s \in \mathbb{R}$
and  $p, p', q, q' \in [1,\infty]$ such that $\frac1p + \frac1{p'} = \frac1q + \frac1{q'} = 1$. Then, we have
\begin{equation}
\bigg| \int_{\T^d}  uv \, dx \bigg|
\le \| u \|_{B^{s}_{p,q}} \| v \|_{B^{-s}_{p',q'}},
\label{dual}
\end{equation}

\noi
where $\int_{\T^d} u v \, dx$ denotes  the duality pairing between $B^{s}_{p,q}(\T^d)$ and $B^{-s}_{p',q'}(\T^d)$.

\smallskip
	
\noi		
\textup{(vi) (fractional Leibniz rule)} 
Let $p, p_1, p_2, p_3, p_4 \in [1,\infty]$ such that 
$\frac1{p_1} + \frac1{p_2} 
= \frac1{p_3} + \frac1{p_4} = \frac 1p$. 
Then, for every $s>0$, we have
\begin{equation}
\| uv \|_{B^{s}_{p,q}} \les  \| u \|_{B^{s}_{p_1,q}}\| v \|_{L^{p_2}} + \| u \|_{L^{p_3}} \| v \|_{B^s_{p_4,q}} .
\label{prod}
\end{equation}

\end{lemma}

The interpolation \eqref{interp} follows
from the Littlewood-Paley characterization of Sobolev norms via the square function
and H\"older's inequality.

\begin{lemma}[paraproduct and resonant product estimates]
\label{LEM:para}
Let $s_1, s_2 \in \R$ and $1 \leq p, p_1, p_2, q \leq \infty$ such that 
$\frac{1}{p} = \frac 1{p_1} + \frac 1{p_2}$.
Then, we have 
\begin{align}
\| f\pl g \|_{B^{s_2}_{p, q}} \les 
\|f \|_{L^{p_1}} 
\|  g \|_{B^{s_2}_{p_2, q}}.  
\label{para2a}
\end{align}

\noi
When $s_1 < 0$, we have
\begin{align}
\| f\pl g \|_{B^{s_1 + s_2}_{p, q}} \les 
\|f \|_{B^{s_1 }_{p_1, q}} 
\|  g \|_{B^{s_2}_{p_2, q}}.  
\label{para2}
\end{align}

\noi
When $s_1 + s_2 > 0$, we have
\begin{align}
\| f\pe g \|_{B^{s_1 + s_2}_{p, q}} \les 
\|f \|_{B^{s_1 }_{p_1, q}} 
\|  g \|_{B^{s_2}_{p_2, q}}  .
\label{para3}
\end{align}

\end{lemma}

The product estimates \eqref{para2a},  \eqref{para2},  and \eqref{para3}
follow easily from the definition \eqref{para1} of the paraproduct 
and the resonant product.
See \cite{BCD, MW2} for details of the proofs in the non-periodic case
(which can be easily extended to the current periodic setting).

We also recall the following product estimate from \cite{GKO, BOZ}.

\begin{lemma}\label{LEM:gko}
Let $s> 0$.

\smallskip

\noi
\textup{(i)}
Let  $1<p_j,q_j,r\le\infty$, $j=1,2$ such that $\frac{1}{r}=\frac{1}{p_j}+\frac{1}{q_j}$.
Then, we have 
$$
\|\jb{\nb}^s(fg)\|_{L^r(\T^3)}\lesssim\| \jb{\nb}^s f\|_{L^{p_1}(\T^3)} \|g\|_{L^{q_1}(\T^3)}+ \|f\|_{L^{p_2}(\T^3)} 
\|  \jb{\nb}^s g\|_{L^{q_2}(\T^3)}.
$$

\smallskip

\noi
\textup{(ii)}
Let   $1<p \le \infty$ and $1 < q,r<\infty$ such that $s \geq   3\big(\frac{1}{p}+\frac{1}{q}-\frac{1}{r}\big)$
and $q, r' \ge p'$.
%
Then, we have
$$
\|\jb{\nb}^{-s}(fg)\|_{L^r(\T^3)}
\lesssim\| \jb{\nb}^{-s} f\|_{L^{p}(\T^3)} \| \jb{\nb}^{s} g\|_{L^{q}(\T^3)} .
$$
\end{lemma}

%
%

\subsection{On discrete convolutions}

Next, we recall the following basic lemma on a discrete convolution.

\begin{lemma}\label{LEM:SUM}
\textup{(i)}
Let $d \geq 1$ and $\al, \be \in \R$ satisfy
\[ \al+ \be > d  \qquad \text{and}\qquad  \al < d .\]
\noi
Then, we have
\[
 \sum_{n = n_1 + n_2} \frac{1}{\jb{n_1}^\al \jb{n_2}^\be}
\les \jb{n}^{- \al + \ld}\]

\noi
for any $n \in \Z^d$, 
where $\ld = 
\max( d- \be, 0)$ when $\be \ne d$ and $\ld = \eps$ when $\be = d$ for any $\eps > 0$.

\smallskip

\noi
\textup{(ii)}
Let $d \geq 1$ and $\al, \be \in \R$ satisfy $\al+ \be > d$.
\noi
Then, we have
\[
 \sum_{\substack{n = n_1 + n_2\\|n_1|\sim|n_2|}} \frac{1}{\jb{n_1}^\al \jb{n_2}^\be}
\les \jb{n}^{d - \al - \be}\]

\noi
for any $n \in \Z^d$.

\end{lemma}

Namely, in the resonant case (ii), we do not have the restriction $\al, \be < d$.
Lemma \ref{LEM:SUM} follows
from elementary  computations.
See, for example,  
 \cite[Lemma 4.2]{GTV} and \cite[Lemmas~4.1 and 4.2]{MWX}.

We also need the following lemma,
where we establish
 a uniform bound with respect to  the coefficients for
 a non-integer variable $n_0$
 defined in \eqref{SUM2}.

\begin{lemma} \label{LEM:SUM2}
Let $ \be \le \frac 12$.
Then, given $\eps > 0$, we have 
\begin{align}
 \sum_{n_1,n_2\in \Z^3}\frac 1 {\jb{n_1}^2 \jb{n_2}^2\jb{n_0}^{2\beta} \jb{n_1-n_2}^{2-2\beta +\eps}} 
\le C_\eps < \infty, 
\label{SUM1}
\end{align}

\noi
uniformly in $t \gg s > 0$, where $n_0$ is defined by 
\begin{align}
n_0 = \frac{t n_1 + s n_2}{t + s}.
\label{SUM2}
\end{align}

\end{lemma}

\begin{proof}
Given dyadic numbers
 $N,M \ge 1$, 
 we separately estimate the contributions 
 from 
  $\jb{n_1} \sim N$ and $\jb{n_2} \sim M$. 
Note that  we have $n_0 \sim n_1 + \frac st n_2$
under $t\gg s > 0$.

\smallskip

\noi
$\bullet$ {\bf Case 1:}
$N \gg M$.
\quad 
In this case, we have
\[ \ld: = {\jb{n_1}^2 \jb{n_2}^2\jb{n_0}^{2\beta} \jb{n_1-n_2}^{2-2\beta +\eps}} \sim N^{4+\eps} M^2.\]

\noi
Thus, we have
\[ \text{LHS of }\eqref{SUM1}
\les 
 \sum_{\substack{N, M \geq 1, \text{ dyadic}\\N \gg M}} N^{-1-\eps}M \les 1.\]

\smallskip

\noi
$\bullet$ {\bf Case 2:}
 $N \sim M$.
 \quad 
 In this case, we have 
$\ld 
\sim N^{4+2\beta} \jb{n_1-n_2}^{2-2\beta +\eps}. $
Thus, we have 
\begin{align*}
\text{LHS of }\eqref{SUM1}
& \les \sum_{\substack{N\ge 1\\ \text{dyadic}}} N^{-4-2\beta} 
\sum_{n_1,n_2 \sim N} \frac 1 {\jb{n_1-n_2}^{2-2\beta +\eps}}\\
& \les \sum_{\substack{N\ge 1\\ \text{dyadic}}} N^{-4-2\beta} N^3 N^{3-(2-2\beta +\eps)} \\
& = \sum_{\substack{N\ge 1\\ \text{dyadic}}}N^{-\eps} \les 1.
\end{align*}

\smallskip

\noi
$\bullet$ {\bf Case 3:} 
 $\frac{t}{s}N \gg M \gg N$.
 \quad 
 In this case, we have 
$\ld 
\sim N^{2+2\beta} M^{4-2\beta+\eps}. $
Thus, for $\be \le \frac 12$, we have
$$ \text{LHS of }\eqref{SUM1}
 \les \sum_{M \gg N} N^{1-2\beta}M^{-1+2\beta -\eps} \les 1. $$

\smallskip

\noi
$\bullet$ {\bf Case 4:}
 $\frac t s N \sim M \gg N$.
 \quad 
 In this case, we have 
$\ld 
 \sim N^{2} M^{4-2\beta+\eps} \jb{n_0}^{2\beta}.$
Recalling $\jb{n_0} \les N$,  we have 
\begin{align*}
\text{LHS of }\eqref{SUM1}
& \les \sum_{\substack{N, M\ge 1, \text{ dyadic}\\ \frac ts N \sim  M \gg N}}
 N^{-2} M^{-4+2\beta-\eps} \sum_{\jb{n_2} \sim M} \sum_{\jb{n_1} \sim N} \frac 1 {\jb{n_0}^{2\beta}}\\
& \les \sum_{\substack{N, M\ge 1, \text{ dyadic}\\ \frac ts N \sim  M \gg N}}
N^{1-2\beta} M^{-1+2\beta -\eps} \les 1, 
\end{align*}

\noi
provided that $\be \le \frac 12$.

\smallskip

\noi
$\bullet$ {\bf Case 5:}
 $M \gg \frac ts N$.
 \quad 
 In this case, we have 
$\ld 
 \sim \big(\frac st\big)^{2\be} N^{2} M^{4+\eps}. $
Thus, we have 
$$ \text{LHS of }\eqref{SUM1}
 \les \bigg( \frac ts\bigg)^{2\be}
 \sum_{\substack{N, M\ge 1, \text{ dyadic}\\ M \gg \frac ts N}}
 NM^{-1-\eps} \les 1, $$

\noi
provided that $\be \le \frac 12$.
This proves Lemma \ref{LEM:SUM2}.
\end{proof}

\subsection{Tools from stochastic analysis}

We conclude this section by recalling useful lemmas
from stochastic analysis.
See \cite{Bog, Shige} for basic definitions.
Let $(H, B, \mu)$ be an abstract Wiener space.
Namely, $\mu$ is a Gaussian measure on a separable Banach space $B$
with $H \subset B$ as its Cameron-Martin space.
Given  a complete orthonormal system $\{e_j \}_{ j \in \N} \subset B^*$ of $H^* = H$, 
we  define a polynomial chaos of order
$k$ to be an element of the form $\prod_{j = 1}^\infty H_{k_j}(\jb{x, e_j})$, 
where $x \in B$, $k_j \ne 0$ for only finitely many $j$'s, $k= \sum_{j = 1}^\infty k_j$, 
$H_{k_j}$ is the Hermite polynomial of degree $k_j$, 
and $\jb{\cdot, \cdot} = \vphantom{|}_B \jb{\cdot, \cdot}_{B^*}$ denotes the $B$--$B^*$ duality pairing.
We then 
denote the closure  of 
polynomial chaoses of order $k$ 
under $L^2(B, \mu)$ by $\mathcal{H}_k$.
The elements in $\H_k$ 
are called homogeneous Wiener chaoses of order $k$.
We also set
\begin{align}
 \H_{\leq k} = \bigoplus_{j = 0}^k \H_j
\notag
\end{align}

\noi
 for $k \in \N$.

Let $L = \Dl -x \cdot \nabla$ be 
 the Ornstein-Uhlenbeck operator.\footnote{For simplicity, 
 we write the definition of the Ornstein-Uhlenbeck operator $L$
 when $B = \R^d$.}
Then, 
it is known that 
any element in $\mathcal H_k$ 
is an eigenfunction of $L$ with eigenvalue $-k$.
Then, as a consequence
of the  hypercontractivity of the Ornstein-Uhlenbeck
semigroup $U(t) = e^{tL}$ due to Nelson \cite{Nelson2}, 
we have the following Wiener chaos estimate
\cite[Theorem~I.22]{Simon}.
See also \cite[Proposition~2.4]{TTz}.

\begin{lemma}\label{LEM:hyp}
Let $k \in \N$.
Then, we have
\begin{equation*}
\|X \|_{L^p(\O)} \leq (p-1)^\frac{k}{2} \|X\|_{L^2(\O)}
 \end{equation*}
 
 \noi
 for any $p \geq 2$
 and any $X \in \H_{\leq k}$.

\end{lemma}

The following lemma will be used in studying regularities of stochastic objects.
We say that a stochastic process $X:\R_+ \to \mathcal{D}'(\T^d)$
is spatially homogeneous  if  $\{X(\cdot, t)\}_{t\in \R_+}$
and $\{X(x_0 +\cdot\,, t)\}_{t\in \R_+}$ have the same law for any $x_0 \in \T^d$.
Given $h \in \R$, we define the difference operator $\dl_h$ by setting
\begin{align}
\dl_h X(t) = X(t+h) - X(t).
\label{diff1}
\end{align}

\begin{lemma}\label{LEM:reg}
Let $\{ X_N \}_{N \in \N}$ and $X$ be spatially homogeneous stochastic processes
$:\R_+ \to \mathcal{D}'(\T^d)$.
Suppose that there exists $k \in \N$ such that 
  $X_N(t)$ and $X(t)$ belong to $\H_{\leq k}$ for each $t \in \R_+$.

\smallskip
\noi\textup{(i)}
Let $t \in \R_+$.
If there exists $s_0 \in \R$ such that 
\begin{align*}
\E\big[ |\ft X(n, t)|^2\big]\les \jb{n}^{ - d - 2s_0}
\end{align*}

\noi
for any $n \in \Z^d$, then  
we have
$X(t) \in W^{s, \infty}(\T^d)$, $s < s_0$, 
almost surely.
Furthermore, if there exists $\g > 0$ such that 
\begin{align*}
\E\big[ |\ft X_N(n, t) - \ft X(n, t)|^2\big]\les N^{-\g} \jb{n}^{ - d - 2s_0}
\end{align*}

\noi
for any $n \in \Z^d$ and $N \geq 1$, 
then 
$X_N(t)$ converges to $X(t)$ in $W^{s, \infty}(\T^d)$, $s < s_0$, 
almost surely.
The following bound also holds:
\begin{align}
\E \big[ \| X_N(t) - X(t) \|_{W^{s,\infty}}^p \big]
\les p^{\frac{kp}{2}} N^{-\g p}. 
\label{reg22}
\end{align}

\smallskip
\noi\textup{(ii)}
Let $T > 0$ and suppose that \textup{(i)} holds on $[0, T]$.
If there exists $\s \in (0, 1)$ such that 
\begin{align}
 \E\big[ |\dl_h \ft X(n, t)|^2\big]
 \les \jb{n}^{ - d - 2s_0+ \s}
|h|^\s, 
\label{reg2}
\end{align}

\noi
for any  $n \in \Z^d$, $t \in [0, T]$, and $h \in [-1, 1]$,\footnote{We impose $h \geq - t$ such that $t + h \geq 0$.}
then we have 
$X \in C^\al([0, T]; W^{s, \infty}(\T^d))$, 
$\al<\s$ and $s < s_0 - \frac \s2$,  almost surely.
Furthermore, 
if there exists $\g > 0$ such that 
\begin{align*}
 \E\big[ |\dl_h \ft X_N(n, t) - \dl_h \ft X(n, t)|^2\big]
 \les N^{-\g}\jb{n}^{ - d - 2s_0+ \s}
|h|^\s, 
\end{align*}

\noi
for any  $n \in \Z^d$, $t \in [0, T]$,  $h \in [-1, 1]$, and $N \geq 1$, 
then 
$X_N$ converges to $X$ in $C^\al([0, T]; W^{s, \infty}(\T^d))$, $\al<\s$ and $s < s_0 - \frac{\s}{2}$,
almost surely.

\end{lemma}

Lemma \ref{LEM:reg} follows
from a straightforward application of the Wiener chaos estimate
(Lemma~\ref{LEM:hyp}).
For the proof, see Proposition 3.6 in \cite{MWX}
and  Appendix in \cite{OOTz}.
As compared to  Proposition 3.6 in \cite{MWX}, 
we made small adjustments.
In studying the time regularity, we 
made the following modifications:
$\jb{n}^{ - d - 2s_0+ 2\s}\mapsto\jb{n}^{ - d - 2s_0+ \s}$
and $s < s_0 - \s \mapsto s < s_0 - \frac \s2$  
so that it is suitable
for studying  the wave equation.
Moreover, while the result in \cite{MWX} is stated in terms of the
H\"older-Besov space $\mathcal{C}^s(\T^d) = B^s_{\infty, \infty}(\T^d)$, 
Lemma \ref{LEM:reg} handles the $L^\infty$-based Sobolev space $W^{s, \infty}(\T^3)$.
Note that 
the required modification of the proof is straightforward
since $W^{s, \infty}(\T^d)$ and $B^s_{\infty, \infty}(\T^d)$
differ only logarithmically.

Next, we recall the following corollary to 
the Garsia-Rodemich-Rumsey inequality
(\cite[Theorem A.1]{FV}).
See Lemma 2.2 in \cite{GKOT} for the proof.
See also  Corollary A.5 in~\cite{FV}
for the $\al = 2$ case.
This lemma is used to obtain the $L^\infty_t$-regularity 
of stochastic objects.

\begin{lemma}\label{LEM:GRR}

Let $(E, d)$ be a metric space.
Given $u \in C([0, T]; E)$, 
suppose that there exist
$c_0 > 0$, $\ta \in(0, 1)$, and $\al > 0$
such that 
\begin{align}
\int_{t_1}^{t_2}
\int_{t_1}^{t_2} \exp \bigg\{c_0 \bigg( \frac{d(u(t), u(s))}{|t-s|^\ta}\bigg)^\al\bigg\} dt ds
= : F_{t_1, t_2} < \infty
\label{G1}
\end{align}

\noi
for any $0 \leq t_1 \leq t_2 \leq T$ with $t_2 - t_1 \leq 1$.
Then, we have
\begin{align*}
 \exp \bigg\{\frac{c_0}{C} \bigg( 
 \sup_{t_1 \leq s < t \leq t_2} \frac{d(u(t), u(s))}{\zeta(t-s)}\bigg)^\al\bigg\} 
\leq  \max( F_{t_1, t_2}, e)
\end{align*}

\noi
for any $0 \leq t_1 \leq t_2 \leq T$
with $t_2 - t_1 \leq 1$,  
where $\zeta(t)$ is defined by 
\begin{align*}
 \zeta(t) = \int_0^t \tau^{\ta - 1} \bigg\{\log\Big( 1+ \frac 4{\tau^2}\Big)  \bigg\}^\frac{1}{\al} d\tau. 
\end{align*}

\end{lemma}

\medskip

Lastly, we recall the following Wick's theorem.
See Proposition I.2 in \cite{Simon}.

\begin{lemma}\label{LEM:Wick}	
Let $g_1, \dots, g_{2n}$ be \textup{(}not necessarily distinct\textup{)}
 real-valued jointly Gaussian random variables.
Then, we have
\[ \E\big[ g_1 \cdots g_{2n}\big]
= \sum  \prod_{k = 1}^n \E\big[g_{i_k} g_{j_k} \big], 
\]

\noi
where the sum is over all partitions of $\{1, \dots, 2 n\}$
into disjoint pairs $(i_k, j_k)$.
\end{lemma}

Given $n \in \Z^3$ and  $0 \leq t_2\leq t_1$, 
define $\s_{n}(t_1, t_2 )$ by 
\begin{align}
\begin{split}
\s_{n}(t_1, t_2 )  
:\! &=
\E  \big[  \ft{\Psi}(n, t_1)  \,  \ft{\Psi}(-n, t_2) \big] \\
&= \frac{e^{-\frac{t_1-t_2}2}}{\jb{n}^2} \bigg( \cos ((t_1-t_2) \jbb{n}) + \frac{\sin ((t_1-t_2) \jbb{n})}{2\jbb{n}} \bigg), 
\end{split}
\label{sigma2}
\end{align}

\noi
where $\Psi$ is as in \eqref{W2}.
Then, by Wick's theorem (Lemma~\ref{LEM:Wick}) and~\eqref{sigma2}, we have
\begin{align} \label{Wickz}
\begin{split}
&\E \Big[
\Big( \ft \Psi (n_1,t_1) \ft \Psi (n_2,t_1') - \ind_{n_1+n_2=0} \cdot \s_{n_1}(t_1,t_1') \Big) \\
& \hphantom{XXX}
\times
\Big( \cj{\ft \Psi (n_1',t_2) \ft \Psi (n_2',t_2') - \ind_{n_1'+n_2'=0} \cdot \s_{n_1'}(t_2,t_2')} \Big)
\Big] \\
&=
\ind_{\substack{n_1=n_1' \\ n_2=n_2'}} \cdot \s_{n_1}(t_1,t_2) \s_{n_2}(t_1',t_2')
+
\ind_{\substack{n_1=n_2' \\ n_2=n_1'}} \cdot \s_{n_1}(t_1,t_2') \s_{n_2}(t_1',t_2)
\end{split}
\end{align}
for $n_1,n _2, n_1', n_2' \in \Z^3$ and $0 \le t_2' \le t_2 \le t_1' \le t_1$.

\section{On the stochastic terms}
\label{SEC:sto1}

In this section, 
we establish the regularity properties
of the stochastic objects 
$\Psi$, $:\! \Psi^2 \!:$, and $(V \ast :\! \Psi^2 \!:) \pe \Psi$.
We study the paracontrolled operators (and $\Ab$) in Section \ref{SEC:po}.
First, we go over  the regularity properties
of the stochastic convolution $\Psi$ and the Wick power $:\! \Psi^2 \!:$.

\begin{lemma}\label{LEM:stoconv}
Given $k = 1, 2$, let 
$:\! \Psi_N^k \!:$
denote the truncated Wick power defined
in
 \eqref{so4a} for $k = 1$ and 
 \eqref{so4b} for $k = 2$, respectively.
Then, 
given any  $T,\eps>0$ and finite $p \geq 1$, 
 $\{ \, :\! \Psi _N^k \!: \, \}_{N\in \N}$ is a Cauchy sequence in $L^p(\O;C([0,T];W^{-\frac k2-\eps,\infty}(\T^3)))$,
 converging to some limit $:\!\Psi^k\!:$ in $L^p(\O;C([0,T];W^{-\frac k2-\eps,\infty}(\T^3)))$.
Moreover,  $:\! \Psi_N^k \!:$  converges almost surely to the same  limit in $C([0,T];W^{-\frac k2 -\eps,\infty}(\T^3))$.
Given any finite $q\geq 1$, we   have 
the following tail estimate:
\begin{align}
\PP\Big( \|:\! \Psi^k \!:\|_{L^q_T W^{-\frac k2-\eps, \infty}_x} > \ld\Big) 
\leq C\exp\bigg(-c \frac{\ld^{\frac{2}{k}}}{T^{ \frac{2}{kq}}}\bigg)
\label{P0}
\end{align}

\noi
for any $T > 0$ and $\ld > 0$.
When $q = \infty$, we  also have 
the following tail estimate:
\begin{align}
\PP\Big( \|:\! \Psi^k \!:\|_{L^\infty ([j, j+1]; W^{-\frac k2-\eps, \infty}_x)}> \ld\Big) 
\leq C\exp\big(-c \ld^{\frac{2}{k}}\big)
\label{P0z}
\end{align}

\noi
for any $j \in \Z_{\ge 0}$ and $\ld > 0$.
Moreover, the tail estimates \eqref{P0}
and \eqref{P0z}
also  hold
for $:\! \Psi _N^k \!: $, uniformly in $N \in \N$.

When $k = 1$, the convergence results for $\Psi_N$ 
 also hold 
in $C^1([0,T];W^{-\frac 32-\eps,\infty}(\T^3)))$.
Moreover, the tail estimates \eqref{P0} and \eqref{P0z}
hold for $\dt \Psi$ with 
$L^q([0, T];  W^{-\frac 32-\eps, \infty}(\T^3))$ in~\eqref{P0}
and $L^\infty ([j, j+1]; W^{-\frac 32-\eps, \infty}(\T^3))$ in \eqref{P0z}.

\end{lemma}

\begin{proof}
From \eqref{sigma2} and \eqref{Wickz}, 
we have
\begin{align}
\E \big[ |\ft{:\! \Psi^k \!:}(n,t)|^2 \big]
&\les \jb{n}^{-3 + k} 
\label{R1}
\end{align}

\noi
for $n \in \Z^3$ and $0 \le t \le T$.
Then, the first part of the claim follows
from Lemma \ref{LEM:reg}.
Indeed, the difference estimate \eqref{reg2}
for $\dl_h \ft{:\! \Psi^k \!:}(n,t)$
follows from~\eqref{R1} and the mean value theorem
as in the proof of Lemma 3.1 in \cite{GKO2}.
Note that 
our stochastic convolution $\Psi$ in  \eqref{W2} is 
for the damped wave equation 
and thus is slightly different  
from that for the undamped wave equation studied in~\cite{GKO2}.
Furthermore, $\Psi$ in~\eqref{W2} has non-zero random initial data
distributed by $\muu$ in~\eqref{gauss1}.
This difference, however, is marginal and the argument
in the proof of 
Lemma 3.1 in \cite{GKO2} can be easily modified  to establish the convergence results.
See also \cite{GKO, GKOT}.

Next, we prove the tail estimate \eqref{P0z}.
Since $ :\! \Psi^{k}\!: $ is spatially homogeneous (i.e.~its
distribution is invariant under spatial translations), 
we have 
\begin{align}
\E\bigg[ \ft{:\! \Psi^{k}\!:}(n_1, t_1) \cj{\ft{:\! \Psi^{k}\!:}(n_2, t_2)} \bigg] = 0
\label{R2}
\end{align}

\noi
unless $n_1 =  n_2$.
Indeed, by letting 
$F_{t_1, t_2} (x, y)=  \E\big[ \! :\! \Psi^{k}\!:(x, t_1) \, \cj{:\! \Psi^{k}\!:(y, t_2)}\big]$, 
it follows from the spatial homogeneity that 
\begin{align*}
\text{LHS of }\eqref{R2} 
& = 
\int_{\T^3}\int_{\T^3}
F_{t_1, t_2} (x, y)
e_{n_1}(x)e_{-n_2}(y) dy dx\\
& = \int_{\T^3}
\bigg( \int_{\T^3}
F_{t_1, t_2} (0, y-x)
e_{-n_2}(y- x)dy  \bigg)
e_{n_1- n_2}(x)
dx
\end{align*}

\noi
which equals $0$ unless $n_1 = n_2$ since
the inner integral on the right-hand side is a constant independent of $x$.  This proves \eqref{R2}.
Now, from \eqref{R1} 
and 
\eqref{R2}, we have
\begin{align}
\E\big[| \jb{\nb}^{-\frac{k}{2} -\eps} :\! \Psi^{k} (x, t) \!:  |^2 \big] 
= 
\sum_{n \in \Z^3} \jb{n}^{-k-2\eps}
\E \big[ |\ft{:\! \Psi^k \!:}(n,t)|^2 \big]
\les \sum_{n \in \Z^3} \jb{n}^{-3-2\eps} \leq C_\eps
\label{R3}
\end{align}

\noi
for any $\eps > 0$, uniformly in 
$x \in \T^3$ and $t \geq 0$.
Then, Minkowski's integral inequality and the Wiener chaos estimate (Lemma \ref{LEM:hyp}), 
we obtain
\begin{align}
 \Big\| \| :\! \Psi^{k}  \!:\|_{L^q_T W^{-\frac k2 -\eps, \infty}_x}\Big\|_{L^p(\O)}
\les p ^\frac{k}{2} T^{ \frac{1}{q}}
\label{P0a}
\end{align}

\noi
for any sufficiently large $p \gg1 $ (depending on $q \geq 1$).
The exponential tail estimate \eqref{P0}
follows from \eqref{P0a} and Chebyshev's inequality
(see also Lemma 4.5 in \cite{TzBO}).

Fix $j \in \Z_{\ge 0}$ and $\ld > 0$. Then, we have
\begin{align}
\begin{split}
\PP\Big( \|:\! \Psi^k \!:\|_{L^\infty ([j, j+1]; W^{-\frac k2-\eps, \infty}_x)}
& > \ld\Big) 
\leq
 \PP\Big( \|:\! \Psi^k(j) \!:\|_{ W^{-\frac k2 -\eps, \infty}_x}> \tfrac{\ld}{2}\Big) \\
+ 
& \PP\Big( \sup_{t \in [j, j+1]}\|:\! \Psi^k(t) \!: - :\! \Psi^k(j) \!:\|_{ W^{-\frac k2 -\eps, \infty}_x}> \tfrac \ld2\Big). 
\end{split}
\label{P0b}
\end{align}

\noi
The first term on the right-hand side of \eqref{P0b}
is for a fixed time $t = j$ and thus 
can be  controlled by the right-hand side of \eqref{P0z}
as above, using \eqref{R3}.
As for the second term on the right-hand side
of \eqref{P0b}, 
a straightforward adaptation of the argument in  the proof of \cite[Proposition~2.1]{GKO}
to the current three-dimensional setting
yields
\begin{align*}
\Big\| |h|^{-\rho} \|\dl_h (:\! \Psi^{k} (t) \!:) \|_{W^{-\frac k2-\eps, \infty}_x} \Big\|_{L^p(\O)}
\les p^\frac{k}{2} 
\end{align*}

\noi
for any sufficiently large $p \gg1 $, $t \in [j, j+1]$, 
and $|h| \leq 1$, where $\dl_h$ is as in \eqref{diff1} and $0 < \rho < \eps$.
Then, by applying Lemma 4.5 in \cite{TzBO}, 
we obtain the following exponential bound:
\begin{align}
\E \Bigg[\exp \bigg\{ \bigg( \frac{
\|:\! \Psi^{k} (\tau_2)\!: - :\! \Psi^{k} (\tau_1) \!: \|_{W^{-\frac k2-\eps, \infty}_x}}{|\tau_2-\tau_1|^\rho}\bigg)^\frac{2}{k}\bigg\}\Bigg] 
\leq C < \infty, 
\label{P0d}
\end{align}

\noi
uniformly in  $j \leq \tau_1 < \tau_2 \leq  j+1$ (and $j \in \Z_{\ge 0}$).
By integrating~\eqref{P0d} in $\tau_1$ and $\tau_2$, 
this verifies
the hypothesis~\eqref{G1} of Lemma \ref{LEM:GRR}
(under an expectation).
Finally, applying 
Lemma~\ref{LEM:GRR}
and then  Chebyshev's inequality, we conclude that 
\begin{align*}
\PP\Big( \sup_{t \in [j, j+1]}\|:\! \Psi^k(t) \!: - :\! \Psi^k(j) \!:\|_{ W^{-\frac k2-\eps, \infty}_x}> \tfrac \ld2\Big)
\leq C\exp\big(-c \ld^{\frac{2}{k}}\big).
\end{align*}

\noi
This proves \eqref{P0z}.

Lastly, when $k = 1$, 
we note that, unlike the heat or Schr\"odinger case,  the truncated stochastic convolution $\Psi_N$
is differentiable in time and its time derivative is given by 
\begin{align}
\dt \Psi_N(t) = 
  \pi_N \dt^2\D(t)\phi_0 +  \pi_N \dt \D(t) (\phi_0+\phi_1) + \sqrt 2 \pi_N \int_0^t\dt \D(t-t')dW(t').
\label{Psiz}
\end{align}

\noi
The formula \eqref{Psiz} can be easily verified
by writing the Fourier coefficient of  the stochastic convolution 
with the zero initial data  as a Paley-Wiener-Zygmund integral
and taking a time derivative.
With $  \ft \D_n(t) $ as in \eqref{W3}, integration by parts gives 
\begin{align*}
  \int_0^t \ft \D_n(t-t')dB_n(t')
=  -   \int_0^t B_n(t') \dd_{t'} (\ft \D_n(t-t'))dt'
=    \int_0^t B_n(t')  \ft \D_n'(t-t')dt', 
\end{align*}

\noi
since $B_n(0) =  \ft \D_n(t-t')|_{t' = t} = 0$.
Then, by taking a time derivative
and integrating by parts again, we obtain
\begin{align*}
\dt \bigg(  \int_0^t \ft \D_n(t-t')dB_n(t')\bigg)
& =  B_n(t)  \ft \D_n'(0) 
+  \int_0^t B_n(t')  \ft \D_n''(t-t'))dt'\\
& = 
 \int_0^t \dt \ft \D_n(t - t') d B_n(t'). 
\end{align*}

\noi
This proves \eqref{Psiz}.
Once we have \eqref{Psiz} for $\dt \Psi_N$, 
we can simply repeat the computation above
and obtain the claimed convergence and tail estimates.
\end{proof}

Next, we study the regularity of the resonant product $(V \ast :\! \Psi^2 \!:) \pe \Psi$
in \eqref{Psi2}.
Note that when $\be > \frac 32$, 
we can make sense of this resonant product in the deterministic manner
and thus the following lemma is not needed in the focusing case.

\begin{lemma}\label{LEM:IV}
Let $V$ be the Bessel potential of order $\b>1$
and set 
\[Z_N = ( V \ast :\! \Psi^2_N \!:) \pe \Psi_N\]

\noi
for $N \in \N$.
Then, 
given any  $T,\eps>0$ and finite $p \geq 1$, 
 $\{ Z_N \}_{N\in \N}$ is a Cauchy sequence in $L^p(\O;C([0,T];W^{\beta - \frac 32 -\eps,\infty}(\T^3)))$,
 converging to some limit 
 \[  Z = ( V \ast :\! \Psi^2 \!:) \pe \Psi\]

\noi
 in $L^p(\O;C([0,T];W^{\beta - \frac 32 -\eps,\infty}(\T^3)))$.
Moreover,  $Z_N$  converges almost surely to the same  limit in $C([0,T];W^{\beta - \frac 32 -\eps,\infty}(\T^3))$.
Given any finite $q\geq 1$, we   have 
the following tail estimate:
\begin{align}
\PP\Big( \| Z \|_{L^q_T W^{\beta - \frac 32 -\eps, \infty}_x} > \ld\Big) 
\leq C\exp\bigg(-c \frac{\ld^{\frac{2}{3}}}{T^{\frac{2}{3q}}}\bigg)
\label{YS0}
\end{align}

\noi
for any $T > 0$ and $\ld > 0$.
When $q = \infty$, we  also have 
the following tail estimate:
\begin{align}
\PP\Big( \|Z\|_{L^\infty ([j, j+1]; W^{\beta - \frac 32 -\eps, \infty}_x)}> \ld\Big) 
\leq C\exp\big(-c \ld^{\frac{2}{3}}\big)
\label{YS0a}
\end{align}

\noi
for any $j \in \Z_{\ge 0}$ and $\ld > 0$.
Moreover, the tail estimates \eqref{YS0}
and \eqref{YS0a}
also  hold
for $Z_N $, uniformly in $N \in \N$.

\end{lemma}

\begin{proof}

Note that  $( V \ast :\! \Psi^2 \!:) \pe \Psi \in \H_{\le 3}$.
Thus, in view of Lemma \ref{LEM:reg}, 
it suffices to show 
\begin{align}
\E \big[ |\ft Z(n,t)|^2 \big]
&\les \jb{n}^{-2\beta}, \label{YS1} 
\end{align}

\noi
for $n \in \Z^3$ and  $0 \le t \le T$.
As mentioned above, 
 the difference estimate \eqref{reg2}
for $\dl_h \ft Z(n,t)$
follows from~\eqref{YS1} and the mean value theorem
as in the proof of Lemma 3.1 in \cite{GKO2}.
Also, an adaptation of the argument in 
 the proof of Lemma 3.1 in \cite{GKO2}
 yields the claimed convergence results.
As for the exponential tail estimates \eqref{YS0} and \eqref{YS0a}, 
from the spatial homogeneity of $Z$
and \eqref{YS1}, we first obtain
\begin{align*}
\E\big[| \jb{\nb}^{\be -\frac{3}{2} -\eps}Z (x, t)  |^2 \big] 
\les \sum_{n \in \Z^3} \jb{n}^{-3-2\eps} \leq C_\eps
\end{align*}

\noi
for any $\eps > 0$, uniformly in 
$x \in \T^3$ and $t \geq 0$.
Then, we can proceed as in the proof of Lemma~\ref{LEM:stoconv}
to conclude 
the exponential tail estimates \eqref{YS0} and \eqref{YS0a}.

In the following, we focus on proving the bound \eqref{YS1}.
Using \eqref{so4b}, we write $\ft Z (n,t)$ as follows:
\begin{align}
\begin{aligned}
\ft Z (n,t)
&= \sum_{\substack{n_1,n_2,n_3 \in \Z^3 \\ n=n_1+n_2+n_3 \\ |n_1+n_2| \sim |n_3|}}
\ft V(n_1+n_2) \Big( \ft{\Psi} (n_1,t) \ft{\Psi}(n_2,t) - \ind_{n_1+n_2=0} \cdot \jb{n_1}^{-2}\Big) \ft{\Psi} (n_3,t) \\
&= \sum_{\substack{n_1,n_2,n_3 \in \Z^3 \\ n=n_1+n_2+n_3 \\    |n_3|\sim |n_1+n_2|\ne0}}
\ft V(n_1+n_2) \ft{\Psi} (n_1,t) \ft{\Psi}(n_2,t) \ft{\Psi} (n_3,t) \\
&\quad + \sum_{\substack{n_1 \in \Z^3}} \ind_{|n| \sim 1}
\ft V(0) \Big( |\ft{\Psi} (n_1,t)|^2 - \jb{n_1}^{-2} \Big) \ft{\Psi} (n,t)\\
&=: \ft Z_1 (n,t) + \ft Z_2(n,t), 
\end{aligned}
\label{YS1b}
\end{align}

\noi
where we used $|n_1+n_2| \sim |n_3|$ and $|n|\sim 1$
to signify 
 the resonant product $\pe$ in 
the definition of  $Z = ( V \ast :\! \Psi^2 \!:) \pe \Psi$.
From \eqref{Wickz} with \eqref{sigma2}, we have
\begin{align}
\E \big[ | \ft Z_2 (n,t)|^2 \big]
\les \ind_{|n| \sim 1} \sum_{n_1 \in \Z^3} \frac{1}{\jb{n_1}^4}
\les \ind_{|n| \sim 1},
\label{YS2}
\end{align}
verifying \eqref{YS1} for $Z_2$.
We now decompose $\ft Z_1 (n,t)$ as
\begin{align}
\ft Z_1 (n,t)
&= \sum_{\substack{n_1,n_2,n_3 \in \Z^3 \\ n=n_1+n_2+n_3 \\  
 |n_3|\sim |n_1+n_2|\ne0
\\ |n_2+n_3| |n_3+n_1| \ne 0}}
\ft V(n_1+n_2) \ft{\Psi} (n_1,t) \ft{\Psi}(n_2,t) \ft{\Psi} (n_3,t) \notag\\
&\quad + 2 \ft{\Psi} (n,t) \sum_{\substack{n_2 \in \Z^3 \\  |n+n_2| \sim |n_2|\ne 0 }}
\ft V(n+n_2)  \Big( |\ft{\Psi}(n_2,t)|^2 - \jb{n_2}^{-2} \Big) \notag\\
&\quad + 2 \ft{\Psi} (n,t) \sum_{\substack{n_2 \in \Z^3 \\ |n+n_2| \sim |n_2|\ne 0 }}
\ft V(n+n_2)  \jb{n_2}^{-2}\notag\\
&\quad - \ind_{n \neq 0}
\ft V(2n) |\ft{\Psi}(n,t)|^2 \ft{\Psi} (n,t) \notag\\
&
=: \ft Z_{11}(n,t) + \ft Z_{12} (n,t) + \ft Z_{13} (n,t) + \ft Z_{14}(n,t).
\label{Z13}
\end{align}

\noi
Here, $Z_{12}$ denotes the renormalized contribution from $n_3 = n_1, n_2$,
while $Z_{13}$ is the counter term.

From \eqref{Bessel1} and \eqref{W2}, we have 
\begin{align}
\E \big[ |\ft Z_{14} (n,t)|^2 \big]
\les \jb{n}^{-2\beta-6},
\label{YS3}
\end{align}

\noi
verifying \eqref{YS1}.
Under the condition $|n+n_2| \sim |n_2|$, 
we have $|n_2|\ges |n|$.
Then, 
it follows from~\eqref{Bessel1}, \eqref{sigma2}, 
and Lemma \ref{LEM:SUM} that 
\begin{align}
\E \big[ |\ft Z_{13} (n,t)|^2 \big]
&= 4 \jb{n}^{-2} 
\bigg( \sum_{\substack{n_2 \in \Z^3 \\  |n+n_2| \sim |n_2| }} \jb{n+n_2}^{-\be}  \jb{n_2}^{-2} \bigg)^2
\les \jb{n}^{-2\beta}
\label{YS4}
\end{align}

\noi
provided that $\be > 1$.
Similarly, 
we have 
\begin{align}
\E \big[ |\ft Z_{12} (n,t)|^2 \big]
&\les \jb{n}^{-2} \sum_{\substack{n_2 \in \Z^3 \\ |n+n_2| \sim |n_2|}} \jb{n+n_2}^{-2\beta} \jb{n_2}^{-4}
\les \jb{n}^{-2\beta-3}
\label{YS5}
\end{align}

\noi
for $\be > -\frac 12$.
Finally, we consider the estimate for $\ft Z_{11}(n,t)$.
The condition $|n_1+n_2| \sim |n_3|$ implies $|n_1+n_2| \sim |n_3| \ges |n|$.
From \eqref{Bessel1}, Wick's theorem (Lemma \ref{LEM:Wick}), and Lemma \ref{LEM:SUM},
we have
\begin{align}
\E \big[ |\ft Z_{11} (n,t)|^2 \big]
&= \sum_{\substack{n_1,n_2,n_3 \in \Z^3 \\ n=n_1+n_2+n_3 \\  
 |n_3|\sim |n_1+n_2|\ne0
  \\ |n_2+n_3| |n_3+n_1| \neq 0}}
\sum_{\substack{n_1',n_2',n_3' \in \Z^3 \\ n=n_1'+n_2'+n_3' \\  
 |n_3'|\sim |n_1'+n_2'|\ne0
  \\ |n_2'+n_3'| |n_3'+n_1'| \neq 0}}
\ft V(n_1+n_2) \ft V(n_1'+n_2') \notag \\
&  \hphantom{XXX}
\times \E \Big[ \ft{\Psi} (n_1,t) \ft{\Psi}(n_2,t) \ft{\Psi} (n_3,t)
\cj{\ft{\Psi} (n_1',t) \ft{\Psi}(n_2',t) \ft{\Psi} (n_3',t)} \Big] \notag \\
&\les \sum_{\substack{n_1,n_2,n_3 \in \Z^3 \\ n=n_1+n_2+n_3 \\ |n_1+n_2| \sim |n_3| \ges |n|}}
\jb{n_1+n_2}^{-2\beta} \jb{n_1}^{-2} \jb{n_2}^{-2} \jb{n_3}^{-2} \notag  \\
&\quad + \sum_{\substack{n_1,n_2,n_3 \in \Z^3 \\ n=n_1+n_2+n_3 \\ |n_1+n_2| \sim |n_3| \ges |n| \\ |n_2+n_3| \sim |n_1| \ges |n|}}
\jb{n_1+n_2}^{-\beta} \jb{n_2+n_3}^{-\beta} \jb{n_1}^{-2} \jb{n_2}^{-2} \jb{n_3}^{-2} \notag\\
&\les \jb{n}^{-\beta} \sum_{\substack{n_1,n_3 \in \Z^3 \\ |n-n_3| \sim |n_3|}}
\jb{n-n_3}^{-\beta} \jb{n_1}^{-2} \jb{n-n_1-n_3}^{-2} \jb{n_3}^{-2} \notag  \\
&\les \jb{n}^{-\beta} \sum_{\substack{n_3 \in \Z^3 \\ |n-n_3| \sim |n_3|}} \jb{n-n_3}^{-\beta-1} \jb{n_3}^{-2} 
\notag \\
&\les \jb{n}^{-2\beta}
\label{YS6}
\end{align}

\noi
for $\beta>0$.
Putting \eqref{YS1b} - \eqref{YS6} together, we obtain the desired bound 
\eqref{YS1}.
\end{proof}

\begin{remark} \label{REM:Z} \rm
The assumption $\be>1$ 
was used to estimate $Z_{13}$ in \eqref{YS4}, 
while the other terms can be controlled under $\be > 0$.
Note that when $\be \le 1$, 
\eqref{YS4} yields
\[
\E \big[ |\ft Z_{13} (n,t)|^2 \big]
\ges \jb{n}^{-2} \bigg( \sum_{\substack{n_2 \in \Z^3}} \jb{n_2}^{-\b-2} \bigg)^2
= \infty.
\]

\noi
From this, we   conclude that $Z \notin C([0,T]; \D'(\T^3))$ almost surely when $\be \le 1$.
See, for example, Subsection 4.4 in \cite{OOcomp}.
For $\be \le 1$, we introduce a renormalization to remove
this problematic term $Z_{13}$. See \eqref{YZ6a} below.

\end{remark}

\section{Construction of the Gibbs measures}
\label{SEC:Gibbs}

In this section, we present the construction
and non-normalizability 
of the Gibbs measures.
We first discuss the defocusing case (Theorem \ref{THM:Gibbs})
for $\be > 1$.
Then, 
we present the full proof 
of Theorem~\ref{THM:Gibbs2}
in the focusing case.
The remaining part of the defocusing case ($0 < \be \le 1$)
is presented in Section~\ref{SEC:def}.
Our proofs rely on the variational formulation
of the partition function due to Barashkov-Gubinelli \cite{BG}.
See Lemma \ref{LEM:var2} 
and  the  Bou\'e-Dupuis variational formula (Lemma~\ref{LEM:var3})
below.

We first consider the defocusing case.
In the following, we study
 the truncated Gibbs measure $\rho_N$ 
defined in \eqref{GibbsN}: 
\[
d\rho_N = Z_N^{-1} e^{-R_N(u)} d \mu,
\]

\noi
where  
$R_N$ is as in \eqref{K1}
and $Z_N$ denotes the partition function:
\begin{align}
Z_N = \int e^{-R_N(u)} d \mu.
\label{P1}
\end{align}

\noi
In what follows, we prove various statements in terms of $\mu$
but they can be trivially upgraded to the corresponding statement for $\muu = \mu_1\otimes \mu_0$.

First, we state the convergence property of $R_N$.

\begin{lemma} \label{LEM:conv}
Let $V$ satisfy \eqref{Vb} with $\b>1$.
Then, given any finite $p\geq 1$, $R_N$ defined in~\eqref{K1} converges to some $R$ in $L^p(\mu)$ as $N \to \infty$.
Moreover, for $\g>0$, 
 $\RR_N$ defined in~\eqref{K2} converges to some $\RR$ 
in $L^p(\mu)$ as $N \to \infty$.

\end{lemma}

We point out  that
the proof of Lemma \ref{LEM:conv} does not rely on
the positivity of $\ft V$.
See Subsection~\ref{SUBSEC:R} for the details.
Note that 
Lemma \ref{LEM:conv}  
implies convergence in measure of $\big\{e^{-R_N(u)}\big\}_{N \in \N}$ (in the defocusing case).
Then, 
the desired convergence \eqref{exp2} of the density follows from a standard argument, 
once we prove the uniform exponential integrability \eqref{exp1}. 
See Remark 3.8 in~\cite{Tz08}.  See also the proof of Proposition 1.2 in \cite{OTh}.
The same comment applies to the focusing case.

In establishing the uniform exponential integrability bounds \eqref{exp2}
and \eqref{exp3}, we employ the variational approach 
as in \cite{BG}.
In Subsection \ref{SUBSEC:var}, 
we briefly go over the setup for the variational
formulation of the partition function
from \cite{BG, GOTW}.
In Subsection~\ref{SUBSEC:def}, 
we then present the uniform exponential integrability~\eqref{exp1}
for $\be > 1$ in the defocusing case.
We then move onto the focusing case.
We go over the construction of the focusing Gibbs measure 
for $\be > 2$
in Subsection \ref{SUBSEC:foc1} 
and the non-normalizability in Subsection \ref{SUBSEC:foc2}.
In Subsection \ref{SUBSEC:foc3}, 
we prove the uniform exponential integrability \eqref{exp3}
in  the weakly nonlinear regime
at the critical value $\be = 2$.

Recall our convention that $\s = -1$ in the defocusing case, 
since a precise value of $\s< 0$ does not play an important role.

\subsection{Proof of Lemma \ref{LEM:conv}} \label{SUBSEC:R}

We only consider the case $p = 2$.
The convergence for general $p \geq1$
follows from the $p = 2$ case and  the Wiener chaos estimate (Lemma \ref{LEM:hyp}).
Furthermore,  in the following, we only show 
\begin{align}
\sup_{N \in \N} \| R_N(u) \|_{L^2(\mu)}< \infty
\qquad \text{and}
\qquad
\sup_{N \in \N} \| \RR_N(u) \|_{L^2(\mu)}< \infty
\label{B2}
\end{align}

\noi
since a slight modification of the argument presented below 
implies the convergence of  $\{ R_N \}_{N\in \N}$ 
and $\{ \RR_N \}_{N\in \N}$ in $L^2(\mu)$.

Define $Q_N(u)$ by
\begin{align}
Q_N(u)
:= \int_{\T^3} (V \ast :\! u_N^2 \!:) :\! u_N^2 \!:  dx
- \ft V(0) \bigg( \int_{\T^3} :\! u_N^2 \!: dx\bigg)^2
- 2 \al_N
\label{KN}
\end{align}

\noi
where  $\al_N$ is as in \eqref{alN}.
Then, we can write  $R_N(u)$ and $\RR_N(u)$ in \eqref{K1} and \eqref{K2} as 
\begin{align}
R_N (u)
& = \frac 14 Q_N(u) + \frac{\ft V(0)}4 \bigg( \int_{\T^3} :\! u_N^2 \!: dx\bigg)^2, 
\label{K3}\\
\RR_N (u)
& = \frac \s4 Q_N(u) + \frac{\s \ft V(0)}4 \bigg( \int_{\T^3} :\! u_N^2 \!: dx\bigg)^2
 - A \, \bigg| \int_{\T^3} :\! u_N^2 \!: dx\bigg|^\g.
\label{K4}
\end{align}

\noi
By the Wiener chaos estimate (Lemma \ref{LEM:hyp}), 
we have
\[
\Bigg\| \bigg| \int_{\T^3} :\! u_N^2 \!: dx \bigg|^p \Bigg\|_{L^2(\mu)}
\les C_p  \bigg\| \int_{\T^3} :\! u_N^2 \!: dx \bigg\|_{L^2(\mu)}^p
\les C_p \bigg( \sum_{n \in \Z^3} \jb{n}^{-4} \bigg)^\frac{p}{2}
<\infty
\]
for any finite $p>0$.
Hence, the desired bounds \eqref{B2} follow once we prove
\begin{align}
\sup_{N \in \N} \| Q_N(u) \|_{L^2(\mu)}
< \infty.
\label{B3}
\end{align}

From Parseval's identity (see \eqref{Wick2}) with  \eqref{KN}, \eqref{Wick1}, and \eqref{alN}, we have
\begin{align}
Q_N(u)
&=
\sum_{\substack{n_1,n_2,n_3,n_4 \in \Z^3 \\ n_1+n_2+n_3+n_4=0 \\ |n_1+n_2| |n_1+n_3| |n_1+n_4| \neq 0}}
\ft V(n_1+n_2)
\ft u_N (n_1) \ft u_N (n_2) \ft u_N (n_3) \ft u(n_4) \notag\\
&\quad +
2\sum_{\substack{n_1,n_2 \in \Z^3 \\ n_1+n_2 \neq 0 \\ |n_1|, |n_2| \le N}}
\ft V(n_1+n_2)
\Big( |\ft u_N (n_1)|^2 - \jb{n_1}^{-2} \Big) \Big( |\ft u_N (n_2)|^2 - \jb{n_2}^{-2} \Big) \notag\\
&\quad
+4\sum_{\substack{n_1,n_2 \in \Z^3 \\ n_1+n_2 \neq 0 \\ |n_1|, |n_2| \le N}}
\ft V(n_1+n_2)
\Big( |\ft u_N (n_1)|^2 - \jb{n_1}^{-2} \Big) \jb{n_2}^{-2} \notag\\
&\quad
-
\sum_{\substack{n_1 \in \Z^3 \\ n_1 \neq 0}}
\ft V(2n_1) |\ft u_N (n_1)|^4 \notag\\
&=: Q_{N,1}(u) + Q_{N,2}(u) + Q_{N,3}(u) +Q_{N,4}(u).
\label{B4}
\end{align}
Here,
the condition $|n_1+n_2| |n_1+n_3| |n_1+n_4| \neq 0$ 
together with $n_1 + n_2 + n_3 + n_4 = 0$ in
the definition of  $Q_{N,1}$ implies that
$n_i + n_j \neq 0$ for all $i \neq j$.

From  \eqref{Vb} and Wick's theorem (Lemma \ref{LEM:Wick}),
$Q_{N,1}(u)$ is estimated as follows:
\begin{align}
\| Q_{N,1} (u) \|_{L^2(\mu)}^2
&\les \sum_{\substack{n_1,n_2,n_3,n_4 \in \Z^3 \\ n_1+n_2+n_3+n_4=0}}
\jb{n_1+n_2}^{-2\b} \jb{n_1}^{-2} \jb{n_2}^{-2} \jb{n_3}^{-2} \jb{n_4}^{-2} 
\notag \\
&\quad +\sum_{\substack{n_1,n_2,n_3,n_4 \in \Z^3 \\ n_1+n_2+n_3+n_4=0}}
\jb{n_1+n_2}^{-\b} \jb{n_1+n_3}^{-\b} \jb{n_1}^{-2} \jb{n_2}^{-2} \jb{n_3}^{-2} \jb{n_4}^{-2} 
\notag \\
\intertext{By Cauchy's inequality, symmetry, and Lemma \ref{LEM:SUM},}
&\les
\sum_{\substack{n_1,n_2,n_3 \in \Z^3}}
\jb{n_1+n_2}^{-2\b} \jb{n_1}^{-2} \jb{n_2}^{-2} \jb{n_3}^{-2} \jb{n_1+n_2+n_3}^{-2}  
\notag \\
&\les
\sum_{\substack{n_1,n_2\in \Z^3}}
\jb{n_1+n_2}^{-2\b-1} \jb{n_1}^{-2} \jb{n_2}^{-2} 
\notag \\
&\les
\sum_{n_1 \in \Z^3}
\jb{n_1}^{-2-\min(2\b,2-\eps)}
< \infty
\label{B5}
\end{align}

\noi
for any small $\eps > 0$, provided that $\be > \frac 12$.
From \eqref{B4} and \eqref{Wickz},
we have
\begin{align}
\begin{split}
\| Q_{N,2} (u) \|_{L^2(\mu)}^2
& \les
\sum_{\substack{n_1,n_2 \in \Z^3}}
\jb{n_1+n_2}^{-2\b}
\jb{n_1}^{-4} \jb{n_2}^{-4}
+ \bigg(\sum_{n_1 \in \Z^3} \jb{2n_1}^{-\be} 
\jb{n_1}^{-4}\bigg)^{2}\\
& \les
\bigg(
\sum_{\substack{n_1 \in \Z^3}}
\jb{n_1}^{-4} \bigg)^2
<\infty
\end{split}
\label{Bb2}
\end{align}

\noi
for $\be \ge 0$.
As for  $Q_{N,3}(u)$, we first note that 
\begin{align*}
\begin{split}
Q_{N,3}(u)
&= 4\sum_{\substack{n_1 \in \Z^3 \\ |n_1| \le N}}
\Big( |\ft u_N (n_1)|^2 - \jb{n_1}^{-2} \Big) \kappa_N (n_1),
\end{split}
\end{align*}

\noi
where $\kappa_N$ is defined in \eqref{kappa1}.
Hence, from \eqref{Wickz}
and the uniform boundedness of $\kk_N$ for  $\b>1$, 
 we obtain
\begin{align}
\begin{split}
\| Q_{N,3} (u) \|_{L^2(\mu)}^2
\les
\sum_{n_1 \in \Z^3}
\kappa (n_1)^2 \jb{n_1}^{-4}
\les \sum_{n_1 \in \Z^3}
\jb{n_1}^{-4}
<\infty.
\end{split}
\label{B6}
\end{align}

\noi
Lastly, we have
\begin{align}
&\| Q_{N,4} (u) \|_{L^2(\mu)}^2
\les  \bigg(\sum_{n_1 \in \Z^3} \jb{n_1}^{- \be -4}\bigg)^2
<\infty.
\label{B7}
\end{align}

\noi
Therefore, putting 
\eqref{B4} - 
\eqref{B7}
together, 
we obtain \eqref{B3}.
This proves Lemma \ref{LEM:conv}.

\begin{remark} \label{REM:Bb} \rm
For a potential $V$ satisfying 
 $\ft V(n) \ges \jb{n}^{-\b}$, $n \in \Z^3$, 
 for some $\be \le 1$,  we have
\begin{align*} 
\lim_{N \to \infty } \| Q_N(u) \|_{L^2(\mu)}
=\infty.
\end{align*}

\noi
The argument above shows that 
while $Q_{N,1}$, $Q_{N,2}$, and $Q_{N,4}$
are uniformly bounded in $L^2(\mu)$ for $\b> \frac 12$, 
 $Q_{N, 3}$ becomes divergent for $\be \le 1$
 due to the unboundedness of $\kk_N$.
For $\frac 12 < \be \le 1$, 
we can  introduce  the second renormalization
as in \eqref{K1r}.
This precisely 
 removes the divergent term $Q_{N, 3}$, 
allowing us to prove an analogue of Lemma \ref{LEM:conv}
for $R^\dia_N(u)$ defined in 
\eqref{K1r}.
For this renormalized potential energy 
 $ R^\dia_N(u)$, 
 the uniform exponential integrability holds true for $\be > \frac 12$.
See Section \ref{SEC:def}.

For $0 < \be \le \frac 12$, the first term $Q_{N, 1}$ in \eqref{B4}
also becomes divergent.
This term, however, constitutes 
the main contribution for the potential energy and thus can not be removed
by a renormalization, 
causing the singularity 
of the resulting Gibbs measure to the base Gaussian measure in this case.
See Subsection \ref{SUBSEC:def5}.


\end{remark}

\subsection{Variational formulation}
\label{SUBSEC:var}

In order to  prove \eqref{exp1}, 
we follow the argument in \cite{BG,GOTW} 
and derive a variational formula for the partition function
$Z_N$ in \eqref{P1}. 
Let us first introduce some notations.
See also Section 4 in  \cite{GOTW}. 
Let $W(t)$ be the cylindrical 
 Wiener process  in~\eqref{W1}.
 We  define a centered Gaussian process $Y(t)$
by 
\begin{align}
Y(t)
=  \jb{\nabla}^{-1}W(t).
\label{P2}
\end{align}

\noi
Then, 
we have $\Law(Y(1)) = \mu$. 
By setting  $Y_N = \pi_NY $, 
we have   $\Law(Y_N(1)) = (\pi_N)_\#\mu_1$. 
In particular, 
we have  $\E [Y_N^2(1)] = \s_N$,
where $\s_N$ is as in~\eqref{sigma1}.

Next, let $\Ha$ denote the space of drifts, which are the progressively measurable processes that belong to
$L^2([0,1]; L^2(\T^3))$, $\PP$-almost surely. 
Given a drift $\dr \in \Ha$, 
we  define the measure $\Q_\dr$ 
whose Radon-Nikodym derivative with 
respect to $\PP$ is given by the following stochastic exponential:
\begin{align*}
\frac{d\Q_\dr}{d\PP} = e^{\int_0^1 \jb{\dr(t),  dW(t)} - \frac{1}{2} \int_0^1 \| \dr(t) \|_{L^2_x}^2dt},
\end{align*}

\noi
where $\jb{\cdot, \cdot}$ stands for the usual  inner product on $L^2(\T^3)$.
Then, by letting  $\Hc$ denote the subspace of $\Ha$ consisting of drifts such that $\Q_\dr(\O) = 1$,
it follows from Girsanov's theorem (\cite[Theorems 1.4 and 1.7 in Chapter VIII]{RV})
 that $W$ is a semi-martingale under $\Q_\dr$
 with the following  decomposition:
\begin{align}
 W(t) = W_\dr(t) + \int_0^t \dr(t')dt',
\label{P3}
\end{align}

\noi
 where $W_\dr$ is now an $L^2(\T^2)$-cylindrical Wiener process under
 the new measure $\Q_\dr$.
Substituting~\eqref{P3} in \eqref{P2} leads to the decomposition:
\begin{align*}
 Y = Y_\dr +I(\dr), 
\end{align*}

\noi
where 
\begin{align}
Y_\dr(t) = \jb{\nabla}^{-1} W_\dr(t)\qquad \textup{and}\qquad I(\dr)(t) = \int_0^t \jb{\nabla}^{-1} \dr(t') dt'.
\label{P3a}
\end{align}
In the following, 
we use $\E_{\Q_\dr}$ for an expectation with respect to 
$\Q_\dr$.

Proceeding  as  in \cite[Lemma 1]{BG} and \cite[Proposition 4.4]{GOTW}, we then have the following variational formula for the partition function $Z_N$ in \eqref{P1}.

\begin{lemma} \label{LEM:var2}
For any $N \in \N$, we have
\begin{equation} \notag
- \log Z_N = \inf_{\dr \in \Hc} \E_{\Q_\dr} 
\bigg[ R_N (Y_\dr(1) + I(\dr)(1)) + \frac{1}{2} \int_0^1 \| \dr(t) \|_{L^2_x}^2 dt \bigg].
\end{equation}
\end{lemma}

We state a useful lemma on the  pathwise regularity estimates  of 
$Y_\dr(1)$ and $I(\dr)(1)$.
See Lemmas 4.6 and 4.7 in \cite{GOTW}.

\begin{lemma}  \label{LEM:Dr}

\textup{(i)} 
Let $V$ be the Bessel potential of order $\be >1$.
Then, given any finite $p \ge 1$, 
 we have 
\begin{align}
\begin{split}
\sup_{\dr \in \Hc} \E_{\Q_\dr} 
\Big[ & \|Y_\dr(1)\|_{\C^{-\frac 12 - \eps}}^p
+ \|:\!Y_\dr^2(1)\!:\|_{\C^{-1 - \eps}}^p
 + 
\big\|(V \ast :\! Y_\dr^2(1) \!: ) Y_\dr(1)\big\|_{\C^{-\frac 12-\eps}}^p
\Big]
 <\infty
 \end{split}
 \label{P4}
\end{align}

\noi
for any $\eps > 0$.

\smallskip

\noi
\textup{(ii)} For any $\dr \in \Hc$, we have
\begin{align*}
\| I(\dr)(1) \|_{H^{1}}^2 \leq \int_0^1 \| \dr(t) \|_{L^2}^2dt.
\end{align*}
\end{lemma}

As for (i), 
the main point is to note that,  
for any $\dr \in \Hc$,
 $W_\dr$ is a cylindrical Wiener process in $L^2(\T^2)$ under $\Q_\dr$.
 Thus,  the law of $Y_\dr(1) = \jb{\nabla}^{-1} W_\dr(1)$ under $\Q_\dr$ is always given by  $\mu$, so in particular,  it is independent of $\dr\in\Hc$. 
 This fact is also used in \eqref{YZ10} below.
As for the last term in \eqref{P4}, 
the same argument as in the proof of Lemma \ref{LEM:IV} yields that 
$(V \ast :\! Y_\dr^2(1) \!:) \pe Y_\dr(1)$ is in $ \C^{\b-\frac 32-\eps}(\T^3)$ almost surely
for $\be > 1$.
By  the paraproduct decomposition \eqref{para1}
and  Lemma~\ref{LEM:para},  we then conclude that 
 $(V \ast :\! Y_\dr^2(1)\!:) Y_\dr(1)\in \C^{-\frac 12-\eps}(\T^3)$ almost surely
for  $\be>1$.

\begin{remark}\rm
In the discussion above, we used the formulation, following the work \cite{GOTW}
rather than the original work by Barashkov and Gubinelli \cite{BG}.
In \cite{BG}, a Gaussian process was localized in a frequency annulus, 
depending on the value of $t$ (which is not restricted to $[0, 1]$ in \cite{BG}), 
in order to treat a cubic term which would be divergent without such a frequency cutoff.
In our current problem, however, there is no such 
issue thanks to the smoothing coming from the Hartree potential $V$,
allowing us to work with a simpler formulation as in \cite{GOTW}.

\end{remark}

\subsection{Exponential integrability in the defocusing case for $\be > 1$}
\label{SUBSEC:def}

In this section, we consider the defocusing case.
We use the variational formulation 
of the partition function $Z_N$ (Lemma \ref{LEM:var2})
and prove  the uniform exponential integrability 
\eqref{exp1} for $\be > 1$ in 
Theorem~\ref{THM:Gibbs}\,(i).
Since the argument is identical for any finite $p \ge 1$,
we only present details for the case $p=1$.

Fixing an arbitrary drift $\dr \in \Hc$, 
our main goal is to establish a uniform (in $N$) lower bound on 
\begin{equation}
\W_N(\dr) = \E_{\Q_\dr} 
\bigg[ R_N(Y_\dr(1) + I(\dr)(1)) + \frac{1}{2} \int_0^1 \| \dr(t) \|_{L^2_x}^2 dt \bigg].	
\label{v_N0}
\end{equation}
Since the drift $\dr \in \Hc$ is fixed, 
we suppress the dependence on the drift $\dr$ henceforth 
and denote $Y = Y_\dr(1) $ and $\Dr = I(\dr)(1)$
with the understanding that an expectation is taken under the measure $\Q_\dr$.
We also set $Y_N = \pi_N Y$ and $\Dr_N = \pi_N \Dr$.
By setting
\begin{align}
V_0 = V -\ft V(0) = V-1, 
\label{V0}
\end{align}

\noi
it follows 
from \eqref{K1}, \eqref{K3},  and \eqref{KN}
that 
\begin{align}
\begin{split}
R_N (Y + \Dr)  & = 
\frac 14 Q_N(Y)
+ \int_{\T^3} ( V_0 \ast :\! Y_N^2 \!: ) Y_N \Dr_N dx
+\frac 12 \int_{\T^3} (V_0 \ast :\! Y_N^2 \!: ) \Dr_N^2 dx
\\
&\hphantom{X}
+ \int_{\T^3} ( V_0 \ast ( Y_N \Dr_N)) Y_N \Dr_N dx
+ \int_{\T^3} ( V_0 \ast \Dr_N^2) Y_N \Dr_N dx
\\
&\hphantom{X}
+\frac 14 \int_{\T^3} (V_0 \ast \Dr_N^2) \Dr_N^2 dx
+ \frac 14 \bigg\{ \int_{\T^3} \Big( :\! Y_N^2 \!: + 2 Y_N \Dr_N + \Dr_N^2 \Big) dx \bigg\}^2.
\end{split}
\label{Y0}
\end{align}

\noi
From \eqref{B4} and Wick's theorem (Lemma \ref{LEM:Wick} and  \eqref{Wickz}), we have
\begin{align}
\begin{split}
\E_{\Q_\dr} [Q_{N,1}(Y)] &=  \E_{\Q_\dr} [Q_{N,3}(Y)] =0, \\
\E_{\Q_\dr}  [Q_{N,2}(Y)
+ Q_{N,4}(Y)]
&= 2\sum_{\substack{n_1 \in \Z^3 \\ |n_1| \le N}} \ft V (2n_1) \jb{n_1}^{-4}
-2 \sum_{\substack{n_1 \in \Z^3 \\ |n_1| \le N}} \ft V (2n_1) \jb{n_1}^{-4} = 0.
\end{split}
\label{Y0a}
\end{align}

\noi
As a consequence, we  have
\begin{align}
\E_{\Q_\dr} [Q_N(Y)]
=0.
\label{Y1}
\end{align}

\noi
Hence, from  \eqref{v_N0}, \eqref{Y0}, and \eqref{Y1}, we obtain
\begin{align}
\begin{split}
\W_N(\dr)
&=\E_{\Q_\dr} 
\bigg[
\int_{\T^3} ( V_0 \ast :\! Y_N^2 \!:) Y_N \Dr_N dx
+ \frac 12 \int_{\T^3} (V_0 \ast :\! Y_N^2 \!: ) \Dr_N^2 dx
\\
&\hphantom{XXXX}
+ \int_{\T^3} (V_0 \ast (Y_N \Dr_N )) Y_N \Dr_N dx
+ \int_{\T^3} (V_0 \ast \Dr_N^2) Y_N \Dr_N dx
\\
&\hphantom{XXXX}
+ \frac 14 \int_{\T^3} (V_0 \ast \Dr_N^2 ) \Dr_N^2 dx
+ \frac 14 \bigg\{ \int_{\T^3} \Big( :\! Y_N^2 \!: + 2 Y_N \Dr_N + \Dr_N^2 \Big) dx \bigg\}^2 \\
&\hphantom{XXXX}
+ \frac{1}{2} \int_0^1 \| \dr(t) \|_{L^2_x}^2 dt
\bigg].
\end{split}
\label{v_N0a}
\end{align}

\noi
The main strategy is to bound $\W_N(\dr)$ from below pathwise,
uniformly in $N \in \N$ and independently of the drift $\dr$, by utilizing the positive terms:
\begin{equation}
\U_N(\dr) = \E_{\Q_\dr}\bigg[ \frac 14 \int_{\T^3} (V_0 \ast \Dr_N^2) \Dr_N^2 dx 
+ \frac 1{16} \bigg( \int_{\T^3} \Dr_N^2 dx \bigg)^2 + \frac{1}{2} \int_0^1 \| \dr(t) \|_{L^2_x}^2 dt\bigg].
\label{v_N1}
\end{equation}

\noi
As pointed out in  Remark \ref{REM:Vp}, the first term on the right-hand side of \eqref{v_N1} 
is non-negative
and is in fact equal to $\frac 14 \| \Dr_N^2 \|_{\dot H^{-\frac \be 2}}^2$.
As for the second term, see Lemma \ref{LEM:Dr3} below.

In the following, we first state two lemmas, controlling the other terms appearing \eqref{v_N0a}.
We present the proofs of these lemmas at the end of this subsection.

\begin{lemma} \label{LEM:Dr2}
Give $\be > 1$, let the potential $V$ satisfy \eqref{Vb}.
Then, there exist   small $\eps>0$ and   a constant  $c  >0$ 
 such that
\begin{align}
\begin{split}
\bigg| \int_{\T^3}  (V_0 \ast \Dr_N^2) Y_N \Dr_N  dx\bigg|
&\le c \Big(1 + 
\| Y_N \|_{\C^{-\frac 12-\eps}}^{c}\Big)\\
& \quad + \frac 1{100} \Big(
\| \Dr_N^2\|_{\dot H^{-\frac{\be}{2}}}^2 + 
 \| \Dr_N \|_{L^2}^4 +  \| \Dr_N \|_{H^1}^2\Big), 
\end{split}
\label{YY1}\\
\bigg| \int_{\T^3} (V_0 \ast :\! Y_N^2 \!:) \Dr_N^2 dx \bigg|
&\le c \| :\! Y_N^2 \!: \|_{\C^{-1-\eps}}^{2}  
+ \frac 1{100} 
\| \Dr_N \|_{L^2}^4, 
\label{YY2} \\
\begin{split}
\bigg| \int_{\T^3}  (V_0 \ast (Y_N \Dr_N) ) Y_N \Dr_N dx \bigg|
&\le c \Big(1 + 
\| Y_N \|_{\C^{-\frac 12-\eps}}^{c}\Big)\\
& \quad + \frac 1{100} \Big(
 \| \Dr_N \|_{L^2}^4 +  \| \Dr_N \|_{H^1}^2\Big),
\end{split}
\label{YY3}\\
\begin{split}
\bigg| \int_{\T^3} (V_0 \ast :\! Y_N^2 \!:) Y_N \Dr_Ndx  \bigg|
&\le c \| (V_0 \ast :\! Y_N^2 \!:) Y_N \|_{\C^{-\frac 12-\eps}}^2 + \frac 1{100} \| \Dr_N \|_{H^1}^2, 
\end{split}
\label{YY4}
\end{align}

\noi
uniformly in $N \in \N$.


\end{lemma}

\begin{lemma} \label{LEM:Dr3}
Given any small $\eps > 0$, 
there exists $c = c(\eps)>0$ such that
\begin{align}
\begin{split}
\bigg\{ \int_{\T^3}&  \Big( :\! Y_N^2 \!: + 2 Y_N \Dr_N + \Dr_N^2 \Big) dx \bigg\}^2 \\
&\ge \frac 14 \| \Dr_N \|_{L^2}^4 - \frac 1{100} \| \Dr_N \|_{H^1}^2 - c \bigg\{ \| Y_N \|_{\C^{-\frac 12-\eps}}^c 
 + \bigg( \int_{\T^3} :\! Y_N^2 \!:  dx \bigg)^2 \bigg\}, 
\end{split}
\label{YY5}
\end{align}

\noi
uniformly in $N \in \N$.

\end{lemma}

We now prove 
  the uniform exponential integrability 
\eqref{exp1} in 
Theorem \ref{THM:Gibbs}.
In view of Lemma~\ref{LEM:var2}, it suffices
to establish a finite lower bound on 
$\W_N(\dr) $ uniformly in $N \in \N$ and $\dr \in \Hc$.
From \eqref{v_N0a}, \eqref{v_N1}, Lemmas  \ref{LEM:Dr2} and \ref{LEM:Dr3}
with Lemma \ref{LEM:Dr},  we obtain
\begin{align*}
\inf_{N \in \mathbb{N}} \inf_{\dr \in \Hc} \W_N(\dr) 
\geq 
\inf_{N \in \mathbb{N}} \inf_{\dr \in \Hc}
\Big\{ -C_0 + \frac{1}{10}\U_N(\dr)\Big\}
 \geq - C_0 >-\infty.
\end{align*}

\noi
This proves the uniformly exponential integrability \eqref{exp1} 
for $\be > 1$ and hence Theorem~\ref{THM:Gibbs}\,(i).

\medskip

We conclude this subsection by 
 presenting the proofs of Lemmas \ref{LEM:Dr2} and \ref{LEM:Dr3}.

\begin{proof}[Proof of Lemma \ref{LEM:Dr2}]

From  \eqref{Vb}, Young's inequality, 
and the product estimate (Lemma \ref{LEM:para}), 
we have
\begin{align}
\begin{split}
\bigg| \int_{\T^3}  (V_0 \ast \Dr_N^2) Y_N \Dr_N  dx\bigg|
&\le \frac{1}{100} \| \Dr_N^2\|_{\dot H^{-\frac{\be}{2}}}^2 + 
c\| Y_N\Dr_N\|_{H^{-\frac{\be}{2}}}^2\\
&\le \frac{1}{100} \| \Dr_N^2\|_{\dot H^{-\frac{\be}{2}}}^2 + 
c\| Y_N\|_{\C^{-\frac12 - \eps}}^{\frac{2(1+\eps)}\eps}
+ \frac 1{100}\|\Dr_N\|_{H^{\frac{1}{2}+2\eps}}^{2(1+\eps)}
\end{split}
\label{YZ1}
\end{align}

\noi
for $\be > 1$.
Then, the estimate \eqref{YY1} follows
from the interpolation \eqref{interp}
and Young's inequality.

Next, we consider the second estimate \eqref{YY2}.
When $\be > 1$, 
it follows from  \eqref{dual} and    \eqref{embed}
that 
\begin{align}
\begin{split}
\bigg| \int_{\T^3} (V_0 \ast :\! Y_N^2 \!:) \Dr_N^2 dx \bigg|
&\le \|  :\! Y_N^2 \!: \|_{\C^{-1-\eps}} \| \Dr_N^2 \|_{B^{-\eps}_{1,1}} \\
&\le c \|  :\! Y_N^2 \!: \|_{\C^{-1-\eps}}\| \Dr_N^2 \|_{L^1}.
\end{split}
\label{YZ1a}
\end{align}

\noi
Then, the estimate \eqref{YY2} follows from  Cauchy's inequality.

As for \eqref{YY3}, 
we have, from \eqref{Vb}, 
\begin{align*}
\bigg| \int_{\T^3}  (V_0 \ast (Y_N \Dr_N) ) Y_N \Dr_N dx \bigg| 
\les \| Y_N \Dr_N \|_{\dot H^{-\frac{\be}{2}}}^2 .
\end{align*}

\noi
Then, the rest follows as in \eqref{YZ1}, provided that $\be > 1$.

Lastly, 
from \eqref{dual}, \eqref{embed}, and Young's inequality that
\begin{align}
\begin{split}
\bigg| \int_{\T^3} (V_0 \ast :\! Y_N^2 \!:) Y_N \Dr_Ndx  \bigg|
&\le \| (V_0 \ast :\! Y_N^2 \!:) Y_N \|_{\C^{ -\frac 12-\eps}} \| \Dr_N \|_{B^{\frac 12 +\eps}_{1,1}} \\
&\le c \| (V_0 \ast :\! Y_N^2 \!:) Y_N \|_{\C^{ -\frac 12-\eps}}^2 + \frac 1{100} \| \Dr_N \|_{H^1}^2.
\end{split}
\label{YZ2}
\end{align}

\noi
Here, the condition 
 $\be > 1$ is needed to guarantee the finiteness
 of the first term on the right-hand side of \eqref{YZ2}.
 See Lemma \ref{LEM:Dr}.
This completes the proof of Lemma \ref{LEM:Dr2}.
\end{proof}

Next, we present the proof of Lemma \ref{LEM:Dr3}.

\begin{proof}[Proof of Lemma \ref{LEM:Dr3}]
From Cauchy's inequality,  there exists a constant $C>0$ such that
\begin{align*}
(a+b+c)^2
\ge \frac 12 c^2 -C (a^2+b^2)
\end{align*}

\noi
for any real numbers $a,b,c$.
Thus,  we have 
\begin{align}
\begin{split}
&\bigg\{ \int_{\T^3} \Big( :\! Y_N^2 \!: + 2 Y_N \Dr_N + \Dr_N^2 \Big) dx \bigg\}^2 \\
&\ge \frac 12 \bigg( \int_{\T^3} \Dr_N^2dx \bigg)^2 - C_0 \bigg\{ \bigg( \int_{\T^3} :\! Y_N^2 \!: dx \bigg)^2
+ \bigg( \int_{\T^3} Y_N \Dr_N dx \bigg)^2 \bigg\}
\end{split}
\label{YZ3}
\end{align}

\noi
for some $C_0 > 0$.
From  \eqref{dual}, \eqref{embed}, \eqref{interp}, and Young's inequality, we have
\begin{align}
\begin{split}
\bigg| \int_{\T^3} Y_N \Dr_N dx \bigg|^2
&\le \| Y_N \|_{\C^{-\frac 12-\eps}}^2 \| \Dr_N \|_{H^{{\frac 12+2\eps}}}^2
\les \| Y_N \|_{\C^{-\frac 12-\eps}}^2 \| \Dr_N \|_{L^2}^{1-4\eps} \| \Dr_N \|_{H^1}^{1+4\eps} \\
&\le c \| Y_N \|_{\C^{-\frac 12-\eps}}^{\frac 8{1-4\eps}} + \frac 1{4C_0} \| \Dr_N \|_{L^2}^4 + \frac1{100 C_0} \| \Dr_N \|_{H^1}^2.
\end{split}
\label{YZ4}
\end{align}

\noi
Hence, \eqref{YY5} follows from \eqref{YZ3} and \eqref{YZ4}.
\end{proof}


\subsection{Exponential integrability for the focusing case: the non-endpoint case $\be > 2$}
\label{SUBSEC:foc1}

In this subsection, we present the construction of the 
focusing Hartree Gibbs measure $\rho$ in~\eqref{Gibbs9} 
in the non-endpoint case 
 $\be > 2$ (Theorem~\ref{THM:Gibbs2}\,(i)).
In view of Lemma \ref{LEM:conv}
and the comments  following the lemma, 
it suffices to prove the uniform exponential integrability~\eqref{exp3}.

In the focusing case, the potential 
energy $\frac 14 \int_{\T^3} (V_0 \ast \Dr_N^2) \Dr_N^2 dx$ 
has a wrong sign.
Thus, 
we need to reprove
\eqref{YY1} 
in Lemma~\ref{LEM:Dr2}
without using the potential energy.

\begin{lemma} \label{LEM:Dr4}
Let $V$ satisfy \eqref{Vb} with $\be\ge 2$.
Then, there exist  small $\eps>0$ and   a constant $c  >0$ 
 such that
\begin{align}
\bigg| \int_{\T^3}  (V_0 \ast \Dr_N^2) Y_N \Dr_N  dx\bigg|
&\le c
\| Y_N \|_{\C^{-\frac 12-\eps}}^{c}
+ \frac 1{100} \Big(\| \Dr_N \|_{L^2}^4 + \| \Dr_N \|_{H^1}^2\Big), 
\label{YY6}
\end{align}

\noi
uniformly in $N \in \N$.

\end{lemma}

\begin{proof}
From \eqref{dual},  \eqref{prod}, 
and Sobolev's inequality with $\be \ge 2$, we have 
\begin{align*}
\text{LHS of } \eqref{YY6}
&\le 
\| Y_N \|_{\C^{-\frac 12-\eps}}
\| (V_0*\Dr_N^2)\Dr_N\|_{B^{\frac{1}{2}+\eps}_{1, 1}}\\
&\le 
\| Y_N \|_{\C^{-\frac 12-\eps}}
\| \Dr_N^2\|_{H^{\frac 12 - \be + 2\eps}}\|\Dr_N\|_{H^{\frac{1}{2}+2\eps}}\\
&\le 
\| Y_N \|_{\C^{-\frac 12-\eps}}
\| \Dr_N\|_{H^{\eps}}^2\|\Dr_N\|_{H^{\frac{1}{2}+2\eps}}.
\end{align*}

\noi
Then, \eqref{YY6} follows from \eqref{interp} and Young's inequality.
%
\end{proof}

\begin{lemma} \label{LEM:Dr5}
Let $0<\g<3$ and $A> 0$.
There exist small $\eps>0$ and a constant  $c>0$ such that
\begin{align}
\begin{split}
A \bigg| \int_{\T^3} & \Big( :\! Y_N^2 \!: + 2 Y_N \Dr_N + \Dr_N^2 \Big) dx \bigg|^\g \\
&\ge \frac A4 \| \Dr_N \|_{L^2}^{2\g} - \frac 1{100} \| \Dr_N \|_{H^1}^2 - c \bigg\{ \| Y_N \|_{\C^{-\frac 12-\eps}}^{\frac{4\g}{3+4\eps-\g(1+4\eps)}} + \bigg| \int_{\T^3} :\! Y_N^2 \!:  dx \bigg|^\g \bigg\}, 
\end{split}
\label{YY8}
\end{align}

\noi
uniformly in $N \in \N$.

\end{lemma}

\begin{proof}
Note that there exists a constant $C>0$ such that
\begin{align}\label{YY9}
|a+b+c|^\g
\ge \frac 12 |c|^\g -C (|a|^\g+|b|^\g)
\end{align}

\noi
for any $a, b, c\in \R$.
Indeed, if $|c|^\g <2 C(|a|^\g+|b|^\g)$, \eqref{YY9} is trivial.
When $|c|^\g \ge 2 C(|a|^\g+|b|^\g)$, 
by $|c| \ge (2C)^{\frac 1\g} \max (|a|, |b|)$ and the triangle inequality, we have
\[
|a+b+c|
\ge |c| -|a| - |b|
\ge \big( 1-2 (2C)^{-\frac 1\g} \big) |c|
\ge 2^{-\frac 1\g} |c|,
\]

\noi
provided that a constant $C>0$ is sufficiently large.
Hence, we obtain \eqref{YY9}.

By \eqref{YY9}, there exists a constant $C_0>0$ such that
\begin{align}
\begin{split}
A\bigg| \int_{\T^3} & \Big( :\! Y_N^2 \!: + 2 Y_N \Dr_N + \Dr_N^2 \Big) dx \bigg|^\g \\
&\ge \frac A2 \bigg( \int_{\T^3} \Dr_N^2dx \bigg)^\g - C_0 A \bigg\{ \bigg| \int_{\T^3} :\! Y_N^2 \!: dx \bigg|^\g
+ \bigg| \int_{\T^3} Y_N \Dr_N dx \bigg|^\g \bigg\}.
\end{split}
\label{YY10}
\end{align}

\noi
From \eqref{dual}, \eqref{embed}, \eqref{interp}, and Young's inequality,
we have
\begin{align}
\begin{split}
\bigg| \int_{\T^3} Y_N \Dr_N dx \bigg|^\g
&\les \| Y_N \|_{\C^{-\frac 12-\eps}}^\g \| \Dr_N \|_{B^{\frac 12+\eps}_{1,1}}^\g
\les \| Y_N \|_{\C^{-\frac 12-\eps}}^\g \| \Dr_N \|_{L^2}^{\frac{\g(1-4\eps)}{2}} \| \Dr_N \|_{H^1}^{\frac{\g(1+4\eps)}{2}} \\
&\le c \| Y_N \|_{\C^{-\frac 12-\eps}}^{\frac{4\g}{3+4\eps-\g(1+4\eps)}} + \frac 1{4 C_0} \| \Dr_N \|_{L^2}^{2\g} + \frac 1{100C_0A} \| \Dr_N \|_{H^1}^2, 
\end{split}
\label{YY11}
\end{align}

\noi
provided that $0<\g<\frac{3+4\eps}{1+4\eps}$,
namely $0<\g<3$ and $0< \eps \ll 1$.
Hence, \eqref{YY8} follows from~\eqref{YY10} and \eqref{YY11}.
\end{proof}

We now present the proof of the uniform exponential integrability \eqref{exp3}
for $\be > 2$, 
using the variational formulation. 
As in the previous section, we only consider the case $p = 1$.
Set \begin{align}
\W_N(\dr)
= \E_{\Q_\dr} 
\bigg[ - \RR_N(Y_\dr(1) + I(\dr)(1)) + \frac{1}{2} \int_0^1 \| \dr(t) \|_{L^2_x}^2 dt \bigg],
\label{YY12}
\end{align}

\noi
where $\RR_N(u)$ is as in \eqref{K4}.
In view of Lemma \ref{LEM:Dr5}, we also set 
\begin{equation}
\U_N(\dr) =
\E_{\Q_\dr} \bigg[\frac A 4\bigg| \int_{\T^3} \Dr_N^2 dx \bigg|^\g + \frac{1}{2} \int_0^1 \| \dr(t) \|_{L^2_x}^2 dt\bigg].
\label{YY13}
\end{equation}

\noi
In the focusing case, the potential energy
 $\int_{\T^3} (V_0 \ast \Dr_N^2) \Dr_N^2 dx$
 appears with a wrong sign and thus we need to control this term by $\U_N$ in \eqref{YY13}.
When $1< \be<  3$, 
it follows from Sobolev's inequality,~\eqref{interp}, 
and Young's inequality that 
\begin{align}
\begin{split}
\bigg| \int_{\T^3} (V_0 \ast \Dr_N^2) \Dr_N^2 dx \bigg|
& \les
\| \Dr_N^2 \|_{H^{-\frac \b2}}^2 
\les
 \| \Dr_N \|_{L^\frac{12}{3+\b}}^4
 \les \| \Dr_N \|_{L^2}^{1+\b} \| \Dr_N \|_{H^1}^{3-\b} \\
&\le  c_0 + \frac{A}{100} \| \Dr_N \|_{L^2}^{2\g } + \frac 1{100} \| \Dr_N \|_{H^1}^2,  
\end{split}
\label{YY14}
\end{align}

\noi
provided that  $\g\ge \frac{\be + 1}{\be-1}$ and $A>0$ is sufficiently large.
When $\be = 3$, 
\eqref{YY14} holds with a strict inequality
 $\g >  \frac{\be + 1}{\be-1} = 2$.
When $\be > 3$, 
applying Hausdorff-Young's inequality twice, we have
\begin{align}
\begin{split}
\bigg| \int_{\T^3} (V_0 \ast \Dr_N^2) \Dr_N^2 dx \bigg|
& \le \| V_0* \Dr_N^2 \|_{L^\infty}\| \Dr_N \|_{L^2}^2
\le \|   \jb{n}^{-\be} \ft {\Dr_N^2 }\|_{\l^1_n}\| \Dr_N \|_{L^2}^2\\
& \les \|   \ft {\Dr_N^2 }\|_{\l^\infty_n}\| \Dr_N \|_{L^2}^2
\les \| \Dr_N \|_{L^2}^4.
\end{split}
\label{YY15}
\end{align}

From \eqref{YY12} and \eqref{YY13} with 
Lemmas \ref{LEM:Dr}, 
\ref{LEM:Dr2}, 
\ref{LEM:Dr4}, and \ref{LEM:Dr5},  
 \eqref{YY14}, 
  and  \eqref{YY15}, 
 and $\max (\frac{\b+1}{\b-1},2+\eps) \le \g<3$
 with $\g > 2$ when $\be = 3$,
we obtain
\begin{align*}
\inf_{N \in \mathbb{N}} \inf_{\dr \in \Hc} \W_N(\dr) 
\geq 
\inf_{N \in \mathbb{N}} \inf_{\dr \in \Hc}
\Big\{ -C_0 + \frac{1}{10}\U_N(\dr)\Big\}
 \geq - C_0 >-\infty.
\end{align*}

\noi
Therefore, 
 from  
an analogue of Lemma \ref{LEM:var2} for $\RR_N(u)$, 
 we conclude 
the uniform exponential integrability \eqref{exp3}, 
provided that $\frac{\be + 1}{\be - 1} <  3$, namely, 
$\be > 2$.

\subsection{Non-normalizability of the focusing Gibbs measure}
\label{SUBSEC:foc2}

In this subsection, we prove the non-normalizability
of the focusing Hartree Gibbs measure
for $\be < 2$ with any $\s > 0$ (Theorem~\ref{THM:Gibbs2}\,(ii))
and for $\be = 2$ with $\s \gg 1$
(Theorem \ref{THM:Gibbs2}\,(iii.a)).
When $\be < 2$, 
the non-normalizability follows 
 from 
 the next proposition.

\begin{proposition} \label{PROP:Gibbs5}
Given  $1 < \beta < 2$,  let $V$ be the Bessel potential of order $\b$.
Then,  for any $\s > 0$, there exists $K > 0$ such that\footnote{It is indeed possible to prove
Proposition \ref{PROP:Gibbs5}
for any $K > 0$.
See Remark \ref{REM:OS}.}  
\[ \lim_{L \to \infty} \liminf_{N \to \infty} \E\bigg[\exp\Big(\min{(\s R_N(u),L)} \Big) 
\cdot \ind_{\{ |\int_{\T^3} \, : \, u_N^2 :\, dx | \le K\}} \bigg] =  \infty,\]

\noi
where $R_N(u)$ is as  in \eqref{K1}.
\end{proposition}

We first present the proof of 
 Theorem \ref{THM:Gibbs2}\,(ii)
by assuming  
 Proposition \ref{PROP:Gibbs5}.

\begin{proof}[Proof of Theorem \ref{THM:Gibbs2}\,(ii)]
It follows from \eqref{K1} and \eqref{K2} that
\[
\s  R_N(u) = \RR_N(u) + A \bigg| \int_{\T^3} :\! u_N^2 \!: dx \bigg|^\gamma. 
\]

\noi
In view of \eqref{pa00}, 
 we have,  for any $L>0$,
\begin{align*}
\E \Big[e^{\RR_N (u)} \Big]
&= \E\bigg[\exp\bigg( \s R_N(u) - A \bigg| \int_{\T^3} :\! u_N^2 \!: dx \bigg|^\gamma \bigg)\bigg] \\
&\ge \E\bigg[\exp\bigg( \min{(\s R_N(u),L)}
- A \bigg| \int_{\T^3} :\! u_N^2 \!: dx \bigg|^\gamma \bigg)  \bigg]\\ 
&\ge\exp\big(-A K^\g\big) 
\E\bigg[\exp\Big(\min{(\s  R_N(u),L)} \Big) \cdot \ind_{\{ |\int_{\T^3}\, : u_N^2 :\, dx | \le K\}} \bigg].
\end{align*}

\noi
Then, \eqref{Gibbs10} follows from 
Proposition \ref{PROP:Gibbs5}.
\end{proof}

\begin{remark}\label{REM:end1}\rm
(i) Proposition \ref{PROP:Gibbs5} holds true at the critical value  $\be = 2$, 
provided that $\s \gg 1$.  See Remark \ref{REM:end2}.
Then, the argument above proves 
Theorem \ref{THM:Gibbs2}\,(iii.a).

\smallskip

\noi
(ii)  
Proposition \ref{PROP:Gibbs5} and Part (i) of this remark 
establish the non-normalizability of the focusing Hartree Gibbs measure
$\rho$ in \eqref{Gibbs7}
with a Wick-ordered $L^2$-cutoff, considered by Bourgain~\cite{BO97}, 
(a) for $\be < 2$ with any $\s > 0$
and (b) for $\be = 2$ with $\s \gg1$:
\begin{align}
\sup_{N \in \N}  \E_{\mu} \bigg[e^{\frac \s 4 \int_{\T^3} (V \ast \, : \, u_N^2 :) \, : u_N^2 : \, dx
 - \frac \s2 \al_N} 
\ind_{\{ |\int_{\T^3} \, : \, u_N^2 :\, dx | \le K\}}
\bigg] = \infty, 
\label{pa00a}
\end{align}

\noi
provided that $K \gg 1$.\footnote{In a recent preprint \cite{OS}, 
the first and third authors with K.~Seong developed further the strategy
introduced in this subsection on non-normalizability
of focusing Gibbs measures.
In particular, by adapting the approach in~\cite{OS}, 
we can remove the assumption $K \gg 1$
and thus prove the non-normalizability \eqref{pa00a}
for any $K > 0$.
See Remark \ref{REM:OS} below.}

In view of \eqref{pa00}, 
if we replace the Wick-ordered $L^2$-cutoff $\ind_{\{\int_{\T^3} :\,|u|^2: \, dx\leq K\}}$
in \eqref{Gibbs7} by $\ind_{\{|\int_{\T^3} :\,|u|^2: \, dx|\leq K\}}$,
namely, with an absolute value,  
then 
the construction of the focusing  Hartree Gibbs measure:
\begin{align*}
d\rho(u) = Z^{-1} \ind_{\{|\int_{\T^3} :\,|u|^2: \, dx|\leq K\}} \, e^{\frac \s4 \int_{\T^3} (V*:|u|^2:)\,  :|u|^2 :\, dx} d\mu(u)
\end{align*}

\noi
in the weakly nonlinear regime ($0 < \s \ll 1$) at the critical value $\be = 2$
follows from the corresponding construction 
for the focusing Gibbs measure in \eqref{Gibbs9}
presented in 
Subsection~\ref{SUBSEC:foc3}.
As for the focusing  Hartree Gibbs measure
with the Wick-ordered $L^2$-cutoff $\ind_{\{\int_{\T^3} :\,|u|^2: \, dx\leq K\}}$
in \eqref{Gibbs7}, 
we start with  the truncated Gibbs measure $\rho_N$ in \eqref{GibbsN1}
with a slightly different potential energy $\RR_N(u)$ (compare this with \eqref{K2}):
\begin{align*}
\RR_N (u)
&:= \frac \s 4 \int_{\T^3} (V \ast :\! u_N^2 \!:) :\! u_N^2 \!: dx
 - A \, \Bigg(\bigg| \int_{\T^3} :\! u_N^2 \!: dx\bigg|^{\g-1} \int_{\T^3} :\! u_N^2 \!: dx\Bigg) - \frac \s2 \al_N
\end{align*}

\noi
and repeat the analysis presented in Subsections~\ref{SUBSEC:foc1}
and~\ref{SUBSEC:foc3}.
Then, an inequality
\begin{align*}
 \ind_{\{x \le K\}} \le \exp\big( -  A |x|^{\gamma-1}x \big) \exp\big(A K^\g\big)
\end{align*}

\noi
for any $x\in \R$, $K>0$, $\g>0$, and $A>0$
yields the normalizability of the 
focusing  Hartree Gibbs measure
in \eqref{Gibbs7} 
as in Theorem \ref{THM:Gibbs2}\,(i) and (iii.b).
In particular, this extends Bourgain's construction  to the critical case $\be = 2$
in the weakly nonlinear regime.
 
\end{remark}

\medskip

The rest of the section is  devoted  to the proof of Proposition \ref{PROP:Gibbs5}. We first note that 
\begin{align}
\begin{split}
 \E\Big[ & \exp\Big(\min{(\s R_N(u),L)} \Big) 
 \cdot\ind_{\{ |\int_{\T^3} \, : u_N^2 :\, dx | \le K\}} \Big]\\
 &
\ge  \E\Big[\exp\Big(\min{(\s R_N(u),L)}
\cdot \ind_{\{ |\int_{\T^3} \, : u_N^2 :\, dx | \le K\}}\Big)   \Big] 
- \PP\bigg(\bigg|\int_{\T^3} :\! u_N^2 \!: dx \bigg| > K\bigg) \\
&\ge \E\Big[\exp\Big(\min{(\s R_N(u),L)} 
\cdot \ind_{\{ |\int_{\T^3} \, : u_N^2 :\, dx | \le K\}}\Big)   \Big] 
- 1,
\end{split}
\label{pax}
\end{align}

\noi
and thus it suffices to prove 
\begin{align}
\lim_{L \to \infty} \liminf_{N \to \infty}  \E\Big[\exp\Big(\min{(\s  R_N(u),L)} 
\cdot \ind_{\{ |\int_{\T^3} \, : u_N^2 :\, dx | \le K\}}\Big)   \Big] =  \infty.
\label{pa0}
\end{align}

As in the previous subsections, 
 we will use a variational formulation.
In this part, however, 
we take a drift $\dr$ depending on $Y$
and thus
we need to use a variational formula,
where an  expectation is taken
with respect to  the underlying probability $\PP$, rather than the modified one $\Q_\theta$
(as in Lemma \ref{LEM:var2}). 
For this purpose, 
we first recall the  Bou\'e-Dupuis variational formula \cite{BD, Ust};
in particular, see Theorem 7 in \cite{Ust}.

\begin{lemma}\label{LEM:var3}
Let $Y(t)
=  \jb{\nabla}^{-1}W(t)$ be as in \eqref{P2}.
Fix $N \in \N$.
Suppose that  $F:C^\infty(\T^3) \to \R$
is measurable such that $\E\big[|F(\pi_NY(1))|^p\big] < \infty$
and $\E\big[|e^{-F(\pi_NY(1))}|^q \big] < \infty$ for some $1 < p, q < \infty$ with $\frac 1p + \frac 1q = 1$.
Then, we have
\begin{align*}
- \log \E\Big[e^{-F(\pi_N Y(1))}\Big]
= \inf_{\dr \in \mathbb H_a}
\E\bigg[ F(\pi_N Y(1) + \pi_N I(\dr)(1)) + \frac{1}{2} \int_0^1 \| \dr(t) \|_{L^2_x}^2 dt \bigg], 
\end{align*}

\noi
where  $I(\dr)$ is  as in  \eqref{P3a}
and the expectation $\E = \E_\PP$
is an 
expectation with respect to the underlying probability measure $\PP$.

\end{lemma}

In our current context, Lemma \ref{LEM:var3}, 
together with 
Lemma~\ref{LEM:conv}, 
yields
\begin{align}
\begin{split}
-& \log {\E\Big[\exp\Big(\min{(\s  R_N(u),L)} \cdot \ind_{\{ |\int_{\T^3} \, : u_N^2 :\, dx | \le K\}}\Big)   \Big]} \\
&= \inf_{\dr \in \mathbb H_a} \E\bigg[ -\min\big(\s  R_N(Y(1) + I(\dr)(1)),L\big)\\
&\hphantom{XXXXX} \times  \ind_{\{ |\int_{\T^3} : (\pi_N Y(1))^2: 
+ 2 (\pi _N Y(1)) (\pi_N I(\dr)(1)) + (\pi_N I(\dr)(1))^2 dx | \le K\}} \\
&\hphantom{XXXXX}
+ \frac 12 \int_0^1 \| \dr(t)\|_{L^2_x} ^2 dt \bigg],
\label{DPf}
\end{split}
\end{align}

\noi
where $Y(1)$ is as in \eqref{P2} and 
$\mathbb H_a$ is as in Subsection \ref{SUBSEC:var}.
For simplicity, we denote  $\pi_N Y(1)$ by $Y_N$ in the following.

Fix a parameter $M \gg 1$.
Let $f: \R^3 \to \R$ be a real-valued Schwartz function
such that 
the Fourier transform $\ft f$ is an even smooth function
supported  on $\big\{\frac 12 <  |\xi| \le 1 \big\}$, satisfying 
  $\int_{\R^3} |\ft f (\xi)|^2 d\xi = 1$.
Define a function $f_M$  on $\T^3$ by 
\begin{align}
f_M(x) &:= M^{-\frac 32} \sum_{|n| >  M/2 } \ft f\Big( \frac n M\Big) e_n, \label{fMdef} 
\end{align}

\noi
where $\ft f$ denotes the Fourier transform on $\R^3$ defined by 
\[ \ft f(\xi) = \frac{1}{(2\pi)^\frac{3}{2}} \int_{\R^3} f(x) e^{-i\xi\cdot x} dx.\]

\noi
Then, a direct computation yields  the following lemma.

\begin{lemma}\label{LEM:leo1}
Let $0<\b<3$.
Then, given any  $\alpha > 0$, we have
\begin{align}
\int_{\T^3} f_M^2 dx &= 1 + O(M^{-\alpha}), \label{fM0} \\
\int_{\T^3} (\jb{\nabla}^{-1} f_M)^2 dx &\les M^{-2}, \label{fm2} \\
\int_{\T^3} (V \ast f_M^2) f_M^2 dx &\sim M^{3-\beta}. \label{fM1}
\end{align}
\end{lemma}

\begin{proof}
Define a function $F_M$ on $\R^3$ by 
$$ F_M(x) := M^\frac 32 f(M x).$$

\noi
Then, 
by the Poisson summation formula,\footnote{Recall our convention
of using the normalized Lebesgue measure 
$dx_{\T^3}= (2\pi)^{-3} dx$
on $\T^3\cong[-\pi, \pi)^3$.
For simplicity of notation, 
we use $dx$ to denote the standard Lebesgue measure $\R^3$
and the normalized Lebesgue measure on $\T^3$ in the following.}
we have 
\begin{equation}
f_M(x)
= (2\pi)^\frac{3}{2}\sum_{k \in \Z^3}F_M(x + 2\pi k)
= \sum_{k \in \Z^3} T_k f(x), \label{fma}
\end{equation}
where
\begin{align}
T_kf (x) := (2\pi)^\frac{3}{2}M^\frac32 f(M(x + 2\pi k)).
\label{fmaa}
\end{align}

\noi
Since $f$ is a Schwartz function, if $|x| \le \pi$ and $k \in \Z^3 \setminus\{0\}$, we have 
\[ |f(M(x + 2\pi k))| \les  (M|k|)^{-\alpha - 3}\]

\noi
for any $\alpha > 0$,  
from which we obtain, for $k \in \Z^3 \setminus\{0\}$, 
\begin{align}
\int_{\T^3} (T_k f(x))^2 dx \les |k|^{-2\alpha -6} M^{-2\alpha-3}.
\label{fma1}
\end{align}

\noi
For $k = 0$, we have
\begin{align}
\int_{\T^3} (T_0 f(x))^2 dx = \int_{|x| \le \pi M} f^2(x) dx 
= 1 - \int_{|x| > \pi M} f^2(x) dx = 1 - O(M^{-\alpha}).
\label{fma2}
\end{align}

\noi
Hence, 
it follows from \eqref{fma}, \eqref{fma1}, and \eqref{fma2} that
\begin{align*}
 \bigg| & \int_{\T^3} f_M^2(x) dx - 1\bigg| \\
&= \bigg|\sum_{k,j \in \Z^3} \int_{\T^3} T_kf(x)T_jf(x) dx - 1\bigg| \\
&= \bigg|\int_{\T^3} (T_0f(x))^2 dx - 1 + 2 \sum_{k \neq 0} \int_{\T^3} T_0f(x)T_kf(x) dx 
+ \sum_{k,j \neq 0} \int_{\T^3} T_kf(x)T_jf(x) dx\bigg| \\
&\les M^{-\alpha} \bigg(1 + M^{-\frac 32} \sum_{k\neq0} |k|^{-\alpha-3} 
+ M^{-\alpha-3}\sum_{k,j \neq 0}|k|^{-\alpha-3}|j|^{-\alpha-3} \bigg) \\
&\les M^{-\alpha},
\end{align*}

\noi
for any $\al > 0$.  This proves \eqref{fM0}.

By Plancherel's identity, \eqref{fMdef},  and \eqref{fM0},
we have 
\begin{align*}
\int_{\T^3} (\jb{\nabla}^{-1} f_M (x))^2 dx
&= \sum_{|n| > M/2} M^{-3} \Big|\ft f \Big(\frac n M\Big)\Big|^2 \frac 1 {\jb{n}^2}\\
&\les M^{-5} \sum_{n \in \Z^3}  \Big|\ft f \Big(\frac n M\Big)\Big|^2 \\
&= M^{-2} \|f_M\|_{L^2}^2 \\
&\les M^{-2}.
\end{align*}

\noi
This proves \eqref{fm2}.

It remains to prove \eqref{fM1}.
By Hausdorff-Young's inequality,  \eqref{fma}, \eqref{fma1}, and \eqref{fma2}, we have 
\begin{align}
\begin{split}
&\sup_{n \in \Z^3} \Big( (1 + |n|^4) \big|\ft{f_M^2}(n) - \ft{(T_0 f)^2}(n)\big| \Big) \\
&\les \Big\| (1+\Dl^2) \big( f_M^2 - (T_0 f)^2 \big) \Big\|_{L^1} \\
&= \bigg\| (1+ \Delta^2) \Big( 2 T_0f \sum_{k\neq 0} T_k f
 + \sum_{k,j\neq 0}  T_k f T_j f \Big) \bigg\|_{L^1}\\
&\lesssim M^{-\wt \alpha+\frac 52}
\les M^{-\al}
\end{split}
\label{fma4}
\end{align}

\noi
for any $\wt \al > 0$ such that  $\wt \al>\al+\frac 52$.
Hence,  Plancherel's identity,  \eqref{fmaa}, \eqref{fma4}, 
and Hausdorff-Young's inequality with \eqref{fM0} and \eqref{fma2} 
yields that
\begin{align}
\begin{split}
\bigg| & \int_{\T^3} (V \ast f_M^2) f_M^2 dx 
- \sum_{n \in \Z^3} \widehat V(n) 
\big|\widehat{(T_0f)^2}(n)\big|^2\bigg|\\
&=\bigg| \sum_{n \in \Z^3} \widehat V(n)  \Big(\big|\widehat{f_M^2}(n)\big|^2 
- \big|\widehat{(T_0f)^2}(n)\big|^2\Big)\bigg|\\
&\les \sum_{n \in \Z^3} \jb{n}^{-\b-4} \Big( (1 + |n|^4) \big|\widehat{f_M^2}(n) - \widehat{(T_0 f)^2}(n)\big| \Big) \Big( \big|\widehat{f_M^2}(n) \big| + \big| \widehat{(T_0 f)^2}(n)\big| \Big) \\
&\lesssim  \sum_{n \in \Z^3} M^{-\alpha} \jb n^{-\beta - 4}\\
&\lesssim M^{-\alpha}.
\end{split}
\label{fma5}
\end{align}

\noi
By the  assumption,  
 $\widehat{f^2}
= \ft f * \ft f $ is an even Schwartz  function with 
$\supp \ft{f^2} \subset  \{ |\xi| \le 2 \}$ and $\widehat{f^2}(0) = 1$.
Moreover, from \eqref{fmaa}, we have $\widehat{(T_0f)^2}(\xi) = (2\pi)^3 \ft {f^2}\big(\frac{\xi}{M}\big)$.
Thus, 
 we have 
\[
 \frac12 \cdot \ind_{\{|\,\cdot\,| \le c_1 M \}}(\xi)
 \le \big|\widehat{(T_0f)^2}(\xi)\big|
\le c_2 \cdot \ind_{\{ |\,\cdot\,| \le 2M \}} (\xi) \]

\noi
for some $c_1, c_2 > 0$.
Thus, we obtain
\begin{align}
\begin{split}
&\sum_{n \in \Z^3} \widehat V(n) 
\big|\widehat{(T_0f)^2}(n)\big|^2
\les \sum_{|n| \le 2M} \jb{n}^{-\beta} 
\les M^{3-\beta}, \\
& \sum_{n \in \Z^3} \widehat V(n) 
\big|\widehat{(T_0f)^2}(n)\big|^2
\gtrsim \sum_{|n| \le c_1 M} \jb{n}^{-\beta}
\sim M^{3-\beta}.
\end{split}
\label{fma6}
\end{align}

\noi
Therefore, 
from \eqref{fma5} and \eqref{fma6}, we obtain \eqref{fM1}.
\end{proof}

Let $Y$ be as in \eqref{P2}.
We define $Z_M$ and $\wt \s_M$ by 
\begin{align}
Z_M := \sum_{|n| \le M} \widehat{Y\big(\tfrac 12)} (n) e_n
\qquad \text{and}\qquad 
\wt \s_M := \E \big[ Z_M^2(x) \big]. 
\label{fmb1}
\end{align}

\noi
Note that $\wt \s_M$ is independent of $x \in \T^3$
thanks to 
 the spatial translation invariance of $Z_M$.

\begin{lemma} \label{LEM:leo2}
Let $ M\gg 1$ and let $ 1 \le p < \infty$.
Then, we have
\begin{align}
&\wt \s_M \sim M,\label{NRZ0}\\
&\E\bigg[\int_{\T^3} |Z_M|^p dx \bigg] \le C(p) M^\frac p2, \label{NRZ1}\\
&\E\bigg[\Big(\int_{\T^3} Z_M^2 dx - \wt \s_M \Big)^2\bigg]
+\E\bigg[ \Big( \int_{\T^3} Y_N Z_M dx - \int_{\T^3} Z_M^2 dx \Big)^2  \bigg] \les 1,    \label{NRZ3}\\
&\E\bigg[\Big( \int_{\T^3} Y_N  f_M dx \Big)^2\bigg] 
+ \E\bigg[\Big( \int_{\T^3} Z_M  f_M dx \Big)^2\bigg] \les M^{-2}   \label{NRZ5}
\end{align}

\noi
for any $N \ge M$.

\end{lemma}

\begin{proof}
From  \eqref{fmb1} and \eqref{P2}, we have
\begin{align*}
\wt \s_M
= \sum_{n \in \Z^3} \E \Big[ |\ft Z_M(n)|^2 \Big]
\sim \sum_{|n| \le M} \frac{1}{\jb{n}^2} \sim M,
\end{align*}

\noi
yielding  \eqref{NRZ0}.
The second estimate \eqref{NRZ1} follows
from Minkowski's integral inequality, 
the Wiener chaos estimate (Lemma \ref{LEM:hyp}), 
and \eqref{NRZ0}.

By the independence of $\big\{|\ft Z_M(n)|^2 - \E \big[ | \ft Z_M(n)|^2\big]\big\}_{n \in \Ld_0}$, 
where $\Ld_0$ is as in \eqref{index}, 
we have
\begin{align*}
\E\bigg[ \Big( \int_{\T^3} Z_M^2 dx - \wt \s_M \Big)^2\bigg]
&=\E \Bigg[ \bigg(\sum_{n \in \Z^3} \Big( |\ft Z_M(n)|^2 - \E \big[ | \ft Z_M(n)|^2 \big] \Big) \bigg)^2 \Bigg]\\
&\sim  \sum_{n \in \Z^3} \frac 1 {\jb{n}^4} \les 1.
\end{align*}

\noi
Using the independence of 
$B_n(1) - B_n(\frac 12) $ and $B_n(\frac 12)$, 
we obtain
\begin{align*}
\E\bigg[ \Big( \int_{\T^3} Y_N Z_M dx  - \int_{\T^3} Z_M^2 dx \Big)^2  \bigg]
&= \E \Bigg[ \bigg(\sum_{n \in \Z} \Big( \ft Y_N(n)\overline{\ft Z_M(n)} 
- |\ft Z_M(n)|^2 \Big) \bigg)^2 \Bigg] \\
&= \E \Bigg[ \bigg(\sum_{|n| \le M} \frac{(B_n(1) - B_n(\frac 12)) 
\cj{B_n(\frac 12)}}{\jb{n}^2} \bigg)^2 \Bigg] \\
&\les \sum_{n \in \Z^3} \frac{\E \big[ |B_n(1) - B_n(\frac 12)|^2 \big] \E \big[ |B_n(\frac 12)|^2 \big]}{\jb{n}^4}\\
&\les \sum_{n \in \Z^3} \frac 1 {\jb{n}^4} \les 1.
\end{align*}

\noi
This proves  \eqref{NRZ3}.

Lastly, from \eqref{fm2}, we have
\begin{align*}
\E\bigg[ \Big( \int_{\T^3} Y_N  f_M dx  \Big)^2\bigg]
&= \E \bigg[ \Big( \sum_{|n| \le N} \ft Y_N(n) \ft f_M(n) \Big)^2 \bigg]
= \sum_{|n| \le N} \frac 1{\jb{n}^2} |\ft f_M(n)|^2 \\
&\le \int_{\T^3} \big(\jb{\nabla}^{-1} f_M (x)\big)^2 dx
\les M^{-2}.
\end{align*}

\noi
A similar computation shows the same bound
holds for the second term in \eqref{NRZ5}.
\end{proof}

We are now ready to  prove Proposition \ref{PROP:Gibbs5}.

\begin{proof}[Proof of Proposition \ref{PROP:Gibbs5}]
For $M \gg 1$,
we set $f_M$, $Z_M$, and $\wt \s_M$ as in \eqref{fMdef} and \eqref{fmb1}.
We choose  a drift $\dr= \dr^0$ for \eqref{DPf}, defined by  
\begin{align}
\dr^0 (t) = 2 \cdot \ind_{t > \frac 12} \jb{\nabla} \big( -Z_M + \sqrt{\wt \s_M} f_M \big)
\label{paax}
\end{align}

\noi
such  that 
\begin{align}
\Dr^0 := I(\dr^0)(1) 
= \int_0^1 \jb{\nb}^{-1} \dr^0(t) dt = - Z_M + \sqrt{\wt \s_M} f_M.
\label{paa0}
\end{align}

\noi
Furthermore, define $Q(u$) by 
\begin{align}
Q(u) := \frac 14 \int_{\T^3} (V_0 \ast u^2) u^2 dx
 = \frac 14 \|u^2\|_{\dot H^{-\frac{\be}{2}}} ,
\label{paa1}
\end{align}

\noi
where $V_0 = V - \ft V(0)$ as in \eqref{V0}.

\begin{remark} \label{REM:choice}\rm
Our choice of the drift in \eqref{paax} (or rather \eqref{paa0}) is based on the following.
In view of~\eqref{DPf}, 
we would like to choose (the space-tine integral of) a drift 
as ``$-Y(1)~+ $ a deterministic perturbation'', 
where the deterministic perturbation drives $R_N \to \infty$.
In view of the regularity condition on drifts, however, we can not use $-Y(1)$
as it is. This gives rise to $-Z_M$ in~\eqref{paa0}, which is
nothing but a smooth approximation\footnote{It is possible to introduce
a more refined approximation of  $-Y(1)$.  See Remark \ref{REM:OS} below.} of $-Y(1)$.
The cutoff function $\ind_{t > \frac 12}$  in \eqref{paax} is inserted to guarantee
the progressive measurability of the drift $\dr^0$.
As for the choice of the deterministic perturbation, 
noting that the main part of the renormalized potential energy $R_N$ in \eqref{K1}
is given by $Q$ in \eqref{paa1}, 
we chose a function $f_M$ such that 
$Q( \sqrt{\wt \s_M} f_M)$ provides the desired divergence.
See \eqref{paa2} below.
Lastly, our choice of $\wt \s_M$ in~\eqref{fmb1}
allows us to view $Z_M^2 - \wt \s_M$ as a Wick renormalization, 
which plays a crucial role in the proof of \eqref{pa8}, presented at the end of this subsection.
%
\end{remark}

Let us first make some preliminary computations.
From 
\eqref{paa0}, \eqref{paa1}, and 
Young's inequality, we have 
\begin{align}
\begin{split}
& Q(\Dr^0) - \wt \s_M^2 Q(f_M) \\
&= - \int_{\T^3} \big(V_0\ast ( \sqrt{\wt \s_M} f_M)^2\big)\sqrt{\wt \s_M} f_M Z_M dx
+ \frac12 \int_{\T^3} \big( V_0\ast (\sqrt{\wt \s_M} f_M)^2 \big) Z_M^2 dx \\
&\quad + \int_{\T^3} \big( V_0\ast (\sqrt{\wt \s_M} f_M Z_M) \big) \sqrt{\wt \s_M} f_M Z_M dx
- \int_{\T^3} (V_0\ast Z_M^2) \sqrt{\wt \s_M} f_M Z_M dx \\
&\quad
+ Q(Z_M) \\
&\ge -\delta \wt \s_M^2 Q(f_M)
- C_\delta  \int_{\T^3} \big( V_0\ast (\sqrt{\wt \s_M} f_M)^2 \big) Z_M^2 dx +(1-\delta) Q(Z_M)  \\
&\ge -\delta \wt \s_M^2 Q(f_M) - C_\delta  \int_{\T^3} \big( V_0\ast (\sqrt{\wt \s_M} f_M)^2 \big) Z_M^2 dx
\end{split}
\label{pa1}
\end{align}
for any $0<\dl<1$.
From Lemmas \ref{LEM:leo1} and \ref{LEM:leo2}, 
we have
\begin{align}
\begin{split}
\E \bigg[ \int_{\T^3} \big( V_0\ast (\sqrt{\wt \s_M} f_M)^2 \big) Z_M^2 dx \bigg]
&=
\int_{\T^3} \big( V_0\ast (\sqrt{\wt \s_M} f_M)^2 \big) \wt \s_M dx \\
&\les \wt \s_M^2 \| f_M \|_{L^2}^2
\les M^2.
\end{split}
\label{pa2}
\end{align}

\noi
Then,
for any  measurable set $E$ with $\PP(E) > \frac 12$
and  any  $L \gg \s \cdot \wt \s_M^2 Q(f_M)$,
it follows from 
 \eqref{pa1}, \eqref{pa2}, \eqref{NRZ0}, and \eqref{fM1}
 that 
\begin{align}
\E \Big[\min\big(\tfrac \s 2 Q(\Dr^0), L\big) \cdot \ind_E\Big]
\ge  \s (1-\delta) \wt \s_M^2 Q(f_M) \PP(E) - C'_\delta \s  M^2
\gtrsim  \s  M^{5-\beta}, 
\label{paa2}
\end{align}

\noi
provided that $0< \be<3$.

%
Recall that  both $\ft Z_M$ and $\ft f_M$ are supported on $\{|n|\leq M\}$.
Then, 
 by Lemma \ref{LEM:Dr}, \eqref{paax},  \eqref{paa0}, 
and Lemmas \ref{LEM:leo1} and \ref{LEM:leo2},  we have 
\begin{align}
&\E \big[ \| \Dr^0\|_{H^1}^2 \big]
\le \E \bigg[ \int_0^1 \|\dr^0(t)\|_{L^2}^2  dt \bigg]
\les M^2 \E \big[ \| \Dr^0\|_{L^2}^2 \big]
\les M^3. \label{pa4}
\end{align}

We now impose $\be>1$. 
Then, it follows from \eqref{Y0} and Lemmas \ref{LEM:Dr2} and \ref{LEM:Dr3}
that 
\begin{align}
\begin{split}
\s R_N(Y + \Dr^0)
&\ge \frac \s 2 Q(\Dr^0) \\
& \quad - c(\s) \Big( 1 + \|Y_N\|_{\C^{-\frac 12 -\eps}} + \|:\! Y_N^2 \!:\|_{\C^{-1 -\eps}} 
+ \| (V_0 \ast :\! Y_N^2 \!:) Y_N \|_{\C^{-\frac 12-\eps}}
\Big)^c \\
&\quad
+ \frac \s{32} \|\Dr^0\|_{L^2}^4 -c_0 \|\Dr^0\|_{H^1}^2
- \frac \s 4 |Q_N(Y)|, 
\end{split}
\label{paa}
\end{align}

\noi
where $c_0$ is independent of $\s > 0$.
Suppose that\footnote{From \eqref{paa0} and $N > M$, we have $\pi_N  \Dr^0= \Dr^0$.}
\begin{align}
\PP\bigg( \Big|\int_{\T^3} (:{Y_N^2}: + 2 Y_N \Dr^0 + (\Dr^0)^2) dx \Big| \le K \bigg) > \frac 12, 
\label{pa5}
\end{align}

\noi
uniformly in $M \gg 1$ and $N\ge M$, 
and $L \gg \s \cdot \wt \s_M^2 Q(f_M) \sim \s  M^{5 - \be}$.
Then,
putting together,  
\eqref{DPf},  \eqref{paa2},  \eqref{pa4}, \eqref{paa}
with 
 Lemma \ref{LEM:conv} (in particular \eqref{B3}),  
there exist constants $C_1, C_2, C_3 > 0 $ such that 
\begin{align}
& -\log{\E\Big[\exp\Big(\min{( \s R_N(u),L)} \cdot \ind_{\{ |\int_{\T^3} \, : u_N^2 :\, dx | \le K\}}\Big)   \Big]} \notag \\
&\le \E\bigg[ -\min\big(\s R_N(Y + \Dr^0),L\big)\cdot 
\ind_{\{ |\int_{\T^3} ( :{Y_N^2}: + 2 Y_N \Dr^0 + (\Dr^0)^2) dx | \le K\}} + \frac 12 \int_0^1 \| \dr^0(t)\|_{L^2_x} ^2 dt \bigg] \notag \\
&\le  \E\bigg[ -\min\big(\tfrac \s  2 Q(\Dr^0),L\big)\cdot \ind_{\{ |\int_{\T^3} ( :{Y_N^2}: + 2 Y_N \Dr^0 + (\Dr^0)^2) dx | \le K\}} \notag \\
&\hphantom{XXX}
 + c(\s) \Big( 1 + \|Y_N\|_{\C^{-\frac 12 -\eps}} + \|:\! Y_N^2 \!:\|_{\C^{-1 -\eps}} 
+ \| (V_0 \ast :\! Y_N^2 \!:) Y_N \|_{\C^{-\frac 12-\eps}}
\Big)^c \notag \\
&\hphantom{XXX}
+ c_0\|\Dr^0\|_{H^1}^2
+ c(\s) |Q_N(Y)|
+ \frac 12 \int_0^1 \| \dr^0(t)\|_{L^2_x} ^2 dt \bigg] \notag \\
&\le -\s C_1 M^{5-\beta} + C_2 M^{3} + C_3
\label{pa6}
\end{align}

\noi
for any  $N \ge M \gg1 $.
Therefore, 
we conclude from \eqref{pax} and \eqref{pa6}
that 
\begin{align}
\begin{split}
 \lim_{L \to \infty} & \liminf_{N \to \infty} \E\Big[\exp\Big(\min{(\s R_N(u),L)} \Big) 
 \ind_{\{ \left|\int_{\T^3} \, : u_N^2 :\, dx \right| \le K\}} \Big]\\
\ge&~  \exp\Big( \s C_1 M^{5-\beta} - C_2 M^{3} -C_3(\s)\Big) \too  \infty
\end{split}
\label{pa7}
\end{align}

\noi
as $M \to \infty$, 
provided that $\be < 2$.
This proves \eqref{pa0} by assuming \eqref{pa5}.

Now, it remains to prove \eqref{pa5} for some $K \gg1$, namely, 
\begin{equation}
\PP\bigg( \Big|\int_{\T^3} \Big( :{Y_N^2}: + 2 Y_N \Dr^0 + (\Dr^0)^2 \Big) dx \Big| > K\bigg) \le \frac 12, \label{pa8}
\end{equation}

\noi
uniformly in $M \gg 1$ and $N\ge M$.
From \eqref{paa0} and 
Lemmas \ref{LEM:Dr} and \ref{LEM:leo2} with \eqref{fM0}, 
we have 
\begin{align}
\begin{split}
& \E \bigg[ \Big|\int_{\T^3} \Big(  :{Y_N^2}: + 2 Y_N \Dr^0 + (\Dr^0)^2 \Big) dx \Big|^2 \bigg]  \\
&= \E \bigg[ \Big|\int_{\T^3} :{Y_N^2}: dx - 2 \int_{\T^3} Y_N Z_M dx+ 2 \sqrt{\wt \s_M} \int_{\T^3} Y_N f_M dx \\
&\qquad + \int_{\T^3} Z_M^2 dx - 2 \sqrt{\wt \s_M} \int_{\T^3} Z_Mf_M dx  + \wt \s_M \int_{\T^3} f_M^2 dx \Big|^2 \bigg] \\
&\les \E \bigg[ \Big(\int_{\T^3} :{Y_N^2}: dx \Big)^2 \bigg]
+ \E \bigg[ \Big( -\int_{\T^3} Y_N Z_M dx + \int_{\T^3} Z_M^2 dx \Big)^2 \bigg] \\
&\quad + \E \bigg[ \Big( - \int_{\T^3} Z_M^2 dx + \wt \s_M \Big)^2 \bigg] 
+ \wt \s_M^2 \bigg( - 1 + \int_{\T^3} f_M^2 dx \bigg)^2 \\
&\quad + \wt \s_M \bigg( \E \bigg[ \Big( \int_{\T^3} Y_N f_M dx \Big)^2 \bigg]
 + \E \bigg[ \Big( \int_{\T^3} Z_Mf_M dx \Big)^2 \bigg] \bigg) \\
&\les 1, 
\end{split}
\notag
\end{align}

\noi
provided that $\al > 1$.
Then, by choosing $K \gg1$, 
the bound \eqref{pa8} follows from 
 Chebyshev's inequality.
This concludes the proof of Proposition \ref{PROP:Gibbs5}.
\end{proof}

\begin{remark}\label{REM:end2}\rm
When $\be = 2$, 
 \eqref{pa7} still holds true  as long as $\s \gg 1$, 
thus  yielding \eqref{pa0} 
 in the strongly nonlinear regime
 at the critical value $\be = 2$.

\end{remark}

\begin{remark}\label{REM:OS} \rm
In  the proof of Proposition \ref{PROP:Gibbs5}, 
we needed the assumption $K \gg1 $
in guaranteeing~\eqref{pa8}.
In a recent preprint \cite{OS}, 
the first and third authors with K.~Seong refined
the approach presented in this subsection
and proved \eqref{pa8} for {\it any} $K>0$
(in the two-dimensional setting).
Hence, by using this refined approach, 
we can show that 
Proposition~\ref{PROP:Gibbs5} (and \eqref{pa00a}) remains true 
for any $K > 0$.
See Subsection 3.2 in \cite{OS}
for the details.

\end{remark}

\subsection{Focusing Gibbs measure at the critical value $\be = 2$}
\label{SUBSEC:foc3}

We consider the focusing Hartree Gibbs measure
at the critical value $\be = 2$.
In the previous section, we prove the non-normalizability 
for $\be = 2$ in the strongly nonlinear regime ($\s \gg 1$);
see
Remarks  \ref{REM:end1} and~\ref{REM:end2}.
In this subsection,  we
show that the focusing Gibbs measure
is indeed normalizable 
for $\be = 2$ 
in the weakly nonlinear regime (i.e.~$0 < \s \ll 1$).

Let $\be = 2$.
In view of  \eqref{YY14}, 
we set  $\g = 3$
in \eqref{K2}.
More precisely, we consider the following renormalized potential energy:
\begin{align}
\RR_N (u)
&= \frac \s 4 \int_{\T^3} (V \ast :\! u_N^2 \!:) :\! u_N^2 \!: dx
 - A \, \bigg| \int_{\T^3} :\! u_N^2 \!: dx\bigg|^3 - \frac 12 \al_N, 
\label{DN1}
\end{align}

\noi
where $\s>0$ is a small constant.
Then, it suffices to prove 
\begin{align}
\inf_{N \in \N} \inf_{\dr \in \Hc}
\E_{\Q_\dr} 
\bigg[ -\RR_N(Y_\dr(1) + I(\dr)(1)) + \frac{1}{2} \int_0^1 \| \dr(t) \|_{L^2_x}^2 dt \bigg]
>-\infty.
\label{DN2}
\end{align}

\noi
In the following, we use
 the same notations as in Subsections~\ref{SUBSEC:def}
 and \ref{SUBSEC:foc1}.
The main difficulty comes from 
the failure of 
Lemma \ref{LEM:Dr5} when $\g = 3$.
See the case \eqref{DN9a} below.

From 
 \eqref{DN1},  Lemmas  \ref{LEM:Dr2}, \ref{LEM:Dr3}, and  \ref{LEM:Dr4}, 
and 
\eqref{YY14}
with Lemma \ref{LEM:Dr}, 
we reduce \eqref{DN2} to showing
\begin{align}
\sup_{N \in \N} \sup_{\dr \in \Hc}
\E
\bigg[ c_0\s  \| \Dr_N \|_{L^2}^6
- A \, \bigg| \int_{\T^3} \Big( :\! Y_N^2 \!: + \, 2 Y_N \Dr_N + \Dr_N^2 \Big) dx\bigg|^3
- \frac{1}{4} \| \Dr_N \|_{H^1}^2 \bigg] < \infty.
\label{DN3}
\end{align}

\noi
Suppose that  we have
\begin{align}
\| \Dr_N \|_{L^2}^2 \gg \bigg| \int_{\T^3} Y_N \Dr_N dx \bigg|.
\notag
\end{align}

\noi
Then, 
from 
 \eqref{YY9},
there exists a constant $c>0$ such that
\begin{align}
\begin{split}
&\bigg| \int_{\T^3} \Big( :\! Y_N^2 \!: + 2 Y_N \Dr_N + \Dr_N^2 \Big) dx\bigg|^3
\ge
\frac 14
\bigg( \int_{\T^3} \Dr_N^2 dx\bigg)^3
-c \bigg| \int_{\T^3} :\! Y_N^2 \!: dx \bigg|^3.
\end{split}
\label{DN10}
\end{align}

\noi
Hence, 
by choosing  $\s > 0$  sufficiently small,\footnote{This case works even for $\s = 1$
simply by taking $A \gg 1$.}
\eqref{DN3} follows from 
\eqref{DN10} and 
Lemma \ref{LEM:Dr}.

Next, we consider the case:
\begin{align}
\| \Dr_N \|_{L^2}^2 \les \bigg| \int_{\T^3} Y_N \Dr_N dx \bigg|.
\label{DN9a}
\end{align}

\noi
Define the sharp frequency projections $\{\proj_j\}_{j \in \N}$
by setting 
 $\proj_1 = \pi_2$
 and 
$\proj_j = \pi_{2^j} - \pi_{2^{j-1}}$
for $j \geq 2$.
Then,
write  $\Dr_N$ as 
\begin{align*}
\Dr_N  = \sum_{j=1}^\infty (\ld_j \proj_j Y_N + w_j),
\end{align*}
where
\begin{align*}
\ld_j &:=
\begin{cases}
\frac{\jb{\Dr_N, \proj_j Y_N}}{\|\proj_j Y_N\|_{L^2}^2},  & \text{if } \| \proj_j Y_N \|_{L^2} \neq 0, \\
0, & \text{otherwise},
\end{cases}
\qquad
\text{and}
\qquad  
w_j :=
\proj_j \Dr_N - \ld_j \proj_j Y_N.
\end{align*}

\noi
Note that $w_j$ is orthogonal to $\proj_j Y_N$ and $Y_N$ in $L^2(\T^3)$.
Thus, we have 
\begin{align}
\| \Dr_N \|_{L^2}^2
&= \sum_{j=1}^\infty \Big( \ld_j^2 \| \proj_j Y_N \|_{L^2}^2 + \| w_j \|_{L^2}^2 \Big), \label{DN7} \\
\int_{\T^3} Y_N \Dr_N dx
&= \sum_{j=1}^\infty \ld_j \| \proj_j Y_N \|_{L^2}^2.
\label{DN8}
\end{align}

\noi
Hence, 
from \eqref{DN9a}, \eqref{DN7}, and \eqref{DN8}, we have
\begin{align}
\sum_{j=1}^\infty \ld_j^2 \| \proj_j Y_N \|_{L^2}^2
\les \bigg| \sum_{j=1}^\infty \ld_j \| \proj_j Y_N \|_{L^2}^2 \bigg|.
\label{DN11}
\end{align}

Fix  $j_0 \in \N$ (to be chosen later).
By  Cauchy-Schwarz inequality and \eqref{DN7}, we have
\begin{align}
\begin{split}
\bigg| \sum_{j=j_0+1}^\infty \ld_j \| \proj_j Y_N \|_{L^2}^2 \bigg|
&\le \bigg( \sum_{j=1}^\infty \ld_j^2 2^{2j} \| \proj_j Y_N \|_{L^2}^2 \bigg)^{\frac 12}
\bigg( \sum_{j=j_0+1}^\infty 2^{-2j} \| \proj_j Y_N \|_{L^2}^2 \bigg)^{\frac 12} \\
&\le \bigg( \sum_{j=1}^\infty 2^{2j} \| \proj_j \Dr_N \|_{L^2}^2 \bigg)^{\frac 12}
\bigg( \sum_{j=j_0+1}^\infty 2^{-2j} \| \proj_j Y_N \|_{L^2}^2 \bigg)^{\frac 12} \\
&\sim\| \Dr_N \|_{H^1}
\bigg( \sum_{j=j_0+1}^\infty 2^{-2j} \| \proj_j Y_N \|_{L^2}^2 \bigg)^{\frac 12}.
\end{split}
\label{DN12}
\end{align}

\noi
Moreover, 
from
 Cauchy-Schwarz inequality
and \eqref{DN11} followed by 
Cauchy's inequality, we have 
\begin{align}
\begin{split}
\bigg| \sum_{j=1}^{j_0} \ld_j \| \proj_j Y_N \|_{L^2}^2 \bigg|
&\le \bigg( \sum_{j=1}^\infty \ld_j^2 \| \proj_j Y_N \|_{L^2}^2 \bigg)^{\frac 12}
\bigg( \sum_{j=1}^{j_0} \| \proj_j Y_N \|_{L^2}^2 \bigg)^{\frac 12} \\
&\le C\bigg| \sum_{j=1}^\infty \ld_j \| \proj_j Y_N \|_{L^2}^2 \bigg|^{\frac 12}
\bigg( \sum_{j=1}^{j_0} \| \proj_j Y_N \|_{L^2}^2 \bigg)^{\frac 12} \\
&\le
\frac 12 \bigg| \sum_{j=1}^\infty \ld_j \| \proj_j Y_N \|_{L^2}^2 \bigg|
+ C' \sum_{j=1}^{j_0} \| \proj_j Y_N \|_{L^2}^2.
\end{split}
\label{DN14}
\end{align}

\noi
Hence,  from \eqref{DN12} and \eqref{DN14}, 
we conclude that 
\begin{align}
\begin{split}
\bigg| \sum_{j=1}^\infty \ld_j \| \proj_j Y_N \|_{L^2}^2 \bigg|
\les
\| \Dr_N \|_{H^1}
\bigg( \sum_{j=j_0+1}^\infty 2^{-2j} \| \proj_j Y_N \|_{L^2}^2 \bigg)^{\frac 12}
+
\sum_{j=1}^{j_0} \| \proj_j Y_N \|_{L^2}^2.
\end{split}
\label{DN15}
\end{align}

Now, write  as follows:
\begin{align}
\sum_{j=j_0+1}^\infty 2^{-2j} \| \proj_j Y_N \|_{L^2}^2
&= \sum_{j=j_0+1}^\infty 2^{-2j} \int_{\T^3} :\! (\proj_j Y_N)^2 \!: dx
+ \sum_{j=j_0+1}^\infty 2^{-2j} \E \big[ (\proj_j Y_N)^2 \big].
\label{DN15x}
\end{align}

\noi
For the first term,
it follows from \eqref{P2} and \eqref{Wickz} that
\begin{align*}
&\E \bigg[ \bigg( \sum_{j=j_0+1}^\infty 2^{-2j} \int_{\T^3} :\! (\proj_j Y_N)^2 \!: dx \bigg)^2 \bigg]
\sim \sum_{j=j_0+1}^\infty 2^{-5j}
\sim 2^{-5j_0}.
\end{align*}

\noi
Set
an almost surely finite constant $B_1(\o)$ by 
\begin{align}
B_1(\o) = 
\bigg(\sum_{k = 1}^\infty 2^{4 k}
\Big( \sum_{j=k+1}^\infty 2^{-2j} \int_{\T^3} :\! (\proj_j Y_N)^2 \!: dx\Big)^2 \bigg)^\frac{1}{2}.
\label{DN15z}
\end{align}

\noi
By the Wiener chaos estimate (Lemma \ref{LEM:hyp}), 
we see that $\E \big[B_1^p\big] \leq C_p < \infty$
for any finite $p \geq 1$.
From \eqref{DN15x} and \eqref{DN15z}, we obtain
\begin{align}
\sum_{j=j_0+1}^\infty 2^{-2j} \| \proj_j Y_N \|_{L^2}^2
\les  2^{-2 j_0} B_1(\o)+ 2^{-j_0}.
\label{DN15a}
\end{align}

\noi
Similarly, we have
\begin{align}
\begin{split}
\sum_{j=1}^{j_0} \| \proj_j Y_N \|_{L^2}^2
&= \sum_{j=1}^{j_0} \int_{\T^3} :\! (\proj_j Y_N)^2 \!: dx
+ \sum_{j=1}^{j_0} \E \big[ (\proj_j Y_N)^2 \big] \\
&\les B_2(\o) + 2^{j_0}
\end{split}
\label{DN15b}
\end{align}

\noi
for some $B_2(\o)\ge0$, 
satisfying  $\E \big[B_2^p\big] \leq C_p < \infty$
for any finite $p \geq 1$.

Hence,  from \eqref{DN15} with \eqref{DN15a} and \eqref{DN15b} that
\[
\bigg| \sum_{j=1}^\infty \ld_j \| \proj_j Y_N \|_{L^2}^2 \bigg|
\les
\Big( 2^{-\frac 12 j_0} +  2^{- j_0} B_1^\frac{1}{2}(\o)\Big) \| \Dr_N \|_{H^1} +   B_2(\o)
+ 2^{j_0} .
\]

\noi
By choosing $2^{j_0} \sim 1 + \| \Dr_N \|_{H^1}^{\frac 23}$, 
it follows from \eqref{DN9a} and \eqref{DN8} 
and Cauchy's inequality that 
\begin{align}
\| \Dr_N \|_{L^2}^6
\les
\| \Dr_N \|_{H^1}^2 + 
B_1^3(\o)
+ B_2^3(\o) + 1.
\label{DN16}
\end{align}

\noi
Therefore, 
by  taking $\s > 0$ sufficiently small, 
the desired bound \eqref{DN3}
in this case follows
from~\eqref{DN16}.

\section{Further analysis in the defocusing case: $0 < \be \leq 1$}
\label{SEC:def}

\subsection{Construction of the defocusing Gibbs measure: $\frac 12 < \be \le  1$}
\label{SUBSEC:def2}

In this subsection, we present the proof of Theorem \ref{THM:Gibbs}\,(ii.a)
for $\frac 12 < \be \le 1$.
As pointed out 
in Remark \ref{REM:ren}, 
we introduce another  renormalization 
and consider a new renormalized potential energy $R^\dia_N(u)$
in~\eqref{K1r}.
Then, as in the case $\be > 1$, it suffices
to prove the uniform exponential integrability~\eqref{exp1a}
for this new potential energy $R^\dia_N(u)$.

We first extend the estimates \eqref{YY1} and \eqref{YY2}
in Lemma \ref{LEM:Dr2} to the range $0 < \be \le 1$.

\begin{lemma}\label{LEM:Dr6}
Let $V$ be the Bessel potential of order 
 $0 < \be \le 1$.
Then, there exist   small $\eps>0$ and   a constant  $c  >0$ 
 such that
\begin{align}
\begin{split}
\bigg| \int_{\T^3}  (V_0 \ast \Dr_N^2) Y_N \Dr_N  dx\bigg|
&\le c \Big(1 + 
\| Y_N \|_{\C^{-\frac 12-\eps}}^{c}\Big)\\
& \quad + \frac 1{100} \Big(
\| \Dr_N^2\|_{\dot H^{-\frac{\be}{2}}}^2 + 
 \| \Dr_N \|_{L^2}^4 +  \| \Dr_N \|_{H^1}^2\Big), 
\end{split}
\label{YY1x}\\
\bigg| \int_{\T^3} (V_0 \ast :\! Y_N^2 \!:) \Dr_N^2 dx \bigg|
&\le c \| :\! Y_N^2 \!: \|_{\C^{-1-\eps}}^{c}  + \frac 1{100} \Big(
\| \Dr_N \|_{L^2}^4 + \| \Dr_N \|_{H^1}^2\Big), 
\label{YY2x} 
\end{align}

\noi
uniformly in $N \in \N$.

\end{lemma}

\begin{proof}

The second estimate \eqref{YY2x}
follows from  a small modification 
of \eqref{YZ1a}.
From  \eqref{dual},  \eqref{prod}, \eqref{interp}, and Young's inequality, 
we have 
\begin{align}
\begin{split}
\bigg| \int_{\T^3} (V_0 \ast :\! Y_N^2 \!:) \Dr_N^2 dx \bigg|
&\le \|  :\! Y_N^2 \!: \|_{\C^{-1-\eps}} \| \Dr_N^2 \|_{B^{1 -\beta+\eps}_{1,1}} \\
&\les \| :\! Y_N^2 \!: \|_{\C^{-1-\eps}} \| \Dr_N \|_{L^2} \| \Dr_N \|_{H^{1- \be  + 2\eps}} \\
&\les \| :\! Y_N^2 \!: \|_{\C^{-1-\eps}} \| \Dr_N \|_{L^2}^{1 + \be - 2\eps } \| \Dr_N \|_{H^1}^{1 - \be +  2\eps} \\
&\le c \| :\! Y_N^2 \!: \|_{\C^{-1-\eps}}^{\frac 4{1+\be - 2\eps}} + \frac 1{100} \| \Dr_N \|_{L^2}^4 + \frac 1{100} \| \Dr_N \|_{H^1}^2,
\end{split}
\label{YZ4a}
\end{align}

\noi
verifying \eqref{YY2x} 
when $0 < \be \le 1$.

As for the first estimate \eqref{YY1x}, 
it suffices to 
 control $\| (V_0 \ast \Dr_N ^2) \Dr_N \|_{W^{\frac 12 +\eps,1}}$, using the terms 
 appearing in \eqref{v_N1}.
From \eqref{Be2} and \eqref{V0}, there exists a constant $K_0>0$ such that 
$V_+ := V_0 + K_0 > 0$.
Then, we have 
\begin{align}
\| (V_0 \ast \Dr_N ^2) \Dr_N \|_{W^{\frac 12 +\eps,1}}
\leq \| (V_+ \ast \Dr_N ^2) \Dr_N \|_{W^{\frac 12 +\eps,1}}
+ K_0 \Big(c+  \| \Dr_N \|_{H^1}^2 + \|\Dr_N\|_{L^2}^4\Big)^{1-\eps_0}
\label{YZ4b}
\end{align}

\noi
for some $0 < \eps_0 < 1$.
Letting 
\[ Q (\Dr_N) :=  \int_{\T^3} (V_+ \ast \Dr_N^2) \Dr_N^2 dx, \]

\noi
 we  have
\begin{align}
Q (\Dr_N) \le
\bigg| \int_{\T^3} (V_0 \ast \Dr_N^2) \Dr_N^2 dx\bigg|
+ K_0 \|\Dr_N\|_{L^2}^4.
\label{YZ4c}
\end{align}

\noi
We also note that 
\begin{align}
\begin{split}
\| V_+ *\Dr_N^2\|_{L^2}
& \les \| \Dr_N^2\|_{\dot H^{-\be}}
+ K_0 \| \Dr_N\|_{L^2}^2
 \les \| \Dr_N^2\|_{\dot H^{-\frac \be2}}
+ K_0 \| \Dr_N\|_{L^2}^2\\
& \les Q^{\frac 12}(\Dr_N) 
+ K_0 \| \Dr_N\|_{L^2}^2.
\end{split}
\label{YZ4c1}
\end{align}

Given $\ld > 0$, from \eqref{YZ4c1}, we have
\begin{align*}
\| (V_+ \ast \Dr_N^2) \Dr_N \|_{L^1} 
&= \int_{\T^3} |V_+ \ast \Dr_N^2| |\Dr_N| dx \\
&\les \int_{\T^3} |V_+ \ast \Dr_N^2| (\ld + \ld^{-1}\Dr_N^2)dx\\
&\les \ld \big(Q^\frac12 (\Dr_N) + \|\Dr_N\|_{L^2}^2\big) + \ld^{-1} Q(\Dr_N) .
\end{align*}

\noi
By choosing $\ld \sim Q^\frac14(\Dr_N) $, 
we obtain 
\begin{align}
\| (V_+ \ast \Dr_N^2) \Dr_N \|_{L^1} \les Q^\frac 34(\Dr_N)
+ 
\|\Dr_N\|_{L^2}^3.
\label{YZ4d}
\end{align}

\noi
Moreover, we have 
\begin{align}
\begin{split}
\| (V_+ \ast \Dr_N^2) \Dr_N \|_{\dot W^{1,1}} 
&\leq   \int_{\T^3} |V_+ \ast \Dr_N^2| |\nb \Dr_N| dx
+ \int_{\T^3} |V_+\ast(\Dr_N \nabla \Dr_N)| |\Dr_N|dx\\
&\le \int_{\T^3} |V_+ \ast \Dr_N^2| |\nabla \Dr_N| dx
+ \int_{\T^3} |\Dr_N| |\nabla \Dr_N| (V_+ \ast |\Dr_N|)dx \\
& \les \big(Q^\frac12(\Dr_N)   + \| \Dr_N \|_{L^2}^2\big) \| \Dr_N \|_{H^1} + \big\| |\Dr_N| (V_+ \ast |\Dr_N|)\big\|_{L^2}\| \Dr_N \|_{H^1}, 
\end{split}
\label{YZ4e}
\end{align}

\noi
where we used  \eqref{YZ4c1} in the last step.
By Cauchy's inequality, we have 
\begin{align}
\begin{split}
\big\| |\Dr_N| (V_+ \ast |\Dr_N|)\big\|_{L^2}^2 
&= \int_{\T^3} (V_+ \ast |\Dr_N|)^2(x) \Dr_N^2 (x) d x \\
&=  \iiint V_+(x-y) V_+(x-z) |\Dr_N(y)||\Dr_N(z)| dydz \, \Dr_N^2 (x) d x\\
&\les   \iiint V_+(x-y) V_+(x-z) \big(\Dr_N^2(y) + \Dr_N^2(z)\big) dydz \, \Dr_N^2 (x) d x \\
&\sim  \ft{ V_+} (0)  \cdot Q(\Dr_N)\\
&\leq  \big(\ft{ V_0} (0) +K_0\big) \cdot Q(\Dr_N).
\end{split}
\label{YZ4f}
\end{align}

\noi
From \eqref{YZ4e} and \eqref{YZ4f}, we obtain
\begin{align}
 \| (V_+ \ast \Dr_N^2) \Dr_N \|_{\dot W^{1,1}} \les \big(Q^\frac12(\Dr_N) +  \| \Dr_N \|_{L^2}^2\big) \| \Dr_N \|_{H^1}.
 \label{YZ4g}
\end{align}

\noi
Hence, by interpolating \eqref{YZ4d} and \eqref{YZ4g}, we have
\begin{align}
\begin{split}
 \| (V_+ \ast \Dr_N^2) \Dr_N \|_{\dot W^{\frac 12 +\eps ,1}} 
& \les  \big(Q^{\frac {5-2\eps}{8}}(\Dr_N)  +  \| \Dr_N \|_{L^2}^{\frac {5-2\eps}{2}}\big)
\| \Dr_N \|_{H^1}^{\frac 12 +\eps}\\
& \les \Big(1 + Q(\Dr_N) + \| \Dr_N \|_{H^1}^2
+ \|\Dr_N\|_{L^2}^4\Big)^{1-\eps_0}
\end{split}
\label{YZ4h}
\end{align}

\noi
for some $0 < \eps_0 < 1$.
Finally, the desired estimate \eqref{YY1x}
follows from 
\eqref{YZ4b}, 
\eqref{YZ4c}, 
\eqref{YZ4d}, 
\eqref{YZ4h}, and Young's inequality.
\end{proof}

In order to handle \eqref{YY3} and \eqref{YY4} for $\be \le 1$, 
we need to introduce a further renormalization. 
Namely, we need to use $R_N^\dia$ in  \eqref{K1r} instead of $R_N$ in \eqref{K1}.
The additional term in \eqref{K1r} is divided into 
the following three terms:
\begin{align}
 - \int_{\T^3}  :\!(K_N^\frac{1}{2}*Y_N)^2\!: dx , 
\quad - 2\int_{\T^3}(K_N* Y_N) \Dr_Ndx , 
\quad \text{and}\quad 
- \int_{\T^3} (K_N*\Dr_N )\Dr_N
dx, 
\label{YZ5}
\end{align}

\noi
where $K_N$ and $K_N^\frac 12$ are defined 
 in \eqref{kappa2}
in terms of the multiplier $\kk_N(n)$.
One can easily check that 
the first term in \eqref{YZ5}
is $0$ under an expectation.
By writing the second term in~\eqref{YZ5} as 
\begin{align*}
- 2\int_{\T^3}(K_N* Y_N) \Dr_Ndx 
= -2 \sum_{n \in \Z^3} \big(\kk_N(n)  \ft Y_N(n)\big) \cj{\ft \Dr_N (n)},  
\end{align*}

\noi
we see that this term  in particular cancels the divergent 
contribution from the left-hand side of~\eqref{YY4}, 
coming from 
$(V_0 \, \ast \! :\! Y_N^2 \!:)\pe  Y_N $
(which corresponds to 
$Z_{13}$
defined in \eqref{Z13}).
In view of Remark \ref{REM:Z}
with \eqref{YS2}, \eqref{YS3}, 
\eqref{YS5}, and \eqref{YS6},
it follows from Lemma \ref{LEM:reg} 
and the paraproduct decomposition \eqref{para1} that 
the renormalized cubic term:
\begin{align}
[(V_0 \, \ast \! :\! Y_N^2 \!:)  Y_N]^\dia : = (V_0 \, \ast\! :\! Y_N^2 \!:)  Y_N  - 2 K_N* Y_N
\label{YZ6a}
\end{align}

\noi
belongs to  $\C^{\be - \frac 32 - \eps}(\T^3)$
with a uniform bound  in $N \in \N$, provided that $0 < \be \le 1$.
See also Appendix \ref{SEC:C}.
Then, by modifying \eqref{YZ2}, we have 
\begin{align}
\begin{split}
\bigg| \int_{\T^3} [(V_0 \, \ast \! :\! Y_N^2 \!:) Y_N ]^\dia \, \Dr_Ndx  \bigg|
&\le \big\| [(V_0 \, \ast \! :\! Y_N^2 \!:) Y_N ]^\dia \big\|_{\C^{\be -\frac 32-\eps}} 
\| \Dr_N \|_{B^{\frac 32 - \be +\eps}_{1,1}} \\
&\le c \big\| [(V_0 \, \ast \! :\! Y_N^2 \!:) Y_N ]^\dia
\big\|_{\C^{ \be -\frac 32-\eps}}^2 + \frac 1{100} \| \Dr_N \|_{H^1}^2, 
\end{split}
\label{YZ7}
\end{align}

\noi
provided that $\be > \frac 12$.

The third term in \eqref{YZ5}
removes the divergence 
for $\be \le 1$
in 
\begin{align*}
\int_{\T^3}  (V_0 \, \ast \, & (Y_N \Dr_N) ) Y_N \Dr_N dx \\
& = \sum_{\substack{n_1 + n_2 + n_3 + n_4 = 0\\n_1 + n_2 \ne 0}}
\ft V(n_1 + n_2) \ft Y_N(n_1) \ft \Dr_N(n_2)  \ft Y_N(n_3) \ft \Dr_N(n_4),  
\end{align*}

\noi
coming from the case $n_1 + n_3 = 0$.
We set 
\begin{align}
\int_{\T^3}  [(V_0 \ast   (Y_N \Dr_N )) Y_N \Dr_N]^\dia dx 
: = \int_{\T^3}  (V_0 \ast   (Y_N \Dr_N )) Y_N \Dr_N dx 
- \int_{\T^3} (K_N*\Dr_N )\Dr_N
dx.
\label{YZ8a}
\end{align}

\noi
Define a function $\Y_N$ on $\T^3\times \T^3$ by its Fourier coefficient:
\begin{align}
\ft \Y_N( n_2, n_4)
:= \sum_{\substack{n_1 \in \Z^3\\n_1 \ne -n_2}}
 \jb{n_1 + n_2}^{-\be} 
 \Big(\ft Y_N(n_1) \ft Y_N(-n_1 -n_2 -n_4) - \ind_{n_2 + n_4 = 0} \cdot \jb{n_1}^{-2}\Big).
\label{YZ9}
\end{align}

\noi
Then, with $\wt \Dr_N(x) = \Dr_N(-x)$, 
it follows from Parseval's identity, \eqref{interp}, and Young's inequality that 
\begin{align*}
|\eqref{YZ8a}|
& = \bigg|\int_{\T^3\times \T^3} \Y_N(x, y) \wt \Dr_N (x) \wt \Dr_N (y) dx dy\bigg|\\
& = \bigg|\int_{\T^3\times \T^3} 
\big(\jb{\nb_x}^{-1+\eps}\jb{\nb_y}^{-1+\eps}
\Y_N(x, y)\big) \\
& \hphantom{XXX}
\times \big(\jb{\nb_x}^{1-\eps}
 \wt \Dr_N (x) \big)\big(\jb{\nb_y}^{1-\eps}\wt \Dr_N (y) \big)dx dy\bigg|\\
& \leq  C \|\Y_N\|_{H^{-1+\eps}(\T^3\times \T^3)}^\frac{2}{\eps}
+ \frac{1}{100}\Big(\| \Dr_N\|_{H^1(\T^3)}^2 + \| \Dr_N\|_{L^2(\T^3)}^4 \Big).
\end{align*}

\noi
Note that $\Y_N \in \H_2$.
Then, in view of  the Wiener chaos estimate (Lemma \ref{LEM:hyp}), 
it suffices to bound the second moment of $\|\Y_N\|_{H^{-1+\eps}(\T^3\times \T^3)}$.
By symmetry, we assume $|n_2| \ges |n_4|$.
Then, 
from \eqref{YZ9}, Young's inequality,  and Lemma \ref{LEM:SUM}, we have 
\begin{align}
\begin{split}
\E_{\Q_\dr} \big[\|& \Y_N\|_{H^{-1+\eps}(\T^3\times \T^3)}^2\big]\\
& \les \sum_{\substack{n_1, n_2, n_4 \in \Z^3\\|n_2| \ges |n_4|}}
\frac{1}{\jb{n_2}^{2-2\eps}\jb{n_4}^{2-2\eps}}\frac{1}{\jb{n_1+n_2}^{2\be}}
\frac{1}{\jb{n_1}^2\jb{n_1+n_2+n_4}^2}\\
& \les 
\sum_{\substack{n_1, n_2, n_4 \in \Z^3\\|n_2| \ges |n_4|}}
\frac{1}{\jb{n_2}^{2-2\eps}\jb{n_4}^{2-2\eps}}
\frac{1}{\jb{n_1}^2}
\bigg(\frac{1}{\jb{n_1+n_2+n_4}^{2+2\be }}
+ \frac{1}{\jb{n_1+n_2}^{2+2\be}}\bigg)\\
& \les 1, 
\end{split}
\label{YZ10}
\end{align}

\noi
uniformly in $N \in \N$, provided that $\be > \frac 12$.

Putting everything together, we conclude that, 
with an additional renormalization \eqref{K1r},  
an analogue of Lemma \ref{LEM:Dr2} holds
for $\be > \frac 12$
and thus, in view of Lemma \ref{LEM:var2}, 
we conclude the uniform exponential integrability 
\eqref{exp1a} for $R_N^\dia(u)$.
Finally, together with 
Remark \ref{REM:Bb}, 
this proves \eqref{exp2a}, 
allowing us to construct the limiting Gibbs measure $\rhoo$ in 
\eqref{Gibbs2a}
for $\be > \frac 12$.

\subsection{Tightness for $0 < \be \le \frac 12$}
\label{SUBSEC:def3}

In the remaining part of this section, 
we consider the case $0 < \be \le \frac 12$.
In this subsection, 
we   extend the uniform exponential integrability 
and prove tightness of the truncated Gibbs measures
 $\{\rho_N\}_{N\in \N}$
for $0 < \be \le \frac 12$.
In this case, 
the estimate~\eqref{YZ7}
fails since $\big[(V_0 \, \ast \! :\! Y_N^2 \!:) Y_N \big]^\dia$ 
defined in~\eqref{YZ6a} is too irregular.
This forces us to introduce a 
 further renormalization (see \eqref{K5}), in an analogous manner to the case of 
 the $\Phi^4_3$-measure 
 studied in \cite{BG}.
The resulting measure 
 will not be absolutely continuous with respect to the base Gaussian free field $\mu$;
 see Subsection \ref{SUBSEC:def5}.
  We point out that this extra renormalization
 appears only at the level of the measure. 
  In the following, we use the following short-hand notations: $Y_N(t) = \pi_N Y(t)$ and $\Dr_N(t) = \pi_N \Dr(t)$.
 Recall also that 
 $Y_N = \pi_N Y(1)$ and $\Dr_N = \pi_N \Dr(1)$.

The Ito product formula yields
\begin{align}
\E\bigg[\int_{\T^3} [(V_0 \, \ast \!:\!Y_N^2\!:) Y_N]^\dia \Dr_N dx \bigg] 
 = \E\bigg[ \int_0^1 \int_{\T^3}  
 [(V_0 \, \ast \! :\!Y_N(t)^2\!:) Y_N(t)]^\dia\dot \Dr_N(t) dt \bigg], 
 \label{YZ11}
\end{align}
 
 \noi
 where we have 
 $\dot \Theta_N (t) = \jb{\nabla}^{-1} \pi_N \theta(t)$ by the definition \eqref{P3a}.
Define $\ZZ^N$ with $\ZZ^N(0) = 0$ by its time derivative:
\begin{align}
\dot \ZZ^N (t) = (1-\Delta)^{-1} [(V_0 \, \ast \!:\!Y_N(t)^2\!:) Y_N(t)]^\dia
\label{YZ12}
\end{align}

\noi
and set 
$\ZZ_N = \pi_N \ZZ^N$.
Then, we
perform a 
  change of variables: 
\begin{align}
\dot \Ups^N(t) : = \dot \Dr(t)  +  \dot \ZZ_N(t)
\label{YZ13}
\end{align}

\noi
with $\Ups^N(0) = 0$
and set $\Ups_N = \pi_N \Ups^N$.
Then, 
from \eqref{YZ11}, \eqref{YZ12}, and \eqref{YZ13}, 
we have 
\begin{align}
\begin{split}
\E\bigg[\int_{\T^3} & [(V_0 \, \ast \!:\!Y_N^2\!: )Y_N]^\dia  \Dr_N dx 
+ \frac12 \int_0^1 \|\dr(t)\|_{L^2_x}^2 dt \bigg] \\
& = \frac 12 \E\bigg[  \int_0^1 \|\dot \Ups^N (t)\|_{H^1_x}^2 dt
\bigg] - C_N,
\end{split}
\label{YZ13a}
\end{align}

\noi
where the divergent constant $C_N$ is given by 
\begin{align}
C_N : = \frac 12  \E\bigg[\int_0^1   \| \dot \ZZ_N(t) \|_{H^1_x}^2 dt \bigg]
\too \infty, 
\label{YZ14}
\end{align}

\noi
 as $N \to \infty$ for $\beta \le \frac 12$.
The divergence in \eqref{YZ14}
can be easily seen from the spatial regularity 
$\be + \frac 12 - \eps$ of $\dot \ZZ_N(t)$
(with a uniform bound in  $N \in \N$) for  $0 < \be \le \frac 12$.

This motivates us to introduce a further renormalization:
\begin{align}
R_N^{\dia \dia}(u)  = R_N^{\dia }(u)  + C_N, 
\label{K5}
\end{align}

\noi
where $R_N^{\dia}(u) $ and $C_N$ are  as in 
\eqref{K1r} and \eqref{YZ14}, respectively.
With a slight abuse of notation, we define the truncated Gibbs measure $\rho_N$ in this case by 
\begin{align}
d\rho_N(u) = Z_N^{-1} e^{-R_N^{\dia\dia}(u)} d \mu(u),
\label{K5a}
\end{align}

\noi
where 
the partition function $Z_N$ is given by 
\begin{align}
Z_N = \int e^{-R_N^{\dia\dia}(u)} d \mu.
\label{K6}
\end{align}

\noi
Then, 
by the  Bou\'e-Dupuis variational formula
 (Lemma \ref{LEM:var3}), we have
\begin{equation}
- \log Z_N = \inf_{\dr \in \Ha} \E
\bigg[ R_N^{\dia\dia} (Y(1) + I(\dr)(1)) + \frac{1}{2} \int_0^1 \| \dr(t) \|_{L^2_x}^2 dt \bigg]
\label{K7}
\end{equation}

\noi
for any $N \in \N$.
By setting
\begin{equation}
\W_N^{\dia\dia}(\dr) = \E
\bigg[ R_N^{\dia\dia}(Y(1) + I(\dr)(1)) + \frac{1}{2} \int_0^1 \| \dr(t) \|_{L^2_x}^2 dt \bigg], 
\label{K8}
\end{equation}

\noi
it follows from 
\eqref{v_N0} and \eqref{v_N0a}
with 
\eqref{K1r}, 
\eqref{YZ6a},   \eqref{YZ8a}, 
\eqref{YZ13a}, and 
\eqref{K5}
that 
\begin{align}
\begin{split}
\W_N^{\dia\dia}(\dr)
&=\E
\bigg[ \frac 12 \int_{\T^3} (V_0 \ast :\! Y_N^2 \!: ) \Dr_N^2 dx
+ \int_{\T^3} [(V_0 \ast (Y_N \Dr_N )) Y_N \Dr_N]^\dia dx
\\
&\hphantom{XXXXX}
+ \int_{\T^3} (V_0 \ast \Dr_N^2) Y_N \Dr_N dx
+ \frac 14 \int_{\T^3} (V_0 \ast \Dr_N^2 ) \Dr_N^2 dx
\\
&\hphantom{XXXXX}
+ \frac 14 \bigg\{ \int_{\T^3} \Big( :\! Y_N^2 \!: + 2 Y_N \Dr_N + \Dr_N^2 \Big) dx \bigg\}^2 \\
&\hphantom{XXXXX}
+ \frac{1}{2} \int_0^1 \| \dot \Ups^N(t) \|_{H^1_x}^2 dt
 \bigg].
\end{split}
\label{K9}
\end{align}

\noi
We also set 
\begin{align}
\Ups_N = \Ups_N(1)
= \pi_N  \Ups^N(1) \qquad \text{and} \qquad \ZZ_N = \ZZ_N(1) = \pi_N \ZZ^N(1).
\label{K9a}
\end{align}

\noi
In view of the change of variables \eqref{YZ13}, 
we view $\dot \Ups^N$ as a drift
and 
study each term in \eqref{K9}
by writing $\Dr_N$ as 
\begin{align}
\Dr_N = \Ups_N -    \ZZ_N.
\label{K9b}
\end{align}

\noi
The positive terms for the current problem are given by 
\begin{equation}
\U_N^{\dia\dia}(\dr) = \E\bigg[ \frac 18 \int_{\T^3} (V_0 \ast \Ups_N^2) \Ups_N^2 dx 
+ \frac 1{32} \bigg( \int_{\T^3} \Ups_N^2 dx \bigg)^2 + \frac{1}{2} \int_0^1 \| \dot \Ups^N(t) \|_{H^1_x}^2 dt\bigg].
\label{K10}
\end{equation}

\noi
As for the first term on the right-hand side of \eqref{K10}, see Lemma \ref{LEM:Dr7} below.

In the following, we state several lemmas, 
controlling the terms appearing in \eqref{K9}.

\begin{lemma}\label{LEM:Dr7}
Let $V$ be the Bessel potential  of order 
 $0 < \be \le \frac 12$
 and $V_0$ be as in \eqref{V0}.
Then, there exist   small $\eps>0$ and   a constant  $c  >0$ 
 such that
\begin{align}
\begin{split}
 \int_{\T^3}  & (V_0 \ast \Dr_N^2 ) \Dr_N^2 dx
 \geq   \frac {1}{2} \int_{\T^3} (V_0 \ast \Ups_N^2 ) \Ups_N^2 dx
 - \frac{1}{1000}\| \Ups_N\|_{L^2}^4  -c  \|\ZZ_N\|_{\C^{\be + \frac 12 - \eps}}^4
\end{split}
\label{KZ1}
\end{align}

\noi
and 
\begin{align}
\begin{split}
\bigg\{ \int_{\T^3}  \Big( :\! Y_N^2 \!:  & + 2 Y_N \Dr_N + \Dr_N^2 \Big) dx \bigg\}^2
 \ge \frac 18 \| \Ups_N \|_{L^2}^4 - \frac 1{100} \| \Ups_N \|_{H^1}^2 \\
& \quad - c \bigg\{1+  \| Y_N \|_{\C^{-\frac 12-\eps}}^{c} + \bigg( \int_{\T^3} :\! Y_N^2 \!:  dx \bigg)^2
+ \|\ZZ_N\|_{\C^{\be + \frac 12 - \eps}}^c \bigg\}
\end{split}
\label{KZ2}
\end{align}

\noi
for 
$\Dr_N = \Ups_N - \ZZ_N$ as in \eqref{K9b}, 
uniformly in $N \in \N$.

\end{lemma}

\begin{proof}
The first estimate \eqref{KZ1}
can be easily seen from 
\[ \| (\Ups_N + \ZZ_N) \ZZ_N \|_{H^{-\frac{\be}{2}}}
\les \| \Ups_N \|_{L^2}^2 + \|\ZZ_N\|_{\C^{\be + \frac 12 - \eps}}^2.
\]

The second estimate \eqref{KZ2} follows
from a slight modification of the proof of 
Lemma \ref{LEM:Dr3}.
Indeed, it follows from \eqref{YZ3}
along with the following two estimates:
\begin{align*}
\frac 12 \bigg(\int_{\T^3} \Dr_N^2 dx \bigg)^2
& = \frac 12 \bigg(\int_{\T^3} \Ups_N^2 dx 
- 2 \int_{\T^3} \Ups_N \ZZ_N  dx 
+ \int_{\T^3}  \ZZ_N^2  dx\bigg)^2 \\
& \geq \frac{1}{5} 
\|\Ups_N\|_{L^2}^4
-  C \|\ZZ_N\|_{L^2}^4
\end{align*}

\noi
and
\begin{align*}
\begin{split}
\bigg| \int_{\T^3}  & Y_N \Dr_N dx \bigg|^2
\le \| Y_N \|_{\C^{-\frac 12-\eps}}^2 
\Big(\| \Ups_N \|_{H^{{\frac 12+2\eps}}}^2+ \| \ZZ_N \|_{\C^{{\frac 12+2\eps}}}^2\Big)\\
&\le C \| Y_N \|_{\C^{-\frac 12-\eps}}^{c}
 + \frac 1{100C_0} \| \Ups_N \|_{L^2}^4 + \frac1{100 C_0} \| \Ups_N \|_{H^1}^2
 + \| \ZZ_N \|_{\C^{{\be + \frac 12-\eps}}}^c.
\end{split}
\end{align*}

\noi
This proves Lemma \ref{LEM:Dr7}.
\end{proof}

\medskip

\begin{lemma}\label{LEM:Dr8}
Let $V$ be the Bessel potential of order 
 $0 < \be \le \frac 12$
 and $V_0$ be as in \eqref{V0}
Then, there exist   small $\eps>0$ and   a constant  $c  >0$ 
 such that
\begin{align}
\begin{split}
\bigg| \int_{\T^3}  (V_0 \ast \Dr_N^2) Y_N \Dr_N  dx\bigg|
&\le c \Big(1 + 
\| Y_N \|_{\C^{-\frac 12-\eps}}^{c}
+ \|\ZZ_N\|_{\C^{\be + \frac 12 -\eps}}^c\Big)\\
& \quad + \frac 1{1000} \Big(
\| \Ups_N^2\|_{\dot H^{-\frac{\be}{2}}}^2 + 
 \| \Ups_N \|_{L^2}^4 +  \| \Ups_N \|_{H^1}^2\Big), 
\end{split}
\label{YY1y}\\
\begin{split}
\bigg| \int_{\T^3} (V_0 \ast :\! Y_N^2 \!:) \Dr_N^2 dx \bigg|
&\le c \| :\! Y_N^2 \!: \|_{\C^{-1-\eps}}^{c} 
+ \| (V_0*:\!Y_N^2\!:) \ZZ_N^2\|_{\C^{\be - 1- \eps}}\\
& \quad
+ \| (V_0*:\!Y_N^2\!:) \ZZ_N\|_{\C^{\be - 1- \eps}}^c
 + \frac 1{1000} \Big(
\| \Ups_N \|_{L^2}^4 + \| \Ups_N \|_{H^1}^2\Big)
\end{split}
\label{YY2y} 
\end{align}

\noi
for 
$\Dr_N = \Ups_N -    \ZZ_N $ as in \eqref{K9b},
uniformly in $N \in \N$.
Furthermore, 
the stochastic terms $ (V_0*:\!Y_N^2\!:) \ZZ_N^2$
and $ (V_0*:\!Y_N^2\!:) \ZZ_N$
have uniformly bounded \textup{(}in $N$\textup{)} moments \textup{(}under the $\C^{\be - 1- \eps}$-norm\textup{)}.

\end{lemma}

\begin{proof}
In the following, we focus on proving the estimates \eqref{YY1y} and \eqref{YY2y}.
See Appendix \ref{SEC:C}
for analysis on 
 the stochastic terms
$ (V_0*:\!Y_N^2\!:) \ZZ_N^2$
and $ (V_0*:\!Y_N^2\!:) \ZZ_N$.

We prove \eqref{YY1y} and \eqref{YY2y}
by replacing each $\Dr_N$ with $\Ups_N$ or $\ZZ_N$ 
and  carrying out case-by-case analysis.
When all the occurrences of $\Dr_N$ are replaced by $\Ups_N$, 
 the estimates \eqref{YY1y} and \eqref{YY2y} follow from Lemma \ref{LEM:Dr6}.
From \eqref{CZ2} and Lemma \ref{LEM:reg}, 
we have $\ZZ_N \in   \C^{\be + \frac 12 - \eps}(\T^3)$
almost surely with a uniform bound   in $N \in \N$.

From \eqref{dual}, \eqref{prod},  and Lemma \ref{LEM:para}
(with $\be > 2\eps$), 
we have
\begin{align*}
\bigg| \int_{\T^3}  (V_0 \ast \Ups_N^2) Y_N \ZZ_N  dx\bigg|
&\le \| V_0* \Ups_N^2\|_{B^{\frac 12 +\eps}_{1, 1}}
\| Y_N \ZZ_N\|_{\C^{-\frac 12 - \eps}}\\
&\les \| \Ups_N\|_{H^{\frac 12 -\be + 2\eps}}^2
\| Y_N \|_{\C^{-\frac 12 - \eps}}
\|\ZZ_N\|_{\C^{\be + \frac 12 -\eps}}.
\end{align*}

\noi
Then, \eqref{YY1y} in this case follows from \eqref{interp} and Young's inequality.
Similarly, we have 
\begin{align*}
\bigg| \int_{\T^3}  ( & V_0 \ast (\Ups_N \ZZ_N)) Y_N (\Ups_N - \ZZ_N)  dx\bigg|\\
&\les \| \Ups_N \|_{H^{\frac{1}{2}-\be + 2\eps}}
\|\ZZ_N\|_{\C^{\frac{1}{2} - \be + 3\eps}}
\| Y_N (\Ups_N - \ZZ_N)\|_{H^{-\frac{1}{2} - 2\eps}}\\
&\les \| \Ups_N \|_{H^{\frac{1}{2}-\be + 2\eps}}
\|\ZZ_N\|_{\C^{\frac{1}{2} - \be + 3\eps}}
\| Y_N\|_{\C^{-\frac 12 - \eps}}
 \|\Ups_N - \ZZ_N\|_{H^{\frac{1}{2} +3\eps}}.
\end{align*}

\noi
Then, \eqref{interp}  and Young's inequality yields \eqref{YY1y} in this case.
The remaining case (with $V_0 * \ZZ_N^2$) follows in an analogous manner
since $V_0 * \ZZ_N^2 \in \C^{\frac 12 + 2\be -\eps}(\T^3)$.

The second estimate \eqref{YY2y}
for $(V_0*:\!Y_N^2\!:) \ZZ_N\Ups_N$
follows 
from \eqref{dual}, \eqref{interp}, and Young's inequality.
\end{proof}

Lastly,  we estimate 
the contribution from the renormalized term defined in \eqref{YZ8a}.
Given small $\eps > 0$, 
define an integral operator $T_N$
by 
\begin{align}
 T_Nf(x) = \int_{\T^3}k_N(x, y) f(y) dy
\label{KZ4}
\end{align}

\noi
with  the kernel $k_N$  given by 
\begin{align}
k_N(x,y) =  \jb{\nabla_x}^{-1+\eps}\jb{\nabla_y}^{-1+\eps} \Y_N(x,y), 
\label{KZ5}
\end{align}

\noi
where $\Y_N$ is defined in \eqref{YZ9}.
Then, the following estimate holds.

\begin{lemma}\label{LEM:Dr9}
Let $V$ be the Bessel potential of order 
 $0 < \be \le \frac 12$.
Then, there exist   small $\eps>0$ and   a constant  $c  >0$ 
 such that
\begin{align}
\begin{split}
\bigg| \int_{\T^3}  [(V_0 \ast (Y_N \Dr_N) ) Y_N \Dr_N]^\dia dx \bigg|
&\le c \Big(1 + 
\| T_N \|_{\L(L^2; L^2)}^c
+ \| [(V_0 \ast (Y_N \ZZ_N) ) Y_N \ZZ_N]^\dia
\|_{\C^{\be - 1-\eps}}\\
& \quad + 
 \| [(V_0 \ast (Y_N \ZZ_N) ) Y_N ]^\dia
\|_{\C^{\be - 1-\eps}}^c
\Big)\\
& \quad + \frac 1{100} \Big(
 \| \Ups_N \|_{L^2}^4 +  \| \Ups_N \|_{H^1}^2\Big),
\end{split}
\label{YY3y}
\end{align}

\noi
for 
$\Dr_N = \Ups_N -    \ZZ_N $ as in \eqref{K9b}, 
uniformly in $N \in \N$.
Here, $[(V_0 \ast (Y_N \ZZ_N) ) Y_N ]^\dia$ is defined by 
\begin{align}
[(V_0 \ast (Y_N \ZZ_N) ) Y_N ]^\dia
: = (V_0 \ast (Y_N \ZZ_N) ) Y_N - K_N*\ZZ_N,
\label{YY3yy}
\end{align}

\noi
where $K_N$ is as in \eqref{kappa2}.
Furthermore, 
the expectation of the first term, containing 
the stochastic terms $T_N$,  $ [(V_0 \ast (Y_N \ZZ_N) ) Y_N \ZZ_N]^\dia$, 
and $ [(V_0 \ast (Y_N \ZZ_N) ) Y_N ]^\dia$, 
is uniformly bounded in $N \in \N$.

\end{lemma}

\begin{proof}

As in the proof of Lemma \ref{LEM:Dr8}, 
we prove \eqref{YY3y}
by performing case-by-case analysis.
The contribution from the case when both $\Dr_N$'s are  replaced by $\ZZ_N$
is clearly bounded by 
$\| [(V_0 \ast (Y_N \ZZ_N) ) Y_N \ZZ_N]^\dia
\|_{\C^{\be - 1-\eps}}$.
From \eqref{YZ8a} and \eqref{YY3yy} with \eqref{K9b}, 
we have
\begin{align}
\begin{split}
\bigg|\int_{\T^3} [(V_0 \ast (Y_N \ZZ_N) ) Y_N \Ups_N]^\dia dx \bigg|
 = \bigg|\int_{\T^3}[(V_0 \ast (Y_N \ZZ_N) ) Y_N ]^\dia \Ups_N dx \bigg|\\
 \leq 
  \| [(V_0 \ast (Y_N \ZZ_N) ) Y_N ]^\dia\|_{\C^{\be - 1-\eps}}
\|\Ups_N\|_{H^1}.
\end{split}
\label{KZ5a}
\end{align}
 
 \noi
 Then, Cauchy's inequality yields \eqref{YY3y} in this case.
By  the symmetry, 
the contribution from 
  $[(V_0 \ast (Y_N \Ups_N) ) Y_N \ZZ_N]^\dia$
is also bounded by \eqref{KZ5a}.

It remains to consider the case $\Dr_N = \Ups_N$ 
for both entries.
Suppose that, for $\be > 0$,  $T_N$ defined in \eqref{KZ4} is bounded on $L^2(\T^3)$.
Then, with $\wt \Ups_N(x) = \Ups_N(-x)$, 
it follows from Parseval's identity, 
the (assumed) boundedness of $T_N$, \eqref{interp}, and Young's inequality that 
\begin{align*}
\bigg| \int_{\T^3}  [(V_0 \ast (Y_N \Ups_N) ) Y_N \Ups_N]^\dia dx \bigg|
& = \bigg|\int_{\T^3\times \T^3} \Y_N(x, y) \wt \Ups_N (x) \wt \Ups_N (y) dx dy\bigg|\\
& = \bigg|\int_{\T^3} T_N\big(\jb{\nb}^{1-\eps} \wt \Ups_N\big)(x) 
\cdot \jb{\nb}^{1-\eps}\wt \Ups_N (x) dx \bigg|\\
& \leq \|T_N\|_{\L(L^2; L^2)} 
\| \Ups_N\|_{H^{1-\eps}}^2\\
& 
\leq  C\|T_N\|_{\L(L^2; L^2)}^\frac{2}{\eps} 
+ \frac{1}{100}\Big(\| \Ups_N\|_{H^1(\T^3)}^2 + \| \Ups_N\|_{L^2(\T^3)}^4 \Big).
\end{align*}

\noi
This proves \eqref{YY3y} in this case.

We now need to show that 
the expectation of the first term on the right-hand side of~\eqref{YY3yy}, containing 
the stochastic terms $T_N$,  $ [(V_0 \ast (Y_N \ZZ_N) ) Y_N \ZZ_N]^\dia$, 
and $ [(V_0 \ast (Y_N \ZZ_N) ) Y_N ]^\dia$, 
is uniformly bounded in $N \in \N$.
As for
the stochastic terms  $ [(V_0 \ast (Y_N \ZZ_N) ) Y_N \ZZ_N]^\dia$
and $ [(V_0 \ast (Y_N \ZZ_N) ) Y_N ]^\dia$, 
see~Appendix \ref{SEC:C}.
In the remaining part of this proof, 
we focus on proving 
 the boundedness of $T_N$ on $L^2(\T^3)$
(under a high moment). 
In the following, all the estimates are uniform in $N \in \N$.

Suppose that there exists some $0< \alpha < 3$ such that 
\begin{equation}
 |x-y|^{2\alpha} \E\big[k_N^2(x, y)\big] \les 1
\label{KZ6}
\end{equation}

\noi
for any $x, y \in \T^3 \cong[-\pi, \pi)^3$, 
uniformly in $N\in \N$.
Then, by the Wiener chaos estimate (Lemma~\ref{LEM:hyp}), 
we have 
\[\E \Big[\big\| |x-y|^{\alpha}k_N(x,y)\big\|_{L^q_{x,y}}^\frac{2}{\eps} \Big]< \infty\]

\noi
for any finite $q\ge$1 and $\eps > 0$. 
Thus, for 
 $1 < p < q< \frac{3}{\al}$ and  $ \frac 1 r  = \frac 1 p - \frac 1 q$, 
 we have 
\begin{align*}
\E \Big[\| k_N \|_{L^{p'}_xL^p_y}^\frac{2}{\eps} \Big]
 = \E \Big[\| k_N \|_{L^{p'}_yL^p_x}^\frac{2}{\eps}\Big]
& = \E \Big[ \big\||x-y|^{-\alpha}|x-y|^{\alpha}k_N(x,y)\big\|_{L^{p'}_yL^p_x}^\frac{2}{\eps}\Big]\\
&\le \big\||x|^{-\alpha}\big\|_{L^q}^\frac{2}{\eps} 
\E \Big[\| |x-y|^{\alpha}k_N(x,y)\|_{L^{p'}_y L^r_x}^\frac{2}{\eps}\Big]\\
&\les \E \Big[\big\| |x-y|^{\alpha}k_N(x,y)\big\|_{L^{\max(p', r)}_{x,y}}^\frac{2}{\eps}\Big]\\
&<  \infty.
\end{align*}

\noi
Therefore, by Schur's test, we conclude that 
\[ \E \Big[\| T_N \|_{\L(L^2 ; L^2)}^\frac{2}{\eps} \Big]
\le
C_\eps
 \E\Big[ \|k_N\|_{L^{p'}_x L^p_y}^\frac{2}{\eps}  + \|k_N\|_{L^{p'}_y L^p_x}^\frac{2}{\eps} \Big] < \infty.\]

In the following, we prove the bound \eqref{KZ6}.
From the definition of the gamma function and a change of variables, we have 
\begin{align}
\jb{\nabla_x}^{-1+\eps} \sim \int_0^{\infty} t^{-\frac {1+\eps}2} e^{-t(1-\Delta_x)} d t.
\label{KZ7}
\end{align}

\noi
Then, from \eqref{KZ5} and \eqref{KZ7}, we have
$$ k_N(x,y) =  c\int_0^{\infty}\int_0^{\infty} t^{-\frac {1+\eps}2} s^{-\frac {1+\eps}2} 
e^{-t} e^{-s} \big((p_t\otimes p_s) \ast \Y_N\big)(x,y) dt ds, $$

\noi
where $p_t$ is the kernel of the heat semigroup  $e^{t\Dl }$. 
Therefore, in order to show \eqref{KZ6}, 
it suffices to show that 
\begin{equation} 
 \E \Big[ \big( (p_t\otimes p_s) \ast \Y_N\big)^2(x,y)\Big]
 \les |x-y|^{-2\alpha} ( s^{-1+2\eps}\vee 1)
(  t^{-1+2\eps}\vee 1)
\label{KZ8}
\end{equation}

\noi
for any $x, y \in \T^3 \cong[-\pi, \pi)^3$
and $t, s > 0$,  
uniformly in $N\in \N$, where $a\vee b = \max(a, b)$.
Without loss of generality, we assume  $t \ge s > 0$.

\smallskip

\noi
$\bullet$
{\bf Case 1:}
 $0 < s \le t \le 1$.
\quad 
From \eqref{YZ9}
and \eqref{Wickz}, we have 
\begin{align}
\E\big[ & \ft \Y(n,m)\ft \Y(n',m')\big] \notag \\
& = \E\bigg[\sum_{\substack{n_1 \in \Z^3\\n_1 \ne -n}}
 \jb{n + n_1}^{-\be} 
 \Big(\ft Y_N(n_1) \ft Y_N(-n_1 -n -m) - \ind_{n +m  = 0} \cdot \jb{n_1}^{-2}\Big) \notag \\
& \quad \times  \sum_{\substack{n_1' \in \Z^3\\n_1' \ne -n'}}
 \jb{n' + n_1'}^{-\be} 
 \Big(\ft Y_N(n_1') \ft Y_N(-n_1' -n' -m') - \ind_{n' +m'  = 0} \cdot \jb{n_1'}^{-2}\Big)\bigg] \notag \\
& = 
\ind_{n+m + n' + m' = 0} 
\sum_{\substack{
 n_1 \in \Z^3\\ |n_1|, \,  |n_1 + n + m| \le N}} 
\frac{\ft V(n + n_1) \ft V(n' - n_1)}{\jb{n_1}^2 \jb{n+m+n_1}^2}  \notag \\
& \quad + \ind_{n+m + n' + m' = 0}
\sum_{\substack{
 n_1 \in \Z^3\\ |n_1|, \,  |n_1 + n + m| \le N}} 
\frac{\ft V(n + n_1) \ft V(m' - n_1)}{\jb{n_1}^2 \jb{n+m+n_1}^2}  
 + \text{l.o.t.}\notag \\
&  =:  \1 + \II  
 + \text{l.o.t.}.
 \label{KZ9}
\end{align}

\noi
Here, ``l.o.t.'' denotes the lower order terms coming from 
$n_1 = -n$ or  $n_1' = - n'$.
Hence, by 
ignoring the lower order terms in \eqref{KZ9}
(which  can be estimated easily), we have
\begin{align}
\E \Big[ & \big((p_t\otimes p_s) \ast \Y_N\big)^2(x,y) \Big] \notag \\
& =\sum_{n,m,n',m' \in \Z^3} 
e^{-t (|n|^2 + |n'|^2)} e^{-s (|m|^2+ |m'|^2)} 
\E\big[\ft \Y(n,m)\ft \Y(n',m')\big] e_{n+n'}(x)e_{m+m'}( y) \notag \\
& = \sum_{n,m,n',m' \in \Z^3}e^{-t (|n|^2 + |n'|^2)} e^{-s (|m|^2+ |m'|^2)} 
\big( \1 + \II\big) \, 
e_{n+n'} (x-y) \notag \\
& \les \sum_{|n_1|,|n_2| \leq N} \frac 1 {\jb{n_1}^2 \jb{n_2}^2} 
\bigg|\sum_{k\in \Z^3} \ft V (k) \exp(-t|k  -  n_1|^2 - s|k - n_2|^2) e_{k} (x-y) \bigg|^2 \notag \\
&= \sum_{|n_1|,|n_2|\le N} \frac {\exp\big(-2\frac{ts}{t+s} |n_1-n_2|^2\big)} {\jb{n_1}^2 \jb{n_2}^2}\notag \\
& \hphantom{XXXXX}\times  
\bigg|\sum_{k\in \Z^3} \ft  V (k) 
\exp\Big( - (t+s) \Big| k - \frac{tn_1+sn_2}{t+s} \Big|^2\Big) e_k (x-y) \bigg|^2. 
\label{KZ10}
\end{align}

Fix $\dl > 0$ small.  
We first consider the case
 $s \gtrsim t^{\frac1 \dl}$.
Recall that 
 $e^{-t|k - \xi_0|^2}$ is the Fourier transform of the periodization of 
  $e^{-i x \cdot \xi_0} p_t^{\R^3}(x)$, 
  where $p_t^{\R^3}$ is the heat semigroup on $\R^3$.
Then, from the Poisson summation formula, we have
\begin{align}
\begin{split}
\bigg|\sum_{k\in \Z^3} \ft V (k) 
& \exp\Big( - (t+s) \Big| k - \frac{tn_1+sn_2}{t+s} \Big|^2\Big) e_k (x-y) \bigg| \\
& \les \sum_{k \in \Z^3} \big(|V^{\R^3}|*p_{t+s}^{\R^3}\big) (x-y + 2\pi k)\\
& \les |x- y|^{\be - 3}
\end{split}
\label{KZ11}
\end{align}
 
 \noi
 for any $x, y \in \T^3 \cong[-\pi, \pi)^3$, 
 where $V^{\R^3}$ is the Bessel potential of order $\be$ on $\R^3$.
 In the last step, we used the well-known asymptotics
 of the Bessel potential on $\R^3$:
 $V^{\R^3}(x) \sim |x|^{\be - 3}$ as $x\to 0$
 and 
 $V^{\R^3}(x) \sim |x|^{\frac{\be - 4}{2}}e^{-|x|}$ as $|x|\to \infty$;
 see (4,2) and (4,3) in \cite{AS}.
 
 We also have 
\begin{align}
 \exp\Big(-2\frac{ts}{t+s} |n_1-n_2|^2\Big) \les s^{-1-\eps}  \jb{n_1-n_2}^{-2-2\eps}.
 \label{KZ12}
\end{align}
	
\noi
Hence, from \eqref{KZ10},  \eqref{KZ11},  and \eqref{KZ12}
with  $s \gtrsim t^{\frac1 \dl}$, 
we obtain \eqref{KZ8}, 
provided that 
 $\alpha > 3 -\beta$ and $\frac{3\eps}{\dl} < 1 - 2\eps$.
The last condition is guaranteed by choosing sufficiently small $\eps > 0$.

Next, we  consider the case
 $s \ll t^{\frac1 \dl}$.
Recall that  given $\g > 0$, we have 
$e^{-x} \le C_\g \jb{x}^{-\g}$ for any $ x>  0$.
Then,  from  Lemma \ref{LEM:SUM}, we have
\begin{align}
\bigg|\sum_{k\in \Z^3} \ft V (k) 
\exp\Big( - (t+s) | k - n_0 |^2\Big) e_k (x-y) \bigg| \les t^{-\frac 32-\eps} \jb{n_0}^{-\beta}, 
\label{KZ13}
\end{align}

\noi
where
 $n_0 = \frac{tn_1+sn_2}{t+s}$. 
We also have 
\begin{align}
\exp\Big(-2\frac{ts}{t+s} |n_1-n_2|^2\Big) \les s^{-1+\beta-\frac 12 \eps}  \jb{n_1-n_2}^{-2+2\beta-\eps}.
\label{KZ14}
\end{align}

\noi
Therefore, from \eqref{KZ13}, \eqref{KZ14}, 
and 
Lemma \ref{LEM:SUM2} with $t^{-1} \ll s^{-\dl}$, we obtain
\begin{align*}
\text{RHS of }\eqref{KZ10}
& \les \sum_{n_1,n_2} \frac{t^{- 3-2\eps} s^{-1+\beta-\frac 12 \eps}} {\jb{n_1}^2 \jb{n_2}^2\jb{n_0}^{2\beta} \jb{n_1-n_2}^{2-2\beta +\eps}} \\
& \les t^{-3- 2\eps} s^{-1+\beta-\frac \eps 2} \les s^{-1+2\eps} t^{-1+2\eps}, 
\end{align*}

\noi
provided that 
 $\frac 52\eps + (2 + 4\eps) \delta \le \beta \le \frac 12 $.
This proves \eqref{KZ8} in this case.

\medskip

\noi
$\bullet$
{\bf Case 2:}
 $t\ge s \ge 1$.
\quad 
In this case, the bound \eqref{KZ8} follows from 
\eqref{KZ10}, \eqref{KZ11}, and \eqref{KZ12}
with $s^{-1-\eps} \le 1$.

\medskip

\noi
$\bullet$
{\bf Case 3:}
 $t\ge 1 \ge s > 0$.
\quad 
In this case, 
 from \eqref{KZ13}, \eqref{KZ14}, 
and 
Lemma \ref{LEM:SUM2}, we have 
\begin{align*}
\text{RHS of }\eqref{KZ10}
& \les \sum_{n_1,n_2} \frac{ s^{-1+\beta-\frac 12 \eps}} {\jb{n_1}^2 \jb{n_2}^2\jb{n_0}^{2\beta} \jb{n_1-n_2}^{2-2\beta +\eps}} \\
& \les  s^{-1+\beta-\frac \eps 2} \les s^{-1+2\eps}, 
\end{align*}

\noi
provided that 
 $\frac 52\eps \le \beta \le \frac 12 $.
This completes the proof of Lemma \ref{LEM:Dr9}.
\end{proof}

Putting everything together, 
we conclude from 
\eqref{K8},  \eqref{K10}, 
and 
Lemmas \ref{LEM:Dr7}, \ref{LEM:Dr8}, 
and  \ref{LEM:Dr9}
with Lemmas \ref{LEM:Dr} and  \ref{LEM:CZ}
that 
\begin{align}
\inf_{N \in \mathbb{N}} \inf_{\dr \in \Ha} \W_N^{\dia\dia}(\dr) 
\geq 
\inf_{N \in \mathbb{N}} \inf_{\dr \in \Ha}
\Big\{ -C_0 + \frac{1}{10}\U^{\dia\dia}_N(\dr)\Big\}
 \geq - C_0 >-\infty.
\label{KZ14a}
\end{align}

\noi
Then, the uniform exponential integrability
\eqref{exp1c}
for 
$R_N^{\dia\dia}(u)$ defined in 
\eqref{K5}
follows from 
the  Bou\'e-Dupuis variational formula~\eqref{K7}.

\begin{remark}\rm
As mentioned in Section \ref{SEC:1}, 
 the uniform exponential integrability
 \eqref{exp1c} holds
 only for the first moment but not for 
 higher moments.
 This is because, in 
the renormalization~\eqref{K5}, 
 the  constant $C_N$ was introduced
to cancel a divergent interaction 
in computing the first moment,
(which is not suitable for higher moments). 

\end{remark}

Finally, we prove tightness of the truncated Gibbs measures $\{\rho_N\}_{N\in \N}$.
Fix small $\eps > 0$ and 
let 
$ B_R \subset H^{-\frac 12 - \eps}(\T^3)$ 
be the closed ball of radius $R> 0$ centered at the origin.
Then, by Rellich's compactness lemma, 
we see that  $B_R$ is compact in $H^{-\frac 12 - 2\eps}(\T^3)$.
In the following, we establish a uniform bound on $\rho_N(B_R^c)$, $N \in \N$, 
by assuming that 
a unique limit $Z = \lim_{N\to \infty} Z_N \in (0, \infty)$ exists.\footnote{More precisely, 
we need a uniform (in $N$) lower bound on $Z_N$. See \eqref{T1}.}
We will prove this latter fact in the next subsection. See \eqref{KZ15}.

Given $M \gg1 $, 
let $F$ be a bounded smooth non-negative function such that $F(u) = 0 $
if $\|u\|_{H^{-\frac 12 - \eps} }> R$
and 
$F(u) = M $
if $\|u\|_{H^{-\frac 12 - \eps}} \leq \frac R2$.
Then, we have 
\begin{align}
\rho_N(B_R^c) \leq Z_N^{-1} \int e^{-F(u)-R_N^{\dia\dia}(u)} d\mu
\les \int e^{-F(u)-R_N^{\dia\dia}(u)} d\mu
=: \ft Z_N, 
\label{T1}
\end{align}

\noi
uniformly in  $N \gg 1$.
Under  the change of variables \eqref{YZ13}, 
define $\wt R^{\dia\dia}_{N}(Y+ \Ups^{N}- \ZZ_{N})$  by 
\begin{align}
\begin{split}
\wt R^{\dia\dia}_{N}(Y+ \Ups^{N}- \ZZ_{N})
& =  \frac 12 \int_{\T^3} (V_0 \ast :\! Y_N^2 \!: ) \Dr_N^2 dx
+ \int_{\T^3} [(V_0 \ast (Y_N \Dr_N )) Y_N \Dr_N]^\dia dx\\
&\hphantom{X}
+ \int_{\T^3} (V_0 \ast \Dr_N^2) Y_N \Dr_N dx
+ \frac 14 \int_{\T^3} (V_0 \ast \Dr_N^2 ) \Dr_N^2 dx
\\
&\hphantom{X}
+ \frac 14 \bigg\{ \int_{\T^3} \Big( :\! Y_N^2 \!: + 2 Y_N \Dr_N + \Dr_N^2 \Big) dx \bigg\}^2, 
\end{split}
\label{KZ16}
\end{align}

\noi
where 
$Y_N = \pi_N Y$
and $\Dr_N = \pi_N \Dr = \pi_N (\Ups^N - \ZZ_N)$.
Then, by 
the Bou\'e-Dupuis formula (Lemma~\ref{LEM:var3}),\footnote{Here, we apply 
the Bou\'e-Dupuis formula (Lemma~\ref{LEM:var3})
for $F$ on rough functions but this can be justified by a limiting argument.
A similar comment applies in the following.}
we have
\begin{align}
\begin{split}
-\log \ft Z_{N}
= \inf_{\dot \Ups^{N}\in  \mathbb H_a^1}
\E \bigg[& F(Y+ \Ups^{N}- \ZZ_{N}) \\
& +  \wt R^{\dia\dia}_{N}(Y+ \Ups^{N}- \ZZ_{N}) 
+ \frac12 \int_0^1 \| \dot \Ups^{N}(t) \|_{H^1_x}^2dt \bigg],
\end{split}
\label{T2}
\end{align}

\noi
where 
$\Ha^1$ denotes 
  the space of drifts, which are the progressively measurable processes 
 belonging to
$L^2([0,1]; H^1(\T^3))$, $\PP$-almost surely, namely,
\begin{align}
\Ha^1 := \jb{\nb}^{-1} \Ha.
\label{TT1}
\end{align}

Recall that 
$Y- \ZZ_N \in \H_{\leq 3}$.
Then, by the Wiener chaos estimate (Lemma \ref{LEM:hyp})
and 
Chebyshev's inequality, we have, for some $c > 0$, 
\begin{align}
\begin{split}
\PP \Big(   \|Y+\Ups^{N}- \ZZ_{N}
\|_{H^{-\frac 12 - \eps}} >  \tfrac R2
\Big) 
&  \leq  \PP \Big( \|Y- \ZZ_N
\|_{H^{-\frac 12 - \eps}} >  \tfrac R4
\Big) 
+ \PP \Big( \|\Ups^{N}\|_{H^1} >  \tfrac R4
\Big) \\
& \leq e^{- c R^\frac{2}{3}}
+ \frac {16}{R^2} \E\Big[\|\Ups^N\|_{H^1_x}^2\Big]
\end{split}
\label{T3}
\end{align}

\noi
uniformly in $N \in \N$ and $R\gg1$.
Thus, by choosing $M = \frac{1}{64} R^2 \gg1$, 
it follows from the definition of $F$, \eqref{T3},  and Lemma \ref{LEM:Dr}
that 
\begin{align}
\begin{split}
\E \Big[ F(Y+ \Ups^{N}- \ZZ_{N})
& \cdot \ind_{ \big\{\|Y+ \Ups^{N}- \ZZ_{N}
\|_{H^{-\frac 12 - \eps}} \leq  \tfrac R2\big\}}\Big]
 \geq \frac{M}{2}
-  \frac {16M }{R^2} \E\Big[\|\Ups^N\|_{H^1_x}^2\Big]\\
&  \geq \frac{M}{2}
-  \frac {1}{4}\E\bigg[ \int_0^1 \| \dot \Ups^{N}(t) \|_{H^1_x}^2dt \bigg].
\end{split}
\label{T4}
\end{align}

\noi
Then, from \eqref{T2}, \eqref{T4}, 
and repeating the computation leading to \eqref{KZ14a}, we obtain
\begin{align}
\begin{split}
-\log \ft Z_{N}
& \geq \frac M{2}
+ 
 \inf_{\dot \Ups^{N}\in  \mathbb H_a^1}
\E \bigg[
   \wt R^{\dia\dia}_{N}(Y+ \Ups^{N}- \ZZ_{N}) 
+ \frac14 \int_0^1 \| \dot \Ups^{N}(t) \|_{H^1_x}^2dt \bigg]\\
& \geq \frac M4, 
\end{split}
\label{T5}
\end{align}

\noi
uniformly $N \in \N$ and $M = \frac{1}{64}  R^2 \gg 1$.
Therefore, given any small $\dl > 0$, by choosing $R = R(\dl ) \gg1 $
and setting $M = \frac 1{64} R^2\gg1$, 
we obtain, from 
\eqref{T1} and~\eqref{T5}, 
\begin{align*}
\sup_{N \in \N} \rho_N(B_R^c)< \dl.
\end{align*}

\noi
This proves 
tightness of the truncated Gibbs measures $\{\rho_N\}_{N\in \N}$.

\subsection{Uniqueness of the limiting Gibbs measure for $0 < \be \le \frac 12$}
\label{SUBSEC:def4}

When $\be > \frac 12$, 
the uniform exponential integrability combined with 
Lemma \ref{LEM:conv} and Remark \ref{REM:Bb}
allowed us to conclude the convergence  of the truncated Gibbs measures.
This argument, however, fails 
for $0< \be \le \frac 12$
due to the non-convergence of 
$\{R_N^{\dia\dia}\}_{N\in \N}$, which can be seen from the proof of 
Lemma \ref{LEM:conv} (see the term $Q_{N, 1}$).
Nonetheless, 
the tightness of the truncated Gibbs measures $\{\rho_N\}_{N\in \N}$, proven in the previous subsection,
together with Prokhorov's theorem implies
existence of a weakly convergent subsequence.
In this subsection, we prove uniqueness of 
the limiting Gibbs measure, which  allows us to conclude
the weak convergence of the whole sequence 
$\{\rho_N\}_{N\in \N}$.

\begin{proposition}\label{PROP:uniq}
Let $0 < \be \le \frac 12$.
Let $\{ \rho_{N^1_k}\}_{k = 1}^\infty$
and $\{ \rho_{N^2_k}\}_{k = 1}^\infty$ be 
two weakly convergent  subsequences
of the truncated Gibbs measures
$\{\rho_N\}_{N\in \N}$ defined in \eqref{K5a}, 
converging weakly to $\rho^{(1)}$ and $\rho^{(2)}$ as $k \to \infty$,  respectively.
 Then, we have 
$\rho^{(1)} = \rho^{(2)}$.

\end{proposition}

\begin{proof}

By taking a further subsequence, we may assume that 
 $N^1_k \ge N^2_k$, $k \in \N$.
We first show that 
\begin{align}
\lim_{k \to \infty} Z_{N^1_k} = \lim_{k \to \infty} Z_{N^2_k}, 
\label{KZ15}
\end{align}

\noi
where $Z_N$ is as in \eqref{K6}.
Let $Y = Y(1)$  be as in \eqref{P2}.
Recall the change of variables~\eqref{YZ13} from the previous section
and
let  $\wt R^{\dia\dia}_{N}(Y+ \Ups^{N}- \ZZ_{N})$ 
be as in \eqref{KZ16}.
Then, by 
the Bou\'e-Dupuis formula (Lemma~\ref{LEM:var3}), we have
\begin{align}
-\log Z_{N^j_k}
= \inf_{\dot \Ups^{N^j_k}\in  \mathbb H_a^1}
\E \bigg[ \wt R^{\dia\dia}_{N^j_k}(Y+ \Ups^{N^j_k}- \ZZ_{N^j_k}) + \frac12 \int_0^1 \| \dot \Ups^{N^j_k}(t) \|_{H^1_x}^2dt \bigg]
\label{KZ17}
\end{align}

\noi
for $j= 1, 2$ and $k \in \N$.
Recall that  $Y$ and $\ZZ_N$ do not depend
on the drift $\dot \Ups^N$
in the Bou\'e-Dupuis formula~\eqref{KZ17}.

Given $\dl > 0$, 
let  $\UUps^{N^2_k}$ be an almost optimizer for \eqref{KZ17}:\footnote{For 
an almost optimizer $\UUps^{N^2_k}$ of the
minimization problem \eqref{KZ17}, we may assume that 
$\UUps^{N^2_k} = \pi_{N^2_k}\UUps^{N^2_k}$.
We, however, do not use this fact 
in view of a more general minimization problem
\eqref{KZ24} below.}
\begin{align}
-\log Z_{N^2_k}
& \ge 
\E \bigg[ \wt R^{\dia\dia}_{N^2_k}(Y+ \UUps^{N^2_k}- \ZZ_{N^2_k}) 
+ \frac12 \int_0^1 \| \dot  \UUps^{N^2_k}(t) \|_{H^1_x}^2dt \bigg] - \dl.
\label{KZ17a}
\end{align}

\noi
Then, 
 by choosing $\Ups^{N^1_k} = \UUps_{N^2_k}: =\pi_{N^2_k}\UUps^{N^2_k}$, 
 we have 
\begin{align}
- & \log  Z_{N^1_k}
+\log Z_{N^2_k} \notag \\
& \leq \inf_{\dot \Ups^{N^1_k}\in  \mathbb H_a^1} 
\E \bigg[ \wt R^{\dia\dia}_{N^1_k}(Y+ \Ups^{N^1_k}- \ZZ_{N^1_k}) + \frac12 \int_0^1 \| \dot \Ups^{N^1_k}(t) \|_{H^1_x}^2dt \bigg]\notag \\
& \quad \quad 
- \E \bigg[ \wt R^{\dia\dia}_{N^2_k}(Y+ \UUps^{N^2_k}- \ZZ_{N^2_k}) 
+ \frac12 \int_0^1 \| \dot  \UUps^{N^2_k}(t) \|_{H^1_x}^2dt \bigg] + \dl\notag \\
& \le 
\E \bigg[ \wt R^{\dia\dia}_{N^1_k}(Y+ \UUps_{N^2_k}- \ZZ_{N^1_k}) + \frac12 \int_0^1 \| \dot \UUps_{N^2_k}(t) \|_{H^1_x}^2dt \bigg]\notag \\
& \quad\quad 
- \E \bigg[ \wt R^{\dia\dia}_{N^2_k}(Y+ \UUps^{N^2_k}- \ZZ_{N^2_k}) 
+ \frac12 \int_0^1 \| \dot  \UUps^{N^2_k}(t) \|_{H^1_x}^2dt \bigg] + \dl\notag \\
& \leq   \E\Big[
\wt R^{\dia\dia}(Y_{N^1_k}+ \UUps_{N^2_k}- \ZZ_{N^1_k})
- \wt R^{\dia\dia}(Y_{N^2_k}+ \UUps_{N^2_k}- \ZZ_{N^2_k}) \Big]+ \dl, 
\label{KZ18}
\end{align}

\noi
where 
$ \wt R^{\dia\dia}$ is defined by 
\begin{align}
\begin{split}
\wt R^{\dia\dia}(Y+ \Ups- \ZZ)
& =  \frac 12 \int_{\T^3} (V_0 \ast :\! Y^2 \!: ) \Dr^2 dx
+ \int_{\T^3} [(V_0 \ast (Y \Dr )) Y \Dr]^\dia dx\\
&\hphantom{X}
+ \int_{\T^3} (V_0 \ast \Dr^2) Y \Dr dx
+ \frac 14 \int_{\T^3} (V_0 \ast \Dr^2 ) \Dr^2 dx
\\
&\hphantom{X}
+ \frac 14 \bigg\{ \int_{\T^3} \Big( :\! Y^2 \!: + 2 Y \Dr + \Dr^2 \Big) dx \bigg\}^2
\end{split}
\label{KZ18a}
\end{align}

\noi
for $\Dr = \Ups - \ZZ$.
At the last inequality in \eqref{KZ18}, we used the fact that 
$\pi_{N^1_k} \UUps_{N^2_k}  = \UUps_{N^2_k} $
under the assumption $N^1_k \ge N^2_k$.

In the following, we discuss how to estimate the difference
\begin{align}
 \E\Big[\wt R^{\dia\dia}(Y_{N^1_k}+ \UUps_{N^2_k}- \ZZ_{N^1_k})
- \wt R^{\dia\dia}(Y_{N^2_k}+ \UUps_{N^2_k}- \ZZ_{N^2_k})\Big].
\label{KZ18b}
\end{align}

\noi
The main point is that differences appear only for $Y$-terms and $\ZZ$-terms
(creating a negative power of $N^2_k$).
The contribution from 
the first term on the right-hand side
in \eqref{KZ18a}  is given by 
\begin{align}
\begin{split}
\E& \bigg[ \frac 12 \int_{\T^3} \big(V_0 \ast (
 :\! Y_{N^1_k}^2 \!: - :\! Y_{N^2_k}^2 \!: ) \big) \UUps_{N^2_k}^2 dx\bigg]\\
& - \E\bigg[ \frac 12 \int_{\T^3} \big(V_0 \ast (
 :\! Y_{N^1_k}^2 \!: - :\! Y_{N^2_k}^2 \!: ) \big) (2\UUps_{N^2_k} - \ZZ_{N^1_k})
 \ZZ_{N^1_k} dx\bigg]\\
& -  \E\bigg[\frac 12 \int_{\T^3} (V_0 \ast :\! Y_{N^2_k}^2 \!: ) 
(2\UUps_{N^2_k} - \ZZ_{N^1_k} - \ZZ_{N^2_k})
(\ZZ_{N^1_k} - \ZZ_{N^2_k}) dx\bigg].
\end{split}
\label{KZ19}
\end{align}

\noi
By slightly modifying the analysis in Subsections \ref{SUBSEC:def2}
and \ref{SUBSEC:def3}, we can bound  each term in~\eqref{KZ19} 
by 
\begin{align}
 (N^2_k)^{-a}\Big( 
C(Y_{N^1_k}, Y_{N^2_k},  \ZZ_{N^1_k}, \ZZ_{N^2_k})
+ 
\U^{\dia\dia}_{N^2_k}\Big)
\les  (N^2_k)^{-a}\Big( 
1 
+ 
\U^{\dia\dia}_{N^2_k}\Big), 
\label{KZ20}
\end{align}

\noi
for some small $a > 0$, where $\U^{\dia\dia}_{N^2_k}$ 
is given by  \eqref{K10} with $\Ups_N = \UUps_{N^2_k}$
and $\Ups^N = \UUps^{N^2_k}$
and $C(Y_{N^1_k}, Y_{N^2_k},  \ZZ_{N^1_k}, \ZZ_{N^2_k})$ denotes
certain high moments of various stochastic terms involving
$Y_{N^j_k}$ and $\ZZ_{N^j_k}$, $j = 1, 2$.
For example, 
proceeding as in \eqref{YZ4a}
together with H\"older's inequality in $\o$, 
we can estimate 
the first term in \eqref{KZ19} by
\begin{align*}
&\les \E\Big[\| :\! Y_{N^1_k}^2 \!: - :\! Y_{N^2_k}^2 \!:  \|_{\C^{-1-\eps}} 
\| \UUps_{N^2_k} \|_{L^2}^{1 + \be - 2\eps } \|\UUps_{N^2_k} \|_{H^1}^{1 - \be +  2\eps}\Big] \\
&\le \| :\! Y_{N^1_k}^2 \!: - :\! Y_{N^2_k}^2 \!:  \|_{L^\frac 4{1+\be - 2\eps}_\o \C^{-1-\eps}_x} 
\| \UUps_{N^2_k} \|_{L^4_\o L^2_x}^{1 + \be - 2\eps } \|\UUps_{N^2_k} \|_{L^2_\o H^1_x}^{1 - \be +  2\eps}
 \\
&\les  (N^2_k)^{-a}
\| \UUps_{N^2_k} \|_{L^4_\o L^2_x}^{1 + \be - 2\eps } \|\UUps_{N^2_k} \|_{L^2_\o H^1_x}^{1 - \be +  2\eps}
\\
&\les  (N^2_k)^{-a}\Big( 1+ \U^{\dia\dia}_{N^2_k}\Big), 
\end{align*}

\noi
where the third inequality follows from
a modification of the proof of Lemma \ref{LEM:stoconv}
and 
 \eqref{reg22} in 
Lemma \ref{LEM:reg}.
By modifying the proofs of 
Lemmas \ref{LEM:Dr8}  and \ref{LEM:CZ},\footnote{\label{FT:xx}In order to obtain a decay  
$ (N^2_k)^{-a}$ from a variant of Lemma \ref{LEM:CZ}, we also need
to control the term $\ZZ_{N^1_k} - \ZZ_{N^2_k}$.
In view of the definitions \eqref{YZ12} and \eqref{K9a},  a modification of Lemma \ref{LEM:IV}
yields a decay  $ (N^2_k)^{-a}$ in estimating  
$\ZZ_{N^1_k} - \ZZ_{N^2_k}
= (\pi_{N^1_k} - \pi_{N^2_k})
\ZZ_{N^1_k}  + \pi_{N^2_k}(\ZZ_{N^1_k} - \ZZ_{N^2_k})$.}
we can also estimate 
the other two terms in~\eqref{KZ19} by~\eqref{KZ20}.

Similarly, 
the contribution to the difference \eqref{KZ18b}
from the third, fourth, and fifth terms in~\eqref{KZ18a}
can also be estimated by~\eqref{KZ20}.
As for the contribution from the second term in~\eqref{KZ18a}, 
we need to check that the difference $T_{N^1_k} - T_{N^2_k}$
of the operator $T_N$ defined in~\eqref{KZ4}
gives a decay $ (N^2_k)^{-a}$.
It follows from \eqref{KZ9} that 
in studying the difference $T_{N^1_k} - T_{N^2_k}$, 
we have an extra condition $\max(|n_1|, |n_2|) > N^2_k$ in 
\eqref{KZ10}, which allows us to gain a small negative power of  $N^2_k$.
Thus, we can also bound 
the contribution from the second term in \eqref{KZ18a}
by~\eqref{KZ20}.
Hence, we conclude that \eqref{KZ18b} is bounded by 
\eqref{KZ20}.

 It follows from (a slight modification of)
 Lemmas \ref{LEM:Dr7}, \ref{LEM:Dr8}, and \ref{LEM:Dr9}
 together with Lemmas~\ref{LEM:stoconv}, \ref{LEM:IV}, and 
\ref{LEM:CZ} that 
$ C(Y_{N^1_k}, Y_{N^2_k},  \ZZ_{N^1_k}, \ZZ_{N^2_k})$
in \eqref{KZ20} 
 is uniformly bounded in $N^1_k$ and $N^2_k$, $k \in \N$.
 Furthermore, 
 from 
 the discussion in Subsection \ref{SUBSEC:def3} (see \eqref{KZ14a}), 
 \eqref{KZ17a}, 
 and 
 \eqref{KZ17}, we have
\begin{align}
\begin{split}
\sup_{k \in \N} \, \U^{\dia\dia}_{N^2_k}
& \leq 10 C_0 + 10 
\sup_{k \in \N}\E \bigg[ \wt R^{\dia\dia}_{N^2_k}(Y+ \UUps^{N^2_k}- \ZZ_{N^2_k}) 
+ \frac12 \int_0^1 \| \dot  \UUps^{N^2_k}(t) \|_{H^1_x}^2dt \bigg] \\
& \leq 10(C_0 + \dl) 
+ 10 \sup_{k \in \N}
\E \bigg[ \wt R^{\dia\dia}_{N^2_k}(Y+ 0- \ZZ_{N^2_k})  \bigg]\\
& \les 1.
\end{split}
\label{KZ20a}
\end{align}

\noi
Therefore, we conclude that 
\begin{align}
 \E\bigg[\wt R^{\dia\dia}(Y_{N^1_k}+ \UUps_{N^2_k}- \ZZ_{N^1_k})
- \wt R^{\dia\dia}(Y_{N^2_k}+ \UUps_{N^2_k}- \ZZ_{N^2_k})\bigg]
\les  (N^2_k)^{-a}
\too 0
\label{KZ21}
\end{align}

\noi
as $k \to \infty$.
 Since the choice of $\dl > 0$ was arbitrary, 
we obtain, from \eqref{KZ18} and \eqref{KZ21}, 
\begin{align}
 \lim_{k \to \infty} Z_{N^1_k} \geq  \lim_{k \to \infty} Z_{N^2_k}. 
 \label{KZ22}
\end{align}

\noi
We proved \eqref{KZ22} under the assumption
 $N^1_k \ge N^2_k$, $k \in \N$.
By extracting a further subsequence, still denoted by 
 $N^1_k$ and $N^2_k$, 
we can assume that 
 $N^1_k \le  N^2_k$, $k \in \N$, which leads
 to 
\begin{align}
 \lim_{k \to \infty} Z_{N^1_k} \leq  \lim_{k \to \infty} Z_{N^2_k}, 
 \label{KZ22a}
\end{align}

\noi
since the limit,  
$ \lim_{k \to \infty} Z_{N^j_k}$, $j = 1, 2$, remains
the same under the extraction of subsequences.
Hence, from \eqref{KZ22} and \eqref{KZ22a}, 
we  conclude
\eqref{KZ15}.

Next, we show $\rho^{(1)} = \rho^{(2)}$.
This claim follows from a small variation of the argument presented above.
For this purpose, it suffices  to prove that for every bounded Lipschitz continuous function $F: \C^{-100}(\T^3) \to \R$,
we have  
\begin{align*}
 \lim_{k\to \infty} \int \exp(F(u)) d \rho_{N^1_k} 
 \ge  \lim_{k\to \infty} \int \exp(F(u)) d \rho_{N^2_k}.
\end{align*}

\noi
In view of \eqref{KZ15}, it suffices to show
\begin{align}
\begin{split}
\limsup_{k\to \infty} \bigg[  -\log\bigg(\int  & \exp(  F(u)
 -R_{N^1_k}^{\dia\dia}(u))
d\mu \bigg) \\
& + \log\bigg(\int \exp(F(u) -R_{N^2_k}^{\dia\dia}(u)) d \mu\bigg)\bigg] \le 0.
 \label{KZ23}
\end{split}
\end{align}

\noi
As before, we assume  $N^1_k \ge N^2_k$, $k \in \N$, without loss of generality.
By the Bou\'e-Dupuis formula (Lemma \ref{LEM:var3}), we have 
\begin{align}
\begin{split}
- & \log  \bigg(\int \exp(F(u)
-R_{N^j_k}^{\dia\dia}(u))
 d \mu  \bigg)\\
&  = \inf_{\dot \Ups^{N^j_k}\in  \mathbb H_a^1}\E \bigg[ 
- F(Y+ \Ups^{N^j_k}- \ZZ_{N^j_k})\\
& \hphantom{XXXXXXl} + \wt R^{\dia\dia}_{N^j_k}(Y+ \Ups^{N^j_k}- \ZZ_{N^j_k}) + \frac12 \int_0^1 \| \dot \Ups^{N^j_k}(t) \|_{H^1_x}^2dt \bigg].
\end{split}
\label{KZ24}
\end{align}

Given $\dl > 0$, 
let  $\UUps^{N^2_k}$ be an almost optimizer for \eqref{KZ24}
with $j =2$:
\begin{align*}
-\log  & \bigg(\int \exp(F(u)
-R_{N^2_k}^{\dia\dia}(u))
 d \mu  \bigg)\\
& \ge 
\E \bigg[ 
- F(Y+ \UUps^{N^2_k}- \ZZ_{N^2_k})
+ 
\wt R^{\dia\dia}_{N^2_k}(Y+ \UUps^{N^2_k}- \ZZ_{N^2_k}) 
+ \frac12 \int_0^1 \| \dot  \UUps^{N^2_k}(t) \|_{H^1_x}^2dt \bigg] - \dl.
\end{align*}

\noi
Then, by choosing $\Ups^{N^1_k} = \UUps_{N^2_k} = \pi_{N^2_k}\UUps^{N^2_k}$
and proceeding as in \eqref{KZ18}, 
 we have 
\begin{align}
  - & \log\bigg(\int   \exp(  F(u)
 -R_{N^1_k}^{\dia\dia}(u))
d\mu \bigg) 
 + \log\bigg(\int \exp(F(u) -R_{N^2_k}^{\dia\dia}(u)) d \mu\bigg) \notag \\
& \le 
\E \bigg[ 
- F(Y+ \UUps_{N^2_k}- \ZZ_{N^1_k})
+ \wt R^{\dia\dia}_{N^1_k}(Y+ \UUps_{N^2_k}- \ZZ_{N^1_k}) + \frac12 \int_0^1 \| \dot \UUps_{N^2_k}(t) \|_{H^1_x}^2dt \bigg]\notag \\
& \quad
- \E \bigg[ 
- F(Y+ \UUps^{N^2_k}- \ZZ_{N^2_k})
+ \wt R^{\dia\dia}_{N^2_k}(Y+ \UUps^{N^2_k}- \ZZ_{N^2_k}) 
+ \frac12 \int_0^1 \| \dot  \UUps^{N^2_k}(t) \|_{H^1_x}^2dt \bigg] + \dl\notag \\
& \le  
\Lip (F)\cdot 
 \E\Big[ \|  \pi_{N^2_k}^\perp\UUps^{N^2_k}
 +  \ZZ_{N^1_k} -  \ZZ_{N^2_k}\|_{\C^{-100}} \Big] \notag \\
 & \quad +  \E\Big[
\wt R^{\dia\dia}(Y_{N^1_k}+ \UUps_{N^2_k}- \ZZ_{N^1_k})
- \wt R^{\dia\dia}(Y_{N^2_k}+ \UUps_{N^2_k}- \ZZ_{N^2_k}) \Big]+ \dl, 
\label{KZ26}
\end{align}

\noi
where $\pi_N^\perp = \text{Id} - \pi_N$.
We can proceed as before to show that  the second term
on the right-hand side of \eqref{KZ26}
satisfies \eqref{KZ21}. 
Here, we need to use the boundedness of $F$ in showing an analogue
of \eqref{KZ20a} in the current context.

By writing
\begin{align*}
 \E\Big[   \|  \pi_{N^2_k}^\perp\UUps^{N^2_k}
 +  \ZZ_{N^1_k} -  \ZZ_{N^2_k}\|_{\C^{-100}} \Big] 
 \leq
 \E\Big[ \|  \pi_{N^2_k}^\perp\UUps^{N^2_k}\|_{\C^{-100}} \Big] 
 + 
 \E\Big[ \| 
  \ZZ_{N^1_k}  -   \ZZ_{N^2_k}\|_{\C^{-100}} \Big], 
\end{align*}

\noi
we see
from Footnote \ref{FT:xx}
 that the second term on the right-hand side tends
to 0 as $k \to \infty$.  
As for the first term, 
from Lemma \ref{LEM:Dr} and (an analogue of) \eqref{KZ20a}, we obtain
\begin{align*}
 \E\Big[ \|  \pi_{N^2_k}^\perp\UUps^{N^2_k}\|_{\C^{-100}} \Big] 
 \les
(N^2_k)^{-a} \|  \UUps^{N^2_k}\|_{L^2_\o H^1} 
 \les
(N^2_k)^{-a} \Big(\sup_{k \in \N} \, \U_{N^2_k}^{\dia\dia}\Big)^\frac{1}{2}
\too 0 , 
\end{align*}

\noi
as $k \to \infty$.
Since the choice of $\dl > 0$ was arbitrary, 
we conclude \eqref{KZ23}
and hence $\rho^{(1)} = \rho^{(2)}$ by symmetry.
This completes the proof of Proposition \ref{PROP:uniq}.
\end{proof}

\subsection{Singularity of the defocusing Gibbs measure for $0 < \be \le \frac 12$}
\label{SUBSEC:def5}

In this subsection,
we prove that 
the Gibbs measure $\rho$ for $0<\be \le \frac 12$ is singular with respect to the reference Gaussian free field $\mu$.
While our proof is inspired by the discussion in Section 4 of \cite{BG2}, 
we directly prove singularity without referring to  the shifted measure.
In  Appendix \ref{SEC:D}, we show that the Gibbs measure is indeed absolutely continuous with respect 
to the shifted measure, namely, the law of~$Y(1) - \ZZ(1) + \W(1)$, 
where the auxiliary process $\W = \W(Y)$ is defined in~\eqref{AC0}.

Given $N \in \N$, define $A_N$ and $B_N$ by 
\begin{align}
A_N := \sum_{|n| \le N} \jb{n}^{-2\be-2} 
\sim
\begin{cases}
\log N, & \text{if } \be=\frac 12, \\
 N^{1-2\be}, & \text{if } \be<\frac 12, 
\end{cases}
\label{Ks0a}
\end{align}

\noi
and
\begin{align}
B_N := (\log N)^{-\frac 14} A_N^{-\frac 12}
\sim
\begin{cases}
(\log N)^{- \frac 34}, & \text{if } \be=\frac 12, \\
(\log N)^{-\frac 14} N^{\be-\frac 12}, & \text{if } \be<\frac 12.
\end{cases}
\label{Ks0}
\end{align}

\begin{proposition}\label{PROP:sing}
Let  $0 < \be \le \frac 12$, $\eps>0$, 
and  $R_N^{\dia}$ be as  in \eqref{K1r}.
Then, 
there exists a strictly increasing sequence $\{ N_k \}_{k =1}^\infty \subset \N$ such that
the set
\begin{align*}
S := \big\{ u \in H^{-\frac 12-\eps}(\T^3) : \lim_{k \to \infty} B_{N_k} R_{N_k}^{\dia}(u) =0 \big\}
\end{align*}
has $\mu$-full measure\textup{:} $\mu (S) = 1$.
Furthermore, 
we have
\begin{align}
\rho (S) =0.
\label{Ks2}
\end{align}

\noi
In particular, the Gibbs measure $\rho$ and 
the Gaussian free field $\mu$ are mutually singular for $0 < \be \le \frac 12$.

\end{proposition}

\begin{proof}
By repeating the computation as in Subsection \ref{SUBSEC:R}, we have
\begin{align}
\| R_N^{\dia}(u) \|_{L^2(\mu)}^2
\les \sum_{|n| \le N} \jb{n}^{-2\be-2} = A_N.
\label{Ks3}
\end{align}

\noi
Then, from 
 \eqref{Ks0} and \eqref{Ks3},
we have
\[
\lim_{N \to \infty} B_N
\| R_N^{\dia}(u) \|_{L^2(\mu)}
\les \lim_{N \to \infty}
(\log N)^{-\frac 14}
= 0.
\]

\noi
Hence, there exists a subsequence such that 
\[
\lim_{k \to \infty} B_{N_k} R_{N_k}^{\dia}(u)
=0
\]
almost surely with respect to $\mu$.

Define $G_k(u)$ by 
\[ G_k(u) = B_{N_k} R_{N_k}^{\dia}(u) - \|u\|_{\C^{-\frac 12 - \eps}}^{10}\]

\noi
for some small $\eps > 0$.
In the following, 
 we show that 
$e^{G_k(u)}$ tends to $0$ in $L^1(\rho)$.
This will imply that there exists a subsequence
of $G_k(u)$ 
 tending to $- \infty$, 
almost surely with respect to the Gibbs measure $\rho$.
Recalling the almost sure boundedness of $\|u\|_{\C^{-\frac 12 - \eps}}^{10}$
under the Gibbs measure $\rho$, 
this shows that  
$B_{N_k} R_{N_k}^{\dia}$ tends $\rho$-almost surely to $- \infty$
along this subsequence,
 which in turn yields~\eqref{Ks2}.

Let $\phi$ be a smooth bump function as in Subsection \ref{SUBSEC:21}.
By Fatou's lemma, the weak convergence of $\rho_M$ to $\rho$, 
and the boundedness of $\phi$, we have
\begin{align}
\begin{split}
\int  e^{G_k(u)}d\rho(u)
& \leq \liminf_{K \to \infty}
\int \phi \bigg(\frac{G_k(u) }{K}\bigg)e^{G_k(u)}d\rho(u)\\
& = \liminf_{K \to \infty}
\lim_{M \to \infty}
\int \phi \bigg(\frac{G_k(u)}{K}\bigg)e^{G_k(u)}d\rho_M(u)\\
& \leq 
\lim_{M \to \infty}
\int e^{G_k(u)}d\rho_M(u)
=: Z^{-1} \lim_{M\to \infty} C_{M, k}, 
\end{split}
\label{Ks4}
\end{align}

\noi
provided that  $\lim_{M\to \infty} C_{M, k} $ exists.
Here,  $Z = \lim_{M \to \infty} Z_M$ denotes the partition function for~$\rho$,
which is well defined thanks to \eqref{KZ15}.
In the following, we show that the right-hand side of~\eqref{Ks4}
tends to $0$ as $k \to \infty$.

As in the previous subsection, we proceed with 
 the change of variables \eqref{YZ13}: 
$ \dot \Ups^M(t) : = \dot \Dr(t)  +  \dot \ZZ_M(t)$.
Then, 
by the Bou\'e-Dupuis formula (Lemma \ref{LEM:var3}), we have 
\begin{align}
\begin{aligned}
- \log C_{M, k} 
& 
= \inf_{\dot \Ups^{M}\in  \Ha^1}
\E \bigg[ 
- B_{N_k} R_{N_k}^{\dia}(Y+ \Ups^{M}- \ZZ_{M}) 
+   \|Y+ \Ups^{M}- \ZZ_{M}\|_{\C^{-\frac 12 - \eps}}^{10}
\\
& \hphantom{XXXXX}+ 
\wt R^{\dia\dia}_{M}(Y+ \Ups^{M}- \ZZ_{M}) + \frac12 \int_0^1 \| \dot \Ups^{M}(t) \|_{H^1_x}^2dt \bigg]\\
& 
=:  \inf_{\dot \Ups^{M}\in  \Ha^1}
\ft \W_{M, k}^{\dia\dia}(\Ups^M), 
\end{aligned}
\label{Ks5a}
\end{align}

\noi
where 
$\wt R^{\dia\dia}_{N}$ is defined in \eqref{KZ16}.
Let $Q_N^\dia := Q_N - Q_{N, 3}$, where $Q_N$
and $Q_{N, 3}$ are as in \eqref{KN} and \eqref{B4}.
Then, from 
\eqref{K1r}, \eqref{Y0}, 
\eqref{YZ6a},   and \eqref{YZ8a},
we have
\begin{align}
\begin{split}
R_{N_k}^{\dia} & (Y + \Ups^M - \ZZ_M)\\
&=
\frac 14 Q_{N_k}^{\dia}(Y)
+ \int_{\T^3} [(V_0 \ast :\! Y_{N_k}^2 \!: ) Y_{N_k} ]^\dia \Dr_{N_k} dx \\
&\hphantom{X}
+ \frac 12 \int_{\T^3} (V_0 \ast :\! Y_{N_k}^2 \!: ) \Dr_{N_k}^2 dx
+ \int_{\T^3} [(V_0 \ast (Y_{N_k} \Dr_{N_k} )) Y_{N_k} \Dr_{N_k}]^\dia dx
\\
&\hphantom{X}
+ \int_{\T^3} (V_0 \ast \Dr_{N_k}^2) Y_{N_k} \Dr_{N_k} dx
+ \frac 14 \int_{\T^3} (V_0 \ast \Dr_{N_k}^2 ) \Dr_{N_k}^2 dx\\
&\hphantom{X}
+ \frac 14 \bigg\{ \int_{\T^3} \Big( :\! Y_{N_k}^2 \!: + 2 Y_{N_k} \Dr_{N_k} + \Dr_{N_k}^2 \Big) dx \bigg\}^2
\end{split}
\label{Ks6}
\end{align}

\noi
for $N_k \leq M$, 
where $\Dr_{N_k}$ is given by 
\begin{align}
\Dr_{N_k} := \pi_{N_k} \Dr = \pi_{N_k} \Ups^M - \pi_{N_k} \ZZ_M.
\label{Ks61}
\end{align}

We can handle the contribution from the last two terms
on the right-hand side of \eqref{Ks5a}
  as in Subsection \ref{SUBSEC:def3}
(see \eqref{KZ14a}) and obtain
\begin{align}
\E \bigg[ 
\wt R^{\dia\dia}_{M}(Y+ \Ups^{M}- \ZZ_{M}) + \frac12 \int_0^1 \| \dot \Ups^{M}(t) \|_{H^1_x}^2dt \bigg]
\geq 
 -C_0 + \frac{1}{10}\U^{\dia\dia}_M, 
\label{Ks62}
\end{align}

\noi
where
 $\U^{\dia\dia}_M$
is given by  \eqref{K10} with $\Ups_N = \pi_M \Ups^M$
and 
 $\Ups^N = \Ups^M$.
The main contribution to~\eqref{Ks5a} comes from the first term.
More precisely, 
 under an expectation, 
the second term on the right-hand side of \eqref{Ks6} gives
a (negative) divergent contribution; see \eqref{Ks8} below.
From~\eqref{Y0a}, the first term 
on the right-hand side of \eqref{Ks6}
gives $0$
under an expectation, 
while we can bound the other terms 
(excluding the first and second terms) as in Subsection~\ref{SUBSEC:def3}
and obtain
\begin{align}
\begin{split}
\E& \bigg[   \bigg|R_{N_k}^{\dia}  (Y + \Ups^M - \ZZ_M)
- \frac 14 Q_{N_k}^{\dia}(Y)
-  \int_{\T^3} [(V_0 \ast :\! Y_{N_k}^2 \!: ) Y_{N_k} ]^\dia \Dr_{N_k} dx \bigg|\bigg]\\
& \les 
C(Y_{N_k},   \pi_{N_k} \ZZ_{M})
+ 
\U^{\dia\dia}_{N_k}
 \les 
1+ 
\U^{\dia\dia}_{N_k}
\end{split}
\label{Ks63}
\end{align}

\noi
where 
$C(Y_{N_k}, \pi_{N_k}  \ZZ_{M})$ denotes
certain high moments of various stochastic terms involving
$Y_{N_k}$ and $  \ZZ_{N_k}$
and 
$\U^{\dia\dia}_{N_k} = \U^{\dia\dia}_{N_k}( \pi_{N_k} \Ups^M)$ 
is given by  \eqref{K10} with $\Ups_N = 
\Ups^N =  \pi_{N_k} \Ups^M $:
\begin{equation}
\begin{split}
\U_{N_k}^{\dia\dia} = \E\bigg[ & \frac 18 \int_{\T^3} (V_0 \ast (\pi_{N_k} \Ups^M)^2) 
(\pi_{N_k} \Ups^M )^2 dx \\
& + \frac 1{32} \bigg( \int_{\T^3} (\pi_{N_k} \Ups^M )^2 dx \bigg)^2 
+ \frac{1}{2} \int_0^1 \| \dt (\pi_{N_k} \Ups^M) (t) \|_{H^1_x}^2 dt\bigg].
\end{split}
\label{Ks64}
\end{equation}

\noi
Note that, in view of the smallness of $B_{N_k}$
in \eqref{Ks5a},  the second and third terms in \eqref{Ks64}
can be controlled by the positive terms $\U_M^{\dia\dia}$ coming from the last two terms 
in \eqref{Ks5a}.
On the other hand, the first term on the right-hand side of  \eqref{Ks64}
can not be controlled by the corresponding potential energy\footnote{Recall the notation
$\Ups_M = \pi_M \Ups^M$.}
$\frac 18 \int_{\T^3} (V_0 \ast \Ups_{M}^2) \Ups_{M}^2 dx $
in $\U_M^{\dia\dia}$.
Here, the second term on the right-hand side of \eqref{Ks5a} comes to the rescue.
From Sobolev's inequality, 
the interpolation~\eqref{interp}
(with $0 = \ta\cdot 1 + (1-\ta) (-\frac 12 - 2\eps)$ for differentiability), 
and 
Young's inequality, we have
\begin{align}
\begin{split}
 \int_{\T^3} (V_0 \ast (\pi_{N_k} \Ups^M )^2) (\pi_{N_k} \Ups^M )^2 dx 
& =  \| (\pi_{N_k} \Ups^M )^2\|_{\dot H^{-\frac \be 2}}^2
\les \| \pi_{N_k} \Ups^M \|_{L^\frac{12}{3+\be}}^4\\
& \les \| \pi_{N_k} \Ups^M \|_{H^1}^\frac{4+16\eps}{3+4\eps}
\| \pi_{N_k} \Ups^M \|_{\C^{-\frac 12 - \eps}}^\frac{8 }{3+4\eps}\\
& \les 1 + 
\| \Ups^{M}\|_{H^1}^2
+ \| \Ups^M\|_{\C^{-\frac 12 - \eps}}^{10}.
\end{split}
\label{Ks65}
\end{align}

\noi
Hence, from \eqref{Ks0}, \eqref{Ks5a}, \eqref{Ks62}, \eqref{Ks63}, 
and \eqref{Ks65} 
with the following bound:
\begin{align*}
\|Y+ \Ups^{M}- \ZZ_{M}\|_{\C^{-\frac 12 - \eps}}^{10}
\ges 
\| \Ups^{M}\|_{\C^{-\frac 12 - \eps}}^{10}
- \|Y \|_{\C^{-\frac 12 - \eps}}^{10}
- \| \ZZ_{M}\|_{\C^{-\frac 12 - \eps}}^{10}, 
\end{align*}

\noi
we obtain
\begin{align}
\ft \W_{M, k}^{\dia\dia}(\Ups^M)
\geq -    B_{N_k}\E \bigg[ 
 \int_{\T^3} [(V_0 \ast :\! Y_{N_k}^2 \!: ) Y_{N_k} ]^\dia \Dr_{N_k} dx \bigg]
- C_1 + \frac{1}{20}\U^{\dia\dia}_M
\label{Ks66}
\end{align}

\noi
for any $1 \ll N_k \leq M$.

It remains to estimate the contribution from the second term
on the right-hand side of~\eqref{Ks6}.
Let us first state a lemma whose proof is presented at the end of this subsection.

\begin{lemma}\label{LEM:Ks3}
Let  $0<\be<1$.
 Then, 
we have
\begin{align}
\bigg| \E \bigg[ \int_0^1\int_{\T^3} (1-\Dl) \dot \ZZ_N (t)\cdot (\dot \ZZ_N-\dot \ZZ_M)(t) dx dt \bigg] \bigg| \les 1
\label{Ks7}
\end{align}

\noi
for  $1\leq N \le M$, where
 $\dot \ZZ_N = \pi_N \dot \ZZ^N$.
\end{lemma}

Note that when $0 < \be \le \frac 12$
(namely, in the case where  we need to apply the change of variables \eqref{YZ13}), 
the pathwise regularity $\be + \frac 12 - \eps$ of 
$\dot \ZZ_N $ is not sufficient to prove Lemma \ref{LEM:Ks3}.
Instead,   we prove Lemma \ref{LEM:Ks3} 
by exploiting  orthogonality, coming from the frequency support consideration.

\medskip

By \eqref{YZ11}, \eqref{YZ12}, \eqref{Ks61}, Lemma \ref{LEM:Ks3}, Cauchy's inequality, 
\eqref{CZ2a}, and \eqref{Ks0a}, 
we have 
\begin{align}
\begin{aligned}
 \E & \bigg[\int_{\T^3} [(V_0 \, \ast \!:\!Y_{N_k}^2\!:) Y_{N_k}]^\dia \Dr_{N_k} dx \bigg] 
 = \E\bigg[ \int_0^1 \int_{\T^3}  
 [(V_0 \, \ast \! :\!Y_{N_k}(t)^2\!:) Y_{N_k}(t)]^\dia\dot \Dr_{N_k}(t) dt \bigg]\\
& 
 = 
   \E \bigg[ \int_0^1\int_{\T^3} (1-\Dl) \dot \ZZ_{N_k}(t) \cdot (\dot \ZZ_{N_k}-\dot \ZZ_{M}) (t)dx dt \bigg]\\
& \quad   - \E\bigg[ \int_0^1 \| \dot \ZZ_{N_k}(t)\|_{H^1_x}^2 dt   \bigg]
+ \E\bigg[   \int_0^1   \jb{\dot \ZZ_{N_k}(t), \dt \pi_{N_k} \Ups^M (t)}_{H^1_x} dt \bigg]\\
& 
 \leq C - \frac 12 \E\bigg[ \int_0^1 \| \dot \ZZ_{N_k}(t)\|_{H^1_x}^2 dt   \bigg]
 + \frac 12 \E\bigg[ \int_0^1 \| \dot \Ups^M(t)\|_{H^1_x}^2 dt   \bigg]\\
& \le C - c A_{N_k}
 + \frac 12 \E\bigg[ \int_0^1 \| \dot \Ups^M(t)\|_{H^1_x}^2 dt \bigg]
\end{aligned}
\label{Ks8}
\end{align}

\noi
for $1\leq N_k \le M$.
Thus, putting 
\eqref{Ks5a}, 
\eqref{Ks66}, 
and 
\eqref{Ks8} together, 
we have 
\begin{align}
\begin{aligned}
- \log C_{M, k} 
& 
\geq \inf_{\dot \Ups^{M}\in  \Ha^1}\Big\{ 
cB_{N_k} A_{N_k} - C_2 + \frac{1}{40}\U^{\dia\dia}_M\Big\}
\geq 
c B_{N_k} A_{N_k} - C_2.
\end{aligned}
\label{Ks9}
\end{align}

\noi
Hence, from \eqref{Ks9} with \eqref{Ks0a} and \eqref{Ks0}, 
we obtain 
\begin{align}
C_{M, k}
\les
\begin{cases}
\exp \Big( -c (\log N_k)^{\frac 14} \Big), & \text{if } \be=\frac 12, \\[5pt]
\exp \Big( -c (\log N_k)^{-\frac 14} N_k^{-\be+\frac 12} \Big), & \text{if } 0<\be<\frac 12
\end{cases}
\label{Ks10}
\end{align}
for $1\ll N_k \le M$, uniformly in $M \in \N$.
Therefore, 
by taking limits in $M \to \infty$ and then $k \to \infty$, 
we conclude from \eqref{Ks4} and \eqref{Ks10} that
\[ 
\lim_{k \to \infty} \int  e^{G_k(u)}d\rho(u) = 0\]

\noi
as desired.
This completes the proof of Proposition \ref{PROP:sing}.
\end{proof}

We conclude this subsection by presenting the proof of Lemma \ref{LEM:Ks3}.

\begin{proof}[Proof of Lemma \ref{LEM:Ks3}]
In the following, we drop the time variable. Let
\begin{align}
:\!\ft Y_N(n_1) \ft Y_N(n_2)\!:
\ = 
\ft Y_N(n_1) \ft Y_N(n_2) - \ind_{n_1+n_2=0} \cdot \jb{n_1}^{-2}.
\label{Ks11}
\end{align}

\noi
Then, proceeding as in \eqref{YS1b} and \eqref{Z13}
with 
  \eqref{YZ12}, \eqref{YZ6a}, and \eqref{kappa2},
we have
\begin{align}
\ft {\dot \ZZ}_N (n)
&= \jb{n}^{-2} \bigg(
\sum_{\substack{n_1, n_2, n_3 \in \Z^3 \\ n=n_1+n_2+n_3}} \ft V_0(n_1+n_2) 
\big( :\!\ft Y_N(n_1) \ft Y_N(n_2)\!:\big) \ft Y_N(n_3)
- 2 \kk_N(n) \ft Y_N(n)
\bigg) \notag\\
&= \jb{n}^{-2} 
\sum_{\substack{n_1, n_2, n_3 \in \Z^3 \\ n=n_1+n_2+n_3 \\ 
|n_2+n_3| |n_3+n_1| \ne 0}}
 \ft V_0(n_1+n_2) \big( :\!\ft Y_N(n_1) \ft Y_N(n_2)\!:\big) \ft Y_N(n_3) \notag \\
&\quad
+ 2 \jb{n}^{-2} \ft Y_N(n) \sum_{\substack{n_1 \in \Z^3 \\ |n_1| \le N}} \ft V_0 (n+n_1) \Big( |\ft Y_N(n_1)|^2 - \jb{n_1}^{-2} \Big) \notag\\
&\quad
- \jb{n}^{-2} \ft V_0 (2n) |\ft Y_N (n)|^2 \ft Y_N(n) \notag\\
&=:
\ft {\dot \ZZ}_{N,1} (n)
+ \ft {\dot \ZZ}_{N,2} (n)
+ \ft {\dot \ZZ}_{N,3} (n)
\label{Kt2}
\end{align}

\noi
for $|n| \le N$.
By repeating the proof of Lemma \ref{LEM:IV},
we have
\begin{align}
\E \Big[ |\ft {\dot \ZZ}_{N,1} (n)|^2 \Big]
\sim \jb{n}^{-2\be-4}.
\label{Kt3}
\end{align}

\noi
Also, by a computation analogous to \eqref{YS5} and \eqref{YS3}, 
we have
\begin{align}
\begin{aligned}
\E \Big[ |\ft {\dot \ZZ}_{N,2} (n)|^2 \Big]
+ \E \Big[ |\ft {\dot \ZZ}_{N,3} (n)|^2 \Big]
&\les \jb{n}^{-2\be-6}.
\end{aligned}
\label{Kt4}
\end{align}

\noi
Hence, from \eqref{Kt2},  \eqref{Kt3}, and \eqref{Kt4}, 
we have
\begin{align}
&\E \bigg[  \int_{\T^3} (1-\Dl) \dot \ZZ_N \cdot (\dot\ZZ_N-\dot\ZZ_M) dx  \bigg] \notag\\
&= \E \Bigg[ \sum_{n \in \Z^3} \jb{n}^2 \ft {\dot \ZZ}_N (n) \Big( \cj{\ft {\dot\ZZ}_N (n) - \ft {\dot\ZZ}_M(n)} \Big) \bigg] \notag\\
&=
\sum_{n \in \Z^3} \jb{n}^2 \E \Big[ \ft {\dot\ZZ}_{N,1} (n) \Big( \cj{\ft {\dot\ZZ}_{N,1} (n) 
- \ft {\dot\ZZ}_{M,1}(n)} \Big) \Big]
+ O \bigg( \sum_{\substack{n \in \Z^3 \\ |n| \le N}} \jb{n}^2 \jb{n}^{-\be-2} \jb{n}^{-\be-3} \bigg) \notag\\
&=
\sum_{n \in \Z^3} \jb{n}^2 \E \Big[ \ft {\dot\ZZ}_{N,1} (n) \Big( \cj{\ft{\dot\ZZ}_{N,1} (n) - \ft {\dot\ZZ}_{M,1}(n)} \Big) \Big]
+ O (1)
\label{Kt5}
\end{align}

\noi
for $\be>0$.

We now write 
 $\ft {\dot \ZZ}_{M,1} (n) - \ft {\dot \ZZ}_{N,1}(n)$ as follows:
\begin{align*}
&\ft {\dot \ZZ}_{M,1} (n) - \ft {\dot \ZZ}_{N,1}(n) \\
&= \sum_{\substack{j,k,\l \in \{1,2,3\} \\ \{ j,k, \l \} = \{ 1,2,3 \}}}
\sum_{\substack{n_1, n_2, n_3 \in \Z^3 \\ n=n_1+n_2+n_3 \\ |n_j|>N, \, |n_k| \le N, \, |n_\l| \le N\\ 
|n_2+n_3| |n_3+n_1| \ne 0}}
 \ft V_0(n_1+n_2) \big(:\!\ft Y_M(n_1) \ft Y_M(n_2) \!:\big) \ft Y_M(n_3) \\
&\quad 
+ \sum_{\substack{j,k,\l \in \{1,2,3\} \\ \{ j,k, \l \} = \{ 1,2,3 \}}}\sum_{\substack{n_1, n_2, n_3 \in \Z^3 \\ n=n_1+n_2+n_3 \\ |n_j|>N, \, |n_k|>N, \, |n_\l| \le N \\ 
|n_2+n_3| |n_3+n_1| \ne 0}}\ft V_0(n_1+n_2) 
\big( :\!\ft Y_M(n_1) \ft Y_M(n_2) \!:\big) \ft Y_M(n_3) \\
&\quad
+ \sum_{\substack{n_1, n_2, n_3 \in \Z^3 \\ n=n_1+n_2+n_3 \\ |n_1|, |n_2|, |n_3|>N \\ 
|n_2+n_3| |n_3+n_1| \ne 0}}
 \ft V_0(n_1+n_2) 
\big( :\!\ft Y_M(n_1) \ft Y_M(n_2) \!:\big) \ft Y_M(n_3)\\
&=: \1(n) + \II(n) + \III(n).
\end{align*}

\noi
By 
 the independence of $\{ \ft Y(n) \}_{n \in \Ld_0}$
where the index set $\Ld_0$ is as in \eqref{index}, 
we have 
\begin{align}
\E \Big[ \ft {\dot \ZZ}_{N,1} (n) \cj{\III(n)}\Big]
= \E \Big[ \ft {\dot \ZZ}_{N,1} (n)\Big] \cdot \E\Big[ \cj{\III(n)}\Big]= 0
\label{Kt6}
\end{align}

\noi
for any $n \in \Z^3$.
We also have 
\begin{align}
\E \Big[ \ft {\dot \ZZ}_{N,1} (n) \cj{\1(n)}\Big]
= 0
\label{Kt7}
\end{align}

\noi
for any $n \in \Z^3$
since only one of the frequencies is larger than $N$ in size.
Noting that there are exactly two frequencies
larger than $N$, 
 we have
\begin{align}
\E \Big[ \ft {\dot \ZZ}_{N,1} (n) \cj{\II(n)}\Big]
= 0
\label{Kt8}
\end{align}

\noi
for any $n \in \Z^3$
since, under the condition $|n_2+n_3| |n_3+n_1| \ne 0$ in $\II$, 
the only possible non-zero contribution $\II$  comes from 
$|n_1|, |n_2| > N$
with $n_1 + n_2 = 0$ in $\II$
but $\ft V_0(n_1 + n_2)= \ft V_0(0) = 0$ in this case.

The desired bound \eqref{Ks7} then follows from 
\eqref{Kt5}, \eqref{Kt6}, \eqref{Kt7}, \eqref{Kt8},
and integrating in time.
\end{proof}

\section{Paracontrolled operators}
\label{SEC:po}

In this section, we study the mapping properties
of 
the paracontrolled operators
$ \If_{\pl}^{(1)}$
and $\If_{\pl, \pe}$
defined in \eqref{X3} and \eqref{X6},
respectively.
Then, we briefly discuss the regularity property of
the stochastic term $\Ab$ defined in \eqref{sto1} at the end of this section.

We first consider the regularity property of the paracontrolled operator
$\If_{\pl}^{(1)}$ defined in~\eqref{X3}.
By writing out the frequency relation  $|n_2|^\theta \les |n_1| \ll |n_2|$
in a more precise manner, we have
\begin{align}
\begin{split}
 \If_{\pl}^{(1)} (w) (t)
 &   =  \sum_{n \in \Z^3}
e_n  \sum_{n =  n_1 +  n_2}
\sum_{ \ta k + c_0 \leq j < k-2}
\varphi_j(n_1) \varphi_k(n_2)  \\
& \hphantom{XXXXX}
\times 
\int_0^t e^{-\frac{t-t'}2} \frac{\sin ((t - t') \jbb{n})}{\jbb{n}} 
\ft w(n_1, t')\,  \ft{\Psi}(n_2, t') dt', 
\end{split}
\label{XX2a}
\end{align}

\noi
where $c_0 \in \R$ is some fixed constant.
In the following, we establish
the mapping property of 
$\If_{\pl}^{(1)}$ in a deterministic manner
by using a pathwise regularity of $\Psi$.

\begin{lemma}\label{LEM:sto3}

Let  $s>0$ and $T > 0$.
Then, given small $\theta > 0$, 
there exists small $\eps= \eps(s, \ta)  > 0$
such that 
the following deterministic estimate holds
for the paracontrolled operator $ \If_{\pl}^{(1)}$ defined in~\eqref{X3}\textup{:}
\begin{align}
\| \If_{\pl}^{(1)}(w) \|_{L^\infty_T H^{\frac12+3\eps}_x}
\les \|w\|_{L^2_T H_x^{s}}
\|\Psi\|_{L^2_T W_x^{-\frac 12 - \eps, \infty}}.
\label{A00}
\end{align}

\noi
In particular, 
 $ \If_{\pl}^{(1)}$ belongs almost surely
to the class
\begin{align}
 \L_3(T) = \L ( L^2([0, T]; H^{s}(\T^3) )
 \, ; \, 
C([0, T]; H^{\frac 12 + 3\eps}(\T^3))).
\notag
\end{align}

\noi
Moreover, by letting
$ \If_{\pl}^{(1), N}$, $N \in \N$, denote the paracontrolled operator
in \eqref{X3} with $\Psi$ replaced by the truncated stochastic convolution $\Psi_N$ in \eqref{so4a}, 
the truncated paracontrolled operator $ \If_{\pl}^{(1), N}$ converges almost surely to $ \If_{\pl}^{(1)}$ 
in $\L_3 (T)$.

\end{lemma}

Lemma \ref{LEM:sto3} follows from  a slight modification of 
the proof of  Lemma 5.1 in \cite{GKO2}.
We present the argument for readers' convenience.

\begin{proof}
Under $|n_2|^\theta \les |n_1| \ll |n_2|$ with $n = n_1 + n_2$,
we have $\jb{n} \sim \jb{n_2}$.
Thus, 
we have
\begin{align}
\jb{n}^{\frac 12+3\eps} \jb{n}^{-1}
\les \jb{n_1}^{\frac{5\eps}{\theta}}\jb{n_2}^{-\frac{1}{2}-2\eps}
\les \jb{n_1}^{s-\eps}\jb{n_2}^{-\frac{1}{2}-2\eps}
\label{A0}
\end{align}

\noi
by choosing $\eps = \eps (s, \theta) > 0$ sufficiently small.

Letting $\ft w_{j}(n_1, t')
= \varphi_j(n_1)\ft w(n_1, t')$
and 
$ \ft{\Psi}_{k}(n_2, t')
=  \varphi_k(n_2)  \ft{\Psi}(n_2, t')$, 
it follows 
from \eqref{XX2a}
and~\eqref{A0} 
with the trivial embedding \eqref{embed} that 
\begin{align*}
\| \If_{\pl}^{(1)}(w)(t) \|_{H^{\frac12+3\eps}}
& \les \int_0^t 
\sum_{j, k =0}^\infty
2^{(s-\eps)j} 2^{(-\frac{1}{2}-2\eps)k}
\bigg\|  \sum_{n =  n_1 +  n_2}
\ft w_{j}(n_1, t')
\ft{ \Psi}_{k}(n_2, t')\bigg\|_{\l^2_n} dt'\\
& \les \int_0^t 
\sum_{j, k =0}^\infty
2^{(s-\eps)j} 2^{(-\frac{1}{2}-2\eps)k}
\| w_{j}(t')\|_{L^{2}_x}
\| \Psi_{k}( t')\|_{L^\infty_x} dt'\\
& \les \|w\|_{L^2_T H_x^{s}}
\|\Psi\|_{L^2_T (B_{\infty, 1}^{-\frac 12 -2 \eps})_x}\\
& \les \|w\|_{L^2_T H_x^{s}}
\|\Psi\|_{L^2_T W_x^{-\frac 12 - \eps, \infty}}
\end{align*}

\noi
for any $t \in [0, T]$, which shows \eqref{A00}.
The continuity in time of $ \If_{\pl}^{(1)}(w)$ and the convergence of $\If_{\pl}^{(1), N}$
follow
from modifying the computation above.
We omit the details.
\end{proof}

Next, we present the proof of Proposition  \ref{PROP:sto4J}
on the  paracontrolled operator
 $\If_{\pl, \pe}$ in \eqref{X6}.
By writing out the frequency relations
more carefully
as in \eqref{XX2a},
we have
\begin{align}
\begin{split}
\If_{\pl, \pe}(w) (t)
&  = \sum_{n \in \Z^3}e_n 
    \int_0^{t} 
\sum_{j = 0}^\infty
\sum_{n_1 \in \Z^3}
\varphi_j(n_1)
\ft w(n_1, t') \A_{n, n_1} (t, t') dt', 
\end{split}
\label{A0a}
\end{align}

\noi
where $\A_{n, n_1} (t, t')$ is given by 
\begin{align}
\begin{split}
\A_{n, n_1} (t, t')
& = \ind_{[0 , t]}(t')
\sum_{\substack{k = 0\\ j \leq \ta k + c_0}}^\infty
\sum_{\substack{\l, m = 0\\|\l-m|\leq 2}}^\infty
\sum_{n - n_1 =  n_2 + n_3} 
 \varphi_k(n_2)  
 \varphi_\l(n_1 + n_2) 
 \varphi_m(n_3) \\
& \hphantom{XXXX}
\times
e^{-\frac{t-t'}2}
\frac{\sin ((t - t') \jbb{n_1+n_2})}{\jbb{n_1+n_2}} 
   \ft{\Psi}(n_2, t')\,   \ft{\Psi}(n_3, t)  .
\end{split}
\label{A0b}
\end{align}

\noi
For simplicity of notations,  however, 
we  use \eqref{X6} and \eqref{X7}
in the following, 
with the understanding that 
the frequency relations
$|n_1| \ll |n_2|^\theta$ and $ |n_1 + n_2|\sim |n_3|$
are indeed characterized by the use of smooth frequency cutoff functions
as in \eqref{A0a} and \eqref{A0b}.

Recall from the definition \eqref{W2}
that $\ft{\Psi}(n_2, t')$ and $\ft{\Psi}(n_3, t)$ in \eqref{X7}
are uncorrelated unless $n_2+ n_3 = 0$,
i.e.~$n = n_1$.
This leads to the following decomposition of $\A_{n, n_1}$:
\begin{align}
\A_{n, n_1} (t, t')
& = \ind_{[0 , t]}(t') \sum_{\substack{n - n_1 =  n_2 + n_3\\ |n_1| \ll |n_2|^\theta \\ |n_1 + n_2|\sim |n_3|}}
e^{-\frac{t-t'}{2}}
\frac{\sin ((t - t') \jbb{n_1+n_2})}{\jbb{n_1+n_2}} \notag\\
& \hphantom{XXXXXXXX}
\times \Big(   \ft{\Psi}(n_2, t') \, \ft{\Psi}(n_3, t) - \ind_{n_2 + n_3 = 0} \cdot \s_{n_2}(t, t') \Big) \notag\\
& \hphantom{X}
+ \ind_{[0 , t]}(t') \cdot \ind_{n=n_1}\cdot \sum_{\substack{n_2 \in \Z^3 \\ |n| \ll |n_2|^\theta}}
e^{-\frac{t-t'}{2}}
\frac{\sin ((t - t')  \jbb{n+n_2})}{\jbb{n+n_2}} 
 \s_{n_2}(t, t') \notag\\
& =: \, 
\A^{(1)}_{n, n_1} (t, t')+    \A^{(2)}_{n, n_1} (t, t').
\label{X9}
\end{align}

\noi
The second term $\A^{(2)}_{n, n_1}$ is a (deterministic) ``counter term''
for  the case $n_2 + n_3 = 0$.
For this second term,  the condition $|n_1 + n_2|\sim |n_3|$
reduces to $|n + n_2|\sim |n_2|$ which is automatically satisfied
under $|n| \ll |n_2|^\theta$ for small $\theta > 0$.

In view of \eqref{sigma2}, 
 the sum in $n_2$ for the second term $\A^{(2)}_{n, n_1}$ is not absolutely convergent.
Hence, we need to exploit the dispersive nature of the problem.
Proceeding as in \cite{GKO2}
with~\eqref{sigma2} and~\eqref{X9}, 
we decompose $ \A^{(2)}_{n, n} (t, t')$ as
\begin{align}
 \A^{(2)}_{n, n} (t, t')
&  =  \ind_{[0 , t]}(t')\cdot e^{-(t-t')} \sum_{\substack{n_2 \in \Z^3 \\ |n| \ll |n_2|^\theta}}
 \frac{\sin ((t - t') \jbb{n+n_2})}{\jbb{n+n_2}} 
   \frac{\cos((t - t') \jbb{n_2}) }{\jb{n_2}^2}  \notag\\
& \hphantom{X}
+ \ind_{[0 , t]}(t')\cdot e^{-(t-t')}
 \sum_{\substack{n_2 \in \Z^3 \\ |n| \ll |n_2|^\theta}}
 \frac{\sin ((t - t') \jbb{n+n_2})}{\jbb{n+n_2} } 
 \frac{\sin ((t-t') \jbb{n_2})}{2 \jb{n_2}^2 \jbb{n_2}}
\notag   \\
& = \ind_{[0 , t]}(t')\cdot  e^{-(t-t')}
\sum_{\substack{n_2 \in \Z^3 \\ |n| \ll |n_2|^\theta}}
\frac{\sin ((t - t')(\jbb{n+n_2} + \jbb{n_2}))}{2 \jbb{n+n_2} \jb{n_2}^2}  \notag\\
& \hphantom{X}
+ \ind_{[0 , t]}(t')\cdot e^{-(t-t')} \sum_{\substack{n_2 \in \Z^3 \\ |n| \ll |n_2|^\theta}}
 \frac{\sin ((t - t')( \jbb{n+n_2} - \jbb{n_2}))}{2 \jbb{n+n_2} \jb{n_2}^2}  \notag\\
& \hphantom{X}
+ \ind_{[0 , t]}(t')\cdot e^{-(t-t')} \sum_{\substack{n_2 \in \Z^3 \\ |n| \ll |n_2|^\theta}}
 \frac{\sin ((t - t') \jbb{n + n_2 })}{\jbb{ n + n_2} } 
 \frac{\sin ((t-t') \jbb{n_2})}{2 \jb{n_2}^2 \jbb{n_2}} \notag\\
& =:
 \A^{(3)}_{n} (t, t')
 +  \A^{(4)}_{n} (t, t')
 +  \A^{(5)}_{n} (t, t').
 \label{A4a}
\end{align}
We denote the contribution to $\If_{\pl, \pe}(w)$
from $\ind_{n = n_1} \cdot \A_{n}^{(j)}$ by $\If_{\pl, \pe}^{(j)}(w)$ for $j=3,4,5$:
\begin{align}
\If_{\pl, \pe}^{(j)} (w) (t)
:= \sum_{n \in \Z^3} e_n 
\int_0^{t}
\ft w(n, t') \A_{n}^{(j)} (t, t') dt'.
\label{X62}
\end{align}

\noi
The analysis for $j = 4, 5$ is analogous to that in~\cite{GKO2}.
As for the $j = 3$ case, 
while the  argument in~\cite{GKO2}
relied  on the time differentiability of the input function $w$, 
we 
present an argument
without using 
 the time differentiability of  $w$.

We now present 
the proof of Proposition \ref{PROP:sto4J}.
Part of the argument follows 
 closely 
 the proof of 
Proposition 1.11 in \cite{GKO2}.

\begin{proof}[Proof of Proposition \ref{PROP:sto4J}]

In the following, we only consider the case $0 < T \leq 1$.
The required modification 
for handling the case  $T > 1$ is straightforward.
As for the random operator $\If_{\pl, \pe}^{(1)}$, 
we carry out the stochastic analysis presented below
on each subinterval $[k, \max(k+1, T)] \subset [0, T]$, 
which gives the extra factor of $T$ in \eqref{pote2}.
As for the deterministic operators $\If_{\pl, \pe}^{(j)}$, $j = 3, 4, 5$,  
the analysis remains essentially the same
even when $T > 1$.

Fix finite $q > 1$ and let $q'$ be its H\"older conjugate.
First, 
we consider the contribution to $\If_{\pl, \pe}$
from $\A^{(1)}_{n, n_1}$ in \eqref{X9} and denote it 
by $\If_{\pl, \pe}^{(1)}$.
Then, 
from \eqref{X6}  and \eqref{X9}, we have
\begin{align}
\begin{split}
\| \If_{\pl, \pe}^{(1)}(w)(t) \|_{ H^{s_3}_x}
& \le \bigg\|\int_0^t 
\jb{n}^{s_3}\sum_{n_1 \in \Z^3}
\ft w(n_1, t') \A_{n, n_1}^{(1)}(t, t') dt'
\bigg\|_{ \l^2_n}\\
& \les  \|w\|_{L^q_T L^2_x}
\|\jb{n}^{s_3}
\A_{n, n_1}^{(1)} (t, t')
\|_{ L^{q'}_{t'}([0, T]; \l^2_{n, n_1})}.
\end{split}
\label{Aop0}
\end{align}

Note that the conditions $|n_1| \ll |n_2|^\theta$ for some small $\theta > 0$
and $|n_1 + n_2|\sim |n_3|$ imply $|n_2| \sim |n_3| \gg |n_1|$.
Moreover, with the condition $n - n_1 =  n_2 + n_3$, 
we have $|n_2| \sim |n_3| \ges |n|$.
Then, from \eqref{X9}, \eqref{Wickz}, and \eqref{sigma2}, 
we have 
\begin{align}
\begin{split}
 \E & \Big[  \|\jb{n}^{s_3} \A_{n, n_1}^{(1)}  (t, t')\|_{\l^2_{n, n_1}}^2\Big]\\
& \le 
\sum_{n, n_1} \jb{n}^{2s_3}
\E
\Bigg[\sum_{\substack{n - n_1 =  n_2 + n_3\\ |n_1| \ll |n_2|^\theta \\ |n_1 + n_2|\sim |n_3|}}
\sum_{\substack{n - n_1 =  n_2' + n_3'\\ |n_1| \ll |n_2'|^\theta \\ |n_1 + n_2'|\sim |n_3'|}}
\frac{\sin((t - t')\jbb{n_1 + n_2})}{\jbb{n_1+n_2}} 
\frac{\sin((t - t')\jbb{n_1 + n'_2})}{\jbb{n_1+n'_2}}  \\
& \quad  \times
\Big(   \ft{\Psi}(n_2, t') \, \ft{\Psi}(n_3, t) - \ind_{n_2+n_3=0} \cdot \s_{n_2}(t,t') \Big)\\
& \quad  \times
\cj{\Big(   \ft{\Psi}(n_2', t') \, \ft{\Psi}(n_3', t) - \ind_{n_2'+n_3'=0} \cdot \s_{n_2'}(t,t') \Big)}
 \Bigg] \\
& \les
\sum_{n, n_1} \jb{n}^{2s_3}
\sum_{\substack{N_2 \geq 1\\ \text{dyadic}} }
\sum_{\substack{n - n_1 =  n_2 + n_3\\ |n_1| \ll |n_2|^\theta \\ |n_1 + n_2|\sim |n_3|\\
|n_2|\sim N_2}}
\frac{1}{\jb{n_2}^4 \jb{n_3}^{2}}
\\
& \les
\sum_{\substack{N_2 \geq 1\\ \text{dyadic}} }
N_2^{-3} 
\sum_{n, n_1} \jb{n}^{2s_3}
\ind_{|n_1|\ll N_2^\theta} \ind_{ |n| \les  N_2}\\
& \les
\sum_{\substack{N_2 \geq 1\\ \text{dyadic}} }
N_2^{2s_3 + 3\ta} 
\les 1, 
\end{split}
\label{Aop3}
\end{align}

\noi
uniformly in $ 0 \leq t' \leq t \leq T$,
provided that
$2s_3 + 3\ta < 0$, 
where, at the second inequality, 
we used the fact that 
non-zero contribution appears only when 
$n_2 = n_2'$ or $n_2 = n_3'$.
Hence, from Minkowski's integral inequality, Lemma \ref{LEM:hyp}, 
and \eqref{Aop3}, we conclude that 
\begin{align}
\Big\|\|\jb{n}^{s_3}
\A_{n, n_1}^{(1)} (t, t')
\|_{  L^{q'}_{t'}([0, T]; \l^2_{n, n_1})}\Big\|_{L^p(\O)}
\les T^{\frac 1{q'}} p
\label{Aop4}
\end{align}

\noi
for any finite $p \geq 2$ and $t \in [0, T]$.
A similar argument yields the following difference estimate;
there exists small $\s_0 > 0$ such that 
\begin{align}
\begin{split}
\Big\|
& \|\jb{n}^{s_3}
\A_{n, n_1}^{(1)} (t_1, t')
\|_{  L^{q'}_{t'}([0, T]; \l^2_{n, n_1})}
- 
\|\jb{n}^{s_3}
\A_{n, n_1}^{(1)} (t_2, t')
\|_{  L^{q'}_{t'}([0, T]; \l^2_{n, n_1})}
\Big\|_{L^p(\O)}\\
& \les T^{\frac 1{q'}} p
|t_1 - t_2|^{\s_0 }
\end{split}
\label{Aop4a}
\end{align}

\noi
for any finite $p \geq 2$ and  $t_1, t_2 \in [0, T]$.
See, for example, the proof of Proposition 1.11 in \cite{GKO2}.
By Kolmogorov's continuity criterion, we conclude that 
\begin{align}
 \|\jb{n}^{s_3}
\A_{n, n_1}^{(1)} (\cdot , t')
\|_{  L^{q'}_{t'}([0, T]; \l^2_{n, n_1})}\in L^\infty([0, T]).
\label{Aop4b}
\end{align}

\noi
The desired mapping property then follows from \eqref{Aop0} and \eqref{Aop4b}.
The tail estimate  \eqref{pote2} for $\If_{\pl, \pe}^{(1)}$
 follows from 
\eqref{Aop4}, 
\eqref{Aop4a}, and 
the Garsia-Rodemich-Rumsey inequality (Lemma \ref{LEM:GRR})
as in the proof of
Lemma \ref{LEM:stoconv}.

Next, we consider 
$\If_{\pl, \pe}^{(3)}$ 
defined in \eqref{X62}.
This is a deterministic operator
with the kernel given by 
 $\A^{(3)}_{n}(t,t')$   in \eqref{A4a}.
 Hence, once we show its boundedness, 
 the tail estimate \eqref{pote2} is automatically satisfied.
The same comment applies to 
$\If_{\pl, \pe}^{(4)}$ 
and $\If_{\pl, \pe}^{(5)}$
 studied below.
In this case, we show
\begin{align}
\If_{\pl, \pe}^{(3)}\in 
 \L(L^q([0, T]; L^2(\T^3) )
\, ;\, 
 L^\infty([0, T]; L^2(\T^3)))
\label{A2a1}
\end{align}

\noi
for any $q > 1$.  In the following, we only consider $1 < q\leq 2$.

Define $\K_n$ by 
\begin{align}
\K_n(t) & = \ind_{[0, 1]}(t)\cdot e^{-t}
\sum_{\substack{n_2 \in \Z^3 \\ |n| \ll |n_2|^\theta}}
\frac{\sin (t(\jbb{n+n_2} + \jbb{n_2}))}{2 \jbb{n+n_2} \jb{n_2}^2}.
\label{A2a}
\end{align}

\noi
Then, 
from \eqref{A4a},
we have
\begin{align}
\A_{n}^{(3)} (t, t') = \ind_{[0, t]}(t') \cdot \K_n(t - t')
\label{A22}
\end{align}

\noi
for $0 \leq t \le 1$.
Thus, we have 
\begin{align}
\begin{split}
\If_{\pl, \pe}^{(3)} (w) (t)
& = \ind_{[0, T]}(t)\cdot \sum_{n \in \Z^3} e_n 
\int_0^{t}
\big(\ind_{[0, T]}(t')\cdot \ft w(n, t')\big)
 \K_n(t - t') dt'\\
& = \ind_{[0, T]}(t)
\cdot \sum_{n \in \Z^3} e_n 
\big(\ind_{[0, T]}\cdot \ft w(n, \cdot)\big)
*_t \K_n
\end{split}
\label{A2aa}
\end{align}

\noi
for $0 \leq t \le T \le 1$.

From \eqref{A2a}, we have 
\begin{align*}
\ft \K_n(\tau) 
& = \frac{1}{\sqrt{2\pi}}\int_0^1 \K_n(t) e^{- i t\tau} dt \\
&= \frac{1}{4i \sqrt{2\pi}}\sum_{\s \in \{1, -1\}}\sum_{\substack{n_2 \in \Z^3 \\ |n| \ll |n_2|^\ta}} 
\frac{1}{ \jbb{n+n_2} \jb{n_2}^2} \\
& \hphantom{XXX}
\times \frac{\exp \big(i (\s(\jbb{n+n_2}+ \jbb{n_2}) -  \tau) - 1\big) - 1}{
i(\s(\jbb{n+n_2}+ \jbb{n_2}) -  \tau) - 1}.
\end{align*}

\noi
In the following, we only bound the contribution from $\s = 1$.
The contribution from $\s = -1$ can be estimated in an analogous manner.
Let 
\[\phi_{n, \tau} (n_2) := \big|\jbb{n+n_2}+ \jbb{n_2} -  \tau\big|.\] 

\noi
Then, 
for $M \ge 4$ dyadic and $\tau\geq 1$, we have 

\smallskip

\begin{itemize}
\item If $M \ll \tau$, then $\#\{n_2 \in \Z^3: \phi_{n, \tau}(n_2) \sim M\} \les M\tau^2$.
In this case, we have  $|n_2| \sim \tau$.

\smallskip
\item If $M \sim \tau$, then $\#\{n_2 \in \Z^3: \phi_{n, \tau}(n_2) \sim M\} \les \tau^3$.
In this case, we have  $|n_2| \les \tau$.

\smallskip

\item If $M \gg \tau$, then $\#\{n_2 \in \Z^3: \phi_{n, \tau}(n_2) \sim M\} \les M^3$.
In this case, we have $\phi_{n, \tau}(n_2) \sim |n_2| \sim M \gg \tau$.
\end{itemize}

\smallskip
\noi
Hence, we have 
\begin{align}
\begin{split}
|\ft \K_n(\tau) | 
&\les \sum_{\substack{n_2\in \Z^3\\\phi_{n, \tau}(n_2) \le 4}} \frac 1 {\jb{n_2}^3} 
+ \sum_{\substack{M \ll \tau\\\text{dyadic}} }
\sum_{\substack{n_2\in \Z^3\\\phi_{n, \tau}(n_2) \sim M}} \frac 1 {\jb{\tau}^3}  \frac 1 M \\
& \quad + 
\sum_{\substack{M \sim \tau\\\text{dyadic}}} \sum_{\substack{n_2\in \Z^3\\\phi_{n, \tau}(n_2) \sim M} }
\frac 1 {\jb{n_2}^3} \frac 1 {\jb{\tau} }
+ \sum_{\substack{M \gg \tau\\\text{dyadic}}} \sum_{\substack{n_2\in \Z^3\\\phi_{n, \tau}(n_2) \sim M} }\frac 1 {M^4}\\
&\les \frac {1}{\jb{\tau}}  + \frac {\log \jb{\tau}}{\jb{\tau}} + \frac {\log \jb{\tau}}{\jb{\tau}} + \frac 1 {\jb{\tau}} \\
&\les \frac{\log \jb{\tau}}{\jb{\tau}}, 
\end{split}
\label{A2c}
\end{align}

\noi
uniformly in $n$ with $|n|\ll |n_2|^\ta$, 
when $\tau \ge 1$.
When $\tau \le 1$, 
we have 
$\phi_{n, \tau} (n_2) \ges \jb{n_2}\gg 1$
and thus 
\begin{align}
|\ft \K_n(\tau) | 
\les \sum_{n_2\in \Z^3}
\frac 1 {\jb{n_2}^3\max(\jb{n_2}, \jb{\tau})} \les \frac {\log \jb{\tau}}{\jb{\tau}}, 
\label{A2d}
\end{align}

\noi
uniformly in $n \in \Z^3 $ with $|n|\ll |n_2|^\ta$.
From \eqref{A2c} and \eqref{A2d}, we conclude that 
$\ft \K_n \in L^q(\R)$ for any $q>1$.
Then, by  Hausdorff-Young's inequality, we obtain
\begin{align}
\big\|\big(\ind_{[0, T]}\cdot \ft w(n, \cdot)\big)
*_t \K_n\big\|_{L^\infty_T}
\leq  \| \ft W_n \ft \K_n\|_{L^1_\tau}
\leq  \| \ft W_n\|_{L^{q'}_\tau} \|\ft \K_n\|_{L^q_\tau}
\les \|  W_n\|_{L^{q}_T} 
\label{A2e}
\end{align}

\noi
for any $1 < q \le 2$, 
uniformly in $n \in \Z^3$, 
where 
$W_n = \ind_{[0, T]}\cdot \ft w(n)$.
Therefore, from \eqref{A2aa}, \eqref{A2e}, and Minkowski's integral inequality, we
conclude that 
\begin{align*}
\|\If_{\pl, \pe}^{(3)} (w) \|_{L^\infty_T L^2_{x}}
\leq 
\Big\|\big\|\big(\ind_{[0, T]}\cdot \ft w(n, \cdot)\big)
*_t \K_n\|_{L^\infty_T} \Big\|_{\l^2_n}
\les \| \ft w(n, \cdot) \|_{\l^2_n L^q_T }
\leq \| w\|_{L^q_T L^2_x}
\end{align*}

\noi
for any $1 < q \le 2$.
This proves \eqref{A2a1}.

Lastly,  we consider
$\If_{\pl, \pe}^{(4)}$ 
and $\If_{\pl, \pe}^{(5)}$
defined in \eqref{X62}.
These are deterministic operators
with the kernels given by 
 $\A^{(4)}_{n}(t,t')$ and $\A^{(5)}_{n}(t,t')$  in \eqref{A4a}.
 Hence, once we show their boundedness, 
 the tail estimate \eqref{pote2} is automatically satisfied.
For now, we assume  that
\begin{align}
\| \jb{n}^{-2+ \frac 1 \theta} \A^{(j)}_{n} (t, t') \|_{\l^\infty_n} \les 1
\label{A3}
\end{align}

\noi
for any $0 \leq t' \leq t \leq T\le  1$, $j = 4, 5$, 
and show 
\begin{align*}
\If_{\pl, \pe}^{(j)}\in 
\L( L^1([0, T]; L^2(\T^3) )
\, ;\, L^\infty([0, T]; H^{s_3}(\T^3))), \qquad j = 4, 5.
\end{align*}

\noi
From  \eqref{X62} and \eqref{A3}, we have
\begin{align*}
\|\If_{\pl, \pe}^{(j)}(w) \|_{L^\infty_T H^{s_3}_x}
& = 
\bigg\|\int_0^t 
\jb{n}^{s_3}
\ft w(n, t') \A_{n}^{(j)}(t, t') dt'
\bigg\|_{L^\infty_T \l^2_n}\\
& \les    \|w\|_{L^1_T L^2_x} 
\sup_{n\in \Z^3}\jb{n}^{s_3 +2-\frac 1 \theta} \\
&\les   \|w\|_{L^1_T L^2_x},
\end{align*}

\noi
provided that $s_3 \le -2+\frac 1\theta$.
By noting that $\jbb{n+n_2} \sim \jb{n_2}\gg \jb{n}$
under $ |n| \ll |n_2|^\theta$, 
we see that~\eqref{A3} is easily verified for $j = 5$.

 The sum for $\A^{(4)}_{n} (t, t')$ in \eqref{A4a}
is not absolutely convergent.
As in \cite{GKO2}, 
we exploit the symmetry $n_2 \leftrightarrow -n_2$
and the oscillatory nature of the sine kernel.
Set
\begin{align}
\Theta^\pm(n, n_2)
:=
\jbb{n \pm n_2} - \jbb{n_2}
 \mp \frac{\jb{n, n_2}}{\jbb{n_2}} .
\label{Y14}
\end{align}

\noi
Then, 
noting  that $\jbb{n\pm n_2} \sim \jbb{n_2}\gg \jb{n}$
under $ |n| \ll |n_2|^\theta$, 
it follows from 
\eqref{jbbn} and the mean value theorem that 
\begin{align}
\begin{split}
\Theta^\pm(n, n_2)
&= \frac{|n|^2}{\jbb{n\pm n_2} + \jbb{n_2}}  \pm \jb{n, n_2} \frac{\jbb{n_2} - \jbb{n \pm n_2}}{(\jbb{n\pm n_2} + \jbb{n_2}) \jbb{n_2}}  \\
&= O\bigg(\frac{\jb{n}^2}{\jb{n_2}}\bigg).
\end{split}
\label{Y14a}
\end{align}

\noi
Write 
\[ \sum_{n_2 \in \Z^3 \setminus\{0\}} F(n_2) 
= \sum_{n_2 \in  \Ld} \big( F(n_2) + F(-n_2) \big), \]

\noi
where  the index $\Ld$ is as in \eqref{index}.
Then, 
from 
\eqref{A4a},  \eqref{Y14}, the mean value theorem, 
and \eqref{Y14a}, we have
\begin{align}
\begin{split}
& \A^{(4)}_{n} (t, t') \\
& = e^{-(t-t')} \sum_{\substack{n_2 \in \Ld \\ |n| \ll |n_2|^\theta}}
 \frac{\sin ((t - t')( \jbb{n + n_2} - \jbb{n_2}))
 + \sin ((t - t')( \jbb{n - n_2} - \jbb{n_2}) )}{2 \jbb{n + n_2} \jb{n_2}^2}  \\
& \hphantom{X}
- e^{-(t-t')} \sum_{\substack{n_2 \in \Ld \\ |n| \ll |n_2|^\theta }}
 \frac{  \sin ((t - t')( \jbb{n - n_2} - \jbb{n_2}) )}{2 \jb{n_2}^2}
 \bigg(\frac{1}{\jbb{n + n_2}} - \frac{1}{\jbb{n - n_2}}\bigg)
   \\
& = e^{-(t-t')} \sum_{\substack{n_2 \in \Ld \\ |n| \ll |n_2|^\theta}}
 \frac{1}{2 \jbb{n + n_2} \jb{n_2}^2} 
 \bigg\{
\sin \bigg((t - t')\Big(  \frac{\jb{n, n_2}}{\jbb{n_2}} + \Theta^+(n, n_2)\Big)\bigg) \\
& \hphantom{XXXXX}
- \sin \bigg((t - t')\Big(  \frac{\jb{n, n_2}}{\jbb{n_2}} - \Theta^-(n, n_2)\Big)\bigg)
 \bigg\} 
 \\
& \hphantom{X}
 +   O\bigg(  \sum_{\substack{n_2 \in \Ld \\ |n| \ll |n_2|^\theta}}
 \frac{ \jb{n}}{\jb{n_2}^4}\bigg) \\
& \les
 \sum_{\substack{n_2 \in \Ld \\ |n| \ll |n_2|^\theta}}
\frac{1}{\jbb{ n + n_2} \jb{n_2}^2} 
\frac{\jb{n}^{2}}{\jb{n_2}}
+ \sum_{\substack{n_2 \in \Ld \\ |n| \ll |n_2|^\theta}}
 \frac{ \jb{n}}{\jb{n_2}^4} \\
& \les
\jb{n}^{2 - \frac 1 \theta}
\end{split}
\label{A5a}
\end{align}

\noi
for any $0 \leq t' \le t \le 1$ and $0<\theta \le1$.
This proves \eqref{A3} for $j = 4$.

This completes the proof of Proposition \ref{PROP:sto4J}.
\end{proof}

We conclude this section by briefly discussing the regularity property 
of the stochastic term $\Ab$ defined in \eqref{sto1}.
For this purpose, we first define its truncated version:
\begin{align}
\Ab_N(x, t, t') = \sum_{n \in \Z^3} e_n(x) \sum_{\substack{n = n_1 + n_2\\|n_1|\sim| n_2|}}
 e^{-\frac{t-t'}{2}} \frac{\sin ( (t - t') \jbb{n_1})}{\jbb{n_1}}
 \ft \Psi_N(n_1, t') \ft \Psi_N (n_2, t).
\label{sto3}
\end{align}

\begin{lemma}\label{LEM:sto1}

Let $\Ab_N(t, t')$ be as in \eqref{sto3}.
Fix finite $q \geq 2$.
Then, 
given any  $T,\eps>0$ and finite $p \geq 1$, 
 $\{ \Ab_N \}_{N\in \N}$ is a Cauchy sequence in $L^p(\O;L^\infty_{t'}L^q_t(\Dl_2(T); H^{-\eps}(\T^3)))$, 
 converging to some limit $\Ab$ \textup{(}formally defined by \eqref{sto1}\textup{)}
 in $L^p(\O;L^\infty_{t'}L^q_t(\Dl_2(T);H^{-\eps}(\T^3)))$, 
 where $\Dl_2(T)$ is as in \eqref{sto2}.
Moreover,  $\Ab_N$  converges almost surely to the same  limit in $L^\infty_{t'}L^q_t(\Dl_2(T); H^{ -\eps}(\T^3))$.
Furthermore, 
 we   have 
the following uniform tail estimate\textup{:}
\begin{align}
\PP\Big( \| \Ab_N\|_{L^\infty_{t'}L^q_t(\Dl_2(T);  H^{-\eps}_x)} > \ld\Big) 
\leq 
\begin{cases}
C\exp\big(-c \frac{\ld}{T^{ \frac{1}{q}}}\big),  & \text{when }0<T\le1, \\
C T \exp ( -  \ld),  & \text{when }  T > 1
\end{cases}
\label{sto4}
\end{align}

\noi
for any  $\ld \gg 1$, and $N \in \N \cup \{\infty\}$, 
where $\Ab_\infty = \Ab$. 

\end{lemma}

\begin{proof}
As in the proof of Proposition \ref{PROP:sto4J}, 
 we only consider the case $0 < T \leq 1$.
In the following, we simply study the regularity of $\Ab$, i.e.~when $N = \infty$.
The claimed convergence and the tail estimate \eqref{sto4}
follow from a standard argument and the fact that $\Ab_N \in \H_{\le 2}$, $N \in \N \cup\{\infty\}$.
By comparing \eqref{sto1} with \eqref{X7}, we have
\begin{align*}
\ft \Ab(n, t, t') = \A_{n, 0}(t, t')
\end{align*}

\noi
for $(t, t') \in \Dl_2(T)$.
Thus, from 
\eqref{X9}
and 
\eqref{A4a}, we can write
\begin{align*}
\ft \Ab(n, t, t') = \ft \Ab^{(1)}(n, t, t')+
\ind_{n = 0}\cdot \Big(
\ft \Ab^{(3)}(0, t, t')+\ft \Ab^{(4)}(0, t, t')+\ft \Ab^{(5)}(0, t, t')\Big),
\end{align*}

\noi
where 
$\ft \Ab^{(1)}(n, t, t') = \A_{n, 0}^{(1)}(t, t')$
and
$\ft \Ab^{(j)}(n, t, t') = \A_{n}^{(j)}(t, t')$, $j = 3, 4, 5$.

From \eqref{Aop3}, 
we have 
\begin{align*}
\E\Big[ \| \ft \Ab^{(1)}(n, t, t') \|_{L^q_t([0, T])}^2 \Big] \les \jb{n}^{-3}T^{\frac 2q}
\end{align*}

\noi
for $n \in \Z^3$ and $0 \le t' \le t \le T$.
Note that, in \eqref{sto1}, $t'$ appears
in $\sin ( (t - t') \jbb{n_1})$
and $\Psi(n_1, t')$.
Then, proceeding as in the proof of Lemma 3.1 in \cite{GKO2}, we obtain 
\begin{align*}
\E\Big[ \| \ft \Ab^{(1)}(n, t, t'_1) - \ft \Ab^{(1)}(n, t, t'_2) \|_{L^q_t([0, T])}^2 \Big] \les 
|t_1' - t_2'|^{\s_0}\jb{n}^{-3+\s_0}T^{\frac 2q }
\end{align*}

\noi
for some small $\s_0 > 0$.
Then, by
 (a variant of) Lemma \ref{LEM:reg}, we 
conclude that $\Ab^{(1)} \in 
L^\infty_{t'}L^q_t(\Dl_2(T);H^{-\eps}(\T^3)))$
almost surely.
The exponential tail estimate \eqref{sto4} for 
$\Ab^{(1)} $
follows from adapting the proof of Lemma \ref{LEM:stoconv}, using Lemma \ref{LEM:GRR}.

It remains to estimate the deterministic terms
$\Ab^{(j)}$, $j = 3, 4, 5$, which appear only at the zeroth frequency.
Let $\phi$ be a smooth bump function in Section \ref{SEC:2}
and set $\phi^T(t) = \phi(T^{-1} t)$. 
Then, from \eqref{A22}, \eqref{A2c},  \eqref{A2d}, 
Hausdorff-Young's inequality, 
and Young's inequality,  we have
\begin{align*}
 \|  \ft \Ab^{(3)}(0, t, t') \|_{L^\infty_{t'}L^q_t(\Dl_2(T))}
& \le
\| \K_0\|_{L^q_T} 
\le 
\|\ft {\phi^T  \K_0} \|_{L^{q'}_\tau}
\le  \| \ft {\phi^T}\|_{L^\frac{q'}{1+q'\eps}_\tau} \|\ft \K_0  \|_{L^\frac{1}{1-\eps}_\tau}
 \les T^{\frac 1q-\eps}
\end{align*}

\noi
for small $\eps > 0$.
From \eqref{A5a}, we have
\begin{align}
 \|  \ft \Ab^{(4)}(0, t, t') \|_{L^\infty_{t'}L^q_t(\Dl_2(T))}
\les T^\frac{1}{q}.
\label{sto5}
\end{align}

\noi
In view of 
a faster decay in $n_2$ for $j = 5$ in \eqref{A4a}, 
the estimate \eqref{sto5} trivially holds
for $\ft \Ab^{(5)}$.
\end{proof}

\section{Local well-posedness of Hartree SdNLW}
\label{SEC:LWP}

In this section,
we present the proofs of 
Theorems \ref{THM:1} and \ref{THM:2}
on  local well-posedness
of the renormalized Hartree SdNLW
systems \eqref{SNLW5} and \eqref{SNLW6}
in the defocusing case and the focusing case, respectively.

\medskip

\noi
$\bullet$ {\bf Defocusing case for $\be > 1$.}
We first treat the defocusing case \eqref{SNLW5}.
By writing \eqref{SNLW5} in the Duhamel formulation (for $X$ and $Y$), we have
\begin{align}
\begin{split}
X &= \Phi_1(X, Y,  \Res)\\
:\! & = 
S(t)(X_0, X_1) - \I \Big( \big(V \ast 
( \Qxy + 2\Res \, + \! :\! \Psi^2 \!:) \big) \pl \Psi\Big),\\
Y &= \Phi_2(X, Y,  \Res)\\
:\! & = 
S(t)(Y_0, Y_1) 
- \I \Big( \big(V \ast 
( \Qxy + 2\Res \, + \! :\! \Psi^2 \!:) \big)  (X+Y)\Big)\\
& \hphantom{X}
- \I \Big( \big(V \ast 
( \Qxy + 2\Res \, + \! :\! \Psi^2 \!:) \big) \pge \Psi\Big),\\
\Res
&= \Phi_3(X, Y,  \Res)\\ 
:\! &= - \If_{\pl}^{(1)}
 \big(V \ast (\Qxy  + 2\Res + :\! \Psi^2 \!:\,)\big)\pe \Psi\\
& \hphantom{X}
 - \If_{\pl, \pe}
 \big(V \ast (\Qxy  + 2\Res + :\! \Psi^2 \!:\,)\big), 
\end{split}
\label{SNLW7}
\end{align}

\noi
where 
\begin{align}
S(t) (f, g) = \dt\D(t)f +  \D(t) (f + g)
\label{lin1}
\end{align}

\noi
and $\D(t)$ is as in \eqref{D2a}.
In the following, we assume that $-\frac 12 < s_3 < 0 < s_1 < \frac 12 < s_2 < 1$.
Given $0 < T\le 1$, 
let $Z^{s_1, s_2, s_3}(T)$ be as in \eqref{Z1}.
Given an enhanced data set $\Xi$ as in~\eqref{data1}, 
we set
\begin{align*}
\Xi(\Psi)  = \big( \Psi, \,  :\! \Psi^2 \!:,  \,  (V \ast :\! \Psi^2 \!:) \pe \Psi,
\,  \If_{\pl, \pe}\big)
\end{align*}

\noi
and 
\begin{align}
\begin{split}
\| \Xi (\Psi)  \|_{\mathcal{X}^{ \eps}_T}
&: =
\| \Psi \|_{C_T  W^{-\frac 12-\eps,\infty}_x\cap \, C^1_T  W^{-\frac 32-\eps,\infty}_x}
+ \| :\! \Psi^2 \!: \|_{C_T  W^{-1-\eps,\infty}_x} \\
&\hphantom{X}
+ \| ( V \ast :\! \Psi^2 \!:) \pe \Psi \|_{C_T W^{ \beta - \frac 32 -\eps, \infty}_x} 
+  \| \If_{\pl, \pe} \|_{ \L_2(\frac 32, T)} 
\label{data3}
\end{split}
\end{align}

\noi
for some small $\eps = \eps( \be, s_1, s_2, s_3)> 0$.
We assume
\begin{align}
\| \Xi (\Psi)  \|_{\mathcal{X}^{ \eps}_1}
\leq K
\label{data4}
\end{align}

\noi
for some $K \geq 1$.

\begin{remark}\label{REM:conv0}\rm
As for proving  the local well-posedness result stated in Theorem \ref{THM:1}, 
we do not need to use the $C^1_T  W^{-\frac 32-\eps,\infty}_x$-norm
of the stochastic convolution $\Psi$.
However, this norm is needed for constructing global-in-time dynamics
and thus we have included it in the definition  of 
the $\mathcal{X}^{ \eps}_T$-norm in \eqref{data3}.
The same comment applies to 
the first component of the ${\mathcal{Y}^{ \eps}_T}$-norm defined in~\eqref{data3a}.
\end{remark}

We first establish  preliminary estimates.
By Sobolev's inequality, we have
\begin{align}
\| f^2 \|_{H^{-a}}
\les \| f^2 \|_{L^{\frac 6{3+2a}}}
= \| f \|_{L^{\frac{12}{3+2a}}}^2
\les \| f \|_{H^{\frac{3-2a}4}}^2
\label{M1-1}
\end{align}
for any $0\leq a< \frac  32$.
By  \eqref{Bessel1}, \eqref{Pxy}, \eqref{M1-1}, Lemma \ref{LEM:para}, 
Lemma \ref{LEM:gko}, and H\"older's inequality
with~\eqref{data4}, 
we have
\begin{align}
\begin{split}
\| V\ast \Qxy \|_{L^\infty_{T} H^{-s_1+s_2+\frac 12}_x}
& \les \| (X+Y)^2 \|_{L^\infty_{T} H^{-\be-s_1+s_2+\frac 12}_x}
+ \| X \pl \Psi \|_{L^\infty_{T} H^{-\be-s_1+s_2+\frac 12}_x} \\
&\quad 
  + \| X \pg \Psi \|_{L^\infty_{T} H^{-\be-s_1+s_2+\frac 12}_x}
+ \| Y \Psi \|_{L^\infty_{T} H^{-\be-s_1+s_2+\frac 12}_x} \\
%
&\les   \Big(
\| X \|_{L^\infty_{T} H^{\max (1 - \frac{s_1-s_2+\be}2, \eps)}_x}^2
+ \| Y \|_{L^\infty_{T} H^{\max (1 - \frac{s_1-s_2+\be}2, \eps)}_x}^2 \Big) \\
&\quad +
 \Big( \| X \|_{L^\infty_T L^2_x}
+ \| Y \|_{L^\infty_{T} H^{\frac 12+  \eps}_x} \Big)
\| \Psi \|_{L^\infty_{T} W^{-\frac 12-\eps,\infty}_x}
\\
&\les
 \| X \|_{X^{s_1}(T)}^2 + \| Y \|_{X^{s_2}(T)}^2 + K^2 , 
\end{split}
\label{M1-3}
\end{align}

\noi
provided that $\b\ge  \max( -s_1+s_2+1 + \eps, 
-3s_1 + s_2 + 2)$,  $s_1\ge \eps $, and $s_2\ge \frac 12 +2\eps$.
When $\be > 1$, these conditions are satisfied 
for  $0<s_1<\frac 12<s_2$ such that $s_2 - s_1 >0$ is sufficiently close to $0$.
We also record the following estimate, 
which follows from 
 Sobolev's and H\"older's inequality:
\begin{align}
\| fg \|_{H^{s_2-1}}
\les \| fg \|_{L^{\frac 6{5-2s_2}}}
\les \| f \|_{L^{\frac 6{3-2s_1}}} \| g \|_{L^{\frac 3{1+s_1-s_2}}}
\les \| f \|_{H^{s_1}} \| g \|_{H^{-s_1+s_2+\frac 12}}
\label{M1-2}
\end{align}

\noi
for any $0 \le s_1<s_2\le 1 $.
Lastly, we recall the energy estimate:
\begin{align}
\bigg\|\int_0^t \D(t-t') F (t') dt'\bigg\|_{X^s(T)}
\les \| F\|_{L^1_T  H^{s-1}_x}.
\label{EE1}
\end{align}

We now estimate $\Phi_1 (X, Y, \Res)$  in \eqref{SNLW7}.
By the energy estimate \eqref{EE1}, Lemma \ref{LEM:para},  and \eqref{M1-3}
with~\eqref{data4},  we have
\begin{align}
\begin{split}
\| \Phi_1 &  (X, Y, \Res) \|_{X^{s_1}(T)} \\
&\les \| (X_0, X_1)\|_{\H^{s_1}}
+ 
\big\|  \big(V \ast (\Qxy + 2\Res + :\! \Psi^2 \!:\,)\big) \pl \Psi \big\|_{L^1_{T} H^{s_1-1}_x}
\\
&\les
\| (X_0, X_1)\|_{\H^{s_1}}
+ T^{\frac 23} 
 \| V\ast (\Qxy + 2\Res + :\! \Psi^2 \!:\,) \|_{L^3_{T}L^2_x}
\| \Psi \|_{L^\infty_{T} W_x^{-\frac 12 -\eps, \infty}}
\\
&\les
\| (X_0, X_1)\|_{\H^{s_1}}
+ 
T^{\frac 23}K  \Big( \|(X, Y, \Res)\|_{Z^{s_1, s_2, s_3}(T)}^2 +K^2 \Big), 
\end{split}
\label{M1a}
\end{align}

\noi
provided that $\beta \ge  \max (-s_3, 1+\eps)$
and $s_1 < \frac 12-\eps$.

Next,  we estimate  $\Phi_2 (X, Y, \Res)$  in \eqref{SNLW7}.
By  the energy estimate \eqref{EE1}, \eqref{M1-2}, 
Lemma \ref{LEM:para}, and \eqref{M1-3}
with \eqref{data4},   we have
\begin{align}
\begin{split}
\| \Phi_2 & (X, Y, \Res) \|_{X^{s_2}(T)}  \\
&\les
\| (Y_0, Y_1)\|_{\H^{s_2}}
+ 
\big\| \big(V \ast (\Qxy + 2\Res + :\! \Psi^2 \!:\,)\big) (X+Y) \big\|_{L^1_T H^{s_2-1}_x}\\
&\quad
+ \big\| \big(V \ast (\Qxy + 2\Res + :\! \Psi^2 \!:\,)\big) \pge \Psi \big\|_{L^1_T H^{s_2-1}_x}
\\
&\les
\| (Y_0, Y_1)\|_{\H^{s_2}}\\
& \quad
+ 
T^{\frac 23}  \| V \ast (\Qxy +2 \Res) \|_{L^3_T H^{-s_1+s_2+\frac 12}_x}
 \Big( \| X \|_{L^\infty_T H^{s_1}_x} + \| Y \|_{L^\infty_T H^{s_2}_x} \Big) \\
&\quad
+  T \| V \ast :\! \Psi^2 \!: \|_{L^\infty_TL^\infty_x}
\Big( \| X \|_{L^\infty_T L^2_x} + \| Y \|_{L^\infty_T L^2_x} \Big) \\
&\quad
+ T^{\frac 23}  \| V \ast (\Qxy + 2\Res) \|_{L^3_T H^{\frac 12 + 2\eps}_x}
\| \Psi \|_{L^\infty_T W^{-\frac 12-\eps, \infty}_x}
\\
&\quad
+ T \| V \ast :\! \Psi^2 \!: \|_{L^\infty_T W_x^{\be - 1 - \eps,\infty}}
 \| \Psi \|_{L^\infty_T W^{-\frac 12-\eps, \infty}_x}
+ T  \| (V \ast :\! \Psi^2 \!:) \pe \Psi \|_{L^2_{T} H^{s_2-1}_x} \\
&\les
\| (Y_0, Y_1)\|_{\H^{s_2}}
+ 
T^{\frac 23}  \Big(  \|(X, Y, \Res)\|_{Z^{s_1, s_2, s_3}(T)}^3 +K^3 \Big), 
\end{split}
\label{M4}
\end{align}

\noi
provided that $\beta\ge \max\big( 1 + \eps, 
 s_2 + \frac 12 + 4\eps, - s_1 + s_2 - s_3 + \frac 12 \big)$
and $ s_1 + 2\eps \le s_2 \le 1$.

Finally, we estimate $\Phi_3(X, Y,  \Res)$ in \eqref{SNLW7}.
By 
Lemma \ref{LEM:para}, 
Lemma \ref{LEM:sto3} (in particular \eqref{A00}), 
and
\eqref{M1-3}
with \eqref{data4}, 
we have
\begin{align}
\begin{split}
\|\Phi_3(X, Y,  \Res)\|_{L^3_T H^{s_3}_x}
 &\le \big\| \If_{\pl}^{(1)}
 \big(V \ast (\Qxy  + 2\Res + :\! \Psi^2 \!:\,)\big)\pe \Psi
\big\|_{L^3_T H^{\eps}_x}
 \\
& \hphantom{X}
+\big\| \If_{\pl, \pe}
 \big(V \ast (\Qxy  + 2\Res + :\! \Psi^2 \!:\,)\big)\big\|_{L^3_T H^{s_3}_x}\\
&\les 
T^\frac{1}{3}
\| \If_{\pl}^{(1)} \big(V \ast (\Qxy  + 2\Res + :\! \Psi^2 \!:\,)\big)\big\|_{L^\infty_TH^{\frac 12 + 3\eps}_x}
\| \Psi\|_{L^\infty_T W^{-\frac 12 - \eps, \infty}_x} 
\\
& \hphantom{X}
+
T^{\frac 13}
K \|V*(\Qxy  + 2\Res + :\! \Psi^2 \!:\,)\|_{  L^\frac{3}{2}_T  L^2_x}\\
&\les 
T^{\frac 23}K^2
\| V \ast (\Qxy  + 2\Res + :\! \Psi^2 \!:\,)\|_{L^3_TH^{s_0 }_x}
\\
& \hphantom{X}
+T^\frac{2}{3} K \|V*(\Qxy  + 2\Res + :\! \Psi^2 \!:\,)\|_{L^3_T L^2_x}\\
& \les T^\frac{2}{3}
K^2  \Big( \|(X, Y, \Res)\|_{Z^{s_1, s_2, s_3}(T)}^2 +K^2 \Big)
\end{split}
\label{M5}
\end{align}

\noi
for some small positive $s_0 = s_0(\eps)  \sim  \eps$, 
provided that $\be \geq \max( -s_3 + s_0, 1 + s_0)$.

By repeating a similar computation, we also obtain the following difference estimate:
\begin{align}
\begin{split}
\| \vec \Phi &  (X, Y, \Res)
- \vec \Phi   (\wt X, \wt Y, \wt \Res)
 \|_{Z^{s_1, s_2, s_3}(T)} \\
&\les
T^{\frac 23} K^2 \Big(\|(X, Y, \Res)\|_{Z^{s_1, s_2, s_3}(T)}
+ \|(\wt X, \wt Y, \wt \Res)\|_{Z^{s_1, s_2, s_3}(T)}
 + K \Big)^2\\
& \quad \times 
\|(X , Y, \Res)  - (\wt X , \wt Y, \wt \Res)\|_{Z^{s_1, s_2, s_3}(T)}, 
\end{split}
\label{M6}
\end{align}

\noi
where $\vec \Phi := (\Phi_1, \Phi_2, \Phi_3)$.
Let $B_R \subset Z^{s_1, s_2, s_3}(T)$ be 
the closed ball 
of radius $R\sim 
\| (X_0, X_1) \|_{\H^{s_1}}
+ \| (Y_0, Y_1) \|_{\H^{s_2}} + 1$, centered at the origin.
Then, 
 by choosing $T = T(K, R) >0$ sufficiently small, 
we conclude from \eqref{M1a}, \eqref{M4}, \eqref{M5}, and \eqref{M6} that 
$\vec \Phi = (\Phi_1, \Phi_2, \Phi_3)$ 
is a contraction on the closed ball $B_R$.
A similar computation yields 
continuous dependence of the solution $(X, Y, \Res)$
on the enhanced data set $\Xi$
measured in the $\mathcal{X}^{s_1, s_2,\eps}_1$-norm.
This concludes the proof of 
Theorem~\ref{THM:1}.

\medskip

\noi
$\bullet$ {\bf Focusing case for  $\be \ge 2$.}
We conclude this section by briefly going over 
the required modifications in the focusing case.
In view of the Gibbs measure construction (Theorem \ref{THM:Gibbs2}), we take $2 <  \g \le 3$
sufficiently close to $3$ (and $\g = 3$ when $\be = 2$).
As mentioned in Section \ref{SEC:1}, 
 a precise value of $\s > 0$ does not play any role
 in the local well-posedness argument, 
so we simply set $\s = 1$
and 
 consider the system~\eqref{SNLW6}.
By writing  \eqref{SNLW6} in the Duhamel formulation,
 we have
\begin{align*}
X &=  \Psi_1(X, Y,  \Res)\\
:\! & = 
S(t)(X_0, X_1) + \I \Big( \big(V \ast 
( \Qxy + 2\Res \, + \! :\! \Psi^2 \!:) \big) \pl \Psi\Big)\\
& \hphantom{X}
- \I \Big(M_\g(\Qxy + 2\Res \, + \! :\! \Psi^2 \!:) \Psi\Big),\\
Y &= \Psi_2(X, Y,  \Res)\\
:\! & = 
S(t)(Y_0, Y_1) 
+ \I \Big( \big(V \ast 
( \Qxy + 2\Res \, + \! :\! \Psi^2 \!:) \big)  (X+Y)\Big)\\
& \hphantom{X}
+ \I \Big( \big(V \ast 
( \Qxy + 2\Res \, + \! :\! \Psi^2 \!:) \big) \pge \Psi\Big)\\
& \hphantom{X}
- \I \Big(M_\g(\Qxy + 2\Res \, + \! :\! \Psi^2 \!:) (X+Y)\Big),\\
\Res
&= \Psi_3(X, Y,  \Res)\\ 
:\! &=  \If_{\pl}^{(1)}
 \big(V \ast (\Qxy  + 2\Res + :\! \Psi^2 \!:\,)\big)\pe \Psi\\
& \hphantom{X}
 + \If_{\pl, \pe}
 \big(V \ast (\Qxy  + 2\Res + :\! \Psi^2 \!:\,)\big)\\
 & \hphantom{X}
- \I\big( M_\g(\Qxy + 2\Res \, + \! :\! \Psi^2 \!:) \Psi\big) \pe \Psi,
\end{align*}

\noi
where the last term in the $\Res$-equation is interpreted as \eqref{sto1a}.

Comparing with \eqref{SNLW7} from the defocusing case, 
it suffices to estimate the last terms in each equation.
Given an enhanced data set $\Xi$ as in~\eqref{data2}, 
we set
\begin{align}
\Xi(\Psi)  = \big( \Psi, \,  :\! \Psi^2 \!:,  \, \Ab, \, \If_{\pl, \pe}\big)
\label{data33}
\end{align}

\noi
and 
\begin{align}
\begin{split}
\| \Xi (\Psi)  \|_{\mathcal{Y}^{ \eps}_T}
&: =
\| \Psi \|_{C_T  W^{-\frac 12-\eps,\infty}_x\cap \, C_T^1  W^{-\frac 32-\eps,\infty}_x}
+ \| :\! \Psi^2 \!: \|_{C_T  W^{-1-\eps,\infty}_x} \\
&\hphantom{X}
+ \| \Ab  \|_{ L^\infty_{t'}L^3_t(\Dl_2(T); H^{ - \eps}_x)}
+  \| \If_{\pl, \pe} \|_{ \L_2(\frac 32, T)} 
\label{data3a}
\end{split}
\end{align}

\noi
for some small $\eps = \eps( \be, s_1, s_2, s_3)> 0$.
We assume
\begin{align}
\| \Xi (\Psi)  \|_{\mathcal{Y}^{ \eps}_1}
\leq K
\label{data5}
\end{align}

\noi
for some $K \geq 1$.

By the energy estimate \eqref{EE1}, 
\eqref{focusnon}, 
and 
\eqref{data5}, we have
\begin{align}
\begin{split}
\Big\|\I \Big(M_\g(\Qxy + 2\Res  & \, + \! :\! \Psi^2 \!:\, ) \Psi\Big)\Big\|_{X^{s_1}(T)}
 \les
\|M_\g(\Qxy + 2\Res \, + \! :\! \Psi^2 \!:) \Psi\|_{L^1_TH^{s_1-1}_x}\\
& \les
\|\Qxy + 2\Res \, + \! :\! \Psi^2 \!: \|_{L^{\g-1}_TH^{-100}_x}^{\g-1} 
\|\Psi\|_{L^\infty_TH^{s_1-1}_x}\\
& \les
T^{\frac{4 -\g}{3}} K \|\Qxy + 2\Res \, + \! :\! \Psi^2 \!: \|_{L^{3}_TH^{-100}_x}^{\g-1} \\
& \les
T^{\frac{4 -\g}{3}}
K \Big( \|(X, Y, \Res)\|_{Z^{s_1, s_2, s_3}(T)}^2 +K^2 \Big)^{\g-1}.
\end{split}
\label{M7}
\end{align}

\noi
Similarly, we have 
\begin{align}
\begin{split}
\Big\|\I \Big(M_\g & (\Qxy + 2\Res   \, + \! :\! \Psi^2 \!:\, ) (X+Y) \Big)\Big\|_{X^{s_2}(T)}\\
& \les
T^{\frac{4 -\g}{3}}
 \Big( \|(X, Y, \Res)\|_{Z^{s_1, s_2, s_3}(T)}^2 +K^2 \Big)^{\g-1}\|(X, Y, \Res)\|_{Z^{s_1, s_2, s_3}(T)}.
\end{split}
\label{M8}
\end{align}

\noi
By Minkowski's integral inequality
and the proceeding as in \eqref{M7}, we have 
\begin{align}
\begin{split}
\Big\|\I\big( M_\g(\Qxy & + 2\Res \, +  \! :\! \Psi^2 \!:\, ) \Psi\big) \pe \Psi\Big\|_{L^3_TH^{s_3}_x}\\
& \leq  \int_0^T 
| M_\g(\Qxy + 2\Res \, + \! :\! \Psi^2 \!:) (t')| \cdot \| \Ab(t, t')\|_{L^3_t([t', T]; H^{s_3}_x)} dt'\\
& \leq K 
\|\Qxy + 2\Res \, + \! :\! \Psi^2 \!: \|_{L^{\g-1}_TH^{-100}_x}^{\g-1} \\
& \les
T^{\frac{4 -\g}{3}}
K \Big( \|(X, Y, \Res)\|_{Z^{s_1, s_2, s_3}(T)}^2 +K^2 \Big)^{\g-1}.
\end{split}
\label{M9}
\end{align}

Since $\g \le 3$, 
we have a
small power of $T$ in  \eqref{M7}, \eqref{M8},  and \eqref{M9}.
Furthermore, since $\g \ge 2$, 
$|x|^{\g-2}x$ is differentiable
with a locally bounded derivative
and thus difference estimates also hold for these extra terms.
Therefore, 
proceeding as in the defocusing case, 
we can show
that 
 $\vec \Psi := (\Psi_1, \Psi_2, \Psi_3)$
is a contraction on the ball $B_R \subset Z^{s_1, s_2, s_3}(T)$
of radius $R\sim 
\| (X_0, X_1) \|_{\H^{s_1}}
+ \| (Y_0, Y_1) \|_{\H^{s_2}} + 1$.
This proves 
Theorem \ref{THM:2}
in the focusing case.

\section{Invariant Gibbs dynamics}
\label{SEC:GWP}

In this section, we present the proof of Theorem \ref{THM:GWP}
by applying Bourgain's invariant measure argument \cite{BO94, BO96}.
In Subsection \ref{SUBSEC:GWP1}, 
we first study the truncated dynamics 
and establish a long time  a priori bound on the solutions (Proposition \ref{PROP:GWPNB}).
In Subsection \ref{SUBSEC:GWP2}, 
we then prove almost sure global well-posedness of the Hartree SdNLW
and invariance of the Hartree Gibbs measure.
Our presentation closely follows those
in \cite{Tz08, BT, ORTz}, 
in particular~\cite{ORTz}, where a renormalization was required on the nonlinearity.
We, however,  point out that 
a certain part of the argument from~\cite{ORTz} in the two-dimensional setting
can not be applied to  the current  three-dimensional setting, 
where we imposed the paracontrolled structure
for constructing local-in-time solutions.
More precisely, in estimating the difference of two solutions, 
the authors in \cite{ORTz} applied
the product estimates (such as Lemma \ref{LEM:gko})
to bound the difference of the enhanced data sets
with two different initial data
(and iterated the local-in-time argument).
Such an estimate, however, fails in the three-dimensional setting
due to the lower regularity of the noise.
See Remark \ref{REM:fail} below
for a further discussion.
We instead establish a stability result on 
a large time interval $[0, T]$ in a direct manner, 
incorporating the paracontrolled structure.\footnote{In a preprint \cite{Bring2}, 
Bringmann overcame a similar issue via a different approach, 
by establishing a certain stability result of a paracontrolled structure.}
See
Proposition \ref{PROP:GWPNC}.

In the following, we only consider the focusing case ($\s> 0$).
A straightforward  modification yields the corresponding result for
the defocusing case.
Furthermore, we restrict our attention
to the non-endpoint case $\be > 2$ and assume $\s = 1$
for simplicity.
The same argument applies to the critical case $\be = 2$
with $0 < \s \ll 1$.

In the remaining part of this section, 
we fix some notations.
Let $V$ be the Bessel potential of order $\be>2$.
We also fix 
$A > 0$  sufficiently large
and $\g > 0$, satisfying 
 $\max \big(\frac{\be+1}{\be-1},2\big) \le \g <3$
 with $\g > 2$ when $\be = 3$, 
such that the focusing Hartree measure 
$\rho$ in \eqref{Gibbs9} is constructed as the limit of
the truncated Gibbs measures $\rho_N$
in \eqref{GibbsN1}
as in  Theorem \ref{THM:Gibbs2}.
With these parameters
and $\mu_0$  as in  \eqref{gauss0}, 
the Gibbs measure $\rhoo = \rho\otimes \mu_0$
for the focusing Hartree SdNLW~\eqref{SNLWA2}
is  constructed 
as the limit of 
the truncated Gibbs measures:
\begin{align}
\rhoo_N = \rho_N \otimes \mu_0
\label{rhooN}
\end{align}

\noi
for  the truncated focusing Hartree SdNLW \eqref{SNLWA2r};
see  Remark \ref{REM:Gibbs1}.

By assumption,  the Gaussian field $\muu = \mu_1 \otimes \mu_0$ in \eqref{gauss1}
and hence the (truncated) Gibbs measures
are independent of (the distribution of) the space-time white noise $\xi$ in \eqref{SNLWA2}
and \eqref{SNLWA2r}.
Hence, we can write the probability space $\Omega$
as 
\begin{align}
\O = \O_1 \times \O_2
\label{O1}
\end{align}
such that the random Fourier series in \eqref{IV2}  depend only on $\o_1 \in \O_1$, 
while the cylindrical Wiener process $W$ in \eqref{W1}
depends only on $\o_2 \in \O_2$.
In view of \eqref{O1}, 
we also write the underlying probability measure $\PP$ on $\O$
as $\PP = \PP_1 \otimes \PP_2$, 
where $\PP_j$ is the marginal probability measure on $\O_j$, $j = 1, 2$.

With the decomposition \eqref{O1} in mind, 
we set
\begin{align}
\begin{split}
\Psi (t;\vec u_0, \o_2)
&= S(t) \vec u_0 
+ \sqrt 2 \int_0^t \D(t-t') dW(t', \o_2),
\\
\vec \Psi (t;\vec u_0, \o_2)
&= \big( \Psi (t; \vec u_0, \o_2), \dt \Psi (t; \vec u_0, \o_2) \big),
\\
\vphantom{\int}
\Psi_N (t;\vec u_0, \o_2)
&= \pi_N \Psi (t; \vec u_0, \o_2)
\end{split}
\label{Psix}
\end{align}

\noi
for $\vec u_0 =(u_0,u_1) \in \H^{-\frac 12-\eps}(\T^3)$ and $\o_2 \in \O_2$, 
where $S(t)$ is as in \eqref{lin1}.
We may suppress the dependence on $t$ and $\o_2$
and write $\Psi (\vec u_0)$, etc.

In the remaining part of this section, 
we fix 
 $\frac 14<s_1 < \frac 12 < s_2<1$ and $-\frac 12 <s_3<0$, 
 satisfying \eqref{Z0}, 
 as in (the proof of) Theorem~\ref{THM:2}
 on the local well-posedness
 of the focusing Hartree SdNLW system~\eqref{SNLW6}.

\subsection{On the truncated dynamics}
\label{SUBSEC:GWP1}

In this subsection, we 
study the truncated  focusing Hartree SdNLW~\eqref{SNLWA2r}:
\begin{align}
\begin{split}
\dt^2 & u_N + \dt u_N  + (1 -  \Dl)  u_N  \\
& 
-\s  \pi_N \big( (V \ast  :\! (\pi_N u_N)^2 \!:\, ) \pi_N u_N \big)
+  M_\g (\,:\! (\pi_N u_N)^2 \!:\,) \pi_N u_N 
= \sqrt{2} \xi, 
\end{split}
\label{SNLWx}
\end{align}

\noi
where 
$:\! (\pi_N u_N)^2 \!:  \, = 
(\pi_N u_N)^2 -\s_N$
and $M_\g$ is  as in  \eqref{focusnon}.
While local well-posedness of the truncated equation~\eqref{SNLWx}
follows from a small modification of the proof of Theorem \ref{THM:2}, 
we  present a simple argument
 to prove local well-posedness of~\eqref{SNLWx}
 (Lemma \ref{LEM:LWPode}).
Then, we prove almost sure global well-posedness
of the truncated equation \eqref{SNLWx} and 
 invariance of the truncated Gibbs  measure $\rhoo_N$
 in \eqref{rhooN} (Lemma \ref{LEM:GWP4}).

Given  $N \in \N$, 
let 
 $\vec u_0 = (u_0, u_1)$ 
be a pair of random distributions such that 
$\Law( (u_0, u_1)) = \rhoo_N = \rho_N \otimes \mu_0$
defined in \eqref{rhooN}.
Let $u_N$ be a solution 
to the truncated equation \eqref{SNLWx} with $(u_N, \dt  u_N) |_{t = 0} 
= \vec u_0$.
With 
$:\! (\pi_N u_N)^2 \!:  \, = 
(\pi_N u_N)^2 -\s_N$, 
we write \eqref{SNLWx}
as 
\begin{align}
\begin{cases}
\dt^2  u_N + \dt u_N  + (1 -  \Dl)  u_N  
-\s  \pi_N \big( (V \ast   ((\pi_N u_N)^2  - \s_N) ) \pi_N u_N \big)\\
\quad 
+  M_\g ((\pi_N u_N)^2 -\s_N) \pi_N u_N 
= \sqrt{2} \xi, 
\\
(u_N,\dt u_N)|_{ t= 0}  =\vec u_0.
\end{cases}
\label{SNLWode2}
\end{align}

\noi
The
dynamics of \eqref{SNLWode2} 
is decoupled into the 
 high frequency part  $\{|n|>  N\}$
and the low frequency part  $\{|n|\le N\}$.
The high frequency part of the dynamics \eqref{SNLWode2} is given by 
\begin{align}
\begin{cases}
\begin{aligned}
&\dt^2 \pi_N^\perp u_N  + \dt \pi_N^\perp u_N + (1 - \Dl)\pi_N^\perp  u_N =\sqrt{2} \pi_N^\perp \xi
\end{aligned}
\\
(\pi_N^\perp u_N,\dt \pi_N^\perp u_N)|_{ t= 0}  =\pi_N^\perp  \vec u_0 , 
\end{cases}
\label{high1}
\end{align}

\noi where 
$\pi_N^\perp = \text{Id} - \pi_N$, 
and thus
the solution $\pi_N^\perp u_N$ to \eqref{high1} is given by
\begin{align}
\pi_N^\perp u_N 
= \pi_N^\perp  \Psi ( \vec u_0).
\label{high2}
\end{align}

\noi
With $v_N = \pi_N u_N$, 
the low frequency part of the dynamics \eqref{SNLWode2} is given by 
\begin{align}
\begin{cases}
\dt^2  v_N + \dt v_N  + (1 -  \Dl)  v_N  -\s  \pi_N \big( (V \ast   ((\pi_N v_N)^2  - \s_N) ) \pi_N v_N \big)\\
\quad 
+  M_\g ((\pi_N v_N)^2 -\s_N) \pi_N v_N 
= \sqrt{2}\pi_N \xi, 
\\
(v_N,\dt v_N)|_{ t= 0}  = \pi_N \vec u_0.
\end{cases}
\label{SNLWode2p}
\end{align}

\noi
Note that we kept $\pi_N$ in several places to emphasize that 
\eqref{SNLWode2p} depends only on finite many frequencies $\{|n|\le N\}$.
By writing~\eqref{SNLWode2p} 
in the Duhamel formulation, 
we have
\begin{align}
v_N (t) & = \pi_N S(t) \vec u_0 + \int_0^t \D(t - t') \NN_N(v_N)(t') dt'
 + \Psi_N(t; 0), 
\label{high3}
\end{align}

\noi
where the truncated nonlinearity $\NN_N(v_N)$ is given by 
\begin{align*}
\NN_N(v_N) =
\s  \pi_N \big( (V \ast   ((\pi_N v_N)^2  - \s_N) ) \pi_N v_N \big)
-   M_\g ((\pi_N v_N)^2 -\s_N) \pi_N v_N 
\end{align*}

\noi
and  $\Psi_N(t; 0) = \pi_N \Psi_N(t; 0)$ is as in \eqref{Psix} with $\vec u_0 = 0$.
Recall from Lemma \ref{LEM:stoconv}
that $\Psi_N(t; 0) \in C^1(\R_+; C^\infty(\T^3))$.
By viewing   $\Psi_N (t; 0)$  in \eqref{high3} as a perturbation, 
it suffices to study 
the following damped NLW with a deterministic perturbation:
\begin{align}
v_N (t) & = \pi_N S(t) (v_0, v_1) + \int_0^t \D(t - t') \NN_N(v_N)(t') dt'
 +  F, 
\label{high6}
\end{align}

\noi
where $(v_0,v_1) \in \H^1(\T^3)$, $\s_N$ is as in \eqref{sigma1}, and 
$F \in C^1(\R_+; C^\infty(\T^3))$ is a   given deterministic function.

\begin{lemma} \label{LEM:LWPode}
Let  $N \in \N$.
Given any 
$(v_0,v_1) \in \H^1(\T^3)$ and $F \in C([0,1]; H^1 (\T^3))$ with
\[
\| (v_0,v_1) \|_{\H^1} \le R
\qquad \text{and}\qquad 
\| F \|_{C^1([0, 1]; H^1)} \le K
\]
for some $R,K\ge 1$, 
 there exist $\tau= \tau (R, K,  N)>0$ and a unique solution $v_N$ to \eqref{high6} on $[0,\tau]$,
 satisfying the bound:
 \[ \| v_N\|_{X^1(\tau)} \les R + K,  \]

\noi
where  $X^1(\tau)$ is as in \eqref{M0}.
Moreover, the solution $v_N$ is unique 
in $X^1(\tau)$.

\end{lemma}

While the local existence time depends on $N \in \N$, 
Lemma \ref{LEM:LWPode} suffices for our purpose.

\begin{proof}
Let $\Phi_N(v_N)$ denote the right-hand side of \eqref{high6}.
Let $0 < \tau \leq 1$.
Then, from \eqref{W3}, \eqref{EE1}, and Sobolev's inequality 
with~\eqref{focusnon}, we have 
\begin{align*}
\| & \Phi_N  (v_N) \|_{X^1(\tau)} \\
&\les \| (v_0, v_1) \|_{\H^1}
+ \tau \big\| V \ast \big((\pi_N v_N)^2 -\s_N \big) \big\|_{L^\infty_\tau L^3_x}
 \| v_N \|_{L^\infty_\tau L^6_x}\\
&\quad
+ \tau \big\| M_\g \big( (\pi_N v_N)^2 -\s_N \big) \big\|_{L^\infty_\tau} 
\| v_N  \|_{L^\infty_\tau L^2_x} + \| F\|_{L^\infty_\tau H^1_x}\\
&\les R
+ \tau \Big( \| v_N \|_{L^\infty_\tau H^1_x}^{2(\g-1)} 
 +  \s_N^{\g-1}\Big) 
 \| v_N \|_{L^\infty_\tau H^1_x}
+ K\\
& \les  R + K
\end{align*}

\noi
for any $v_N  \in X^1(\tau)$ with  $\|v_N \|_{X^1(\tau)} \le C_0 (R+K)$, 
where the last step holds by choosing 
 $\tau= \tau (R,K, N)> 0$ 
 sufficiently small.
A difference estimate also follows in a similar manner since $\g \ge 2$.
Hence, we conclude that 
 $\Phi_N$ is a contraction on the ball $B_{C_0(R+K)}
\subset X^1(\tau)$ for some $C_0 > 0$.
At this point, the uniqueness holds only in the ball 
$B_{C_0(R+K)}$ but by a standard continuity argument, we can extend the uniqueness to hold in the entire $ X^1(\tau)$. We omit details.
\end{proof}

\begin{remark}\label{REM:tau}\rm
(i) From the proof, we see that 
 $\tau= \tau (R, K, N) 
 \sim( R  + K + N)^{-\ta}$ 
 for some $\ta > 0$.
In particular, the local existence time $\tau$ depends on $N \in \N$.

\smallskip

\noi
(ii) 
Note that the uniqueness statement
for $v_N$  in Lemma \ref{LEM:LWPode} is unconditional, 
namely, the uniqueness of the solution $v_N$ holds in the entire class
$X^1(\tau)$.
Then, from~\eqref{high2} and 
the unconditional uniqueness of the solution $v_N = v_N(\pi_N \vec u_0)$ to \eqref{SNLWode2p}, 
we obtain the {\it unique} representation of $u_N$:
\begin{align}
 u_N =  \pi_N^\perp  \<1>( \vec u_0) + \pi_N v_N(\pi_N \vec u_0).
 \label{O1a}
\end{align}

\noi
This   uniqueness statement
for $u_N$ plays an important role in Proposition \ref{PROP:GWPNB}
and Lemma \ref{LEM:Leo1}.
See Remarks \ref{REM:unique} and \ref{REM:Leo} below.

\end{remark}

Before proceeding further, let us introduce some notations.
Given the cylindrical Wiener process  $W$ in~\eqref{W1}, 
by possibly enlarging the probability space $\Omega_2$, 
there exists a family of translations $\tau_{t_0} : \O_2 \to \O_2$ such that
\begin{align}
W(t, \tau_{t_0}(\o_2))
=  W(t+t_0, \o_2) - W(t_0,\o_2)
\label{fail0}
\end{align}

\noi
for $t, t_0 \ge 0$ and $\o_2 \in \O_2$.
Denote by $\Phi^N(t)$ the stochastic flow map 
to the truncated equation~\eqref{SNLWx}
given in Lemma \ref{LEM:LWPode}
(which is not necessarily global at this point).
Namely, 
\[\vec u_N(t) 
= \Phi^N(t) (\vec u_0,\o_2)\]

\noi
is the solution 
 to \eqref{SNLWx} with 
 $\vec u_N |_{t=0}= \vec u_0$, satisfying  $\Law (\vec u_0) = \rhoo_N$, 
 and the noise $\xi (\o_2)$.
We now 
extend $\Phi^N(t)$ as
\begin{align}
\ft \Phi^N(t) (\vec u_0, \o_2)
= \big( \Phi^N(t)(\vec u_0, \o_2), \tau_t (\o_2) \big).
\label{Phi9}
\end{align}

\noi
Note that 
by the uniqueness of the solution to \eqref{SNLWx},
we have
\begin{align}
\Phi^N (t_1+t_2) (\vec u_0, \o_2)
= \Phi^N (t_2) \big( \Phi^N (t_1) (\vec u_0, \o_2), \tau_{t_1}(\o_2) \big)
= \Phi^N (t_2) \big(\ft \Phi^N (t_1) (\vec u_0, \o_2)\big)
\label{Ba0}
\end{align}

\noi
for $t_1, t_2 \ge 0$
as long as the flow is well defined.

Next, by exploiting invariance of the truncated Gibbs measure $\rhoo_N$, 
we construct global-in-time solutions 
to \eqref{SNLWx} almost surely with respect to the truncated Gibbs measure $\rhoo_N$ in~\eqref{rhooN}.

\begin{lemma} \label{LEM:GWP4}
Let $N \in \N$.
Then, the truncated focusing Hartree SdNLW \eqref{SNLWx} is almost surely globally well-posed
with respect to the random initial data distributed
by the truncated Gibbs measure $\rhoo_N$ in~\eqref{rhooN}.
Furthermore, $\rhoo_N$ is invariant under the resulting dynamics.
\end{lemma}

\begin{proof}

We first discuss the (formal) invariance of the truncated Gibbs measure $\rhoo_N$
under the truncated dynamics \eqref{SNLWx}.
Given $N \in \N$, let $\pi_N^\perp = \text{Id}- \pi_N$.
We define
the marginal probability measures
$\muu_{N}$ and $\muu_{N}^\perp$
on $\pi_N \H^{-\frac 12 -\eps}(\T^3)$
and $\pi_N^\perp \H^{-\frac 12 -\eps}(\T^3)$, respectively, 
as 
   the induced probability measures
under the following maps:
\begin{equation*}
\o_1 \in \O_1 \longmapsto (\pi_N u^{\o_1}, \pi_N v^{\o_1})
 \end{equation*}

\noi
for 
$\muu_{N}$
and
\begin{equation*}
\o_1 \in \O_1 \longmapsto (\pi_N^\perp u^{\o_1}, \pi_N^\perp v^{\o_1})
 \end{equation*}

\noi
for $\muu_{N}^\perp$, 
 where 
$u^{\o_1}$ and $v^{\o_1}$ are as in \eqref{IV2}
with $\o$ replaced by $\o_1$ in view of the decomposition~\eqref{O1}.
Then, we have
\begin{align}
\muu = 
\muu_{ N} \otimes \muu_{N}^\perp.
\label{O2}
\end{align}

\noi
From 
\eqref{rhooN} with 
\eqref{gauss1}, 
\eqref{GibbsN1},  and \eqref{O2}, 
we then  have
\begin{align}
\rhoo_N = \vec \nu_{N} \otimes \muu_{N}^\perp, 
\label{O3}
\end{align}

\noi
where $\vec \nu_N$ is given by
\begin{align}
d \vec \nu_N= \ft Z_N^{-1} e^{\RR_N(u)}d\muu_{N}
\label{nuu}
\end{align}

\noi
with the density $\RR_N$ as in \eqref{K2}.

By writing $u_N$ as $u_N = \pi_N^\perp u_N + \pi_N u_N$, 
we see that
 the high frequency part
$\pi_N^\perp u_N = \pi_N^\perp\Psi(\vec u_0)$
satisfies the linear dynamics \eqref{high1}.
It is easy to check that  the Gaussian measure $\muu_{N}^\perp$
is invariant under the dynamics of \eqref{high1}, 
say,  by studying~\eqref{high1} for each frequency $|n|> N$
on the Fourier side.
The low frequency part $\pi_N u_N$ satisfies
\eqref{SNLWode2p}.
By writing \eqref{SNLWode2p} in the Ito formulation
with $(u_N^1, u_N^2) = (\pi_N u_N , \dt \pi_N u_N )$, 
it is easy to see that 
 the generator $\L^N$ for \eqref{SNLWode2p}
 can be written as $\L^N = \L^N_1 + \L^N_2$, 
 where $\L^N_1$ denotes the generator
 for the undamped  NLW with the truncated nonlinearity:
\begin{align}
\begin{split}
d  \begin{pmatrix}
u_N^1 \\ u_N^2
\end{pmatrix}
 + \Bigg\{
\begin{pmatrix}
0  & -1\\
1-\Dl &  0
\end{pmatrix}
 \begin{pmatrix}
u_N^1 \\ u_N^2
\end{pmatrix}
 +  
\begin{pmatrix}
0 \\ -  \NN_N(\pi_N u_N^1) 
\end{pmatrix}
\Bigg\} dt 
   = 0 
\end{split}
\label{O7}
\end{align}

\noi
and $\L^N_2$ denotes the generator
for the Ornstein-Uhlenbeck process 
(for the second component $u_N^2$):
\begin{align}
\begin{split}
d  \begin{pmatrix}
u_N^1 \\ u_N^2
\end{pmatrix}
   = 
  \begin{pmatrix}
0  \\ - u_N^2 dt + \sqrt 2\pi_N dW
\end{pmatrix} .
\end{split}
\label{O8}
\end{align}

By recalling that the Ornstein-Uhlenbeck process
preserves the standard Gaussian measure, 
we conclude that $\vec \nu_N$ is invariant under the linear dynamics  \eqref{O8}
since
 the measure $\vec \nu_N$ is nothing but the  white noise measure 
(projected onto the low frequencies $\{|n|\leq N\}$)
on the second component $u_N^2$.
As for  \eqref{O7}, we note that it  is a Hamiltonian equation with the Hamiltonian:
\begin{align*}
\mathcal{E}^\sharp_N(u_N^1, u_N^2)
= \frac 12 \int_{\T^3} |\jb{\nb} u_N^1|^2 dx 
+ \frac 12 \int_{\T^3} (u_N^2)^2 dx - \RR_N(u^1_N),
\end{align*}

\noi
where $\RR_N$ is as in \eqref{K2} (with $\s = 1$).
Then, from  the conservation of the Hamiltonian 
$\mathcal{E}^\sharp_N(u_N^1, u_N^2)$ and Liouville's theorem
(on the finite-dimensional phase space $\pi_N \H^{-\frac 12 -\eps}(\T^3)$), 
we conclude that $\vec \nu_N$ in \eqref{nuu} is invariant under the dynamics of \eqref{O7}.
Therefore, we conclude that 
\[(\L^N)^*\vec \nu_N = 
(\L^N_1)^*\vec \nu_N +  (\L^N_2)^*\vec  \nu_N = 0,\]
where
$(\L^N)^\ast$ denotes the 
 the adjoint of 
the infinitesimal 
generator $\L^N = \L^N_1 + \L^N_2$
for \eqref{SNLWode2p}.
This shows invariance of $\vec \nu_N$ under \eqref{SNLWode2p}.

Therefore, from \eqref{O3}
and the invariance of  $\muu^\perp_{N}$  and $\vec \nu_N$ 
under \eqref{high1} and \eqref{SNLWode2p}, respectively,  we conclude that  the truncated Gibbs measure 
$\rhoo_N$ in~\eqref{rhooN}
is {\it formally} invariant 
under the dynamics of 
the   truncated focusing Hartree SdNLW
\eqref{SNLWx}.
Here, by the formal invariance, 
we mean that the $\rhoo_N$-measure of a measurable set 
is preserved under the truncated dynamics~\eqref{SNLWx} as long as the flow is well defined.
In view of the translation invariance of the law
of the Brownian motions $\{B_n\}_{n \in \Z^3}$
in \eqref{W1}, 
we also conclude formal invariance of 
$\rhoo_N\otimes \PP_2$
under the extended stochastic flow map $\ft \Phi^N(t)$ defined in \eqref{Phi9}.

Next, by exploiting this formal invariance of $\rhoo_N\otimes \PP_2$, 
we establish almost sure global well-posedness of 
\eqref{SNLWx}.
By arguing as in \cite{BO94, CO, BOP2}, 
it suffices to show  ``almost'' almost sure global existence.
Namely, we prove that, 
given any $T\geq 1$ and $\kk >0$,
there exists $\Si_{T,\kk} \subset \H^{-\frac 12 - \eps}(\T^3)  \times \O_2$ such that
$\rhoo_N \otimes \PP_2( \Si_{T,\kk}^c) < \kk$ and for any $(\vec u_0, \o_2) \in \Si_{T,\kk}$,
there exists a solution $u_N$ to \eqref{SNLWx}
on the time interval $[0, T]$.

We follow the ideas from \cite{BO94, ORTz}.
Given $T \geq 1$ and  $\kk > 0$, 
let  
\begin{align}
K \sim c_N  \bigg( \log \frac T   \kk + \log C_N \bigg)^{\frac 12}
\label{O9}
\end{align}

\noi
for some suitable $c_N, C_N > 0$.
Then, 
with $\tau = \tau (K,  K, N)>0$ as in Lemma \ref{LEM:LWPode}
(see also Remark~\ref{REM:tau}), 
we set 
\begin{align*}
\Si_{T,\kk}
= \bigcap_{j = 0 }^{ [T/ \tau]}
\Big\{&  (\vec u_0, \o_2) \in \H^{-\frac 12 - \eps}(\T^3)  \times \O_2: 
\|  \Phi^N(j \tau) (\vec u_0, \o_2) \|_{\H^1 } \le  K, \\
& \| \Psi_N (\ft \Phi^N(j \tau) (\vec u_0, \o_2)) \|_{L^\infty_{\tau, x} } \le K \Big\}.
\end{align*}

\noi
By the definition of $\Si_{T,\kk}$
and the local well-posedness argument (Lemma \ref{LEM:LWPode}), 
we see that, given any 
$(\vec u_0, \o_2) \in \Si_{T,\kk}$, 
the corresponding solution $(u_N, \dt u_N) $ to \eqref{SNLWx} exists on $[0, T]$.

By Bernstein's inequality, we have
\begin{align*}
\| \pi_N \vec u_0\|_{\H^1} & \les N^{\frac{3}{2} + \eps}  \| \pi_N \vec u_0\|_{\H^{-\frac{1}{2}-\eps}}, \\
 \| \Psi_N (\vec u_0, \o_2) \|_{L^\infty_{\tau, x} } 
& \les 
N^{\frac{1}{2} + \eps} \| \Psi_N (\vec u_0, \o_2) \|_{L^\infty_{\tau} W^{-\frac 12 - \eps, \infty}_x}. 
\end{align*}

\noi
Then, from the  (formal) invariance of $\rhoo_N \otimes \PP_2$
under the extended stochastic flow map $\ft \Phi^N(t)$ in~\eqref{Phi9}, 
Remark~\ref{REM:tau}, Cauchy-Schwarz inequality
with Theorem \ref{THM:Gibbs2}
(in particular, the bound~\eqref{exp3}), 
 Lemma \ref{LEM:stoconv}, 
and \eqref{O9}, we have
\begin{align*}
\rho_N \otimes \PP_2( \Si_{T,\kk}^c)
& \les  \frac{T}{\tau}  
\bigg\{
\rho_N \otimes \PP_2 \big( (\vec u_0, \o_2): 
\|   \pi_N \vec u_0 \|_{\H^1 } >  K\big) \\
& \hphantom{XXXX}
+\rho_N \otimes \PP_2 \big( (\vec u_0, \o_2): 
\| \Psi_N ( \vec u_0, \o_2) \|_{L^\infty_{\tau, x} } >  K\big)\bigg\}
\\
& \le C_N  TK^\ta
\bigg\{
\mu \otimes \PP_2 \big( (\vec u_0, \o_2): 
\|   \pi_N \vec u_0 \|_{\H^1 } >  K\big) \\
& \hphantom{XXXXXX}
+\mu \otimes \PP_2 \big( (\vec u_0, \o_2): 
\| \Psi_N ( \vec u_0, \o_2) \|_{L^\infty_{\tau, x} } >  K\big)\bigg\}^\frac{1}{2}\\
%
& \le C_N T \cdot C e^{-c'_N K^2}
\ll  \kk.
\end{align*}

\noi
This proves the desired 
 ``almost'' almost sure global existence, and thus almost sure global well-posedness
 of the truncated focusing Hartree SdNLW \eqref{SNLWx}.
Since the dynamics is now globally well defined
almost surely with respect to $\rhoo_N$, 
we conclude invariance of the truncated Gibbs measure $\rhoo_N$
from the formal invariance of $\rhoo_N$ discussed above.
\end{proof}

We now establish a long time a priori bound on 
the solutions to the truncated equation~\eqref{SNLWx}.
We emphasize that the following growth bound \eqref{Ba1ab} with \eqref{Ba1aa} is independent of $N$,
which is contrast to all earlier results of this section.

\begin{proposition}\label{PROP:GWPNB}
Let 
 $i \in \N$ and $N \in \N$.
 Then, 
there exists a $\rhoo_N \otimes \PP_2$-measurable set 
$\Si^i_N \subset \H^{-\frac 12-\eps}(\T^3) \times \O_2$ such that
\begin{align}
\rhoo_N \otimes \PP_2 \big( (\H^{-\frac 12-\eps}(\T^3) \times \O_2) \setminus \Si^i_N \big)
\le 2^{-i}.
\label{Ba1aa}
\end{align}
Moreover,
there exists  $C>0$ such that
for any $(\vec u_0, \o_2) \in \Si^i_N$ and $t \ge 0$,
we have
\begin{align}
\big\| \Phi^N (t) ( \vec u_0, \o_2) \big\|_{\H^{-\frac 12-\eps}}
\le C ( i+ \log (1+t)).
\label{Ba1ab}
\end{align}
\end{proposition}

\begin{proof}
We follow the argument in \cite{ORTz}.
Given $(\vec u_0, \o_2) \in \H^{-\frac 12-\eps}(\T^3) \times \O_2$, 
we set
\begin{align}
\Psi_N =  \Psi_N (\vec u_0, \o_2)
\label{Ba1x}
\end{align}
and 
define
the enhanced data set $ \Xi (\Psi_N (\vec u_0, \o_2))$  by 
\begin{align}
 \Xi (\Psi_N (\vec u_0, \o_2))
  = \big(   \Psi_N , \,  :\! \Psi_N^2  \!:,  \, \Ab_N, \, \wt \If_{\pl, \pe}^N\big), 
\label{data3x}
\end{align}

\noi
where 
$ :\! \Psi_N^2  \!:$ and  $ \Ab_N$
are defined in \eqref{so4b} and
\eqref{sto3}, respectively, 
with the substitution~\eqref{Ba1x}.
The paracontrolled operator
$ \wt \If_{\pl, \pe}^N$
is defined in a manner
analogous to $ \If_{\pl, \pe}^N$
in 
Proposition \ref{PROP:sto4J}, 
but with an extra frequency cutoff $\pi_N$.
Namely, instead of \eqref{X2}, 
we first define
$\wt \If^N_{\pl}$ by 
\begin{align*}
\wt \If^N_{\pl}(w) (t)
    =  \I (\pi_N(w\pl \Psi_N))(t) , 
\end{align*}

\noi
where $\Psi_N$ is as in \eqref{Ba1x}.
We then define
$\wt \If^{(1), N}_{\pl}$ and $\wt \If^{(2), N}_{\pl}$
as in \eqref{X3} with an extra frequency cutoff  $|n| \le N$, 
depending on $|n_1|\ges |n_2|^\ta$
or $|n_1|\ll |n_2|^\ta$.
Note that the conclusion of Lemma~\ref{LEM:sto3}
(in particular the estimate \eqref{A00})
holds for $\wt \If^{(1), N}_{\pl}$.
Finally, we  define 
$\wt \If^N_{\pl, \pe}$ by 
\begin{align*}
\wt \If^N_{\pl, \pe}(w) (t)
   =
\wt  \If_{\pl}^{(2), N}(w)\pe \Psi_N(t) , 
\end{align*}

\noi
namely, by inserting 
a  frequency cutoff  $|n_1+n_2| \le N$ 
and replacing $\Psi$ by $\Psi_N$ in \eqref{X7}.

Fix small $\dl > 0$. Given $i, j, N \in \N$ and $D \gg1 $, 
define a set $B^{i,j}_N (D)$ by\footnote{The third condition in \eqref{Ba2} is used in the proof of Proposition \ref{PROP:GWPNC} below.} 
\begin{align}
\begin{split}
B^{i,j}_N (D) = & \Big\{  (\vec u_0, \o_2) \in \H^{-\frac 12-\eps}(\T^3) \times \O_2 :\\
& \hphantom{X}
\| \vec \Psi (\vec u_0, \o_2) \|_{C([0, 2^j]; \H^{-\frac 12-\eps})
\cap\, C^1([0, 2^j]; \H^{-\frac 32-\eps})} \le D(i+j), \\
& \hphantom{X}
\sup_{k = 0, 1, \dots, j}\| \Xi (\Psi_N (\vec u_0, \o_2))  \|_{\mathcal{Y}^{ \eps}_{2^k}}
\le D(i+j),\\
& \hphantom{X}
\sup_{k = 0, 1, \dots, j}
\| \Xi (\Psi_M (\vec u_0, \o_2)) - \Xi (\Psi_N (\vec u_0, \o_2))  \|_{\mathcal{Y}^{ \eps}_{2^k}}\\
& \hphantom{XXXXXXXXX}
\le M^{-\dl} D(i+j)
\text{ for any } M \le N
\Big\}
\end{split}
\label{Ba2}
\end{align}

\noi
where $\| \Xi (\Psi)  \|_{\mathcal{Y}^{ \eps}_T}$ is as in \eqref{data3a}.
Then, 
by Theorem \ref{THM:Gibbs2}, Cauchy-Schwarz inequality, 
Lemma~\ref{LEM:stoconv}, 
Lemma~\ref{LEM:sto1},
Proposition~\ref{PROP:sto4J},
 and \eqref{Ba2},
we have
\begin{align}
\begin{split}
\rhoo_N \otimes  \PP_2 &  \big( (\H^{-\frac 12-\eps} (\T^3) \times \O_2 )\setminus B^{i,j}_N (D)  \big) \\
&\le C \big\| e^{\RR_N (u)} \big\|_{L^2(\mu)} \Big( \muu \otimes \PP_2 
\big( (\H^{-\frac 12-\eps} (\T^3) \times \O_2 )\setminus B^{i,j}_N (D) \big) \Big)^{\frac 12} \\
&
\le C 2^j \exp \big( -c D (i+j) \big)\\
& 
\le C  \exp \big( -c' D (i+j) \big), 
\end{split}
\label{Ba3}
\end{align}

\noi
uniformly in $i, j, N \in \N$, 
provided that $D \gg 1$.

It follows from a slight modification of (the proof of) Theorem \ref{THM:2}
that 
\begin{align}
\Phi^N (t) \big( B^{i,j}_N (D) \big)
\subset
\Big\{ \vec u \in \H^{-\frac 12-\eps}(\T^3):
\| \vec u \|_{\H^{-\frac 12-\eps}} \le D (i+j+1) \Big\}
\label{Ba4a}
\end{align}

\noi for any $0 \le t \le \tau$,
where $\tau $ is given by 
\begin{align}
\tau = \big( D (i+j) \big)^{-\ta}
\label{Ba4}
\end{align}

\noi
for some $\ta > 0$.
Indeed,
by decomposing the first component
 $\Phi^N_1(t)(\vec u_0, \o_2)$ of  $\Phi^N(t)(\vec u_0, \o_2)$ as in \eqref{decomp3} and \eqref{decomp3a}:
\begin{align}
\Phi_1^N (t) (\vec u_0, \o_2) = \Psi (t; \vec u_0, \o_2) + X_N(t) + Y_N(t), 
\label{Ba4b}
\end{align}

\noi
we see that 
 $X_N$, $Y_N$, and $\Res_N := X_N \pe \Psi_N (\vec u_0, \o_2)$ 
 satisfy  the following system:
\begin{align}
\begin{split}
 (\dt^2 + \dt  +1 - \Dl) X_{N} & =
\pi_{N} \Big(
\big( V \ast 
( \QxyN + 2\Res_{N} \, + \! :\! \Psi_{N}^2 \!: ) \big) \pl \Psi_{N} \Big) \\
&\phantom{X}
-M_\g(\QxyN + 2\Res_N \, + \! :\! \Psi_N^2 \!:) \Psi_N ,\\
 (\dt^2 + \dt +1  - \Dl) Y_{N}
&  = \pi_{N} \Big( \big( V \ast ( \QxyN + 2\Res_{N} \, + \! :\! \Psi_{N}^2 \!:) \big) (X_{N}+Y_{N}) \Big) \\
& \hphantom{X}
+ \pi_{N} \Big( \big( V \ast ( \QxyN + 2\Res_{N} \, + \! :\! \Psi_{N}^2 \!: ) \big) \pge \Psi_{N} \Big) \\
&\phantom{X}
-M_\g(\QxyN + 2\Res_N \, + \! :\! \Psi_N^2 \!:) (X_N+Y_N) ,\\
\Res_{N}
&=  \wt \If_{\pl}^{(1), N}
 \big(V \ast (\QxyN + 2\Res_{N} + :\! \Psi_{N}^2 \!:\,)\big)\pe \Psi_{N}\\
& \hphantom{X}
 + \wt \If_{\pl, \pe}^{N}
 \big(V \ast (\QxyN + 2\Res_{N} + :\! \Psi_{N}^2 \!:\,)\big) \\
& \hphantom{X}
- \I\big( M_\g(\QxyN + 2\Res_N \, + \! :\! \Psi_N^2 \!:) \Psi_N \big) \pe \Psi_N, \\
(X_{N}, \dt X_{N}, Y_{N} ,   \dt Y_{N} , &  \Res_{N})|_{t = 0}  = (0, 0, 0, 0, 0), 
\end{split}
\label{Ba4bb}
\end{align}

\noi
where $M_\g$ is as in \eqref{focusnon}, 
$\QxyN $ is as in 
\eqref{Pxy} with $\Psi$ replaced by 
$\Psi_N = \Psi_N (\vec u_0, \o_2)$
(in particular $\QxyN $ satisfies the bound \eqref{M1-3} 
with $X$, $Y$, and $\Psi$
replaced by  $X_N$, $Y_N$, and $\Psi_N$, 
uniformly in $N \in \N$), 
and 
$\wt \If_{\pl}^{(1), N}$ and $\wt \If_{\pl, \pe}^{N}$
are defined as above.
Then, by repeating 
 the proof of Theorem \ref{THM:2}
(see \eqref{M1a} - \eqref{M5} and \eqref{M7} - \eqref{M9})
with  the uniform boundedness of $\pi_N$, 
$(\vec u_0, \o_2) \in B^{i,j}_N (D)$ (see \eqref{Ba2}) and~\eqref{Ba4}, 
we have
\begin{align}
\begin{split}
\|(X_N, &  Y_N,  \Res_N)\|_{Z^{s_1, s_2, s_3}(\tau)} \\
&\les
\tau^{\frac 23} K
\Big( \|(X_N, Y_N, \Res_N)\|_{Z^{s_1, s_2, s_3}(\tau)}^3 + K^3\Big) \\
&\quad
+\tau^{\frac {4-\g}3}
\Big( \|(X_N, Y_N, \Res_N)\|_{Z^{s_1, s_2, s_3}(\tau)} +K \Big)^{2\g-1} \\
&\les \big( D (i+j) \big)^{1-\frac 23 \ta }
\Big( \|(X_N, Y_N, \Res_N)\|_{Z^{s_1, s_2, s_3}(\tau)}^3 
+ \big( D (i+j) \big)^3 \Big) \\
&\quad
+\big( D (i+j) \big)^{-\frac{4-\g}3 \theta}
\Big( \|(X_N, Y_N, \Res_N)\|_{Z^{s_1, s_2, s_3}(\tau)}
+ \big( D (i+j)\big) \Big)^{2\g-1},
\end{split}
\label{Ba4c}
\end{align}

\noi
where
$K = \| \Xi (\Psi_N (\vec u_0, \o_2))  \|_{\mathcal{Y}^{ \eps}_1} + 1$.
Then, by taking  sufficiently large $\ta \gg 1$
and  $D \gg 1$ (independent of $i, j, N \in \N$), 
a standard 
 continuity argument with \eqref{Ba4c} yields
\begin{align}
\begin{split}
\|(X_N, Y_N, \Res_N)\|_{Z^{s_1, s_2, s_3}(\tau)}
\le 1.
\end{split}
\label{Ba4c1}
\end{align}

\noi
Then, \eqref{Ba4a} follows
from  the decomposition \eqref{Ba4b} with the bounds 
\eqref{Ba2}
and 
\eqref{Ba4c1}.

Next, we set
\begin{align}
\Si^{i,j}_N = \bigcap_{\l=0}^{[2^j/\tau]} 
\big(\ft \Phi^N (\l \tau)\big)^{-1} \big( B^{i,j}_N(D) \big),
\label{Ba5}
\end{align}

\noi
where $\ft \Phi^N(t)$ is the extended stochastic flow map
in \eqref{Phi9}.
Then, from the invariance of $\rhoo_N \otimes\PP_2$ 
under  $\ft \Phi^N(t)$ (from the proof of Lemma  \ref{LEM:GWP4}), \eqref{Ba3}, and \eqref{Ba4},
we have
\begin{align}
\begin{split}
\rhoo_N \otimes & \, \PP_2 \big( (\H^{-\frac 12-\eps} (\T^3) \times \O) \setminus \Si^{i,j}_N \big) \\
&\le C  \frac{2^j}{\tau}\cdot 
\rhoo_N \otimes \PP_2 \big( (\H^{-\frac 12-\eps} (\T^3) \times \O) \setminus B^{i,j}_N(D) \big) \\
&\le
C 2^j D^\ta (i+j)^\ta  \exp \big( -c D (i+j) \big) \\
&\le
2^{-(i+j)},
\end{split}
\label{Ba6}
\end{align}
uniformly in $i, j, N \in \N$, 
provided that $D \gg 1$.
Moreover, from \eqref{Ba5} and \eqref{Ba4a}
with the flow property \eqref{Ba0}, we have 
\begin{align}
\| \Phi^N(t) (\vec u_0, \o_2) \| _{\H^{-\frac 12-\eps}}
\le D (i+j+1)
\label{Ba6a}
\end{align}

\noi
for $(\vec u_0, \o_2) \in \Si^{i,j}_N$ and $0 \le t \le 2^j$.

Finally, we set
\begin{align}
\Si^i_N = \bigcap_{j=1}^\infty \Si^{i,j}_N.
\label{Ba10}
\end{align}

\noi
Then, \eqref{Ba1aa}
follows from \eqref{Ba6}.
The growth bound
\eqref{Ba1ab} follows from 
\eqref{Ba6a}.
\end{proof}

\begin{remark}\label{REM:unique}\rm
(i) In the proof of Proposition \ref{PROP:GWPNB}, 
we used two different decompositions
\eqref{O1a}
 and~\eqref{Ba4b}
for the solution $u_N$ to the truncated equation \eqref{SNLWx}.
The former was used to obtain~\eqref{Ba6}, 
while the latter was used to obtain \eqref{Ba6a}.
The unconditional uniqueness statement 
for $u_N$ in Remark \ref{REM:tau}
was essential to conclude that 
these solutions given by the two different decompositions coincide.

\smallskip
\noi
(ii)
Note that 
the power in the growth bound \eqref{Ba1ab}
comes from the fact that 
the enhanced data set 
$\Xi (\Psi_N (\vec u_0, \o_2))$ in \eqref{data3x} 
belongs to $\H_{\le 2}$
in  the focusing case.
In the defocusing case, 
the associated enhanced data set 
belongs to $\H_{\le 3}$
and thus we need to replace the right-hand side of~\eqref{Ba1ab}
by 
$ C ( i+ \log (1+t))^{\frac 32}$.
\end{remark}

We conclude this subsection by stating a corollary to Proposition \ref{PROP:GWPNB}.

\begin{corollary}\label{COR:Leo2}
Given $i \in \N$ and $N \in \N$, 
let $\Si_N^i$ be as in Proposition \ref{PROP:GWPNB}.
Fix $T \gg 1$
and let $j$ be the smallest integer such  that $2^j \geq T$
and  $\tau > 0$ be as in \eqref{Ba4}.
Then, there exists $C(i, T) > 0$ such that 
\begin{align}
& \sup_{(\vec u_0, \o_2) \in \Si_N^i}
\| V*:\! (\pi_N u_N)^2\!:\|_{\l^\infty_kL^3([k\tau, (k+1)\tau];  
H^{-s_1 + s_2 +\frac{1}{2}} + W^{\be - 1-\eps, \infty})}
\le C(i, T), 
\label{XL1}\\
& \sup_{(\vec u_0, \o_2) \in \Si_N^i}
\|  :\! (\pi_N u_N)^2\!: \|_{\l^\infty_k 
L^3([k\tau, (k+1)\tau]; H^{-100})}
\le C(i, T), 
\label{XL2}
\end{align}

\noi
uniformly in $N \in \N$, 
where 
 $V$ is the Bessel potential of order $\b\geq 2$
 and $u_N
 = \Phi_1^N (t) (\vec u_0, \o_2)$  is as in \eqref{Ba4b}, 
denoting the global-in-time solution to  the truncated Hartree SdNLW  \eqref{SNLWx}.

\end{corollary}

\begin{proof}
We only prove the first bound \eqref{XL1}, since the second bound
\eqref{XL2} follows from \eqref{XL1} and Sobolev's inequality.
Given 
$(\vec u_0, \o_2) \in \Si_N^i$, 
it follows from 
\eqref{Ba10} and \eqref{Ba5} that 
\begin{align*}
\ft \Phi^N (k \tau)(\vec u_0, \o_2) 
= \big( \Phi^N(k \tau )(\vec u_0, \o_2), \tau_{k\tau } (\o_2) \big)
\in  B^{i,j}_N(D)
\end{align*}

\noi
for any $k = 0, 1, \dots, \big[\frac{2^j}\tau\big]$, 
where $\ft \Phi^N $ is as in  \eqref{Phi9}.

Now, consider the 
truncated dynamics
\eqref{SNLWx} on 
the time interval $[k\tau, (k+1) \tau]$
with $(u_N, \dt u_N)|_{t = k\tau} = 
\Phi^N(k \tau )(\vec u_0, \o_2)$
and the noise parameter $\tau_{k\tau } (\o_2)$.
Let $\tf = t - k \tau$ denote the shifted time.
Then, it suffices to study the system \eqref{Ba4bb} 
for $0 \le \tf \le \tau$, 
where the enhanced data set 
$ \Xi (\Psi_N (\vec u_0, \o_2))$ in \eqref{data3x}
is based on 
$\Psi_N$  given by 
\begin{align}
\begin{split}
\Psi_N  (\tf )& = \Psi_N\big(\tf; \ft \Phi^N (k \tau)(\vec u_0, \o_2)\big)\\
&= \pi_N S(\tf) \big(\Phi^N(k \tau )(\vec u_0, \o_2)\big)
+ \sqrt 2 \pi_N  \int_0^\tf \D(\tf-t') dW(t', \tau_{k \tau } (\o_2))
\end{split}
\label{XL3}
\end{align}

\noi
and
$\QxyN $ is as in 
\eqref{Pxy} with $\Psi$ replaced by 
$\Psi_N$ in \eqref{XL3}.
For clarity, let us denote the solution to \eqref{Ba4bb} 
on $[k\tau, (k+1)\tau]$
by $(X_N^{(k)}, Y_N^{(k)}, \Res_N^{(k)})$.
Arguing
as in the proof of  Proposition~\ref{PROP:GWPNB}, 
we obtain (by expressing in terms of the original time $t = \tf + k\tau$)
\begin{align}
\begin{split}
\|(X_N^{(k)}, Y_N^{(k)},  \Res_N^{(k)})\|_{Z^{s_1, s_2, s_3}([k \tau, (k+1)\tau])}
\le 1, 
\end{split}
\label{XL4}
\end{align}

\noi
where, with a slight abuse of notation, 
we used $Z^{s_1, s_2, s_3}(I)$
to denote 
the $Z^{s_1, s_2, s_3}$-norm restricted to a given time interval $I$.
Then, 
by writing 
\begin{align*}
V*:\! (\pi_N u_N)^2\!:
\, 
= V*( Q_{X_N^{(k)}, Y_N^{(k)}} + 2 \Res_N^{(k)} +  :\!\Psi_N^2\!: ), 
\end{align*}

\noi
we obtain from \eqref{XL4} and \eqref{M1-3} with \eqref{Ba2} 
and the regularity ranges $\frac 14<s_1 < \frac 12 < s_2<1$ and $-\frac 12 <s_3<0$
from Theorem \ref{THM:2}
that 
\begin{align*}
\| V*:\!(\pi_N u_N)^2\!:\|_{L^3([k \tau, (k+1)\tau]; 
H^{-s_1 + s_2 +\frac{1}{2}} + W^{\be - 1-\eps, \infty})}
\le C(i, j), 
\end{align*}

\noi
uniformly in $N\in \N$, 
$(\vec u_0, \o_2) \in \Si_N^i$, and
$k = 0, 1, \dots, \big[\frac{2^j}\tau\big]$.
This proves \eqref{XL1}.
\end{proof}

\subsection{Proof of Theorem \ref{THM:GWP}}
\label{SUBSEC:GWP2}

In this subsection, by an approximation argument, 
we  first prove almost sure global well-posedness
of the focusing Hartree SdNLW \eqref{SNLWA2}.
Given $i \in \N$, define
a set $\Si^i$ by 
\begin{align}
\Si^i = \limsup_{N \to \infty} \Si^i_N
= \bigcap_{N = 1}^\infty \bigcup_{M = N}^\infty \Si^i_M. 
\label{Bc1}
\end{align}

\noi
Then, from 
\eqref{Bc1},   Theorem \ref{THM:Gibbs2}, 
and \eqref{Ba1aa}, we have 
\begin{align}
\begin{split}
\rhoo \otimes \PP_2 (\Si^i)
& = \lim_{N\to \infty}  \rhoo \otimes \PP_2 \bigg(\bigcup_{M = N}^\infty \Si^i_M\bigg)\\
& \ge \limsup_{N \to \infty} \, \rhoo \otimes \PP_2 (\Si^i_N )
= \limsup_{N \to \infty} \, \rhoo_N \otimes \PP_2 (\Si^i_N ) \\
&\ge 1- 2^{-i}.
\end{split}
\notag
\end{align}

\noi
Hence, by setting
\begin{align}
\Si = \bigcup_{i=1}^\infty \Si^i,
\label{Bc3}
\end{align}
we obtain
\[
\rhoo \otimes \PP_2 (\Si)=1.
\]

\noi
In view of 
Lemma \ref{LEM:GWP4}, 
without loss of generality, 
we assume that given any $(\vec u_0, \o_2) \in \Si$, 
there exists the global-in-time solution $(u_N, \dt u_N)$ to \eqref{SNLWx}
with $(u_N, \dt u_N)|_{t = 0} = \vec u_0$
and the noise parameter $\o_2$
(i.e.~with the external forcing $\xi(\o_2)$).

Fix $(\vec u_0, \o_2) \in \Si$.
Then, it follows from 
\eqref{Bc3}, \eqref{Bc1}, \eqref{Ba10}, and \eqref{Ba5}
that 
there exist $i \in \N$ and an increasing sequence $\{ N_k \}_{k \in \N}$ such that
\begin{align}
(\vec u_0, \o_2) \in \Si^i_{N_k}
= \bigcap_{j=1}^\infty \bigcap_{\l=0}^{[2^j/\tau]}
\big(\ft  \Phi^{N_k}(\l \tau)\big)^{-1} (B^{i,j}_{N_k} (D))
\label{Ba4c0}
\end{align}

\noi
for any $k \in \N$,
where $\tau = \tau(i, j, D)>0$ is as in \eqref{Ba4}.

In the next proposition, 
we prove convergence of 
the solutions $\{ \Phi^{N_k}(\vec u_0, \o_2) \}_{k \in \N}$
along this particular subsequence $N_k = N_k(\vec u_0, \o_2)$.
In Corollary \ref{COR:XX}, 
we establish convergence of the entire sequence
$\{ \Phi^{N}(\vec u_0, \o_2) \}_{N \in \N}$.
See also Remark \ref{REM:conv}.

\begin{proposition}\label{PROP:GWPNC}
Let $(\vec u_0, \o_2) \in \Si$, $i \in \N$, and $\{ N_k \}_{k \in \N}$
be as above.
Then, 
$\{ \Phi^{N_k}(\vec u_0, \o_2) \}_{k \in \N}$ is a Cauchy sequence 
in $C(\R_+; \H^{-\frac 12-\eps}(\T^3))$
endowed with the
compact-open topology \textup{(}in time\textup{)}.
\end{proposition}

Before proceeding to the proof of 
Proposition~\ref{PROP:GWPNC}, 
we first establish a
growth bound
for the solution
 $(X_N, Y_N, \Res_N)$ to the truncated system~\eqref{Ba4bb} 
with  $\Psi_N = \Psi_N (\vec u_0, \o_2)$
on a given (large) time interval $[0, T]$.

\begin{lemma}\label{LEM:Leo1}
Let $(\vec u_0, \o_2) \in \Si_N^i$
for some  $i \in \N$ and $N \in \N$.
Fix  $T \gg 1$.
Suppose that we have 
\begin{align}
\sup_{k = 0, 1, \dots, j}\| \Xi (\Psi_N (\vec u_0, \o_2))  \|_{\mathcal{Y}^{ \eps}_{2^j}} 
\leq K
\label{YL1}
\end{align}

\noi
for some \textup{(}large\textup{)} $K \geq 1$, 
where 
 $j$ is the smallest integer such  that $2^j \geq T$
 and 
the enhanced data set 
$\Xi (\Psi_N (\vec u_0, \o_2))$ is as in 
\eqref{data3x}
and the $\mathcal{Y}^{ \eps}_t$-norm 
 is as in \eqref{data3a}.
Then, the global-in-time solution $(X_N, Y_N, \Res_N)$ to the truncated system~\eqref{Ba4bb} 
with $\Psi_N = \Psi_N (\vec u_0, \o_2)$ satisfies 
\begin{align}
\|(X_N, Y_N, \Res_N)\|_{Z^{s_1, s_2, s_3}(T)}
\le C(i, T, K), 
\label{YL2}
\end{align}

\noi
where the constant $ C(i, T, K)$ is  independent of 
 $(\vec u_0, \o_2) \in \Si_N^i$ and $N \in \N$.

\end{lemma}

\begin{proof}
With 
$:\! (\pi_N u_N)^2\!: \, = \QxyN + 2\Res_N \, + \! :\! \Psi_N^2 \!:$
and \eqref{sto1a} (for $\Ab_N$), 
we can write \eqref{Ba4bb} as
\begin{align}
\begin{split}
 (\dt^2 + \dt  +1 - \Dl) X_{N} & =
\pi_{N} \Big(
\big( V \ast 
:\! (\pi_N u_N)^2\!: \big) \pl \Psi_{N} \Big) 
-M_\g(:\! (\pi_N u_N)^2\!:) \Psi_N ,\\
 (\dt^2 + \dt +1  - \Dl) Y_{N}
&  = \pi_{N} \Big( \big( V \ast 
:\! (\pi_N u_N)^2\!:\big) (X_{N}+Y_{N}) \Big) \\
& \hphantom{X}
+ \pi_{N} \Big( \big( V \ast :\! (\pi_N u_N)^2\!: \big) \pge \Psi_{N} \Big) \\
& \hphantom{X}
-M_\g(:\! (\pi_N u_N)^2\!:) (X_N+Y_N) ,\\
\Res_{N}
&=  \wt \If_{\pl}^{(1), N}
 \big(V \ast :\! (\pi_N u_N)^2\!:\big)\pe \Psi_{N}
 + \wt \If_{\pl, \pe}^{N}
 \big(V \ast:\! (\pi_N u_N)^2\!:\big) \\
& \hphantom{X}
- \int_0^t  M_\g(:\! (\pi_N u_N)^2\!:)(t') \Ab_N(t, t') dt', \\
(X_{N}, \dt X_{N}, Y_{N} ,   \dt Y_{N} , &  \Res_{N})|_{t = 0}  = (0, 0, 0, 0, 0).
\end{split}
\label{YL3}
\end{align}

\noi
Let  $\tau > 0$ be as in \eqref{Ba4}.
Then, we set 
\[ L^q_{I_k} = L^q(I_k), \quad \text{where}\quad 
I_k = [k\tau, (k+1) \tau].\]

\noi
By writing \eqref{YL3}
in the Duhamel formulation, 
it follows from 
\eqref{EE1}, Lemma~\ref{LEM:para}, 
\eqref{focusnon}, and Corollary \ref{COR:Leo2} with \eqref{Ba4} and \eqref{YL1}
that 
\begin{align}
\begin{split}
\| X_N \|_{X^{s_1}(T)} 
& \les \frac{T}{\tau}
\Big(\| V \ast :\! (\pi_N u_N)^2\!: \|_{\l^\infty_k L^3_{I_k} H^{\frac{1}{2}+2\eps}_x}
+ \|  :\! (\pi_N u_N)^2\!: \|_{\l^\infty_k L^3_{I_k} H^{-100}_x}^{\g-1}\Big)\\
& \quad 
\times
\| \Psi_N \|_{L^\infty_T W^{-\frac{1}{2}-\eps, \infty}_x}\\
& \leq \frac{1}{\tau} C(i, T, K).
\end{split}
\label{YL4}
\end{align}

\noi
Proceeding as in \eqref{M5} and \eqref{M9}
with 
Lemma \ref{LEM:para}, 
Lemma \ref{LEM:sto3} (for $\wt \If_{\pl}^{(1), N}$), 
 \eqref{YL1}, and \eqref{Ba4}, we have
\begin{align}
\begin{split}
\| \Res_N \|_{L^3_{T} H^{s_3}_x} 
&\le
C(i, T)
K^2 
\Big( \| V \ast :\! (\pi_N u_N)^2\!: \|_{\l^\infty_kL^3_{I_k} H^{\frac{1}{2}+2\eps}_x}\\
& \hphantom{XXXXXXX}
+ \|  :\! (\pi_N u_N)^2\!: \|_{\l^\infty_kL^3_{I_k} H^{-100}_x}^{\g-1}\Big)\\
& \le C(i, T, K).
\end{split}
\label{YL5}
\end{align}

\noi
Given $0 < t \le T$, let $k_*(t)$ be the largest integer such that $k_*(t)\tau \leq t$.
Proceeding as in \eqref{M4} and \eqref{M8}
with 
Corollary \ref{COR:Leo2} and
\eqref{YL4},  we have 
\begin{align*}
\|  Y_N \|_{ X^{s_2}(t)} 
& \leq  \|  Y_N \|_{ X^{s_2}((k_*(t) + 1)\tau)}\\
& \le C \sum_{k = 0}^{k_*(t)}
\bigg\{
\tau^\frac{2}{3}\| V \ast :\! (\pi_N u_N)^2\!: \|_{\l^\infty_k L^3_{I_k}
 (H^{-s_1 + s_2 +\frac{1}{2}}_x + W^{\be - 1-\eps, \infty}_x)}\\
& \hphantom{XXXX}\times\Big(\|  X_N \|_{ X^{s_1}(T)} 
+ \|  Y_N \|_{ X^{s_2}((k+1)\tau)}
+ \| \Psi_N \|_{L^\infty_TW^{-\frac{1}{2}-\eps, \infty}_x}
 \Big)
\\
& \hphantom{X}
+ \tau^\frac{4-\g}{3} \|  :\! (\pi_N u_N)^2\!: \|_{\l^\infty_k L^3_{I_k} H^{-100}_x}^{\g-1}
\Big(\|  X_N \|_{ X^{s_1}(T)} 
+ \|  Y_N \|_{ X^{s_2}((k+1)\tau)}\Big)\bigg\}\\
& \le C_1(i,T, K) 
\sum_{k = 0}^{k_*(t) }
\tau^\frac{4-\g}{3} \frac 1 \tau
+  C_2(i,T) 
\sum_{k = 0}^{k_*(t) }
\tau^\frac{4-\g}{3}  \|  Y_N \|_{ X^{s_2}((k+1)\tau)}\\
& \le C_1'(i,T, K) 
\tau^\frac{4-\g}{3} \frac 1 {\tau^2}
+  C_2(i,T) 
\sum_{k = 0}^{k_*(t) }
\tau^\frac{4-\g}{3}  \|  Y_N \|_{ X^{s_2}((k+1)\tau)},
\end{align*}

\noi
where $C_2(i,T) $ does not explicitly depend on $\tau$.
Then, by choosing smaller\footnote{From the proof of Corollary \ref{COR:Leo2}, 
we see that 
 the constant $C_2(i,T)$, bounding 
\[\| V \ast :\! (\pi_N u_N)^2\!: \|_{\l^\infty_k L^3_{I_k}
 (H^{-s_1 + s_2 +\frac{1}{2}}_x + W^{\be - 1-\eps, \infty}_x)}\qquad 
 \text{and}\qquad
\|  :\! (\pi_N u_N)^2\!: \|_{\l^\infty_k L^3_{I_k} H^{-100}_x}^{\g - 1}\]
does not grow even if we choose smaller $\tau > 0$.

} $\tau = \tau(i, T)> 0$ (say, by taking 
sufficiently large $\ta \gg 1$ in \eqref{Ba4}), we obtain
\begin{align*}
   \|  Y_N \|_{ X^{s_2}((k_*(t) + 1)\tau)}
 \le C_3(i,T, K) 
+  C_2(i,T) 
\sum_{k = 0}^{k_*(t)-1 }
\tau^\frac{4-\g}{3}  \|  Y_N \|_{ X^{s_2}((k+1)\tau)}.
\end{align*}

\noi
By applying the discrete Gronwall inequality
with \eqref{Ba4}, we
then obtain
\begin{align}
\begin{split}
\|  Y_N \|_{ X^{s_2}(t)} 
& \leq \|  Y_N \|_{ X^{s_2}((k_*(t) + 1)\tau)}
\le C(i,T, K) 
\exp \bigg(\sum_{k = 0}^{k_*(t) - 1}
\tau^\frac{4-\g}{3}\bigg)\\
& \le C(i, T, K).
\end{split}
\label{YL6}
\end{align}

\noi
Putting \eqref{YL4}, \eqref{YL5}, and \eqref{YL6} together, 
we obtain 
 \eqref{YL2}.
\end{proof}

\begin{remark}\label{REM:Leo}\rm
In the proof of Lemma \ref{LEM:Leo1}, 
we crucially used the 
unconditional uniqueness of 
the solution $u_N$ to the truncated equation \eqref{SNLWx}.
More precisely,  
in  the proof of Lemma \ref{LEM:Leo1}, 
we studied the equation \eqref{YL3}
on a large time interval $[0, T]$, 
where we used the representation
\begin{align}
u_N (t) 
= \Psi_N (t;  \vec u_0, \o_2)
 + X_N(t) +  Y_N(t).
\label{YL7}
\end{align}
 
 \noi
On the other hand, we used Corollary \ref{COR:Leo2} 
for the bound on $V*\!:\! (\pi_N u_N)^2\!:$
whose proof is based on 
studying 
the system \eqref{Ba4bb}
on the subinterval $[k \tau, (k+1) \tau]$
with \eqref{XL3}.
In this case,  the enhanced data set
on each subinterval $[k \tau, (k+1) \tau]$
was constructed from 
$\Psi_N  (\tf ) = \Psi_N\big(\tf; \ft \Phi^N (k \tau)(\vec u_0, \o_2)\big)$, 
where  $\tf = t - k \tau$.
Namely,  we used the representation 
\begin{align}
u_N (t) 
= \Psi_N 
\big(t-k \tau ; \ft \Phi^N (k \tau)(\vec u_0, \o_2)\big) 
 + X^{(k)}_N(t-k \tau) +  Y^{(k)}_N(t-k \tau)
 \label{YL8}
\end{align}

\noi
for $t \in  [k \tau, (k+1) \tau]$.
The unconditional uniqueness of $u_N$
guarantees  that the two representations
\eqref{YL7} and \eqref{YL8} 
agree, thus allowing  us to use
 Corollary \ref{COR:Leo2}
 in the proof of Lemma~\ref{LEM:Leo1}.
 This in turn allowed us to 
 express the $Y_N$-equation in \eqref{YL3}
 linearly in $Y_N$,
 which was crucial in applying
 the discrete Gronwall inequality.
\end{remark}

\begin{remark}\label{REM:fail}\rm
As mentioned at the beginning of this section, 
the authors in \cite{ORTz} presented the details
of the globalization argument in the stochastic PDE setting.
However, the problem considered in \cite{ORTz}
is two-dimensional and thus is not applicable
to our three-dimensional problem.
More precisely, 
in the last part of Subsection 5.2 in \cite{ORTz}, 
the authors estimated the difference
$\Phi^N (t) \ft \Phi^M (\dl) (\vec u_0, \o_2) -
\Phi^N (t) \ft \Phi^N (\dl) (\vec u_0, \o_2)$
(written with the notation of the current paper).
In our problem, 
this leads to estimating a term of the form
\begin{align}
\begin{split}
&  \big(\Psi_N  (\vec v_0, \tau_\dl(\o_2) )
-  \Psi_N(\vec w_0, \tau_\dl(\o_2)) \big)
\pe  X_N(\vec v_0, \tau_\dl(\o_2))\\
&  = 
S(t) (\vec v_0 - \vec w_0)\pe X_N(\vec v_0, \tau_\dl(\o_2)),
\end{split}
\label{fail}
\end{align}

\noi
where 
$\vec v_0 =  \Phi^M (\dl) (\vec u_0, \o_2)$,  $\vec w_0 =  \Phi^N (\dl) (\vec u_0, \o_2)$, 
and  $X_N(\vec v_0, \tau_\dl(\o_2))$
denotes the first component of the solution to the truncated
system \eqref{Ba4bb} with the data $(\vec v_0, \tau_\dl (\o_2))= 
\ft \Phi^M (\dl) (\vec u_0, \o_2)$.
Here, $\tau_\dl$ denotes the translation operator defined in \eqref{fail0}.
Since the first factor
$S(t) (\vec v_0 - \vec w_0)$ has regularity $-\frac 12-$
and the second factor $X_N(\vec v_0, \tau_\dl(\o_2))$ has regularity
$\frac 12-$ (in the limiting sense), 
the resonant product in \eqref{fail} is not well defined
in the limit.
Note that, in the two-dimensional case studied in \cite{ORTz}, 
the first and second factors have regularities $0-$ and $1-$, respectively,
and thus there is no issue in making sense of the resonant product in \eqref{fail}.

\end{remark}

We are  now ready to prove
Proposition~\ref{PROP:GWPNC}.

\begin{proof}[Proof of 
Proposition~\ref{PROP:GWPNC}]

Fix $T\gg 1$.
We prove 
that $\{ \Phi^{N_k}(\vec u_0, \o_2) \}_{k \in \N}$ is a Cauchy sequence 
in $C([0, T]; \H^{-\frac 12-\eps}(\T^3))$.

Given   $\ld = \ld(i, T) \gg 1$ (to be determined later), 
we define 
$ \Zc^{s_1, s_2, s_3}_\ld(T)$ by 
\begin{align}
\|(X, Y, \Res)\|_{\Zc^{s_1, s_2, s_3}_\ld(T)}
= \|(e^{-\ld t}X, e^{-\ld t}Y, e^{-\ld t}\Res)\|_{Z^{s_1, s_2, s_3}(T)}.
\label{ZL1}
\end{align}

\noi
For notational simplicity,
we also set $Z_N = (X_N, Y_N, \Res_N)$
and  $\Xi_N = \Xi (\Psi_N (\vec u_0, \o_2))$, 
where $\Xi (\Psi_N (\vec u_0, \o_2))$ denotes
the enhanced data set defined in \eqref{data3x}.
We consider the truncated system~\eqref{Ba4bb}
on $[0, T]$ with $\Psi_N = \Psi_N(\vec u_0, \o_2)$
to study the difference $e^{-\ld t}(Z_{N_{k_1}} - Z_{N_{k_2}})(t)$.

Given $T \gg1$, 
let  $j = j(T)$ be the smallest integer such  that $2^j \geq T$.
 Recalling that 
$(\vec u_0, \o_2) \in \Si^i_{N_k}$,  $k \in \N$, 
it follows from \eqref{Ba2} that 
\begin{align}
\begin{split}
 \sup_{\l = 0, 1, \dots, j}\| \Xi_{N_k}\|_{\mathcal{Y}^\eps_{2^\l}}
&  \le K  =  D(i+j), \\
\sup_{\l = 0, 1, \dots, j} \| \Xi_{N_{k_1}} - \Xi_{N_{k_2}}\|_{\mathcal{Y}^\eps_{2^\l}}
& \le  N_{k_1}^{-\dl} D(i+j)
\end{split}
\label{ZL2}
\end{align}

\noi
for any $k_2 \ge k_1 \ge 1$.
 We first estimate 
$e^{-\ld t} (X_{N_{k_1}} - X_{N_{k_2}})(t)$.
Using the Duhamel formulation of \eqref{Ba4bb}, 
we have
\begin{align}
e^{-\ld t} X_{N_{k_1}}(t) - e^{-\ld t}X_{N_{k_2}}(t)
= e^{-\ld t}\1_1 (t)+ e^{-\ld t}\1_2(t) + e^{-\ld t}\1_3(t), 
\label{ZL3}
\end{align}

\noi
where (i) $\1_1$ contains 
one of the differences 
$X_{N_{k_1}} - X_{N_{k_2}}$, 
$Y_{N_{k_1}} - Y_{N_{k_2}}$, 
or $\Res_{N_{k_1}} - \Res_{N_{k_2}}$, 
(ii)~$\1_2$ contains 
the difference $\Psi_{N_{k_1}} - \Psi_{N_{k_2}}$
or 
$:\! \Psi_{N_{k_1}}^2 \!: - :\! \Psi_{N_{k_2}}^2 \!:$\,, 
and (iii) 
$\1_3$ contains the terms
with the high frequency projection $\pi_{N_{k_2}} - \pi_{N_{k_1}}$
onto the frequencies $\{N_{k_1} <  |n|\le N_{k_2}\}$, 
which  allows us to gain 
 a small negative power of $N_{k_1}$ 
 by losing a small amount of regularity.

Proceeding as in \eqref{M1a}
and \eqref{M7}
and then applying Lemma \ref{LEM:Leo1}
and \eqref{ZL2}, 
we can estimate the last two terms
on the right-hand side of~\eqref{ZL3}
as
\begin{align}
\begin{split}
e^{-\ld t}  \| \1_2(t) + \1_3(t) \|_{H^{s_1}}
&  \le C(T)
\Big( \sum_{j = 1}^2 \|Z_{N_{k_j}}\|_{Z^{s_1, s_2, s_3}(T)}^2
+ K^2\Big)^{2(\g-1)}\\
& \quad \times 
\Big( \| \Xi_{N_{k_1}} - \Xi_{N_{k_2}}\|_{\mathcal{Y}^\eps_T}
+N_{k_1}^{-\dl_0} K \Big) \\
& \le C(i, T) N_{k_1}^{-\dl_1}
\end{split}
\label{ZL4}
\end{align}

\noi
for any $0 \le t \le T$ and some small $\dl_0, \dl_1 > 0$.
In order to estimate the first term on the right-hand side of \eqref{ZL3}, 
we use the following bound:
\begin{align}
e^{-\ld t }\| e^{\ld t'} \|_{L^q_{t'}([0, t])} \les  \ld^{-\frac 1q}.
\label{ZL5}
\end{align}

\noi
Proceeding as above with \eqref{ZL1}, \eqref{ZL2}, and \eqref{ZL5}
and noting that $K = K(i, j) = K(i, T)$,  we have, 
for some finite $q \geq 1$, 
\begin{align}
\begin{split}
e^{-\ld t} \| \1_1(t) \|_{H^{s_1}}
&  \le C(T) e^{-\ld t }\| e^{\ld t'} \|_{L^q_{t'}([0, t])}  K^2 
\Big( \sum_{j = 1}^2 \|Z_{N_{k_j}}\|_{Z^{s_1, s_2, s_3}(T)}
+ K\Big)^{2\g-3}\\
& \quad \times
\|Z_{N_{k_1}} - Z_{N_{k_2}}\|_{\Zc^{s_1, s_2, s_3}_\ld(T)} \\
& \le C(i, T) \ld^{-\frac 1q}  \|Z_{N_{k_1}} - Z_{N_{k_2}}\|_{\Zc^{s_1, s_2, s_3}_\ld(T)}\\
& \le \frac{1}{10} \|Z_{N_{k_1}} - Z_{N_{k_2}}\|_{\Zc^{s_1, s_2, s_3}_\ld(T)}
\end{split}
\label{ZL6}
\end{align}

\noi
for any $0 \le t \le T$, 
where the last inequality follows from choosing
$\ld = \ld(i, T)$ sufficiently large.
Hence, from \eqref{ZL3}, \eqref{ZL4}, and \eqref{ZL6}, 
we obtain
\begin{align}
\| e^{-\ld t} (X_{N_{k_1}} - X_{N_{k_2}})\|_{X^{s_1}(T)}
\le C(i, T) N_{k_1}^{-\dl_1} + 
\frac{1}{10} \|Z_{N_{k_1}} - Z_{N_{k_2}}\|_{\Zc^{s_1, s_2, s_3}_\ld(T)}.
\label{ZL7}
\end{align}

\noi
A similar computation yields
\begin{align}
\| e^{-\ld t} (Y_{N_{k_1}} - Y_{N_{k_2}})\|_{X^{s_2}(T)}
\le C(i, T) N_{k_1}^{-\dl_1} + 
\frac{1}{10} \|Z_{N_{k_1}} - Z_{N_{k_2}}\|_{\Zc^{s_1, s_2, s_3}_\ld(T)}.
\label{ZL8}
\end{align}

It remains to estimate  $e^{-\ld t}(\Res_{N_{k_1}} - \Res_{N_{k_2}})(t)$.
Once again, using \eqref{Ba4bb}, 
we have
\begin{align}
e^{-\ld t} \Res_{N_{k_1}}(t) - e^{-\ld t}\Res_{N_{k_2}}(t)
= e^{-\ld t}\II_1 (t)+ e^{-\ld t}\II_2 (t)+ e^{-\ld t}\II_3 (t), 
\label{ZL9}
\end{align}

\noi
where (i) $\II_1$ contains 
one of the differences 
$X_{N_{k_1}} - X_{N_{k_2}}$, 
$Y_{N_{k_1}} - Y_{N_{k_2}}$, 
or $\Res_{N_{k_1}} - \Res_{N_{k_2}}$, 
(ii)~$\II_2$ contains 
the difference of one of the terms in 
the enhanced data set $\Xi_{N_{k_j}}$, $j = 1, 2$, 
and (iii)~$\II_3$ contains the terms
with the high frequency projection $\pi_{N_{k_2}} - \pi_{N_{k_1}}$
onto the frequencies $\{N_{k_1} <  |n|\le N_{k_2}\}$, 
allowing us to gain 
 a small negative power of $N_{k_1}$
 by losing a small amount of regularity.

Proceeding as in \eqref{M5}
and \eqref{M9}
and then applying Lemma \ref{LEM:Leo1}, 
we can estimate the last two  term 
on the right-hand side of~\eqref{ZL9}
as
\begin{align}
\begin{split}
  \| e^{-\ld t} \II_2 + e^{-\ld t} \II_3 \|_{L^3_TH^{s_3}_x}
& \le C(i, T) N_{k_1}^{-\dl_1}
\end{split}
\label{ZL10}
\end{align}

\noi
for any $0 \le t \le T$ and some small $\dl_1 > 0$.
As for the first term on the right-hand side of~\eqref{ZL9}, 
let us first estimate the contribution from 
the terms involving
$\wt \If_{\pl, \pe}^{N_{k_1}}$ as an example.
For $j = 1, 2$, 
let 
$Q_{X_{N_{k_j}},Y_{N_{k_j}}}$
be as in~\eqref{Pxy}
 with $X_{N_{k_j}}$ and $Y_{N_{k_j}}$ but with $\Psi_{N_{k_1}}$.
Then, a slight modification of \eqref{M5} with~\eqref{ZL2} and \eqref{ZL5} yields
\begin{align*}
& \Big\|     e^{-\ld t}\, \wt \If_{\pl, \pe}^{N_{k_1}}
 \big(V \ast (Q_{X_{N_{k_1}},Y_{N_{k_1}}}  + 2\Res_{N_{k_1}} + :\! \Psi_{N_{k_1}}^2 \!:\,)\big) \\
& \qquad -    e^{-\ld t}\, \wt \If_{\pl, \pe}^{N_{k_1}}
 \big(V \ast (Q_{X_{N_{k_2}},Y_{N_{k_2}}}  + 2\Res_{N_{k_2}} + :\! \Psi_{N_{k_1}}^2 \!:\,)\big) 
  \Big\|_{L^3_TH^{s_3}_x}\\
& \leq  K \bigg\|e^{- \ld t} \Big\|
e^{\ld t'}\Big( e^{-\ld t'}\big( Q_{X_{N_{k_1}},Y_{N_{k_1}}}(t')  + 2\Res_{N_{k_1}}(t')\big)\\
& \qquad - e^{-\ld t'} \big(Q_{X_{N_{k_2}},Y_{N_{k_2}}} (t') + 2\Res_{N_{k_2}}(t')\big)\Big)
\Big\|_{L^\frac{3}{2}_{t'}([0, t]; H^{-\be }_x)}\bigg\|_{L^3_T}\\
& \le C(i, T) \ld^{-\frac 1q}  \|Z_{N_{k_1}} - Z_{N_{k_2}}\|_{\Zc^{s_1, s_2, s_3}_\ld(T)}\\
& \le \frac{1}{10} \|Z_{N_{k_1}} - Z_{N_{k_2}}\|_{\Zc^{s_1, s_2, s_3}_\ld(T)}.
\end{align*}

\noi
The other terms can be handled in a similar
manner and thus we obtain
\begin{align}
  \| e^{-\ld t} \II_1 \|_{L^3_TH^{s_3}_x}
\le \frac{1}{10} \|Z_{N_{k_1}} - Z_{N_{k_2}}\|_{\Zc^{s_1, s_2, s_3}_\ld(T)}.
\label{ZL11}
\end{align}

Hence, putting 
\eqref{ZL1}, \eqref{ZL7}, \eqref{ZL8}, \eqref{ZL9}, \eqref{ZL10}, and \eqref{ZL11}
together, we obtain
\begin{align}
\begin{split}
\|Z_{N_{k_1}} - Z_{N_{k_2}}\|_{Z^{s_1, s_2, s_3}(T)}
& \le e^{\ld T} \|Z_{N_{k_1}} - Z_{N_{k_2}}\|_{\Zc^{s_1, s_2, s_3}_\ld(T)}\\
& \le C(i, T) e^{\ld T}  N_{k_1}^{-\dl_1} \too 0, 
\end{split}
\label{ZL12}
\end{align}

\noi
as $k_2 \ge k_1 \to \infty$.
Therefore, we conclude from 
\eqref{ZL2} and \eqref{ZL12}
that 
$u_{N_k} = \Psi_{N_k} + X_{N_k} + Y_{N_k}$, $k \in \N$, 
is a Cauchy sequence
in $C([0, T]; \H^{-\frac 12-\eps}(\T^3))$.
This completes
the proof of 
Proposition~\ref{PROP:GWPNC}.
\end{proof}

\begin{remark}\label{REM:conv} \rm 

In Proposition \ref{PROP:GWPNC}, 
we proved that, given any $(\vec u_0, \o_2) \in \Si$,
a subsequence 
 $\{ \Phi^{N_k}(\vec u_0, \o_2) \}_{k \in \N}$
converges 
to some limit 
in $C(\R_+; \H^{-\frac 12-\eps}(\T^3))$.
In fact, 
a slight modification of  the proof of 
Proposition~\ref{PROP:GWPNC}
shows that 
 the solution 
$ (X_{N_k}, Y_{N_k}, \Res_{N_k} )$ to~\eqref{Ba4bb} 
on $[0, T]$ with $N = N_k$, emanating from $(\vec u_0, \o_2)$, 
converges, 
in 
$Z^{s_1, s_2, s_3}(T)$, 
to a limit $(X, Y, \Res)$, satisfying
the focusing Hartree SdNLW system~\eqref{SNLW6}
on $[0, T]$ 
with the zero initial data
and 
 the enhanced data set $\Xi(\Psi)$ in \eqref{data33}
 given as the limit of the enhanced data set\footnote{This is 
 nothing but  the enhanced data set constructed from 
 the limiting stochastic convolution $\Psi (\vec u_0, \o_2)$.}
$\Xi(\Psi_{N_k}(\vec u_0, \o_2))$ in \eqref{data3x},  which is guaranteed to exist
thanks to 
\eqref{Ba4c0} and 
the difference estimate  assumption
in~\eqref{Ba2}.

\end{remark}

Let $\Si$ be the set of full 
$\rhoo \otimes \PP_2$-probability 
defined in \eqref{Bc3}.
Proposition \ref{PROP:GWPNC} shows that
 given any $(\vec u_0, \o_2) \in \Si$,
there exists a subsequence $N_k = N_k(\vec u_0, \o_2) \in \N$
such that 
 $\{ \Phi^{N_k}(\vec u_0, \o_2) \}_{k \in \N}$
converges 
to some limit.
We now show that the entire sequence 
 $\{ \Phi^{N}(\vec u_0, \o_2) \}_{N \in \N}$
 converges (to a unique limit, 
 which we can denote by $(u, \dt u)$ without ambiguity).

\begin{corollary}\label{COR:XX}
Let $(\vec u_0, \o_2) \in \Si$.
Then, the entire sequence 
$\{ \Phi^{N}(\vec u_0, \o_2) \}_{N \in \N}$ 
converges to some limit $(u, \dt u)$
in $C(\R_+; \H^{-\frac 12-\eps}(\T^3))$
endowed with the
compact-open topology \textup{(}in time\textup{)}.

\end{corollary}

\begin{proof}
We use the same notations as in the proof of Proposition \ref{PROP:GWPNC}.
As discussed before, 
given $(\vec u_0, \o_2) \in \Si$, 
there exist $i \in \N$ and an increasing sequence $\{ N_k \}_{k \in \N}$ such that
$(\vec u_0, \o_2) \in \Si^i_{N_k}$ defined in 
\eqref{Ba4c0}.
Denote by 
$\Phi (\vec u_0, \o_2)$  the limit of  $\Phi^{N_k} (\vec u_0,\o_2)$ as $k \to \infty$, 
constructed in Proposition \ref{PROP:GWPNC}.
Fix $T>0$.
By writing
\begin{align}
\begin{split}
\Phi^N (t) (\vec u_0,\o_2) -\Phi (t) (\vec u_0,\o_2) 
&= \big( \Phi^N(t) (\vec u_0,\o_2) - \Phi^{N_k} (t) (\vec u_0,\o_2) \big) \\
&\quad+ \big( \Phi^{N_k}(t) (\vec u_0,\o_2) - \Phi (t) (\vec u_0,\o_2) \big), 
\end{split}
\label{Bd3}
\end{align}
we see that the second term tends to $0$
 in $C([0,T];\H^{-\frac 12-\eps}(\T^3))$, 
 as $k \to \infty$.

From
\eqref{Ba4c0} and \eqref{Ba2},
we have
\begin{align}
\begin{split}
&
\sup_{m = 0, 1, \dots, j} \| \Xi (\Psi_N (\ft \Phi^{N_k} (\l \tau) (\vec u_0, \o_2))) \\
& \hphantom{XXXX}
- \Xi (\Psi_{N_k} ( \ft \Phi^{N_k} (\l \tau) (\vec u_0, \o_2)))  \|_{\mathcal{Y}^{ \eps}_{2^m}}
\le N^{-\dl} D(i+j), \\
&
\sup_{m = 0, 1, \dots, j} \| \Xi (\Psi_N ( \ft \Phi^{N_k} (\l \tau) (\vec u_0, \o_2)))  \|_{\mathcal{Y}^{ \eps}_{2^m}}
\le 2 D(i+j)
\end{split}
\label{Bd4}
\end{align}
for any 
$1\le N \le N_k$, $j \in \N$,  and 
$\l =0, \dots, [\frac{2^j}\tau]$, where $\tau $ is as in \eqref{Ba4}.
The, given any $T \gg 1$, using the second bound in \eqref{Bd4} with $\l = 0$, 
we can repeat the proofs of  
Corollary~\ref{COR:Leo2}
and 
Lemma \ref{LEM:Leo1}
so that Lemma \ref{LEM:Leo1} holds for 
the global-in-time solution $(X_N, Y_N, \Res_N)$ to the truncated system~\eqref{Ba4bb} 
with $\Psi_N = \Psi_N (\vec u_0, \o_2)$.
Then, we can estimate the first term on the right-hand side of \eqref{Bd3}
by repeating the computation in  the proof of 
and Proposition \ref{PROP:GWPNC} with $(N_{k_1}, N_{k_2})$ replaced by $(N, N_k)$.
\end{proof}

Finally, we show invariance of the focusing Hartree Gibbs measure $\rhoo$ in \eqref{Gibbs9}
for the limiting process $\vec u = (u, \dt u)$.
Fix  $F \in C_b (\H^{-\frac 12-\eps} (\T^3); \R)$ and $t > 0$.
It follows from \eqref{O1}, 
the bounded convergence theorem with Corollary \ref{COR:XX}, 
the strong convergence of $\rhoo_N$ to $\rhoo$ (Theorem \ref{THM:Gibbs2}), 
and invariance of $\rhoo_N$ (Proposition \ref{LEM:GWP4})
that
\begin{align*}
\int \E_{\o_2} \big[ F(\Phi (t) (\vec u_0^{\o_1},  \o_2)) \big] d \rhoo (\vec u_0^{\o_1})
&= \lim_{N \to \infty} \int \E_{\o_2}
 \big[ F(\Phi^{N} (t) (\vec u_0^{\o_1}, \o_2)) \big] d \rhoo (\vec u_0^{\o_1}) \\
&= \lim_{N \to \infty} \int \E_{\o_2}
 \big[ F(\Phi^{N} (t) (\vec u_0^{\o_1}, \o_2)) \big] d \rhoo_{N} (\vec u_0^{\o_1}) \\
&= \lim_{N \to \infty} \int F(\vec u_0^{\o_1}) d \rhoo_{N} (\vec u_0^{\o_1}) \\
&= \int F(\vec u_0^{\o_1}) d \rhoo (\vec u_0^{\o_1}).
\end{align*}

\noi
This shows  invariance of $\rhoo$.
This concludes the proof of Theorem \ref{THM:GWP}.

\appendix

\section{On the parabolic stochastic quantization
of the focusing Hartree Gibbs measure}
\label{SEC:B}

In this section,
we consider the parabolic stochastic quantization of the focusing Hartree Gibbs measure
$\rho$ constructed in Theorem \ref{THM:Gibbs2}, 
associated with the energy functional $E^\sharp (u)$ in~\eqref{Energy0a}.
More precisely, we study 
 the following  focusing Hartree stochastic nonlinear heat equation (SNLH) on $\T^3$:
\begin{align}
\dt u + (1 -  \Dl)  u -\s  (V \ast :\! u^2 \!:)u
+ M_\g (:\! u^2 \!:) u = \sqrt{2} \xi,
\label{SNLH}
\end{align}

\noi
where $\s > 0$ and $M_\g$ is as in \eqref{focusnon}.

\begin{theorem}\label{THM:GWPH}
Let $\s > 0$.
Let $V$ be the Bessel potential of order $\be \geq 2$, 
where 
we also assume that  $\s > 0$ is sufficiently small
when $\be = 2$.
Then, the focusing   Hartree SNLH~\eqref{SNLH} 
on the three-dimensional torus $\T^3$ 
 is almost surely globally well-posed
with respect to the random initial data distributed
by the focusing Hartree Gibbs measure $\rho$ in \eqref{Gibbs9}.
Furthermore, the Gibbs measure $\rho$ is invariant under the resulting dynamics.

\end{theorem}

Here, we made a somewhat informal statement in the spirit of Theorem \ref{THM:GWP0}.
A rigorous statement needs to be given in terms 
of a limiting procedure as in 
Theorem \ref{THM:GWP}, which we omit.

As in the wave case, 
the main task is to prove local well-posedness of \eqref{SNLH}.
Once this is achieved, then the rest follows from 
 Bourgain's invariant measure argument
 whose detail we omit.
Thus, we only prove local well-posedness of \eqref{SNLH}
in the following.

\begin{remark}\rm
The defocusing/focusing nature of the problem does not play an important role
in the local well-posedness argument.
By simply setting $\s < 0$ in \eqref{SNLH} and $A = 0$ in \eqref{focusnon}, 
our argument  below proves an analogue of Theorem \ref{THM:GWPH}
in the defocusing case
for $\be > 1$.
See Proposition \ref{PROP:LWPv2} below.

In the defocusing case,  by adapting 
the well-posedness argument 
 \cite{CC, GIP, Kupi, MW1}
 for the parabolic $\Phi^4_3$-model \eqref{heat0},  
 we expect that 
 an analogue of Theorem \ref{THM:GWPH}
 can be extended to $\be > 0$.

\end{remark}

Let  $\Psi$ be
 the stochastic convolution, satisfying
\begin{align*}
\begin{cases}
\dt \Psi + (1 -  \Dl) \Psi  = \sqrt{2} \xi
\\
\Psi |_{t = 0} = \phi_0 \qquad \text{with }\Law (\phi_0) = \mu. 
\end{cases}
\end{align*}

\noi
Then, by repeating the arguments, 
 Lemmas \ref{LEM:stoconv} and \ref{LEM:IV}
for $\Psi$, 
$:\! \Psi^2 \!:$, and $(V \ast :\! \Psi^2 \!:) \pe \Psi$
extend to the parabolic setting when $\be > 1$.

We proceed with the following first order expansion:
\begin{align}
u = \Psi+ v.
\label{decomp3H}
\end{align}

\noi
Then, it follows from \eqref{SNLH} and \eqref{decomp3H} that 
the residual term $v$ 
 satisfies
\begin{align}
(\dt +1 -  \Dl)  v 
&  = \NN_1(v) + \NN_2(v), 
\label{SNLH7}
\end{align}

\noi
where 
$\NN_1(v)$ and $\NN_2(v)$ are given by 
\begin{align}
\begin{split}
\NN_1(v) &:= \s \big( V \ast (v^2 + 2 v \Psi + :\! \Psi^2 \!:) \big) (v+\Psi), 
\\
\NN_2(v) &:= -M _\g (v^2 + 2 v \Psi + :\! \Psi^2 \!:) (v+\Psi). 
\end{split}
\label{NNH1}
\end{align}

\noi
Here,  
$( V \ast  :\! \Psi^2 \!:)\Psi$ in $\NN_1(v)$
is interpreted
as 
\[ ( V \ast  :\! \Psi^2 \!:)\Psi
= 
(V \ast :\! \Psi^2 \!:) \pl \Psi 
+ (V \ast :\! \Psi^2 \!:) \pe \Psi
+ (V \ast :\! \Psi^2 \!:) \pg \Psi,   
\]

\noi
where the second term on the right-hand side is given a meaning via
stochastic analysis for $1 < \be \le \frac 32$.

Since $\Psi \sim -\frac 12-$,
we expect that $v$ has regularity $\frac 32-$.
Hence, $v \Psi$ is well defined
and thus a straightforward computation yields the following local well-posedness
of \eqref{SNLH7}.

\begin{proposition}\label{PROP:LWPv2}
Let $\b>1$,  $\s \in \R\setminus \{0\}$, $2<\g\leq 3$, and $A\in \R$.
Given $s < \frac{3}{2}$ sufficiently close to $\frac 32$, 
%
%
%
there exists  $\eps = \eps(s) > 0$
such that 
if 
\begin{itemize}
\item   $\Psi $ is a distribution-valued function belonging to $C([0, T]; \C^{-\frac 12 - \eps, \infty}(\T^3))$, 

\smallskip
\item
$:\! \Psi^2 \!:
$ is a distribution-valued function belonging to $ C([0, T]; \C^{-1 - \eps, \infty}(\T^3))$,

\smallskip
\item
$(V \ast :\! \Psi^2 \!:) \pe \Psi
$ is a distribution-valued function belonging to $ C([0, T]; \C^{\be -\frac 32 - \eps, \infty}(\T^3))$,

\smallskip

\end{itemize}

\noi
then the Hartree SNLH \eqref{SNLH7} 
  is locally well-posed in $\C^s(\T^3)$.
More precisely, 
given any $v_0 \in \C^s(\T^3)$, 
there exists $T > 0$ 
such that   a unique solution $v$
to \eqref{SNLH7}
exists
 on the time interval $[0, T]$
in the class 
$ C([0, T]; \C^s(\T^3))$.
Furthermore, the solution $v$
depends  continuously 
on the enhanced data set:
\begin{align}
\Xi = 
(v_0, \Xi(\Psi)): = \big(
v_0, \Psi,  :\! \Psi^2 \!:, (V \ast :\! \Psi^2 \!:) \pe \Psi \big)
\label{data2H}
\end{align}

\noi
in the class $\C^s(\T^3) \times \mathcal{X}^{\eps}_T$, 
where 
\begin{align}
\begin{split}
\mathcal{X}^{\eps}_T
 :=
C([0,T]; \C^{-\frac 12-\eps}(\T^3))
\times
C([0,T]; \C^{-1- \eps}(\T^3))
\times
C([0, T]; \C^{\b-\frac 32- \eps}(\T^3)).
\end{split}
\label{data3H}
\end{align}
\end{proposition}

When $\be > \frac 32$, 
the resonant product 
$(V \ast :\! \Psi^2 \!:) \pe \Psi $ makes sense in the deterministic manner
and thus we do not include this term in the enhanced data set.

Before proceeding to the proof of Proposition \ref{PROP:LWPv2},
 we first recall the Schauder estimate for the heat equation.
Let $P (t)  = e^{-t (1-\Dl)}$ denote the linear heat propagator defined 
as a Fourier multiplier operator:
\begin{align}
P(t) f = \sum_{n \in \Z^3} e^{-t \jb{n}^2} \ft f(n) e_n
\notag
\end{align}

\noi
for $t \geq 0$.
Then, we have the following Schauder estimate
on $\T^d$.

\begin{lemma} \label{LEM:Schauder}
Let $ -\infty< s_1 \leq  s_2 < \infty$.
Then, we have 
\begin{align}
\| P(t) f \|_{\C^{s_2}}
\les t^{\frac{s_1 -s_2}{2}} \| f \|_{\C^{s_1}}
\label{regheat}
\end{align}

\noi
for any $t > 0$. 
\end{lemma}

The bound \eqref{regheat} on $\T^d$ 
follows from the decay estimate for the heat kernel on $\R^d$ 
(see Lemma~2.4 in \cite{BCD}) and the Poisson summation formula
to pass such a decay estimate to $\T^d$.

\begin{proof}[Proof of Proposition \ref{PROP:LWPv2}]

Define a map $\Phi$ by 
\begin{align}
\Phi (v) (t) &= P(t) v_0
+ \int_0^t P(t - t') \big(\NN_1 (v)+  \NN_2 (v)\big) (t') dt'.
\label{PhiH}
\end{align}

\noi
Let $0 <  T \le 1$.
We assume
\begin{align}
\| \Xi (\Psi)  \|_{\mathcal{X}^{\eps}_1}
\leq K
\label{data9}
\end{align}

\noi
for some $K \geq 1$, 
where $\Xi(\Psi)$
and $\mathcal{X}^{\eps}_1$ are as in \eqref{data2H} and \eqref{data3H}.

From 
Lemma \ref{LEM:Schauder}, 
\eqref{NNH1}, 
\eqref{alge}, 
and Lemma \ref{LEM:para} with \eqref{data9}, we have
\begin{align}
\begin{split}
\bigg\|\int_0^tP(t - t') \NN_1(v)(t') dt' \bigg\|_{L^\infty_T \C^s_x}
& \les T^\ta
\bigg(  \| v^2 +2 v\Psi\|_{L^\infty_T \C^{\frac 12 - \be +2 \eps}}
\|v + \Psi\|_{L^\infty_T \C^{-\frac 12 - \eps}}\\
& \hphantom{XXXX}
+ \|V \ast  :\! \Psi^2 \!:\|_{L^\infty_T \C^{ \eps}}
\| \Psi\|_{L^\infty_T \C^{-\frac 12 - \eps}}\\
& \hphantom{XXXX}
+ \|(V \ast :\! \Psi^2 \!:) \pe \Psi\|_{L^\infty_T \C^{\be -\frac 32 - \eps}}\\
& \hphantom{XXXX}
+ \|(V \ast :\! \Psi^2 \!:) v\|_{L^\infty_T \C^{\eps}}\bigg) \\
& \les T^\ta \Big(\| v\|_{L^\infty_T \C^s_x}^3 + K^3\Big)
\end{split}
\label{He1}
\end{align}
	
\noi	
for $\be > 1$
and 
$\frac 12 + 2\eps \le s < \frac 3 2 - \eps$.
Similarly, 
we have 
\begin{align}
\begin{split}
\bigg\|\int_0^t P(t - t')\NN_2(v)(t') dt' \bigg\|_{L^\infty_T \C^s_x}
& \les T^\ta
\|M_\g( v^2 + 2v\Psi +   :\! \Psi^2 \!:)\|_{L^\infty_T}
\|v + \Psi\|_{L^\infty_T \C^{-\frac 12 - \eps}}\\
& \les T^\ta
\| v^2 + 2v\Psi  +  :\! \Psi^2 \!:\|_{L^\infty_T\C_x^{-100}}^{\g-1}
\|v + \Psi\|_{L^\infty_T \C^{-\frac 12 - \eps}}\\
& \les T^\ta \Big(\| v\|_{L^\infty_T \C^s_x}^5 + K^5\Big)
\end{split}
\label{He2}
\end{align}
	
\noi
since $\g \leq 3$.
Hence, 
from  \eqref{PhiH}, \eqref{He1}, and \eqref{He2}, 
we have
\begin{align}
\begin{split}
\| \Phi (v) \|_{L^\infty_T \C^s_x}
&\les \| v_0\|_{\C^{s}} +T^\theta \Big( \| v \|_{L^\infty_T \C^s_x}^5+K^5\Big).
\end{split}
\label{LWP1H}
\end{align}

\noi
Moreover, 
since $\g > 2$, 
$\NN_2(v)$ in \eqref{NNH1} is Lipschitz continuous with respect to $v$
and thus a similar computation also yields a  difference estimate:
\begin{align}
\| \Phi (v_1) - \Phi (v_2) \|_{L^\infty_T \C^s_x}
&\les T^\ta \Big( \| v_1 \|_{L^\infty_T \C^s_x} + \| v_2 \|_{L^\infty_T \C^s_x} 
+ K \Big)^4 \| v_1-v_2 \|_{L^\infty_T \C^s_x}.
\label{LWP2H}
\end{align}

\noi
Therefore, 
local well-posedness of \eqref{SNLH7}
follows from a contraction argument with 
 \eqref{LWP1H} and~\eqref{LWP2H}.
An analogous  computation 
shows that  the solution $v$ depends continuously on 
the enhanced data set $\Xi$ in \eqref{data2H}.
\end{proof}

\section{On the regularities of the stochastic terms}
\label{SEC:C}

In the following, we study the regularities of the stochastic terms, 
appearing in Subsection~\ref{SUBSEC:def3}.
From \eqref{YZ12} and \eqref{YZ6a}, 
we have 
\begin{align*}
\dot \ZZ_N  
& = (1-\Delta)^{-1}[(V_0 \, \ast \!:\!Y_N^2\!:\,) Y_N]^\dia\\
& = (1-\Delta)^{-1}\Big((V_0 \, \ast\! :\! Y_N^2 \!:\,)  Y_N  -2  K_N* Y_N\Big).
\end{align*}

\noi
In view of 
\eqref{kappa1} and \eqref{kappa2}, 
we see that the subtraction of 
\[ 2K_N* Y_N = 
2 \ft Y_N (n,t) \sum_{\substack{n_1 \in \Z^3 \\ n_1 \ne -n\\ |n_1|\leq N}}
\ft V(n+n_1)  \jb{n_1}^{-2}, 
\]

\noi
removes  the divergent term 
in $(V_0 \, \ast \! :\! Y_N^2 \!:)\pe  Y_N $
(which corresponds to 
$Z_{13}$
defined in \eqref{Z13}).
See Remark  \ref{REM:Z}.
Then, by repeating the proof of Lemma \ref{LEM:IV}
and taking into account the smoothing by $(1-\Dl)^{-1}$, 
we have 
\begin{align}
\E \big [ |\ft {\dot \ZZ}_N(n, t)|^2\big] \sim  \jb{n}^{-2\be- 4}
\label{CZ2a}
\end{align}

\noi
for $0 < \be \le 1$ and $0 \leq t \leq 1$;
see the proof of  Lemma \ref{LEM:Ks3}.
Thus, by Minkowski's integral inequality, we have 
\begin{align}
\E \big [ |\ft {\ZZ}_N(n)|^2\big] \sim  \jb{n}^{-2\be- 4}
\label{CZ2}
\end{align}

\noi
where $\ZZ_N = \ZZ_N(1)$ is as in \eqref{K9a}.

\begin{lemma}\label{LEM:CZ}
Let $V_0$, $Y_N$, and $\ZZ_N$ be as in Section \ref{SEC:def}
and let $0 < \be \le \frac 12$.
Then, given any $\eps > 0$ and finite $p \geq 1$, we have
\begin{align}
\E\Big[ \|(V_0*:\!Y_N^2\!:) \ZZ_N^2\|_{\C^{\be - 1- \eps}}^p\Big]
\leq C_{p, \eps} < \infty, 
\label{CZ3}\\
\E \Big[\|(V_0*:\!Y_N^2\!:) \ZZ_N\|_{\C^{\be - 1- \eps}}^p\Big]
\leq C_{p, \eps} < \infty, 
\label{CZ4}\\
\E\Big[ \| [(V_0 \ast (Y_N \ZZ_N) ) Y_N \ZZ_N]^\dia
\|_{\C^{\be - 1-\eps}}^p\Big]
\leq C_{p, \eps} < \infty, 
\label{CZ5}\\
\E\Big[ \| [(V_0 \ast (Y_N \ZZ_N) ) Y_N ]^\dia
\|_{\C^{\be - 1-\eps}}^p\Big]
\leq C_{p, \eps} < \infty, 
\label{CZ6}
\end{align}

\noi
uniformly in $N \in \N$.	
Here, the third term is defined as in \eqref{YZ8a}
\textup{(}with $\Dr_N$ replaced by $\ZZ_N$\textup{)}, 
while the fourth term is defined in \eqref{YY3yy}.

\end{lemma}

\begin{proof}
By Proposition 3.6 in \cite{MWX}, 
we only compute the second moment of the Fourier coefficient of each stochastic term.
With $Q_1 = (V_0*:\!Y_N^2\!:) \ZZ_N^2$, 
we have
\begin{align*}
\E\big[ |\ft Q_1(n)|^2\big]
& = \E \bigg[ \sum_{n = n_1 + n_2 + n_3 + n_4}
\jb{n_1 + n_2}^{-\be} 
:\!\ft Y_N(n_1)\ft Y_N(n_2)\!: \ft \ZZ_N(n_3) \ft \ZZ_N(n_4)\\ 
& \quad \times
\sum_{n = m_1 + m_2 + m_3 + m_4}
\jb{m_1 + m_2}^{-\be} \cj{:\!\ft Y_N(m_1)
\ft Y_N(m_2)\!: \ft \ZZ_N(m_3) \ft \ZZ_N(m_4)}\bigg], 
\end{align*}

\noi
where we used the notation introduced in \eqref{Ks11}.
In order to compute the expectation above, 
we need to take all possible pairings between $(n_1, n_2, n_3, n_4)$
and $(m_1, m_2, m_3, m_4)$.
By Jensen's inequality, however, 
we see that it suffices to consider the case
$n_j = m_j$, $j = 1, \dots, 4$.
See the discussion on $\<31p>$ in Section 4 of \cite{MWX}.
See also Section~10 in~\cite{Hairer}. 
Hence, from \eqref{CZ2}, 
we have
\begin{align*}
\E\big[ |\ft Q_1(n)|^2\big]
& \les   \sum_{n = n_1 + n_2 + n_3 + n_4}
\frac{1}{\jb{n_1 + n_2}^{2\be} 
\jb{n_1}^2\jb{n_2}^2 \jb{n_3}^{2\be + 4}\jb{n_4}^{2\be + 4}}.
\end{align*}

\noi
By applying Lemma \ref{LEM:SUM} iteratively, we have
\begin{align*}
\E\big[ |\ft Q_1(n)|^2\big]
& \les   \sum_{n_3, n_4\in \Z}
\frac{1}{\jb{n - n_3 - n_4 }^{1+ 2\be} 
\jb{n_3}^{2\be + 4}\jb{n_4}^{2\be + 4}}
\les \jb{n}^{-3 - 2(\be - 1)}.
\end{align*}

\noi
By applying Proposition 3.6 in \cite{MWX},
we obtain \eqref{CZ3}.
The second estimate \eqref{CZ4} follows in a similar manner.

Let $Q_3 =  [(V_0 \ast (Y_N \ZZ_N) ) Y_N \ZZ_N]^\dia$.
Then, proceeding as above with Jensen's inequality and Lemma \ref{LEM:SUM}, 
we have 
\begin{align*}
\E\big[ |\ft Q_3(n)|^2\big]
& \les   \sum_{n = n_1 + n_2 + n_3 + n_4}
\frac{1}{\jb{n_1 + n_2}^{2\be} 
\jb{n_1}^2\jb{n_2}^{2\be + 4} \jb{n_3}^{2}\jb{n_4}^{2\be + 4}}\\
& \les   \sum_{n_3, n_4\in \Z}
\frac{1}{\jb{n - n_3 - n_4 }^{2+ 2\be} 
\jb{n_3}^{2}\jb{n_4}^{2\be + 4}}\\
& \les \jb{n}^{-3 - 2(\be - 1)}.
\end{align*}

\noi
Similarly, with 
 $Q_4 =  [(V_0 \ast (Y_N \ZZ_N) ) Y_N]^\dia$, 
 we have
\begin{align*}
\E\big[ |\ft Q_4(n)|^2\big]
& \les   \sum_{n = n_1 + n_2 + n_3}
\frac{1}{\jb{n_1 + n_2}^{2\be} 
\jb{n_1}^2\jb{n_2}^{2\be + 4} \jb{n_3}^{2}}\\
& \les   \sum_{n_3\in \Z}
\frac{1}{\jb{n - n_3 }^{2+ 2\be} 
\jb{n_3}^{2}}
 \les \jb{n}^{-3 - 2(\be - 1)}.
\end{align*}

\noi
Therefore, 
these estimates with Proposition 3.6 in \cite{MWX} 
yield  \eqref{CZ5} and \eqref{CZ6}.
\end{proof}

\section{Absolute continuity with respect to the shifted measure}
\label{SEC:D}

In this section,
we prove that the defocusing Hartree Gibbs measure $\rho$  for $0<\be \le \frac 12$
is absolutely continuous
with respect to the shifted measure $\Law (Y(1) - \ZZ(1) + \W(1))$,
where $Y$ is as in \eqref{P2}, $\ZZ$ is defined as the limit of 
the antiderivative of 
$\dot \ZZ^N$ in \eqref{YZ12}, and the auxiliary process $\W$ is defined by 
\begin{align}
\W (t) = (1-\Delta)^{-1} \int_0^t \jb{\nabla}^{-\frac 12 - \eps} 
\big(\jb{\nabla}^{-\frac 12 - \eps} Y(t')\big)^{19} dt'
\label{AC0}
\end{align}

\noi
for some small $\eps > 0$.
For the proof,
we construct a drift as in the discussion in Section 3 of~\cite{BG2}.
Note that the coercive term $\W$ is introduced  to 
guarantee  global existence of a drift on the time interval $[0, 1]$.
See Lemma \ref{LEM:globald} below.

First, we present the following general lemma, 
giving a criterion for absolute continuity.

\begin{lemma} \label{lemma: AC1}
Let $\mu_n$ and $\rho_n$ be probability measures on a Polish space $X$.
Suppose that $\mu_n$ and $\rho_n$  converge weakly to $\mu$ and $\rho$, respectively.
Furthermore, suppose  that for every $\eps > 0$,
there exist $\delta(\eps) >0 $ and $\eta(\eps)>0 $ 
with $\delta(\eps)$, $\eta(\eps) \to 0$ as $\eps \to 0$ such that for every continuous function $F: X \to \R$ with $0 < \inf F \le F \le 1$ satisfying  
\[\mu_n(\{F \le \eps\}) \ge 1- \delta(\eps)\] 

\noi
for any $n \geq n_0 (F)$, 
we have
\begin{align}
\limsup_{n \to \infty} \int F(u) d \rho_n(u) \le \eta(\eps).
\label{J0a}
\end{align}
Then, $\rho$ is absolutely continuous with respect to $\mu$.
\end{lemma}

\begin{proof}
By the inner regularity, it suffices to show that 
for every compact set $K\subset X$ with $\mu(K) = 0$, 
we have  $\rho(K) = 0$. 
Consider the family of Lipschitz functions:
\begin{align}
\chi_m^{K,\eps_*}(u) := \max\big(\eps_*, 1-md(u,K)\big)
\label{J0}
\end{align}

\noi
for $m \in \N$ and small $\eps_*>0$,
where $d(u,K)$ denotes the distance between $u$ and $K$.
Then, we have
\begin{align}
0< \eps_* = \inf \chi_m^{K,\eps_*} \le \chi_m^{K,\eps_*} \le 1.
\label{J1}
\end{align}

It follows from \eqref{J0} that
\begin{align}
\int \chi_m^{K,\eps_*}(u) d \mu(u) \le \eps_* + \int \ind_{\{d(\cdot, K) < m^{-1}\}}(u) d \mu(u) =: \eps_* + \l_m
\label{J2}
\end{align}

\noi
and that $\l_m\to 0$ as $m \to \infty$.
Given  $\eps > 0$,  let $m = m(\eps) \in \N$ and $\eps_*= \eps_*(\eps)>0$ be such that 
$\frac{2(\eps_* + \l_{m})}{\eps} < \delta(\eps)$. 
Let $S^{K,\eps} := \{ \chi_{m}^{K,\eps_*} > \eps\}$. 
By Markov's inequality, the weak convergence of $\mu_n$ to $\mu$,  and \eqref{J2},
we have 
\begin{align}
\mu_n(S^{K,\eps}) \le\frac 1 \eps \int \chi_{m}^{K,\eps_*}(u) d \mu_n(u) \le \frac{2(\eps_* + \l_{m})}{\eps} < \delta(\eps)
\label{J3}
\end{align}

\noi
for any $\eps>0$ and sufficiently large $n \gg 1$.
Therefore, by our hypothesis \eqref{J0a} with \eqref{J1} and \eqref{J3},
we obtain
\begin{align}
\limsup_{n \to \infty}  \int \chi_{m}^{K,\eps_*}(u) d \rho_n(u) \le \eta(\eps)
\label{J4}
\end{align}
for $\eps>0$.
Hence,
it follows from \eqref{J0}, the weak convergence of $\rho_n$ to $\rho$, and \eqref{J4} that
\begin{align*}
\rho(K) &\le \int \chi_{m}^{K,\eps_*}(u) d \rho(u)
= \lim_{n \to \infty}  \int \chi_{m}^{K,\eps_*}(u) d \rho_n(u)
\le \eta(\eps).
\end{align*}

\noi
By taking $\eps \to 0$, 
 we conclude that $\rho(K) = 0$.
\end{proof}

By regarding
$\dot \ZZ^N$ in \eqref{YZ12} and $\W$ in \eqref{AC0}
as functions of $Y$,
we write them as 
\begin{align}
\dot \ZZ^N (Y) (t) &:= (1-\Delta)^{-1} [(V_0 \, \ast \!:\!Y_N^2(t)\!:) Y_N(t)]^\dia, \label{AC01} \\
\W (Y) (t) &:= (1-\Delta)^{-1} \int_0^t \jb{\nabla}^{-\frac 12 - \eps} 
\big(\jb{\nabla}^{-\frac 12 - \eps} Y(t')\big)^{19} dt', \notag 
\end{align}

\noi
and we set $\dot \ZZ_N (Y)  = \pi_N \dot \ZZ^N (Y)$.
Then, from 
 \eqref{AC01} and \eqref{YZ6a}, we have 
\begin{align}
\dot \ZZ_N (Y + \Theta) - \dot \ZZ_N (Y) = (1-\Delta)^{-1} \pi_N P_N(Y,\Theta),
\label{Jb0}
\end{align}

\noi
where $P_N(Y,\Theta)$ is given by
\begin{align}
\begin{split}
P_N(Y,\Theta) &:=
(V_0 \ast :\!Y_N^2\!:) \Theta_N + 2 \big( (V_0 \ast (Y_N \Theta_N))Y_N - K_N \ast \Theta_N \big)\\
&\qquad
+ (V_0 \ast \Theta_N^2)Y_N + 2 (V_0 \ast (\Theta_N Y_N)) \Theta_N + (V_0 \ast \Theta_N^2)\Theta_N.
\end{split}
\label{Jb1}
\end{align}

\noi
Here,  $K_N$ is as in \eqref{kappa2}
and $\Dr_N = \pi_N \Dr$.
We also define $\W_N(Y)(t)$ by 
\begin{align}
\W_N (Y)(t) = (1-\Delta)^{-1} \pi_N \int_0^t \jb{\nabla}^{-\frac 12 - \eps} 
\big(\jb{\nabla}^{-\frac 12 - \eps} Y_N(t')\big)^{19} dt'.
%
\label{AC04}
\end{align}

Next, we state a lemma on the construction of  a drift $\Dr$.

\begin{lemma} \label{LEM:globald}
Let $V$ be the Bessel potential of order $\be>0$.
Let $\dot \Ups \in L^2 ([0,1]; H^1(\T^3))$.
Then, given any $N \in \N$, 
the Cauchy problem for $\Dr$\textup{:}
\begin{equation} \label{ACde}
\begin{cases}
\dot \Theta - (1-\Delta)^{-1} \pi_N P_N(Y,\Theta) + \dot \W_N(Y+\Theta) - \dot \Ups = 0 \\
\Theta(0) = 0
\end{cases}
\end{equation}

\noi
is almost surely globally well-posed in $H^1(\T^3)$
on the time interval $[0, 1]$.
Moreover, if  $\|\dot \Ups\|_{L^2([0,\tau]; H^1_x)}^2 \le M$ for some $M>0$
and for some stopping time $\tau \in [0, 1]$,
then, for any $ 1 \le p <\infty$, there exists $C= C(M,p)>0$ such that 
\begin{equation}\label{AC2}
\E \Big[ \|\dot \Theta\|_{L^2([0,\tau]; H^1_x)}^p \Big]
 \le C(M,p),
\end{equation}
where $C(M,p)$ is independent of $N \in \N$.
\end{lemma}

We first
 prove the absolute continuity of 
the defocusing Hartree Gibbs measure $\rho$ with respect to $\Law (Y(1) - \ZZ(1) + \W(1))$
by assuming Lemma~\ref{LEM:globald}.
 We present the proof of 
Lemma~\ref{LEM:globald} at the end of this section.
Let $\dl(L)$ and $R(L)$ satisfy $\dl (L) \to 0$ and $R(L) \to \infty$ as $L \to \infty$, which will be specified later.
In view of Lemma \ref{lemma: AC1}, it suffices to show that 
if $F: \C^{-100}(\T^3) \to \R$ is a bounded continuous function
with $F \ge 0$ and
\begin{align}
\PP \big(\{F(Y(1) - \ZZ_N(1)+ \W_N(1)) \ge L \}\big) \ge 1 - \delta(L),
\label{AC4b}
\end{align}

\noi
then
we have
\begin{align}
\limsup_{N \to \infty} \int \exp(- F(u)) d \rho_N (u) \le \exp(-R (L)).
\label{AC00}
\end{align}

For simplicity, we use the same short-hand notations as in Subsection \ref{SUBSEC:def3};
for instance,
$Y=Y(1)$, $\ZZ = \ZZ(1)$, and $\W = \W (1)$.
By the Bou\'e-Dupuis formula (Lemma \ref{LEM:var3}) and \eqref{YZ13},
we have
\begin{align*}
- \log &  \bigg( \int \exp(- F(u) - R^{\dia\dia}_N(u)) d \mu (u) \bigg) \\
&= \inf_{\dot \Ups^N \in  \mathbb H_a^1}\E \bigg[ 
F(Y+ \Ups^N- \ZZ_N)
+ \wt R^{\dia\dia}_N (Y+\Ups^N-\ZZ_N) + \frac12 \int_0^1 \| \dot \Ups^N(t) \|_{H^1_x}^2dt \bigg],
\end{align*}
where $\wt R^{\dia\dia}_N$ is as in \eqref{KZ16}.
It follows from Lemmas \ref{LEM:Dr7}, \ref{LEM:Dr8}, and \ref{LEM:Dr9} 
with Lemmas \ref{LEM:Dr} and  \ref{LEM:CZ}
(see \eqref{KZ14a})
that
\begin{align}
\begin{split}
- \log &  \bigg( \int \exp(- F(u) - R^{\dia\dia}_N(u)) d \mu (u)\bigg) \\
&\quad \ge \inf_{\dot \Ups^N \in  \mathbb H_a^1}\E \bigg[ 
F(Y+ \Ups^N- \ZZ_N)
+ \frac1{20} \int_0^1 \| \dot \Ups^N(t) \|_{H^1_x}^2dt \bigg]
-C_1
\end{split}
\label{AC4a}
\end{align}

\noi
for some constant $C_1 > 0$.
For $\dot \Ups^N \in \Ha^1$,
let $\Theta^N$ be the solution to \eqref{ACde} with $\Ups$ replaced by $\Ups^N$.
For any $M>0$, define the stopping time $\tau_M$ as 
\begin{align}
\begin{split}
\tau_M
&=\min\bigg(1,  \, \min \bigg\{ \tau : \int_0^\tau \|\dot\Ups^N (s) \|_{H^1_x}^2 ds = M \bigg\}, \\
&\hphantom{XXXXX}
\min \bigg\{ \tau : \int_0^\tau \|\dot\Theta^N (s) \|_{H^1_x}^2 ds = 2C(M,2)\bigg\}\bigg),
\end{split}
\label{AC4d}
\end{align}
where $C(M,2)$ is the constant appearing in \eqref{AC2} with $p=2$.
Define Let $\Theta^N_M$ by 
\begin{align}
\Theta^N_M(t) := \Theta^N (\min(t,\tau_M)).
\label{AC4e}
\end{align}

\noi
It follows from \eqref{Jb0} and \eqref{ACde} 
with $\Ups^N(0) = \Dr_M^N (0) = \W_N(0) = 0$ that
\begin{align}
Y + \Ups^N - \ZZ_N
= Y +  \Theta^N_M - \ZZ_N(Y + \Theta^N_M) + \W_N(Y + \Theta^N_M)
\label{Ja9}
\end{align}
on the set $\{\tau_M = 1\}$.

Since $\| \dot \Theta_M^N \|_{L^2_t([0,1];H^1_x)}^2$ is bounded by $2C(M,2)$,
Girsanov's theorem yields that $\Law(Y + \Theta_M^N)$ is absolutely continuous with respect to $\Law (Y)$.
Moreover, by Cauchy-Schwarz inequality, we have 
\begin{align}
\PP\big(\{Y + \Theta_M^N \in E\}\big) \le C_M \Big(\PP\big(\{Y \in E\}\big)\Big)^\frac12
\label{AC4c}
\end{align}
for any measurable set $E$.

From \eqref{AC4a}, \eqref{Ja9}, and  the non-negativity of $F$, 
we have
\begin{align}
 - \log & \bigg( \int \exp(- F(u)- R^{\dia\dia}_N(u)) d \mu (u)\bigg) \notag \\
&\ge \inf_{\dot \Ups^N \in  \mathbb H_a^1}\E \bigg[ 
\Big( F \big( Y+  \Theta^N_M - \ZZ_N(Y + \Theta^N_M) + \W_N(Y + \Theta^N_M) \big)\notag \\
&\phantom{= \inf_{\Ups\in  \mathbb H_a}\E \bigg[ }
+ \frac1{20} \int_0^1 \| \dot \Ups^N (t) \|_{H^1_x}^2dt\Big)\ind_{\{ \tau_M = 1 \}}\phantom{]} \notag \\
&\phantom{= \inf_{\Ups\in  \mathbb H_a^1}\E \bigg[ }
+\Big( F(Y+ \Ups^N- \ZZ_N)
+ \frac1{20} \int_0^1 \| \dot \Ups^N (t) \|_{H^1_x}^2dt\Big)\ind_{\{ \tau_M < 1 \}} \bigg]
-C_1 \notag \\
&\ge \inf_{\dot \Ups^N \in  \mathbb H_a^1}\E \bigg[ 
F \big( Y+  \Theta^N_M - \ZZ_N(Y + \Theta^N_M) + \W_N(Y + \Theta^N_M) \big)
\cdot \ind_{\{ \tau_M = 1 \}}\phantom{]} \notag \\
&\phantom{= \inf_{\Ups\in  \mathbb H_a}\E \bigg[ }
+ \frac1{20} \int_0^1 \| \dot \Ups^N (t) \|_{H^1_x}^2dt \cdot \ind_{\{ \tau_M < 1 \}} \bigg] - C_1 \notag \\
\intertext{From \eqref{AC4d} followed by 
\eqref{AC4c} and 
\eqref{AC4b}, 
}
&\ge \inf_{\dot \Ups^N \in  \mathbb H_a^1}
\E \bigg[ L \cdot \ind_{\{\tau_M = 1\} 
\cap \{ F(Y+  \Theta^N_M - \ZZ_N(Y + \Theta^N_M) + \W_N(Y + \Theta^N_M)) \ge L \}} \notag \\
&\phantom{= \inf_{\Ups\in  \mathbb H_a}\E \bigg[ }
+ \frac M {20} \ind_{\{\tau_M < 1\} \cap \{\int_0^1 \| \dot \Theta^N_M (t) \|_{H^1_x}^2 dt < 2C(M,2)\}} \bigg]  - C_1 \notag \\
&\ge \inf_{\dot \Ups^N \in  \mathbb H_a^1}
\Bigg\{ L \Big( \PP(\{\tau_M = 1\}) - C_M \delta (L)^ \frac 12 \Big) \notag \\
&\hphantom{XXXXX}
+ \frac M {20} \PP\bigg(\{\tau_M < 1\} \cap \bigg\{\int_0^1 \| \dot \Theta^N_M (t) \|_{H^1_x}^2 dt < 2C(M,2)\bigg\}\bigg) \Bigg\} - C_1.
\label{ACX1}
\end{align}

In view of  \eqref{AC2} with \eqref{AC4d} and \eqref{AC4e}, 
Markov's inequality gives 
\[
\PP \bigg( \int_0^1 \| \dot \Theta^N_M (t) \|_{H^1_x}^2 dt
=  \int_0^{\tau_M} \| \dot \Theta^N_M (t) \|_{H^1_x}^2 dt
 \ge 2C(M,2) \bigg) \le  \frac 12,
\]

\noi
and thus we have
\begin{align}
\PP \bigg( \{\tau_M < 1\} \cap \bigg\{\int_0^1 \| \dot \Theta^N_M (t) \|_{H^1_x}^2dt < 2C(M,2) \bigg\} \bigg) \ge \PP(\{\tau_M < 1\})- \frac 12. 
\label{ACX2}
\end{align}

\noi
Hence, by choosing   $M = 20L$,
it follows from \eqref{ACX1} and \eqref{ACX2} that 
\begin{align*}
- \log & \bigg( \int \exp(- F(u)- R^{\dia\dia}_N(u)) d \mu (u) \bigg) \\
&\ge  \inf_{\dot \Ups^N \in  \mathbb H_a^1} \bigg\{
L \Big( \PP(\{\tau_M = 1\}) - C'_{L} \delta(L)^ \frac 12 \Big)
+ L \Big( \PP(\{\tau_M < 1\})- \frac 12 \Big) \bigg\} - C_1 \\
&= L \Big( \frac 12 - C_{L}' \delta(L)^\frac 12 \Big) - C_1.
\end{align*}

\noi
Therefore, 
by choosing $\delta(L)>0$ such that $C'_{L} \delta(L)^\frac12 \to 0$ as $L \to \infty$,
this shows \eqref{AC00} with 
\[R(L) = L \Big( \frac 12 - C'_{L} \delta(L)^\frac 12 \Big) - C_1 + \log Z, \]

\noi
where $Z = \lim_{N \to \infty} Z_N$ denotes the normalization
constant for the defocusing Hartree Gibbs measure $\rho$.

\smallskip

We conclude this section by presenting the proof of Lemma \ref{LEM:globald}.

\begin{proof}[Proof of Lemma \ref{LEM:globald}]
For simplicity, we only consider $0< \be \le \frac 12$, which is the relevant case in this section.
First, we estimate each term on the right-hand side of \eqref{Jb1}.
From Lemma~\ref{LEM:gko}, we have
\begin{align}
\begin{split}
\| (V_0 \ast :\!Y_N^2(t)\!:) \Theta_N (t)\|_{H^{-1}_x}
&\les \| V_0 \ast :\!Y_N^2(t)\!:  \|_{W^{-1,\infty}_x} \| \Theta_N (t)\|_{H^1_x} \\
&\les \| :\!Y_N^2(t)\!: \|_{W^{-1-\eps,\infty}_x} \| \Theta (t)\|_{H^1_x},
\end{split}
\label{Jb3}
\end{align}

\noi
provided that $\be\ge\eps>0$.
For the second term on the right-hand side of \eqref{Jb1},
we define $ \Y_N^t $ by replacing $Y_N =Y_N(1)$ in \eqref{YZ9} with $Y_N(t)$.
We also define $T_N^t$ by \eqref{KZ4} and \eqref{KZ5}
where we replaced  $\Y_N$ in 
 \eqref{KZ5}
 with $ \Y_N^t$.
Then, by duality we have 
\begin{align}
\begin{split}
\| (V_0 \ast (Y_N (t) & \Theta_N(t))) Y_N (t) - K_N \ast \Theta_N (t) \|_{H^{-1}_x}\\
& =\sup_{\|h \|_{H^1_x} = 1}
\bigg|\int_{\T^3\times \T^3} 
\Y_N^t (x, y) \wt \Dr_N(y, t)  h(x) dy dx
 \bigg| \\
& =\sup_{\|h \|_{ H^1_x} = 1}
\bigg|\int_{\T^3} T_N^t\big(\jb{\nb}^{1-\eps} \wt \Dr_N (t)\big)(x) 
\cdot \jb{\nb}^{1-\eps}h (x) dx \bigg|\\
& \leq \|T_N^t\|_{\L(L^2; L^2)} 
\| \Dr_N(t)\|_{H^{1-\eps}_x}
\end{split}
\label{Jb4}
\end{align}

\noi
for $\eps>0$, where $\wt \Dr_N(x, t) = \Dr_N(-x, t)$.
By Lemma \ref{LEM:gko} (i) and (ii) and Sobolev's inequality, 
we have
\begin{align}
\begin{split}
\| (V_0 \ast \Theta_N^2(t))Y_N (t) \|_{H^{-1}_x}
&\les \| V_0 \ast \Theta_N^2 (t) \|_{W^{\frac 12+\eps, \frac 3{2}}_x} 
\| Y_N(t)  \|_{W^{-\frac 12-\eps, \infty}_x} \\
&\les \| \Dr_N(t) \|_{H_x^{-\be+\frac 12 + \eps}}
\| \Dr_N (t)\|_{L^6_x} \| Y_N(t) \|_{W^{-\frac 12-\eps,\infty}_x} \\
&\les \| Y_N (t)\|_{W^{-\frac 12-\eps,\infty}_x}\| \Dr(t) \|_{H^1_x}^2 
\end{split}
\label{Jb5}
\end{align}

\noi
for $0 \le \be \le \frac 12$ and $0<\eps \ll 1$.
By Sobolev's inequality and  Lemma \ref{LEM:gko}, we have
\begin{align}
\begin{split}
\| (V_0 \ast (\Theta_N (t)Y_N(t))) \Theta_N(t) \|_{H^{-1}_x}
&\les \| (V_0 \ast (\Theta_N (t)Y_N(t))) \Theta_N (t)\|_{W^{-\frac 12, \frac{3}{2}}_x} \\
&\les \| V_0 \ast (\Theta_N (t)Y_N (t))\|_{W^{-\frac 12, \frac{12}{5}}_x} 
\| \Theta_N(t) \|_{W^{\frac 12, \frac{12}{5}}_x} \\
&\les \| \Theta_N (t) \|_{W^{\frac 12, \frac{12}{5}}_x} 
 \| Y_N (t) \|_{W^{-\frac 12-\eps, \infty}_x} 
\| \Theta_N (t) \|_{H^1_x} \\
&\les \| Y_N(t) \|_{W^{-\frac 12-\eps, \infty}_x}\| \Theta(t)  \|_{H^1_x}^2 
\end{split}
\label{Jb6}
\end{align}

\noi
for $\be \ge \eps > 0$.
Lastly, we have
\begin{align}
\begin{split}
\| (V_0 \ast \Theta_N^2(t))\Theta_N (t)\|_{H^{-1}_x}
\les \| \Theta_N (t)\|_{L^{\frac {18}5}_x}^3
\les \| \Dr(t) \|_{H^1_x}^3
\end{split}
\label{Jb7}
\end{align}
for $\be \ge 0$.

Putting
\eqref{Jb1} and \eqref{Jb3} - \eqref{Jb7} together, 
\begin{align}
\begin{split}
\| (1-\Delta)^{-1}P_N(Y(t),\Theta(t)) \|_{H^{1}_x}
&\les 
\| P_N(Y(t),\Theta(t)) \|_{H^{-1}_x} \\
&\les
\Big( \| :\!Y_N^2(t)\!: \|_{W^{-1-\eps,\infty}_x}
+ \|  T_N^t \|_{\L (L^2; L^2)}  \Big) \| \Theta (t)\|_{H^1_x} \\
& \quad+ \| Y_N(t) \|_{W^{-\frac 12-\eps, \infty}_x} \| \Dr(t) \|_{H^1_x}^2
+ \| \Dr(t) \|_{H^1_x}^3.
\end{split}
\label{Jb8}
\end{align}

\noi
Moreover,
from \eqref{AC0}, we have
\begin{align}
\begin{split}
\| \dot \W_N(Y(t)+\Theta(t)) \|_{H^1_x}
&\les
\| \jb{\nb}^{-\frac 12-\eps}Y(t) \|_{L^\infty_x}^{19}
+ \| \jb{\nb}^{-\frac 12-\eps} \Dr (t)\|_{L^\infty_x}^{19} \\
&\les 
\| Y(t) \|_{W^{-\frac 12-\eps, \infty}_x}^{19}
+ \| \Dr (t)\|_{H^1_x}^{19}
\end{split}
\label{Jb9}
\end{align}

\noi
for $\eps>0$.
Therefore,
by studying the integral formulation of 
\eqref{ACde}, 
a contraction argument 
in $L^\infty([0,T]; H^1(\T^3))$ for some $T>0$
with  \eqref{Jb8} and \eqref{Jb9}
yields local well-posedness.
Here, the local existence time $T$ depends on $\| \Theta (0) \|_{H^1_x}$, 
$\|\dot \Ups\|_{L^2_T H^1_x}$,  and 
the following terms:
\[
\| Y_N \|_{L^\infty_TW^{-\frac 12-\eps, \infty}_x}, \quad
\| :\!Y_N^2\!: \|_{L^\infty_T W^{-1-\eps,\infty}_x}, \quad
\text{and}\quad \| T_N^t \|_{L^2_T\L(L^2; L^2)}
\]

\noi
whose 
almost sure boundedness follows from a small modification 
of the proofs of 
Lemmas \ref{LEM:Dr} and~\ref{LEM:Dr9}.

Next, we prove global existence on $[0, 1]$.
It follows from \eqref{ACde} with \eqref{AC04} that 
\begin{align}
\begin{split}
\frac 12 \frac{d}{dt} \| \Theta(t) \|_{H^1}^2
&= \int_{\T^3} P_N(Y(t),\Theta(t)) \Theta_N(t) dx \\
&\quad
- \int_{\T^3} \big(\jb{\nabla}^{-\frac 12 - \eps} 
(Y_N(t)+ \Theta_N(t))\big)^{19}\jb{\nabla}^{-\frac 12 - \eps} \Theta_N(t) dx \\
&\quad
+ \int_{\T^3} \jb{\nabla} \Theta(t) \jb{\nabla} \dot \Ups(t) dx. \label{AC3}
\end{split}
\end{align}

\noi
From   \eqref{Jb1}, Lemma \ref{LEM:Dr6}, and \eqref{Jb4}, we have
\begin{align}
\begin{split}
\int_{\T^3} P_N(Y(t),\Theta(t)) \Theta_N(t) dx
& \les \| \Theta_N (t) \|_{H^1}^2 + \|\Dr_N(t)\|_{L^2}^4 \\
& \quad + \int_{\T^3} (V \ast \Theta_N^2(t)) \Theta_N^2(t) dx + C_0(Y_N(t))
\end{split}
\label{AC3a}
\end{align}

\noi
for $0<\be \le \frac 12$ and $0<\eps \ll 1$,
where
\begin{align}
C_0(Y_N(t))
: = 1+ \| Y_N (t) \|_{\C^{-\frac 12-\eps}}^c
+ \| :\!Y_N^2 (t) \!: \|_{\C^{-1-\eps}}^c
+ \| T_N^t  \|_{\L (L^2; L^2)}^c
\label{AC3a0}
\end{align}

\noi
for some $c>0$.
We now estimate the last two terms on the right-hand side of~\eqref{AC3a}.
By~\eqref{interp}, we have
\begin{align}
\begin{split}
\|\Dr_N(t)\|_{L^2}^4 &  + \int_{\T^3} (V \ast \Theta_N^2(t)) \Theta_N^2(t) dx\\
&\les \| \Theta_N(t) \|_{L^4}^4
 \les \| \Theta_N(t) \|_{H^1}^{\frac {5} {3}} \|  \Theta_N(t) \|_{W^{-\frac 12 - \eps, 20}}^{\frac {7}{3}} 
\\
&\le \| \Theta_N(t) \|_{H^1}^2 + \eps_0 \| \Theta_N(t) \|_{W^{-\frac 12 - \eps, 20}}^{20} 
+ C_{\eps_0}
\end{split}
\label{AC3b}
\end{align}
for $\be \ge 0$ and small $\eps_0 > 0$.
Moreover, it follows from \eqref{YY9} and Young's inequality that
\begin{align}
\begin{split}
\int_{\T^3}&  \big(\jb{\nabla}^{-\frac 12 - \eps} (Y_N(t)+ \Theta_N(t))\big)^{19}\jb{\nabla}^{-\frac 12 - \eps} \Theta_N(t) dx \\
&\ge
\frac 12 \int_{\T^3} (\jb{\nabla}^{-\frac 12 - \eps} \Theta_N(t))^{20} dx
- c \int_{\T^3} \big|(\jb{\nabla}^{-\frac 12 - \eps} Y_N(t))^{19}
\jb{\nabla}^{-\frac 12 - \eps} \Theta_N(t)\big| dx \\
& \ge 
\frac 12 \| \Theta_N(t) \|_{W^{-\frac 12 - \eps, 20}}^{20}
- c \| Y_N(t) \|_{W^{-\frac12-\eps, 20}}^{19} \| \Theta_N(t) \|_{W^{-\frac12-\eps, 20}}\\
& \ge
\frac 14  \| \Theta_N(t) \|_{W^{-\frac 12 - \eps, 20}}^{20}
- c \| Y_N(t) \|_{W^{-\frac12-\eps, 20}}^{20}.
\end{split}
\label{AC3c}
\end{align}

\noi
Therefore, from \eqref{AC3} - \eqref{AC3c},
we obtain
\[
 \frac{d}{dt} \| \Theta(t) \|_{H^1}^2
\les
\| \Theta(t) \|_{H^1}^2
+ \|\dot \Ups(t)\|_{H^1}^2
+ C_0 (Y_N(t)) + \| Y(t) \|_{W^{-\frac12-\eps, 20}}^{20}.
\]

\noi
By Gronwall's inequality, this implies
\begin{align}
\| \Theta(t) \|_{H^1}^2
\les
\|\dot \Ups\|_{L^2([0,t];H^1_x)}^2 
+ \| C_0 (Y_N(t)) \|_{L_t^1([0,1])}
+ \| Y \|_{L_t^{20}([0,1];W^{-\frac12-\eps, 20}_x)}^{20}, 
\label{AC3ca}
\end{align}

\noi
uniformly in $ 0 \le t \le 1$.
The a priori bound \eqref{AC3ca} allows us to iterate
the local well-posedness argument, 
guaranteeing  existence of the solution $\Dr$ on $[0, 1]$.

It follows from \eqref{AC3a0}  and  a small modification 
of the proofs of 
Lemmas \ref{LEM:Dr} and~\ref{LEM:Dr9} that 
\begin{align}
\E \Big[ \| C_0 (Y_N(t)) \|_{L_t^1([0,1])}^p \Big]
+ \E \Big[ \| Y \|_{L_t^{20}([0,1];W^{-\frac12-\eps, 20}_x)}^p \Big]
< \infty
\label{AC3X}
\end{align}
for any finite $p \ge 1$, uniformly in $N \in \N$.
Then,  from \eqref{Jb8}, \eqref{Jb9}, 
\eqref{AC3ca}, and \eqref{AC3X}, we have
\begin{align}
\|(1-\Delta)^{-1} P_N(Y,\Theta) + \dot \W_N(Y+\Theta) \|_{L^2([0,\tau];H^1_x)}
\les \| \dot \Ups \|_{L^2 ([0,\tau]; H^1_x)}^{19} + \wt C_N,
\label{AC3d}
\end{align}
with $\E \big[ |\wt C_N|^p \big] \le C_p < \infty$ for any finite $p \ge 1$,  uniformly in $N \in \N$.
Therefore, from  \eqref{ACde} and \eqref{AC3d},
we obtain the bound \eqref{AC2}.
\end{proof}

\begin{ackno}\rm
T.O.~was supported by the European Research Council (grant no.~637995 ``ProbDynDispEq''
and grant no.~864138 ``SingStochDispDyn"). 
M.O.~was supported by JSPS KAKENHI  Grant numbers JP16K17624 and JP20K14342.
L.T.~was funded by the Deutsche Forschungsgemeinschaft (DFG, German Research Foundation) under Germany's Excellence Strategy-EXC-2047/1-390685813, through the Collaborative Research Centre (CRC) 1060.
M.O. would like to thank the School of Mathematics at the University
of Edinburgh for its hospitality, where part of this manuscript was prepared. 
T.O.~would like to express gratitude to  
the  Centre de recherches math\'ematiques, Canada, 
for its hospitality, 
where the revision of this manuscript was prepared.
The authors would like to thank Bjoern Bringmann
for pointing out an error in Section \ref{SEC:GWP}
in the previous version.
The authors also would like to thank
the anonymous referees for the helpful comments.

\end{ackno}

\end{document}